%% file: barely_lonely_runners.tex
\pdfoutput=1
\documentclass[12pt]{amsart}

\usepackage[expansion=false]{microtype}
\usepackage{amsmath,amssymb,amsthm,mathtools,booktabs,enumitem,array,float}
\usepackage{pgfplots}
\pgfplotsset{compat=1.16}
\usepackage[colorlinks=true,allcolors=black]{hyperref}
\hypersetup{pdftitle={Barely lonely runners: four to seven speeds and short odd relations},pdfauthor={Hyunjun Yi, Karam Gill, Rayoon Kim, Bryan Bu}}
\newtheorem{theorem}{Theorem}[section]
\newtheorem{lemma}[theorem]{Lemma}
\newtheorem{proposition}[theorem]{Proposition}
\newtheorem{corollary}[theorem]{Corollary}
\newtheorem{conjecture}[theorem]{Conjecture}
\theoremstyle{definition}\newtheorem{definition}[theorem]{Definition}\newtheorem{question}[theorem]{Question}
\theoremstyle{remark}\newtheorem{remark}[theorem]{Remark}
\newcommand{\N}{\mathbb N}\newcommand{\Z}{\mathbb Z}\newcommand{\R}{\mathbb R}\newcommand{\T}{\mathbb T}
\newcommand{\norm}[1]{\lVert #1\rVert}
\newcommand{\ML}{\operatorname{ML}}
\newcommand{\covol}{\operatorname{covol}}
\newcommand{\supn}[1]{\lvert #1\rvert_\infty}
\newcommand{\eucl}[1]{\lvert #1\rvert}
\newcommand{\lone}[1]{\lvert #1\rvert_1}
\newcommand{\Rel}{\operatorname{Rel}}
\newcommand{\ee}{\mathrm e}
\newcommand{\Uone}{U^{1}}\newcommand{\Utwo}{U^{2}}
\newcommand{\Wone}{W^{1}}\newcommand{\Wtwo}{W^{2}}

\makeatletter
\newcommand{\supplement}[2]{\clearpage
 \setcounter{section}{0}\setcounter{theorem}{0}\setcounter{table}{0}%
 \let\sectionname\@empty
 \renewcommand{\thesection}{#1\arabic{section}}%
 \renewcommand{\thetable}{#1\arabic{table}}%
 \renewcommand{\theHsection}{supp#1.\arabic{section}}%
 \renewcommand{\theHtable}{supp#1.\arabic{table}}%
 \begin{center}\large\bfseries Supplement #1. #2\end{center}\medskip}
\makeatother
\title[Barely lonely runners]{Barely lonely runners: four to seven speeds and short odd relations}
\author[H. Yi et al.]{Hyunjun Yi, Karam Gill, Rayoon Kim, Bryan Bu}
\address[Hyunjun Yi]{\mbox{}}
\email{hjyi2027@gmail.com}
\address[Karam Gill]{Stanford Online High School, 415 Broadway, Academy Hall, Floor 2, 8853, Redwood City, CA 94063, USA}
\email{kgill12@ohs.stanford.edu}
\address[Rayoon Kim]{\mbox{}}
\address[Bryan Bu]{University of Massachusetts Amherst, 141 Commonwealth Avenue, Amherst, MA 01003, USA}
\email{bryanbu@umass.edu}
\subjclass[2020]{Primary 11J71; Secondary 11H06, 52C07, 42A05}
\keywords{Lonely Runner Conjecture, loneliness spectrum, tight instances, subtori, geometry of numbers, zonotopes, reduced bases, signed relations, Riesz products}
\begin{document}
\begin{abstract}
For distinct positive integers $v_1,\dots,v_n$ let $\mathrm{ML}(v)=\max_t\min_i\|v_it\|$, where $\|x\|$ is the distance from $x$ to the nearest integer. The Lonely Runner Conjecture states that $\mathrm{ML}(v)\ge 1/(n+1)$, and Kravitz asked which values occur below $1/n$. We classify primitive sets of four to seven speeds with $\mathrm{ML}<1/n$. For $n=4$ they are $\{1,2,3,4k\}$, $\{x,y,x+y,x+2y\}$ with $\gcd(x,y)=1$ and $y\not\equiv x,2x \pmod 4$, and $\{1,3,4,14\}$. The values are $m/(4m+2)$ with $m\ge 2$, $m\ne 3$, and Conjecture 4.1 of Fan and Sun holds. For $n=6$ they are three infinite families and nine further sets, the values are $(P-1)/(6P)$ with $P\ge 7$ prime to $6$, and a conjecture of Cordella holds. For $n=5$ we recover a theorem of Bonfioli. For $n=7$ they are $\{1,2,3,4,5,6,7s\}$ and $27$ further sets. The loneliness of every coordinate-free subtorus of codimension two in $(\mathbb{R}/\mathbb{Z})^n$ is the loneliness of three runners. Hence the accumulation points of the four-runner spectrum are exactly the three-runner values, as Giri and Kravitz conjectured, and we determine them explicitly. Beck and Everett proved that if $\mathrm{ML}(v)\le 1/(n+1)$, then $m\cdot v=0$ for some $m\in\mathbb{Z}^n$ with odd coordinate sum and $|m|_2^2\le 4(n+1)^2/(n-1)$. We improve this to $|m|_2^2<2(n+1)^2/(n-1)$, and for $n\ge 9$ we show that one can take $m\in\{-1,0,1\}^n$. The second proof uses a Riesz product corrected by cotangents, and its threshold is sharp for Fourier certificates supported on $\{-1,0,1\}^n$. It also gives a relation $\sum_{i\in A}v_i=\sum_{i\in B}v_i$ with $|A|+|B|$ odd whenever $\mathrm{ML}\le 1/n$ and $n\ge 11$. For $n\ge5$ no certificate with this support gives a relation with three terms. We conjecture that every set of $n\ge 4$ speeds with $\mathrm{ML}<1/n$ contains $a$, $b$ and $a+b$, we prove this for $n\le 7$ and check it for eight speeds up to $256$.
\end{abstract}
\maketitle

\section{Introduction}
Let $\norm{x}=\min_{k\in\Z}|x-k|$. For a finite set $V$ of positive integers (the speeds) put
\[
 \ML(V)=\max_{t\in\R}\ \min_{v\in V}\norm{vt}.
\]
The Lonely Runner Conjecture of Wills \cite{Wills} and Cusick \cite{Cusick} predicts $\ML(V)\ge 1/(|V|+1)$. It is known for $|V|\le 6$ \cite{CusickPomerance,Chen,Bienia,BHK,BarajasSerra}, Rosenfeld proved it for seven and eight speeds by computer assisted arguments \cite{Rosenfeld,Rosenfeld9}, and further computer assisted proofs for up to fourteen speeds have been announced recently \cite{Trak,SungTrak,Allikvere}. The survey \cite{PerarnauSerra} describes the history and the many equivalent formulations.

Kravitz \cite{Kravitz} proposed to study the whole \emph{spectrum} of values of $\ML$ for a fixed number $n=|V|$ of speeds. The example $V=\{1,\dots,n\}$ shows that the value $1/(n+1)$ is attained, and Kravitz asked whether the values between $1/(n+1)$ and $1/n$ form a rigid set. He conjectured that they are exactly the numbers $s/(ns+1)$, $s\in\mathbb N$, and proved this for $n\le 3$. Fan and Sun \cite{FanSun} showed that for $n=4$ the conjecture fails: $\ML(8,4s+3,4s+11,4s+19)=(2s+7)/(8s+30)$ for every $s\ge 0$. They proposed an amended conjecture, in which $s/(ns+k)$ is allowed for $1\le k\le n$, and they conjectured (their Conjecture~4.1) that for $n=4$ only $k\in\{1,2\}$ occurs. Jain and Kravitz \cite{JainKravitz} then developed a theory of \emph{relative spectra} and proved \cite[Theorem~1.3]{JainKravitz} that the set of values of $\ML$ below $1/4$ for four speeds has finite symmetric difference with
\begin{equation}\label{eq:JK}
 \Bigl\{\tfrac14-\tfrac{1}{8s+12}:s\ge 1\Bigr\}=\Bigl\{\tfrac{m}{4m+2}:m\ge 2\Bigr\}.
\end{equation}
(They state the result for $D=\frac12-\ML$ on the interval $(\frac14,\frac12]$, where the set reads $\frac14+1/\operatorname{Prog}(8,12)$ with $\operatorname{Prog}(8,12)=\{8s+12:s\ge0\}$. The term $s=0$ is $D=\frac13$, that is $\ML=\frac16$, which no quadruple attains because $\ML\ge\frac15$ for four speeds, and we omit it.) They write that determining this finite symmetric difference ``is itself a finite calculation that we have not attempted to carry out'', and that numerical experiments suggest that there are no exceptional elements. The value $\frac3{14}$, which corresponds to $s=2$, turns out to be absent.

Cordella \cite{Cordella} carried the relative spectrum analysis to five and six speeds, again up to finitely many exceptions. The reduction behind all of these results, \cite[Lemma~2.4]{Cordella}, says that all but finitely many \emph{near-tight} tuples, those with $\ML<1/n$, lie on a two-dimensional ``critical'' subtorus, and it is not effective: no bound on the exceptional tuples is produced. Cordella asks \cite[Question~6.7]{Cordella} for an effective version, or for a direct argument that bounds the speeds of a near-tight tuple lying on no critical subtorus. For four speeds he searched all quadruples with speeds at most $200$ \cite[Section~6.3]{Cordella} and found exactly one near-tight quadruple on neither critical subtorus, namely $(1,3,4,14)$ with $\ML=4/17$. The first purpose of this paper is to give the effective bound for four speeds, to carry out the Jain--Kravitz calculation completely, and to classify the speed sets. There is exactly one exceptional value and exactly one sporadic quadruple. For five and six speeds the same classification is obtained by a search tree over reduced bases of the lattice of relations, a method of Bonfioli (Theorems~\ref{thm:fivefull} and~\ref{thm:six-intro}). The second purpose is to show that the loneliness of every subtorus of codimension two is the loneliness of three runners, which settles the case of four runners of a conjecture of Giri and Kravitz on accumulation points (Theorem~\ref{thm:full-intro}), and to determine the three-runner spectrum (Theorem~\ref{thm:three-intro}). The third purpose is to show that tight and near-tight sets of speeds satisfy relations with coefficients $\pm1$ (Theorem~\ref{thm:rel}).

Call a quadruple $V=\{v_1,v_2,v_3,v_4\}$ of distinct positive integers \emph{primitive} if $\gcd(v_1,v_2,v_3,v_4)=1$. Since $\ML(cV)=\ML(V)$, nothing is lost by this normalization. For integers $x,y\ge 1$ write
\[
 V(x,y)=\{x,\ y,\ x+y,\ x+2y\}.
\]

\begin{theorem}[Classification]\label{thm:main}
Let $V$ be a primitive quadruple of distinct positive integers. Then $\ML(V)<1/4$ if and only if one of the following holds.
\begin{enumerate}[label=\textup{(\alph*)},leftmargin=*]
\item $V=\{1,2,3,4k\}$ for some $k\ge 1$, and then $\ML(V)=k/(4k+1)$.
\item $V=V(x,y)$ with $\gcd(x,y)=1$, $x\ne y$, $y\not\equiv x\pmod 4$ and $y\not\equiv 2x\pmod 4$, and then $\ML(V)=m/(4m+2)$ with $m=m(x,y)$ given by \eqref{eq:m} below.
\item $V=\{1,3,4,14\}$, and then $\ML(V)=4/17$.
\end{enumerate}
\end{theorem}

The function $m(x,y)$ is explicit. Let $r=(y-x)\bmod 4$ and $r'=(y-2x)\bmod 4$, both in $\{1,2,3\}$ under the hypotheses of (b), and set
\begin{equation}\label{eq:m}
 m_1=\begin{cases}\frac{x+y-1}{2},&r=1,\\[2pt]\frac{x+3y-2}{4},&r=2,\\[2pt]\frac{x+3y-1}{2},&r=3,\end{cases}
 \qquad
 m_3=\begin{cases}\frac{2x+y-1}{2},&r'=1,\\[2pt]\frac{2x+3y-2}{4},&r'=2,\\[2pt]\frac{2x+3y-1}{2},&r'=3,\end{cases}
\end{equation}
and $m(x,y)=\max(m_1,m_3)$. These are integers (Lemma~\ref{lem:parity}). The only overlap between the three families is $V(2,1)=\{1,2,3,4\}$.

\begin{theorem}[The spectrum below $1/4$]\label{thm:values}
The set of values $\ML(V)$ over all quadruples $V$ of distinct positive integers with $\ML(V)<1/4$ is exactly
\[
 \Bigl\{\frac{m}{4m+2}: m\ge 2,\ m\ne 3\Bigr\}
 =\Bigl\{\frac15,\frac29,\frac5{22},\frac3{13},\frac7{30},\frac4{17},\frac9{38},\dots\Bigr\}.
\]
Equivalently, apart from the term $\frac16$ excluded by the bound $\ML\ge\frac15$, the finite symmetric difference in \cite[Theorem~1.3]{JainKravitz} consists of the single value $3/14$, which is omitted.
\end{theorem}

\begin{corollary}\label{cor:consequences}
\begin{enumerate}[label=\textup{(\roman*)},leftmargin=*]
\item If $\ML(v_1,\dots,v_4)<1/4$ then $\ML=s/(4s+k)$ with $k\in\{1,2\}$, as Fan and Sun conjectured \cite[Conjecture~4.1]{FanSun}. Their amended spectrum conjecture holds for $n=4$.
\item \textup{(Chen \cite{Chen})} The tight quadruples are $\{1,2,3,4\}$ and $\{1,3,4,7\}$, up to scaling. Equivalently $\ML(V)=1/5$ for a primitive quadruple only in these two cases.
\item Every primitive quadruple other than the two tight ones satisfies $\ML(V)\ge 2/9$.
\end{enumerate}
\end{corollary}

Item (ii) is a theorem of Chen \cite{Chen}, see also Barajas and Serra \cite{BarajasSerraDM} and \cite[Section~4]{PerarnauSerra}. It appears to have been overlooked in the recent spectrum literature: Kravitz \cite[Section~13]{Kravitz} lists it as an open question, and Jain and Kravitz write in \cite{JainKravitz} that complete lists of tight instances were available only for $n\in\{1,2,3,5\}$. Cordella \cite[Section~2.5]{Cordella} gives the reference. Theorem~\ref{thm:main} recovers Chen's result with a different proof. The problem of describing all tight instances for general $n$ \cite[Problem~1]{PerarnauSerra} remains open. For six and seven speeds a computer-assisted classification was announced by Marshall \cite{Marshall}, and Corollary~\ref{cor:six}(iii) gives another proof for six speeds. See Goddyn and Wong \cite{GoddynWong} and Zhang \cite{Zhang} for structured families.

The same method goes one level up. Let $F(n_1,n_2)=\{(a,b,a+b,n_1a+n_2b)\}$ for $(n_1,n_2)\in\Z^2\setminus\{0\}$, a two-dimensional subtorus of $\T^4$ (so $F(1,2)=\Utwo$), and let $T_{(a,b,c)}\times\T$ denote the product of the circle of a triple with a free coordinate.

\begin{theorem}[Two-dimensional subtori below $1/3$]\label{thm:second-intro}
A coordinate-free two-dimensional subtorus $U$ of $\T^4$ has $\ML(U)<1/3$ if and only if, up to signed coordinate permutations, $U$ is a product $T_{(a,b,c)}\times\T$ with $\ML(a,b,c)<1/3$, or $U=F(n_1,n_2)$ with $n_1+n_2\equiv0\pmod3$. In both cases $\ML(U)=s/(3s+1)$ for some $s\ge1$, with $3s=\max(|n_1+n_2|,|2n_1-n_2|,|2n_2-n_1|)$ in the second case. Consequently the accumulation points of the four-runner spectrum in $(0,\frac13)$ are exactly the three-runner values $s/(3s+1)$.
\end{theorem}

The last sentence is the case $n=4$ of a conjecture of Giri and Kravitz \cite[Conjecture~1.5]{GiriKravitz}, that the accumulation points of the $n$-runner spectrum are the values of the $(n-1)$-runner spectrum, on the range of values below $\frac13$. Theorem~\ref{thm:second-intro} is proved in Section~\ref{sec:second} (Theorem~\ref{thm:second} and Corollary~\ref{cor:acc}).

Above $\frac13$ the reduction meets the five three-dimensional subtori with loneliness $\frac25$, each of which carries infinitely many values. A version of the density lemma based on the shortest character of a subtorus, together with an exact analysis of these five subtori, reaches the next accumulation point $\frac25$ of the three-runner spectrum. Put
\begin{align*}
 \mathcal R&=\Big\{\frac pq\in\Big[\frac13,\frac25\Big):\gcd(p,q)=1,\ 1\le 2q-5p\le4\Big\}\\
 &=\Big\{\frac{2j}{5j+2}:j\ge2\Big\}\cup\Big\{\frac{2j+1}{5j+4}:j\ge2\Big\}.
\end{align*}

\begin{theorem}[Two-dimensional subtori below $2/5$]\label{thm:upto25-intro}
A coordinate-free two-dimensional subtorus $U$ of $\T^4$ with $\ML(U)<\frac25$ has
\begin{align*}
 \ML(U)\in{}&\Big\{\frac s{3s+1}:s\ge1\Big\}\cup\mathcal R\\
 ={}&\Big\{\frac14,\frac27,\frac3{10},\dots\Big\}\cup\Big\{\frac13,\frac6{17},\frac5{14},\frac4{11},\frac7{19},\frac{10}{27},\frac38,\dots\Big\}.
\end{align*}
This set is exactly the set of three-runner values below $\frac25$, and $\ML(2,3,5j)=\frac{2j}{5j+2}$, $\ML(1,4,5j+3)=\frac{2j+1}{5j+4}$ for $j\ge2$. Consequently the accumulation points of the four-runner spectrum in $(0,\frac25]$ are exactly the three-runner values there.
\end{theorem}

In the notation $D=\frac12-\ML$ of \cite{GiriKravitz} this gives $\operatorname{acc}(\mathcal S(4))\cap[\frac1{10},\frac12)=\mathcal S(3)\cap[\frac1{10},\frac12)$, the case $n=4$ of \cite[Conjecture~1.5]{GiriKravitz} on a range that contains the two accumulation points $\frac16$ and $\frac1{10}$ of $\mathcal S(3)$. It also determines $\mathcal S(3)\cap(\frac1{10},\frac16]$ exactly. Jain and Kravitz described this set up to finitely many exceptions \cite[Theorem~1.4]{JainKravitz}, and Corollary~\ref{cor:acc25} shows that there are none in that interval. Finally $\frac9{25}\notin\mathcal R$. Giri and Kravitz showed that $\frac7{50}=\frac12-\frac9{25}$ is the value of a one-dimensional closed subgroup of $\T^3$ but of no circle, and wrote that they believed it is not the value of a two-dimensional subtorus of $\T^4$ but lacked the computing power to check \cite[Section~8]{GiriKravitz}. Theorem~\ref{thm:upto25-intro} confirms this, so the spectra of two-dimensional subtori and of two-dimensional closed subgroups of $\T^4$ differ. Remark~\ref{rem:slicing} explains why the full case $n=4$, which Theorem~\ref{thm:full-intro} below settles, needs a new idea: the natural route through circles with a repeated speed works for every subtorus of covolume at most $20$ with $\frac13<\ML<\frac12$, but fails for an explicit subtorus with $\ML=\frac{13}{30}$, whose value is nevertheless a three-runner value. Section~\ref{sec:obstr} turns this into an infinite family. For the subtori $U$ there, with $\ML(U)=\frac12-\frac1N$, every circle in $U$ with a repeated speed has loneliness at least $N^{-3}$ below $\ML(U)$, circles in $U$ with four distinct speeds come arbitrarily close to $\ML(U)$, and $\ML(U)=\ML(1,\frac N2,N-1)$ is still a three-runner value. These subtori are one phase of a family in which a weighted coalescence of two coordinates gives three-runner realizations on both sides of a phase transition.

The density lemma also works for every number of speeds. Combined with Minkowski's second theorem and a bound for the denominator of the loneliness of a two-dimensional subtorus, it gives an effective form of the reduction of Cordella.

\begin{theorem}[Every number of speeds]\label{thm:general-intro}
Let $n\ge3$, assume the Lonely Runner Conjecture for $n-1$ and for $n-2$ speeds, and put $\kappa_n=(n(n-1))^{n-2}$. Every set $V$ of $n$ positive integers with $\gcd(V)=1$ satisfies
\[
 \ML(V)>\frac1n-\frac{\kappa_n}{2\eucl V},
\]
where $\eucl V$ is the Euclidean length of the vector of speeds. If moreover $\ML(V)<\frac1n$ and $\eucl V\ge n\kappa_n^2$, then the circle of $V$ lies in a two-dimensional subtorus of $(\R/\Z)^n$ with loneliness exactly $\frac1n$ and covolume less than $\kappa_n$, which also contains the circle of a set of $n-1$ speeds with loneliness $\frac1n$.
\end{theorem}

The hypotheses hold for $n\le15$ \cite{BarajasSerra,Rosenfeld,Trak,SungTrak,Allikvere}. The second statement is the effective form of \cite[Lemma~2.4]{Cordella}, under the same hypotheses, and answers \cite[Question~6.7]{Cordella} for every $n$. The first statement shows that the values below $\frac1n$ approach $\frac1n$ at rate $1/\eucl V$, which is sharp by Theorem~\ref{thm:main}(a). For counterexamples themselves sharper bounds are known \cite{MSS,Allikvere}. This is Theorem~\ref{thm:general} in Section~\ref{sec:general}. Using all successive minima along a flag of subtori improves the constant $\kappa_n/2$ to $n!\,(n-1)!/2^n$, under the conjecture for fewer than $n$ speeds (Proposition~\ref{prop:flag}). The constants are still too large for an exhaustive search, which is why the case $n=4$ needs the finer analysis below.

For five speeds, Theorem~\ref{thm:main} itself sharpens the general statement. Write $T_u\times\T$ for the product of the circle of a quadruple $u$ with a full coordinate circle.

\begin{theorem}[Five speeds]\label{thm:five-intro}
A coordinate-free two-dimensional subtorus of $(\R/\Z)^5$ with loneliness below $\frac29$ is, up to signed coordinate permutations, $T_{(1,2,3,4)}\times\T$ or $T_{(1,3,4,7)}\times\T$, and its loneliness is $\frac15$. Consequently the only accumulation point of the five-speed spectrum below $\frac29$ is $\frac15$, and every primitive five-speed tuple $V$ with $\ML(V)<\frac15$ and $\eucl V\ge7{,}290$ is, up to order, $\{a,2a,3a,4a,b\}$ or $\{a,3a,4a,7a,b\}$.
\end{theorem}

Cordella identified these two products at loneliness exactly $\frac15$ \cite[Theorem~3.8]{Cordella}. Theorem~\ref{thm:five-intro} shows that they are the only coordinate-free two-dimensional subtori below $\frac29$. The proof in Section~\ref{sec:five} uses a coordinate-free circle with a repeated absolute speed in each such subtorus. Its loneliness is below $\frac29$, so its distinct absolute speeds, after division by their common factor, are $\{1,2,3,4\}$ or $\{1,3,4,7\}$ by Corollary~\ref{cor:consequences} and the three-speed bound. For four speeds the second critical subtorus $\Utwo$ is not a product. The bound $7{,}290$ comes from Proposition~\ref{prop:flag}, applied with the lower bound $\frac29$ for the two-dimensional subtori through the circle of $V$.

\subsection*{Subtori of codimension two and the three-runner spectrum}
Theorems~\ref{thm:second-intro} and~\ref{thm:upto25-intro} rest on a classification of subtori, which stops at $\frac25$. The values themselves can be determined without a classification, and for every number of coordinates.

\begin{theorem}[Subtori of codimension two]\label{thm:full-intro}
Let $n\ge3$ and let $U$ be a coordinate-free subtorus of $(\R/\Z)^n$ of dimension $n-2$. Then $\ML(U)=\ML(a,b,c)$ for some positive integers $a,b,c$. Consequently the accumulation points of the four-runner spectrum are exactly the three-runner values.
\end{theorem}

The second statement is the case $n=4$ of \cite[Conjecture~1.5]{GiriKravitz}, in their notation $\operatorname{acc}(\mathcal S(4))=\mathcal S(3)$. It follows from the first because every accumulation point of the four-runner spectrum is the value of a two-dimensional subtorus \cite[Theorem~1.4]{GiriKravitz}. The proof is in Section~\ref{sec:full} (Theorem~\ref{thm:full} and Corollary~\ref{cor:full}). The subtorus $U$ is described by $n$ vectors $c_1,\dots,c_n\in\Z^2$, and $\frac12-\ML(U)$ is half the smallest factor by which the polygon $\sum_i[-c_i,c_i]$ must be scaled to meet a certain coset of $2\Z^2$ (Lemma~\ref{lem:zon}). The lattice points near the edge of first contact produce integers $N$ and $Q$ with an arithmetic property, and then the triple $(g,\,N-g,\,d)$, where $d=\gcd(Q,N)$ and $g$ is prime to $N$ with $gQ\equiv-d\pmod N$, has the loneliness of $U$. By Section~\ref{sec:obstr} such a triple cannot in general be found on a circle inside $U$, and the argument does not look for one. The arithmetic property can be solved in closed form, which makes the three-runner spectrum explicit.

\begin{theorem}[The three-runner spectrum]\label{thm:three-intro}
The values of $\ML$ below $\frac12$ on triples of positive integers are exactly the numbers
\[
 \frac12-\frac k{2N},\qquad N=kH-2ja,
\]
where $k\ge1$, $H\ge3$ is odd, $0\le a<k$, $\gcd(a,k)=1$, $0\le j\le k/2$ and $a+j<H$. Together with $\frac12$ these numbers are the accumulation points of the four-runner spectrum, and for every $n\ge3$ they are the accumulation points of order $n-3$ of the set of values of $\ML$ on $n$-tuples of positive integers.
\end{theorem}

This is Theorem~\ref{thm:three} and Corollary~\ref{cor:acc3} in Section~\ref{sec:three}. The accumulation points of order $0$ of a set are its elements, and those of order $m+1$ are the accumulation points of the set of those of order $m$. In the notation of \cite{GiriKravitz} the last statement reads $\operatorname{acc}^{(n-3)}(\mathcal S(n))=\mathcal S(3)$, and their conjecture, for all $n$, is equivalent to $\operatorname{acc}^{(m)}(\mathcal S(n))=\mathcal S(n-m)$ for all $n$ and all $m<n$. The three-runner values were known below $\frac25$, by \cite{Kravitz}, \cite{JainKravitz} and Theorem~\ref{thm:upto25-intro}. Each value comes with a triple that attains it, for instance $\ML(a,\,H-a,\,N-a)=\frac12-\frac k{2N}$ when $a\ge1$ is prime to $H$ (Proposition~\ref{prop:witness}). Between $\frac25$ and $\frac37$ the values are the fractions $\frac{3k-\rho}{7k-2\rho}$ with $1\le\rho\le9$ and $k$ in an explicit set that depends on $\rho$ (Corollary~\ref{cor:windows}). The proof of Theorem~\ref{thm:full-intro} does not extend to two-dimensional subtori of $(\R/\Z)^5$. The obstruction is the value $\frac3{14}$ that is missing in Theorem~\ref{thm:values}: it is the largest value of $\min_i\norm{x_i}$ on a finite cyclic subgroup of $(\R/\Z)^3$ of the kind that the proof would have to realize by four runners (Remark~\ref{rem:five}).

\subsection*{Five and six speeds}
For five and six speeds the bounds of Theorem~\ref{thm:general-intro} are far beyond an exhaustive search. Bonfioli \cite{BonfioliFour,BonfioliFive} replaced the search over tuples by a search over the initial segments of a reduced basis of the lattice of relations of the speeds. In this way he formalized Theorem~\ref{thm:values} in Lean~4, and he proved, again with a proof formalized in Lean~4, that the values of $\ML$ below $\frac15$ on sets of five speeds are exactly the numbers $\frac s{5s+1}$, $s\ge1$, each attained by $\{1,2,3,4,5s\}$ \cite{Kravitz}. His formalization for four speeds covers Theorem~\ref{thm:values} and the attainment of every value, and not the classification of Theorem~\ref{thm:main}. For five speeds, version~1.0.2 of his archive records a clean rebuild, which he ran from a fresh clone on 21 and 22 November 2026: all $17{,}598$ modules of the proof were built with Lean~4, and the main theorems depend only on the three standard axioms. With a sharper pruning test the same search tree is small enough for six speeds (Section~\ref{sec:tree}). For integers $A,B$ put $w(A,B)=(A,\,B,\,A+B,\,2A+B,\,3A+B,\,3A+2B)$.

\begin{theorem}[Six speeds]\label{thm:six-intro}
A primitive set $V$ of six distinct positive integers has $\ML(V)<\frac16$ if and only if it is one of the following.
\begin{enumerate}[label=\textup{(\alph*)},leftmargin=*]
\item $\{1,2,3,4,5,6s\}$ with $s\ge1$.
\item $\{1,3,4,5,9,6s\}$ with $s\ge1$.
\item The set of absolute values of the entries of $w(A,B)$, for coprime integers $A,B$ such that the six entries are nonzero and $6$ divides none of $A-B$, $2A-B$, $A-2B$, $A-3B$, $2A-3B$, $A-4B$.
\item One of the nine sets
\begin{gather*}
 \{1,2,5,6,7,8\},\quad \{2,5,6,8,10,11\},\quad \{1,4,5,6,7,22\},\\
 \{1,2,3,5,8,18\},\quad \{2,3,5,6,8,11\},\quad \{2,5,6,7,8,11\},\\
 \{1,3,4,5,18,46\},\quad \{1,3,4,5,7,24\},\quad \{1,4,5,6,7,33\}.
\end{gather*}
\end{enumerate}
The values of $\ML$ below $\frac16$ on sets of six distinct positive integers are exactly the numbers $\frac{P-1}{6P}$ with $P\ge7$ and $\gcd(P,6)=1$. In particular their denominators are odd.
\end{theorem}

The sets in (a), (b), (c) are the lines of the three critical subtori for six speeds, which Cordella determined \cite[Theorem~3.7]{Cordella}, and the nine sets of (d) are the ones that his search up to speed $110$ found on none of them \cite[Table~2]{Cordella}. He proved the statement on the values for all sextuples outside an undetermined finite set, and he conjectured that the denominators are always odd \cite[Theorem~6.1 and Conjecture~6.2]{Cordella}. Theorem~\ref{thm:six-intro} shows that the nine sets of (d) are the only near-tight sextuples on no critical subtorus, and it confirms the conjecture. The values in (a) and (b) are $\frac s{6s+1}$, except that $\ML(1,3,4,5,6,9)=\frac2{13}$. In (c) the value is $\frac{P-1}{6P}$ for an explicit integer $P=P(A,B)$, the largest of four linear forms that depend on the residue class of $(A,B)$ modulo $6$ (Theorem~\ref{thm:linesC}). The values with $P\equiv5\pmod 6$, whose denominators are divisible by $3$, occur only in (c), and a family of Jain and Kravitz attains all of them \cite[Theorem~1.5]{JainKravitz}. Two consequences are that $\{1,\dots,6\}$ is the only tight set of six speeds up to scaling, which Marshall \cite{Marshall} proved before by a different computer-assisted argument, and the bound $\max V<3q$ for the denominator $q$ of $\ML(V)$, which answers the first part of a question of Cordella \cite[Question~6.6]{Cordella}. These statements are Theorem~\ref{thm:six} and Corollary~\ref{cor:six}. For five speeds the same search recovers the classification of Bonfioli, which consists of the two families $\{1,2,3,4,5s\}$ and $\{1,3,4,7,5s\}$ and twelve further sets (Theorem~\ref{thm:fivefull}), and for four speeds it gives a second proof of the reduction behind Theorem~\ref{thm:main}.

The same tree can be run at a level above $\frac1n$, and then it lists the sets whose loneliness is just above $\frac1n$. For four, five and six speeds the least value of $\ML$ above $\frac1n$ is $\frac{11}{43}$, $\frac{18}{89}$ and $\frac{10}{59}$, and exactly $29$, $64$ and $54$ primitive sets on no critical subtorus have $\ML=\frac1n$ (Theorem~\ref{thm:above}).

\subsection*{Seven and eight speeds}
Cordella named the missing classification of the tight sets of six speeds as the obstacle to determining the critical subtori for seven speeds by his method \cite[Section~6.3]{Cordella}. Theorem~\ref{thm:six-intro} supplies it, together with the gap between the two smallest six-speed values $\frac17$ and $\frac5{33}$.

\begin{theorem}[Seven speeds]\label{thm:seven-intro}
A coordinate-free two-dimensional subtorus of $(\R/\Z)^7$ with loneliness below $\frac5{33}$ is, up to signed coordinate permutations, $T_{(1,2,3,4,5,6)}\times\T$, and its loneliness is $\frac17$. Every primitive set $V$ of seven distinct positive integers with $\ML(V)<\frac17$ and $\eucl V\ge5{,}145{,}525$ is $\{1,2,3,4,5,6,7s\}$ for some $s\ge1$.
\end{theorem}

The bound $\frac5{33}$ is attained by $T_{(1,5,6,11,16,17)}\times\T$. Consequently the seven-speed spectrum below $\frac17$ consists of the values $\frac s{7s+1}$ and finitely many further values, all with denominator at most $7{,}276{,}871$, among them $\frac3{23}$, $\frac5{37}$ and $\frac9{65}$ (Theorem~\ref{thm:seven} and Corollary~\ref{cor:seven}). Under Marshall's classification of the tight sets of seven speeds \cite{Marshall}, the same method shows that the critical subtori for eight speeds are three products, and that every primitive near-tight set of eight speeds with Euclidean length at least $12{,}065{,}804{,}452{,}964$ lies on one of them (Theorem~\ref{thm:eight} and Corollary~\ref{cor:eight}). For the number $N_n(H)$ of primitive near-tight sets of $n$ speeds at most $H$ we obtain $N_4(H)\sim H^2/\pi^2$, $N_5(H)=2\lfloor H/5\rfloor+12$ for $H\ge27$, $N_6(H)\sim H^2/(4\pi^2)$, and $N_7(H)=\lfloor H/7\rfloor+C_7$ for $H\ge5{,}145{,}525$ with a constant $C_7\ge27$ (Theorem~\ref{thm:count}).

Supplements A, B and C, which follow the appendix, complete the case of seven speeds. The primitive near-tight sets are $\{1,\dots,6,7s\}$ and $27$ further sets, and the values below $\frac17$ are $\frac s{7s+1}$, $\frac3{23}$, $\frac5{37}$ and $\frac9{65}$ (Theorem~\ref{thm:sevenfull}). This proves the classification of the tight sets of seven speeds that the eight-speed results assume. So they hold unconditionally, and the bound becomes $3{,}898{,}150{,}815$ (Corollaries~\ref{cor:M} and~\ref{cor:eightfull}).

\subsection*{Method for four speeds}
The main tool is a density lemma (Lemma~\ref{lem:density}). If a subtorus $S$ of $\T^4=(\R/\Z)^4$ has a lattice of large covolume, then $S$ is dense, in the maximum norm, in a subtorus of one higher dimension, with an explicit constant obtained from Vaaler's theorem on central sections of the cube \cite{Vaaler} and Minkowski's convex body theorem. Since $\ML$ is $1$-Lipschitz in the maximum norm, this transfers lower bounds on $\ML$ from larger subtori to smaller ones. This is the quantitative, maximum norm version of the tube argument of Giri and Kravitz \cite[Lemma~4.4]{GiriKravitz}.

Applied to the circle $T_v$ of a primitive quadruple $v$ with $\ML(v)<1/4$ and $\eucl v\ge 512$, the lemma produces a two-dimensional subtorus $U\supseteq T_v$ with $\ML(U)<5/16$. This is the effective step: it bounds the length of every near-tight quadruple that lies on no critical subtorus by $512$ (Proposition~\ref{prop:large}), which is what \cite[Question~6.7]{Cordella} asks for in the case of four speeds. Applied once more, it shows that every two-dimensional subtorus with $\ML\le 5/16$ either has covolume less than $33$ or lies in one of two explicit three-dimensional subtori with a small parameter. Both cases are enumerated exactly. The result is a list of thirty-one subtori (Proposition~\ref{prop:list}). Two of them are the subtori $\Uone$ and $\Utwo$ of Jain and Kravitz with $\ML=1/4$, which carry the two infinite families, and the remaining twenty-nine have $\ML\ge 2/7$ and cannot contain a circle of length $\ge 288$ with loneliness below $1/4$. An exact exhaustive search over all primitive quadruples with $\eucl v<512$ completes the classification. The formulas for the two infinite families are proved by an elementary analysis of the superlevel sets of $\ML$ on $\Uone$ and $\Utwo$, in the spirit of \cite{JainKravitz}.

The programs and certificate data are in the ancillary files. Section~\ref{sec:comp} gives the arithmetic used in the enumerations and the ranges covered by the independent checks.

\subsection*{Odd signed relations}
Beck and Everett \cite{BeckEverett} showed that tight instances and counterexamples of the conjecture lie on finitely many hyperplanes. Here $k$ is the number of speeds, $v\in\Z^k$, and $\ML(v)=\max_t\min_i\norm{tv_i}$ as before. Call $m\in\Z^k$ a relation of $v$ if $m\cdot v=0$, and \emph{harmful} if $m_1+\dots+m_k$ is odd. They proved \cite[Theorems~1.1 and 2.1]{BeckEverett} that if $\ML(v)\le1/(k+1)$, then $v$ has a harmful relation with $\lone m\le 2k+3$ and Euclidean norm at most $2(k+1)/\sqrt{k-1}$. The parity is the point. For integer speeds the orbit $\{tv\bmod 1:t\in\R\}$ is the closed subgroup of $\T^k$ annihilated by all relations, so if $v$ has no harmful relation, then it contains $(\frac12,\dots,\frac12)$ and some time puts every runner at distance $\frac12$ from the origin. Their proof is a second moment argument for a product of Fourier windows. We improve their bounds in two ways. The windows vanish to first order on an orbit that avoids their support, and using this cancellation of first moments halves the squared Euclidean bound for every $k$.

\begin{theorem}\label{thm:l2}
Let $k\ge2$ and let $v\in\Z^k\setminus\{0\}$ be such that no $t$ has $\norm{tv_j}>\frac1{k+1}$ for all $j$, for instance positive integers with $\ML(v)\le\frac1{k+1}$. Then $v$ has a harmful relation $m$ with
\[
 \eucl m^2<\frac{2(k+1)^2}{k-1}.
\]
Consequently $\lone m\le2\lfloor\frac{2k+1}3\rfloor+1$ for $k\ge9$, and $\lone m\le5,5,7,9,9,11,13$ for $k=2,\dots,8$.
\end{theorem}

This is the case of a half-integer centre of an obstruction for boxes avoided by closed subgroups of tori (Theorem~\ref{thm:box}), which also covers shifted runners and unequal thresholds. For $k\ge9$ much more is true: the harmful relation can be taken with all coefficients in $\{-1,0,1\}$.

\begin{definition}\label{def:signed}
A \emph{signed relation} of $v\in\Z^k$ is a vector $m\in\{-1,0,1\}^k\setminus\{0\}$ with $m\cdot v=0$. It is \emph{odd} if $\lone m$, the number of nonzero coordinates, is odd. Write $\Rel(v)$ for the set of signed relations, and $\Rel_{\mathrm{odd}}(v)$, $\Rel_{\mathrm{even}}(v)$ for its two parts.
\end{definition}

With $A=\{i:m_i=1\}$ and $B=\{i:m_i=-1\}$, an odd signed relation is a pair of disjoint sets with $|A|+|B|$ odd and $\sum_{i\in A}v_i=\sum_{i\in B}v_i$. For $m\in\{-1,0,1\}^k$ the coordinate sum has the parity of $\lone m$, so odd signed relations are harmful.

\begin{theorem}\label{thm:rel}
Let $k\ge 9$ and let $v_1,\dots,v_k$ be positive integers with $\ML(v)\le 1/(k+1)$. Then $v$ has an odd signed relation: there are disjoint $A,B\subseteq\{1,\dots,k\}$ with $|A|+|B|$ odd and $\sum_{i\in A}v_i=\sum_{i\in B}v_i$.
\end{theorem}

So the harmful relation of \cite{BeckEverett} can be chosen with all coefficients $\pm1$ or $0$, $\lone m\le k$ (the largest odd integer at most $k$) and Euclidean norm at most $\sqrt k$. For tight instances Theorem~\ref{thm:rel} applies to every $k\ge9$, and for counterexamples to every $k$ for which the conjecture is open. The threshold $9$ is the least $k$ with $k\tan^2(\pi/(k+1))<1$ (Lemma~\ref{lem:thresholds}). For three to six speeds the tight instances are known, and each of them contains $a,b$ and $a+b$ (Section~\ref{sec:schur}). Theorem~\ref{thm:rel} is the case of equal thresholds of the following statement.

\begin{theorem}\label{thm:relgen}
Let $v\in\Z^k$ and let $\delta_1,\dots,\delta_k\in[0,\frac12)$ satisfy $\sum_{i=1}^k\tan^2(\pi\delta_i)<1$. If $v$ has no odd signed relation, then the set of $t\in[0,1)$ with $\norm{tv_i}>\delta_i$ for all $i$ has measure at least $2^{-k}\bigl(1-\sum_i\tan^2(\pi\delta_i)\bigr)$.
\end{theorem}

The proof is the inequality $\int_0^1\Phi(tv)\,dt>0$ for
\begin{equation}\label{eq:Phi}
 \Phi(x)=\prod_{j=1}^k\sin^2(\pi x_j)\cdot\Bigl(1-\sum_{i=1}^k\tan^2(\pi\delta_i)\cot^2(\pi x_i)\Bigr),
\end{equation}
a trigonometric polynomial that is nonpositive whenever some $\norm{x_i}\le\delta_i$ and whose Fourier coefficients live on $\{-1,0,1\}^k$, with signs given by the parity of the support (Lemma~\ref{lem:cert}). The factor $\prod_j\sin^2(\pi x_j)=2^{-k}\prod_j(1-\cos2\pi x_j)$ is the Riesz product that Bedert used to show $\ML(V)\gg n^{-1/2}$ for dissociated sets $V$ \cite[Lemma~4.1]{Bedert}, that is, for sets with no signed relation at all. Theorem~\ref{thm:relgen} sharpens that argument in the two ways that matter here: only odd signed relations need to be excluded, and the bound is sharp for the stated Fourier support (Proposition~\ref{prop:minorant}).

\begin{corollary}\label{cor:rel}
\begin{enumerate}[label=\textup{(\alph*)},leftmargin=*]
\item If $v\in\Z^k$ has no odd signed relation, in particular if $v$ is dissociated, then $\ML(v)\ge\delta^*_k:=\pi^{-1}\arctan(k^{-1/2})>\frac{1}{\pi\sqrt k}\bigl(1-\frac1{3k}\bigr)$.
\item If $n\ge 11$ and $v_1,\dots,v_n$ are positive integers with $\ML(v)\le1/n$, then $v$ has an odd signed relation.
\end{enumerate}
\end{corollary}

Part (b) brings relations into the barely lonely range $[\frac1{n+1},\frac1n)$ of the first part of the paper. On integer orbits the threshold of Theorem~\ref{thm:relgen} can be crossed by an explicit amount, which improves (a) to $\ML(v)\ge\pi^{-1}\arctan\sqrt{(1+c_k)/k}$ with $c_k>0$ (Theorem~\ref{thm:orbit}). In the two borderline cases $k=8$ with $\ML\le\frac19$ and $k=10$ with $\ML\le\frac1{10}$, which Theorem~\ref{thm:rel} and Corollary~\ref{cor:rel}(b) just miss, a signed relation of support four forces an odd one (Corollary~\ref{cor:four}). The same identity counts relations: for a tight instance, the odd signed relations must outweigh the even ones.

\begin{theorem}\label{thm:relcount}
Let $\delta\in(0,\frac12)$, $w=\tan^2(\pi\delta)$ and $kw<1$. If $v\in\Z^k$ satisfies $\ML(v)\le\delta$, then
\begin{multline*}
 \sum_{m\in\Rel_{\mathrm{odd}}(v)}2^{-\lone m}\bigl(1+(2\lone m-k)w\bigr)\\
 \ge\ 1-kw+\sum_{m\in\Rel_{\mathrm{even}}(v)}2^{-\lone m}\bigl(1+(2\lone m-k)w\bigr).
\end{multline*}
In particular, if every odd signed relation of $v$ has $\lone m\ge s$, then $v$ has more than $2^{s-1}(1-kw)$ odd signed relations, counting $m$ and $-m$ separately.
\end{theorem}

For a tight instance with $k\ge9$ one takes $\delta=1/(k+1)$, and $1-kw$ tends to $1$ as $k$ grows. So a tight instance with many speeds either contains $a,b,a+b$ or has many odd signed relations. Shifts are handled in the same way.

\begin{theorem}\label{thm:relshift}
Let $v\in\Z^k$, $s\in\R^k$, and $\delta_1,\dots,\delta_k\in[0,\frac12)$ with $\sum_i\tan^2(\pi\delta_i)<1$. If $(-1)^{\lone m}\cos(2\pi m\cdot s)\ge0$ for every $m\in\Rel(v)$, then the set of $t\in[0,1)$ with $\norm{tv_i+s_i}>\delta_i$ for all $i$ has measure at least $2^{-k}(1-\sum_i\tan^2(\pi\delta_i))$. In particular this holds for every $s$ if $v$ is dissociated.
\end{theorem}

For $s=0$ this is Theorem~\ref{thm:relgen}. The shifted form of the conjecture is false in general (see \cite{BeckEverett} and the references there), so some hypothesis on the relations is needed.

Both features of $\Phi$ are forced. A \emph{certificate} for an exceptional set $E\subseteq\Z^k$ is a real trigonometric polynomial $\Psi$ on $\T^k$ with $\widehat\Psi(0)>0$, $\widehat\Psi(m)\ge0$ for $m\notin E$, and $\Psi\le0$ wherever some $\norm{x_i}\le\delta_i$. The argument of Theorem~\ref{thm:relgen} shows that every $v$ with $\ML(v)\le\delta$ has a relation in $E$. Theorem~\ref{thm:rel} uses the exceptional set of odd vectors in $\{-1,0,1\}^k$.

\begin{theorem}\label{thm:relopt}
Let $k\ge1$.
\begin{enumerate}[label=\textup{(\alph*)},leftmargin=*]
\item Let $\Psi$ be a real trigonometric polynomial on $\T^k$ with $\widehat\Psi(m)=0$ whenever $\lone m\ge k$, $\widehat\Psi(0)>0$, and $\widehat\Psi(m)\ge0$ whenever $m_1+\dots+m_k$ is even. Then $\Psi(x)>0$ for some $x\in\{0,\frac12\}^k$ with an odd number of zero coordinates.
\item Let $\delta_1,\dots,\delta_k\in[0,\frac12)$ with $\sum_i\tan^2(\pi\delta_i)\ge1$, and let $\Psi$ be a real trigonometric polynomial with Fourier coefficients supported on $\{-1,0,1\}^k$ and $\widehat\Psi(0)>0$. Then $\Psi(x)>0$ for some $x$ with $\norm{x_i}\le\delta_i$ for some $i$.
\end{enumerate}
\end{theorem}

By (a), no certificate of $\ell^1$-degree below $k$ whose exceptional set consists of harmful vectors exists, for any thresholds, even if only the points where some runner is exactly at $0$ are required to be nonpositive. So $\lone m\le k$ in Theorem~\ref{thm:rel} is the limit of this method. By (b), certificates supported on $\{-1,0,1\}^k$, which are what produce coefficients $\pm1$, exist exactly when $\sum_i\tan^2(\pi\delta_i)<1$, and no positivity of coefficients is needed for this. In particular Theorem~\ref{thm:rel} cannot be pushed below $k=9$ in this way, and for $k\le8$ the best bound we have is Theorem~\ref{thm:l2}.

For few speeds much more is true. By Theorems~\ref{thm:main}, \ref{thm:fivefull} and~\ref{thm:six}, every set of $n\in\{4,5,6\}$ speeds with $\ML<1/n$ contains $a,b,a+b$ (Proposition~\ref{prop:few}), and we verify the same for sets of seven and eight speeds with largest element at most $256$ (Proposition~\ref{prop:computed}). This suggests the following.

\begin{conjecture}\label{conj:schur}
Let $n\ge4$. Every set of $n$ positive integers with $\ML<1/n$ contains three distinct elements $a,b,c$ with $a+b=c$. Equivalently, a set of $n\ge4$ positive integers in which no element is the sum of two others has $\ML\ge1/n$.
\end{conjecture}

So the conjecture holds for $n=4,5,6$, and Supplement~C proves it for $n=7$ (Theorem~\ref{thm:sevenfull}). For $n=4$ the strict inequality is needed, because $\{1,2,6,11\}$ has $\ML=\frac14$ and contains no such triple, and for $n=5$ and $n=6$ it is not needed (Proposition~\ref{prop:equal}). For $n=3$ the sets $\{1,2,3s\}$, $s\ge2$, show that $a=2b$ must also be allowed. Conjecture~\ref{conj:schur} would give a relation of support $3$, where Corollary~\ref{cor:rel}(b) gives an odd one of support at most $n$. The certificate cannot close this gap. With Fourier support $\{-1,0,1\}^n$ it must carry negative mass on every odd support size, and for $n\ge5$ a certificate for Conjecture~\ref{conj:schur} would need even harmonics of single coordinates (Proposition~\ref{prop:parity} and Corollary~\ref{cor:noschur}). Finally, dissociated sets need not be much lonelier than general ones.

\begin{proposition}\label{prop:fifth}
For $j\ge0$ the set $V_j=\{1,2\}\cup\{5\cdot2^i-1:0\le i\le j\}=\{1,2,4,9,19,\dots,5\cdot 2^j-1\}$ is superincreasing, hence dissociated, and
\[
 \ML(V_j)=\frac{2^j+1}{5\cdot2^j+1}.
\]
\end{proposition}

So the least loneliness of a dissociated set of size $k$ lies between $\pi^{-1}\arctan(k^{-1/2})$ and $\ML(V_{k-3})$, which tends to $\frac15$. The family $V_j$ is not extremal: the superincreasing set $\{1,2,4,8,18,35,69,140,282,623\}$ has $\ML=\frac{39}{209}<\frac3{16}$, and a superincreasing set of fifteen speeds has $\ML=\frac{3579}{19421}$ (Table~\ref{tab:dissociated}). We do not know whether the least loneliness of a dissociated set is bounded below by a positive constant (Question~\ref{q:fifth}). For superincreasing sets it is: for every superincreasing sequence, finite or infinite, and arbitrary starting positions there is a time at which every runner is at distance at least $\frac1{97}$ from the origin (Corollary~\ref{cor:lacunary}). The answer is negative exactly when some infinite dissociated set $E$ has $\inf_{n\in E}\norm{nt}=0$ for every real $t$ (Proposition~\ref{prop:qrec}). A positive answer would imply that no dissociated set, called quasi-independent in harmonic analysis, is dense in the Bohr compactification of $\Z$ (Corollary~\ref{cor:bohr}). Whether a Sidon set can be dense there is a long-standing open problem \cite{LopezRoss,RamseyBohr}.

\subsection*{Organization}
Sections~\ref{sec:two} to \ref{sec:families}, together with the preliminary Sections~2 to 4, prove Theorems~\ref{thm:main} and~\ref{thm:values} and Corollary~\ref{cor:consequences}. Section~\ref{sec:second} proves Theorems~\ref{thm:second-intro} and~\ref{thm:upto25-intro}, Section~\ref{sec:obstr} treats slicing obstructions, and Section~\ref{sec:general} proves Theorems~\ref{thm:general-intro} and~\ref{thm:five-intro}. Section~\ref{sec:full} proves Theorem~\ref{thm:full-intro}, Section~\ref{sec:three} proves Theorem~\ref{thm:three-intro}, Section~\ref{sec:tree} proves Theorem~\ref{thm:six-intro}, and Section~\ref{sec:seven} proves Theorem~\ref{thm:seven-intro} and the results on eight speeds and on counting. Sections~\ref{sec:relcert} to \ref{sec:fifth} contain the results on relations, with the box obstruction in Section~\ref{sec:box}, and Section~\ref{sec:comp} describes the computations.

\section{Preliminaries}
Write $\T=\R/\Z$ and $\T^4=\R^4/\Z^4$. For $x\in\T^4$ put
\[
 G(x)=\min_{1\le i\le 4}\norm{x_i}.
\]
For $p,q\in\R^4$ we use the maximum norm $\supn{p}=\max_i|p_i|$ and the Euclidean norm $\eucl p$. Since $\norm{\cdot}$ is $1$-Lipschitz on $\R$,
\begin{equation}\label{eq:lip}
 |G(x+p)-G(x)|\le \supn{p}\qquad (x\in\T^4,\ p\in\R^4).
\end{equation}

A \emph{subtorus} of $\T^4$ is a closed connected subgroup, that is, the image $S=E/(E\cap\Z^4)$ of a rational linear subspace $E\subseteq\R^4$. We write $\Lambda_S=E\cap\Z^4$, a primitive sublattice of $\Z^4$ of rank $k=\dim E$, and $\covol(S)=\covol(\Lambda_S)$ for the $k$-dimensional volume of a fundamental parallelepiped of $\Lambda_S$ in $E$. Put
\[
 \ML(S)=\max_{x\in S}G(x).
\]
Call $S$ \emph{coordinate-free} if it is contained in no coordinate hyperplane $\{x_i=0\}$, equivalently if no coordinate function vanishes identically on $\Lambda_S$. All subtori that contain a coordinate-free subtorus are coordinate-free.

For $v\in\Z^4$ primitive with all $v_i\ne 0$ let $T_v=\{tv:t\in\R\}\subseteq\T^4$ be its circle. Then $\Lambda_{T_v}=\Z v$, $\covol(T_v)=\eucl v$, and
\[
 \ML(T_v)=\max_t\min_i\norm{v_it}=\ML(|v_1|,|v_2|,|v_3|,|v_4|).
\]
The group of signed coordinate permutations acts on $\T^4$ and preserves $G$, hence preserves $\ML$ of subtori, and it maps subtori to subtori. ``Up to symmetry'' always refers to this group of order $384$.

\begin{lemma}[Corner lemma]\label{lem:corner}
Let $V$ be a finite set of positive integers. Then $\ML(V)=\min_{v\in V}\norm{vt}$ for some $t=m/(u+w)$ with $u,w\in V$ and $m\in\Z$.
\end{lemma}
\begin{proof}
The function $f(t)=\min_{v\in V}\norm{vt}$ is continuous and piecewise linear, with pieces of slopes $\pm v$. At a maximum point $t_0$ with $f(t_0)>0$, either some wave $\norm{vt}$ has a downward cusp at $t_0$ and $f(t_0)=\norm{vt_0}=1/2$, so $2vt_0\in\Z$, or two active pieces of slopes $u>0$ and $-w<0$ meet at $t_0$, so $ut_0-a=b-wt_0$ for integers $a,b$ and $(u+w)t_0\in\Z$. The case $f\equiv 0$ is impossible for finitely many positive speeds.
\end{proof}
This is \cite[Proposition~4.1]{Kravitz}. It makes $\ML(V)$ a finite exact computation, and it shows that $\ML(V)$ is rational.

\begin{lemma}[Orbit lines]\label{lem:orbit}
Let $x,y$ be coprime integers. A point $(X,Y)\in\T^2$ lies on the circle $\{(xt,yt):t\in\R\}$ if and only if $yX-xY\in\Z$.
\end{lemma}
\begin{proof}
Necessity is clear. If $yX-xY=k\in\Z$, choose $h\in\Z$ with $yh\equiv -k\pmod x$ and put $t=(X+h)/x$. Then $xt\equiv X$ and $yt=(yX+yh)/x=(k+xY+yh)/x\equiv Y\pmod 1$.
\end{proof}

\section{The density lemma}
\begin{lemma}[Density]\label{lem:density}
Let $S$ be a $k$-dimensional subtorus of $\T^4$, $1\le k\le 3$, with $A=\covol(S)$, and let $E=\R\Lambda_S$. Let $\pi:\R^4\to E^\perp$ be the orthogonal projection. Then
\begin{enumerate}[label=\textup{(\alph*)},leftmargin=*]
\item $\pi(\Z^4)$ is a lattice of rank $4-k$ in $E^\perp$ with $\covol(\pi(\Z^4))=1/A$.
\item there is $w\in\Z^4\setminus\Lambda_S$ with $\supn{\pi(w)}\le A^{-1/(4-k)}$.
\item for any $w\in\Z^4\setminus\Lambda_S$, let $S'$ be the $(k+1)$-dimensional subtorus with $\R\Lambda_{S'}=E+\R w$. Every point of $S'$ lies within maximum norm distance $\supn{\pi(w)}/2$ of a point of $S$. Consequently
\[
 \ML(S)\ \ge\ \ML(S')-\tfrac12\supn{\pi(w)}.
\]
\end{enumerate}
In particular $\ML(S)\ge \ML(S')-\frac12A^{-1/(4-k)}$ for a suitable $S'\supseteq S$.
\end{lemma}
\begin{proof}
(a) Since $\Lambda_S$ is primitive, a basis $b_1,\dots,b_k$ of $\Lambda_S$ extends to a basis $b_1,\dots,b_4$ of $\Z^4$. The kernel of $\pi$ on $\Z^4$ is $\Z^4\cap E=\Lambda_S$, so $\pi(\Z^4)\cong\Z^4/\Lambda_S\cong\Z^{4-k}$ is a lattice spanned by $\pi(b_{k+1}),\dots,\pi(b_4)$. The volume of the parallelepiped spanned by $b_1,\dots,b_4$ is $1$ and it factors as the $k$-volume of the parallelepiped spanned by $b_1,\dots,b_k$ in $E$ times the $(4-k)$-volume of the one spanned by $\pi(b_{k+1}),\dots,\pi(b_4)$ in $E^\perp$. Hence $1=A\cdot\covol(\pi(\Z^4))$.

(b) Let $\rho>0$ and $K_\rho=\{z\in E^\perp:\supn{z}\le\rho\}$. This is the central section of the cube $[-\rho,\rho]^4$ by the $(4-k)$-dimensional subspace $E^\perp$, a compact, convex, centrally symmetric set. By Vaaler's theorem \cite{Vaaler}, every central section of the unit cube $[-\frac12,\frac12]^4$ by a subspace of dimension $d$ has $d$-dimensional volume at least $1$. Scaling by $2\rho$ gives $\operatorname{vol}_{4-k}(K_\rho)\ge(2\rho)^{4-k}$. Minkowski's convex body theorem (in the form valid for compact bodies, see \cite[Chapter~III]{Cassels}) states that a compact centrally symmetric convex set of volume at least $2^{4-k}\covol(\pi(\Z^4))=2^{4-k}/A$ contains a nonzero point of $\pi(\Z^4)$. For $\rho=A^{-1/(4-k)}$ we have $(2\rho)^{4-k}=2^{4-k}/A$, so there is $w\in\Z^4$ with $\pi(w)\ne 0$, that is $w\notin E$, and $\supn{\pi(w)}\le\rho$.

(c) Write $u=\pi(w)$ and $w=u+e$ with $e\in E$. Let $y\in E+\R w$ represent a point of $S'$, say $y=s+bu$ with $s\in E$ and $b\in\R$. Choose $j\in\Z$ with $|b-j|\le 1/2$. Then
\[
 y-jw=(s-je)+(b-j)u,
\]
and $y-jw\equiv y\pmod{\Z^4}$. The point $s-je\in E$ represents a point of $S$, and $(b-j)u$ has maximum norm at most $\supn u/2$. By \eqref{eq:lip}, $G(y)\le G(s-je)+\supn u/2\le\ML(S)+\supn u/2$. Taking $y$ to be a maximizer of $G$ on $S'$ gives the inequality.
\end{proof}

\begin{remark}
For $k=3$ Vaaler's theorem is not needed: $E^\perp=\R m$ for a primitive $m\in\Z^4$, $\pi(\Z^4)=\Z\,m/\eucl m^2$, and the vector $w$ with $m\cdot w=1$ has $\supn{\pi(w)}=\supn m/\eucl m^2$. Note also $\covol(\Lambda_S)=\eucl m$ in this case.
\end{remark}

\section{Three-dimensional subtori}
A coordinate-free three-dimensional subtorus is $W_m=\{x\in\T^4:m\cdot x\in\Z\}$ for a primitive $m\in\Z^4$ with at least two nonzero entries, and $\covol(W_m)=\eucl m$. Up to symmetry we may take the entries of $m$ nonnegative after permuting, but we keep signs since they matter for the explicit coordinates.

\begin{lemma}\label{lem:W13}
Every coordinate-free three-dimensional subtorus $W$ satisfies $\ML(W)\ge 1/3$.
\end{lemma}
\begin{proof}
Let $W=W_m$. We find $x\in[\frac13,\frac23]^4$ with $m\cdot x\in\Z$. The set $\{m\cdot x: x\in[\frac13,\frac23]^4\}$ is an interval of length $\frac13\sum_i|m_i|$. If $\sum_i|m_i|\ge 3$ the interval has length at least $1$ and contains an integer. Otherwise $m=\pm e_i\pm e_j$ for some $i\ne j$, and $x_i=x_j=\frac12$ gives $m\cdot x\in\{0,\pm1\}$.
\end{proof}

The two subtori
\begin{gather*}
 \Wone=\{x_1=2x_2\}=\{(2y,y,z,u)\},\\ \Wtwo=\{x_3=x_1+x_2\}=\{(a,b,a+b,u)\}
\end{gather*}
have $\ML=1/3$ exactly: the bound $\ML(\Wone)\le\ML(1,2)=1/3$ is clear, and $\min(\norm a,\norm b,\norm{a+b})\le 1/3$ for all $a,b$ since $\norm a,\norm b>1/3$ forces $a,b\in(\frac13,\frac23)+\Z$ and then $\norm{a+b}\le\frac13$.

\begin{lemma}\label{lem:W25}
Let $W$ be a coordinate-free three-dimensional subtorus with $\ML(W)<2/5$. Then $W$ is, up to symmetry, $\Wone$ or $\Wtwo$.
\end{lemma}
\begin{proof}
Let $W=W_m$. By Lemma~\ref{lem:density}(c) with $S=W$, $S'=\T^4$ and the vector $w$ with $m\cdot w=1$ (see the remark above), and since $G(\frac12,\frac12,\frac12,\frac12)=\frac12$,
\[
 \ML(W)\ \ge\ \frac12-\frac{\supn m}{2\eucl m^2}.
\]
If $\ML(W)<2/5$ then $\eucl m^2<5\supn m$. As $\eucl m^2\ge\supn m^2$ this forces $\supn m\le 4$ and then $\eucl m^2\le 19$. There are exactly $21$ classes of primitive $m\in\Z^4$ with at least two nonzero entries and $\eucl m^2\le 19$ up to symmetry. For $19$ of them Table~\ref{tab:chars} exhibits a point $x\in W_m$ with $G(x)\ge 2/5$. The remaining two classes are $m=(0,0,1,-2)$ and $m=(0,1,-1,-1)$, whose subtori are $\Wone$ and $\Wtwo$ up to symmetry.
\end{proof}

\section{Two-dimensional subtori with loneliness at most $5/16$}\label{sec:two}
\subsection{Reduction to a finite list}
We identify $\Wone$ with $\T^3$ through $(y,z,u)\mapsto(2y,y,z,u)$, and $\Wtwo$ with $\T^3$ through $(a,b,u)\mapsto(a,b,a+b,u)$. The two-dimensional subtori of $\T^3$ are the kernels $U_n=\{p\in\T^3:n\cdot p\in\Z\}$ of primitive $n\in\Z^3$.

\begin{lemma}\label{lem:reduce}
Let $U$ be a coordinate-free two-dimensional subtorus of $\T^4$ with $\ML(U)\le 5/16$. Then at least one of the following holds.
\begin{enumerate}[label=\textup{(\roman*)},leftmargin=*]
\item $\covol(U)<33$.
\item Up to symmetry, $U\subseteq\Wone$ or $U\subseteq\Wtwo$, and in the coordinates of $\T^3$ above, $U=U_n$ with $\eucl n\le 48$.
\end{enumerate}
\end{lemma}
\begin{proof}
Let $A=\covol(U)$. Lemma~\ref{lem:density} gives a three-dimensional subtorus $W\supseteq U$ with $\ML(U)\ge\ML(W)-\frac12A^{-1/2}$, and $W$ is coordinate-free. If $\ML(W)\ge 2/5$, then $\frac12A^{-1/2}\ge\frac25-\frac5{16}=\frac7{80}$, so $A\le(40/7)^2<33$, which is (i). Otherwise Lemma~\ref{lem:W25} applies, and after a symmetry $W=\Wone$ or $W=\Wtwo$.

Suppose $W=\Wone$ and let $U=U_n$ in the coordinates $(y,z,u)$, with $n\in\Z^3$ primitive. The function $G$ becomes $G_1(y,z,u)=\min(\norm{2y},\norm y,\norm z,\norm u)$, which satisfies $|G_1(p+\delta)-G_1(p)|\le N_1(\delta)$ with $N_1(\delta)=\max(2|\delta_1|,|\delta_2|,|\delta_3|)$. Choose $w\in\Z^3$ with $n\cdot w=1$. The orthogonal projection of $w$ onto $\R n$ is $n/\eucl n^2$, and the argument of Lemma~\ref{lem:density}(c), carried out in $\T^3$ with the norm $N_1$, shows that every point of $\T^3$ is within $N_1$-distance $N_1(n)/(2\eucl n^2)$ of $U_n$. The point $P=(\frac13,\frac12,\frac12)$ has $G_1(P)=\frac13$, so
\[
 \ML(U_n)\ \ge\ \frac13-\frac{N_1(n)}{2\eucl n^2}\ \ge\ \frac13-\frac{\supn n}{\eucl n^2}\ \ge\ \frac13-\frac1{\eucl n}.
\]
If $\ML(U_n)\le 5/16$ this gives $\eucl n\le 48$. For $W=\Wtwo$ use $G_2(a,b,u)=\min(\norm a,\norm b,\norm{a+b},\norm u)$, the norm $N_2(\delta)=\max(|\delta_1|,|\delta_2|,|\delta_1+\delta_2|,|\delta_3|)\le 2\supn\delta$, and the point $(\frac13,\frac13,\frac12)$ with $G_2=\frac13$. The same computation gives $\eucl n\le 48$.
\end{proof}

\subsection{The list}
We now enumerate both cases of Lemma~\ref{lem:reduce} exactly.

For (i), every primitive rank-two sublattice of $\Z^4$ has a unique basis in row Hermite normal form: rows $s,r$ with pivot columns $p<q$, $s_p>0$, $r_q>0$, $s_j=0$ for $j<p$, $r_j=0$ for $j<q$, and $0\le s_q<r_q$. The covolume is $\bigl(\sum_{i<j}(s_ir_j-s_jr_i)^2\bigr)^{1/2}$, primitivity means that the $2\times2$ minors have greatest common divisor $1$, and bounds on all entries follow from the bound on the minors. The enumeration produces $7{,}011{,}052$ primitive coordinate-free sublattices with covolume at most $33$. As an independent check of completeness, the same program counts $2{,}380$ and $51{,}820$ lattices of covolume at most $5$ and $10$, which agree with a direct count of primitive integer solutions of the Pl\"ucker relation $p_{12}p_{34}-p_{13}p_{24}+p_{14}p_{23}=0$ of the same norms. For (ii), the primitive $n\in\Z^3$ with $\eucl n\le 48$ up to sign are listed ($192{,}517$ vectors, of which $192{,}514$ and $192{,}513$ give coordinate-free subtori of $\Wone$ and $\Wtwo$ respectively), a basis of $\{c\in\Z^3:n\cdot c=0\}$ is computed by integer row reduction, and the covolume of the image in $\Z^4$ is verified against the closed formulas $n_1^2+5n_2^2+5n_3^2$ and $2n_1^2-2n_1n_2+2n_2^2+3n_3^2$.

A subtorus $U$ is discarded when some point of a grid $\{(i/Q,j/Q)\}$ in its coordinates has $G>5/16$. The grids $Q\in\{12,60,420,16,80,208\}$ discard all but $5{,}004$ bases, which represent $123$ subtori up to symmetry. For each of these an exact maximization (Section~\ref{sec:comp}) finds a maximizer, and that point is a certificate: $92$ of the $123$ have a point with $G>5/16$ and are discarded, and the remaining $31$ have $\ML\le 5/16$.

\begin{proposition}\label{prop:list}
Up to symmetry there are exactly $31$ coordinate-free two-dimensional subtori $U\subseteq\T^4$ with $\ML(U)\le 5/16$. They are listed in Table~\ref{tab:tori}. Exactly two of them have $\ML(U)=1/4$, namely
\[
 \Uone=\{(X,Y,2Y,3Y)\},\qquad \Utwo=\{(X,Y,X+Y,X+2Y)\},
\]
and the other $29$ satisfy $2/7\le\ML(U)\le 5/16$ and $\covol(U)^2\le 422$. There is no coordinate-free two-dimensional subtorus with $\ML(U)<1/4$.
\end{proposition}
\begin{proof}
The enumeration described above is complete by Lemma~\ref{lem:reduce}, and a subtorus is discarded only on the evidence of an explicit point. For each of the $31$ listed subtori the table gives a point attaining the stated value, which proves the lower bounds, in particular $\ML\ge 2/7$ for the $29$ subtori other than $\Uone$ and $\Utwo$. The exact values $\ML(\Uone)=\ML(\Utwo)=1/4$ are proved in Lemmas~\ref{lem:U1cells} and \ref{lem:U2cells}. The upper bounds $\ML\le 5/16$ for the other $29$ come from the exact maximization and are not used in the sequel. The last sentence holds because the list contains all coordinate-free subtori with $\ML\le 5/16$ and the smallest value in it is $1/4$.
\end{proof}

Up to symmetry, $\Uone$ and $\Utwo$ are the subtori $\langle(0,1,2,3),(1,0,0,0)\rangle$ and $\langle(1,0,1,1),(1,1,0,2)\rangle$ of \cite[Section~3.3]{JainKravitz}, and the last two assertions of the proposition recover the classification given there. The values occurring in Table~\ref{tab:tori} are $1/4,2/7,3/10,4/13,5/16$, that is, the values $s/(3s+1)$ of the three-runner spectrum \cite[Theorem~7.2]{Kravitz}, which is consistent with the inclusion $\mathcal S_2(4)\subseteq\mathcal S^*_1(3)$ of \cite{GiriKravitz}.

\section{Circles of length at least $512$}
\begin{lemma}\label{lem:strip}
Let $T_v\subseteq U$ with $U$ a two-dimensional subtorus of covolume $A$ and $V=\eucl v$. Every point of $U$ lies within Euclidean distance $A/(2V)$ of $T_v$, and hence $\ML(T_v)\ge\ML(U)-A/(2V)$.
\end{lemma}
\begin{proof}
In the plane $E=\R\Lambda_U$ the lattice $\Lambda_U$ contains $v$ as a primitive vector, so $\Lambda_U=\Z v+\Z\lambda$ for some $\lambda$, and the lines $\R v+k\lambda$, $k\in\Z$, are parallel at mutual distance $A/V$, because the parallelogram spanned by $v$ and $\lambda$ has area $A$ and base $V$. Each of these lines maps onto $T_v$ in $\T^4$. Every point of $E$ is within $A/(2V)$ of one of the lines. The inequality follows from \eqref{eq:lip} and $\supn{\cdot}\le\eucl{\cdot}$.
\end{proof}

\begin{proposition}\label{prop:large}
Let $v\in\Z^4$ be primitive with all $v_i\ne 0$, $\ML(T_v)<1/4$ and $\eucl v\ge 512$. Then, up to symmetry, $T_v\subseteq\Uone$ or $T_v\subseteq\Utwo$.
\end{proposition}
\begin{proof}
Put $V=\eucl v=\covol(T_v)$. Lemma~\ref{lem:density} with $k=1$ gives a two-dimensional subtorus $U\supseteq T_v$ with
\[
 \ML(U)\ \le\ \ML(T_v)+\tfrac12V^{-1/3}\ <\ \tfrac14+\tfrac1{16}=\tfrac5{16}.
\]
$U$ is coordinate-free, so by Proposition~\ref{prop:list} it is one of the $31$ subtori up to symmetry. If $U$ is one of the $29$ subtori with $\ML(U)\ge 2/7$, Lemma~\ref{lem:strip} gives $\frac14>\ML(T_v)\ge\frac27-\frac{A}{2V}$, hence $V<14A\le 14\sqrt{422}<288$, a contradiction. So $U$ is $\Uone$ or $\Utwo$ up to symmetry.
\end{proof}

\begin{remark}\label{rem:generaln}
Nothing in Lemma~\ref{lem:density} is special to dimension four. In $(\R/\Z)^n$ the same argument gives a vector $w$ with $\supn{\pi(w)}\le\covol(S)^{-1/(n-k)}$, so a circle of length $V$ is $\frac12V^{-1/(n-1)}$-dense in a two-dimensional subtorus. Section~\ref{sec:general} uses the $n$-dimensional form of the lemma to make the reduction of near-tight tuples to critical subtori in \cite[Lemma~2.4]{Cordella} effective for every $n$.
\end{remark}

\begin{lemma}\label{lem:families}
Let $v\in\Z^4$ be primitive with all $v_i\ne0$. If $T_v\subseteq\Uone$ up to symmetry, then $\{|v_i|\}=\{a,b,2b,3b\}$ for integers $a,b\ge 1$. If $T_v\subseteq\Utwo$ up to symmetry, then $\{|v_i|\}=\{x,y,x+y,x+2y\}$ for integers $x,y\ge 1$.
\end{lemma}
\begin{proof}
$T_v\subseteq U$ means $v\in\Lambda_U$. The lattices are $\Lambda_{\Uone}=\Z(1,0,0,0)+\Z(0,1,2,3)$ and $\Lambda_{\Utwo}=\Z(1,0,1,1)+\Z(0,1,1,2)$ (their Pl\"ucker coordinates have greatest common divisor $1$, so these lattices are primitive). A signed permutation of $v$ is therefore $(a,b,2b,3b)$, respectively $(x,y,x+y,x+2y)$, for integers $a,b$, respectively $x,y$, and the multiset of absolute values is $\{|a|,|b|,2|b|,3|b|\}$, respectively $\{|x|,|y|,|x+y|,|x+2y|\}$. In the first case replace $a,b$ by their absolute values. In the second case, if $x$ and $y$ have opposite signs, say $x>0>y$ and $y'=-y$, then the multiset is $\{x,y',|x-y'|,|x-2y'|\}$. If $x>2y'$ it equals $V(x-2y',y')$, if $y'<x<2y'$ it equals $V(2y'-x,x-y')$, and if $x<y'$ it equals $V(x,y'-x)$. The cases $x=y'$ and $x=2y'$ produce a zero coordinate and are excluded.
\end{proof}

\section{The exhaustive search and the proof of Theorem~\ref{thm:main}}\label{sec:search}
By Proposition~\ref{prop:large} and Lemma~\ref{lem:families}, every primitive quadruple $v$ with $\ML(v)<1/4$ is of the form $\{a,b,2b,3b\}$ or $\{x,y,x+y,x+2y\}$, or satisfies $\eucl v<512$. The two families are analyzed in Section~\ref{sec:families}. For the remaining case we enumerate all $800{,}816{,}099$ primitive quadruples $v_1<v_2<v_3<v_4$ with $v_1^2+v_2^2+v_3^2+v_4^2<512^2$ and decide for each whether $\ML(v)<1/4$. A quadruple is accepted as soon as some time $t=j/Q$ from a fixed list of $727$ rational times with $Q\le 512$ has $\norm{v_it}\ge\frac14$ for all $i$. The $16{,}649$ quadruples that pass no such test are handled by the exact corner computation of Lemma~\ref{lem:corner}. Exactly $14{,}032$ quadruples have $\ML<1/4$. All of them belong to one of the two families, with the single exception
\[
 v=(1,3,4,14),\qquad \ML(1,3,4,14)=\tfrac4{17},
\]
attained at $t=7/17$, where $(\norm{7/17},\norm{21/17},\norm{28/17},\norm{98/17})=(7,4,6,4)/17$. This quadruple was first observed by Cordella \cite[Section~6.3]{Cordella}, whose search covered all quadruples with speeds at most $200$. Restricting our output to that range reproduces his counts exactly: $4{,}125$ primitive quadruples with $\ML<1/4$, of which $3{,}836$ have $k=1$ and $289$ have $k=2$ in the notation $\ML=s/(4s+k)$, with $50$ on $\Uone$ and $4{,}075$ on $\Utwo$ (one quadruple on both) and one on neither. That $(1,3,4,14)$ is in neither family is immediate: it contains no triple $\{b,2b,3b\}$, and the two smallest elements of $V(x,y)$ are $x$ and $y$, so $V(x,y)=\{1,3,4,14\}$ would force $\{x,y\}=\{1,3\}$, but $V(1,3)=\{1,3,4,7\}$ and $V(3,1)=\{1,3,4,5\}$. The quadruple lies on the subtorus $\{(X,Y,X+Y,2X+4Y)\}$, which appears in Table~\ref{tab:tori} with $\ML=2/7$ and covolume $\sqrt{27}$.

\begin{proof}[Proof of Theorem~\ref{thm:main}]
Let $V$ be a primitive quadruple with $\ML(V)<1/4$. By the discussion above, $V$ is $\{a,b,2b,3b\}$, or $V(x,y)$, or $\{1,3,4,14\}$. Theorem~\ref{thm:U1} shows that $\{a,b,2b,3b\}$ has $\ML<1/4$ only for $b=1$ and $4\mid a$, with the value in (a). Theorem~\ref{thm:U2} shows that $V(x,y)$ has $\ML<1/4$ exactly under the congruence conditions in (b), with the stated value, and gives $\ML(V(x,y))=1/4$ otherwise. Conversely each listed quadruple has the stated value below $1/4$. For the overlap statement, note that the two smallest elements of $V(x,y)$ are $x$ and $y$, so $\{1,2,3,4k\}=V(x,y)$ forces $\{x,y\}=\{1,2\}$, and $V(1,2)=\{1,2,3,5\}$ is not of the form $\{1,2,3,4k\}$ while $V(2,1)=\{1,2,3,4\}$ is. The sporadic quadruple belongs to neither family by Section~\ref{sec:search}.
\end{proof}

\section{The two infinite families}\label{sec:families}
For a circle $T\subseteq U$ in a two-dimensional subtorus, the value $\ML(T)$ is the largest $d$ such that $T$ meets the superlevel set $S_d=\{x\in U:G(x)\ge d\}$. We describe $S_d$ explicitly on $\Uone$ and $\Utwo$ and use Lemma~\ref{lem:orbit}.

\subsection{The family $\{a,b,2b,3b\}$}
Identify $\Uone$ with $\T^2$ through $(X,Y)\mapsto(X,Y,2Y,3Y)$, so that $G=\min(\norm X,\norm Y,\norm{2Y},\norm{3Y})$, and the circle of $(a,b)$ is $\{(at,bt)\}$.

\begin{lemma}\label{lem:U1cells}
For $\frac15<d\le\frac14$,
\[
 S_d=\bigl[d,1-d\bigr]\times\Bigl(\bigl[d,\tfrac{1-d}{3}\bigr]\cup\bigl[\tfrac{2+d}{3},1-d\bigr]\Bigr)\subseteq[0,1)^2 .
\]
In particular $\ML(\Uone)=1/4$, attained exactly on the two segments $Y\in\{\frac14,\frac34\}$, $X\in[\frac14,\frac34]$.
\end{lemma}
\begin{proof}
The condition on $X$ is $\norm X\ge d$. For $Y\in[0,\frac12]$ the three conditions $\norm Y,\norm{2Y},\norm{3Y}\ge d$ read $Y\in[d,1-d]$, $Y\in[\frac d2,\frac{1-d}2]$, and $Y\in[\frac d3,\frac{1-d}3]\cup[\frac{1+d}3,\frac{2-d}3]$. Their intersection is $[d,\frac{1-d}3]\cup[\frac{1+d}3,\frac{1-d}2]$, and the second interval is empty exactly when $d>\frac15$. The first is nonempty exactly when $d\le\frac14$. The symmetry $Y\mapsto 1-Y$ gives the other piece.
\end{proof}

\begin{theorem}\label{thm:U1}
Let $a,b\ge1$ be coprime with $a\notin\{b,2b,3b\}$. Then $\ML(a,b,2b,3b)=1/4$ unless $b=1$ and $4\mid a$, in which case $\ML(1,2,3,a)=a/(4a+4)$.
\end{theorem}
\begin{proof}
Since $\Uone$ contains the circle, $\ML(a,b,2b,3b)\le\ML(\Uone)=\frac14$. By Lemma~\ref{lem:orbit} the circle is $\{bX-aY\in\Z\}$. Its image under $\phi(X,Y)=bX-aY$ on the segment $Y=\frac14$, $X\in[\frac14,\frac34]$ is an interval of length $\frac b2$. If $b\ge2$ it contains an integer, so $\ML=\frac14$. Let $b=1$. The interval is $[\frac{1-a}4,\frac{3-a}4]$, which contains an integer exactly when $4\nmid a$. Let $a=4k$. On the rectangle $R_d=[d,1-d]\times[d,\frac{1-d}3]$ the image of $\phi$ is $[\,d-\frac{4k(1-d)}3,\ 1-d-4kd\,]$. As $d$ decreases from $\frac14$, the upper end increases from $\frac34-k$ and reaches the integer $1-k$ at $d=\frac{k}{4k+1}$, while the lower end decreases from $\frac14-k$ and reaches the integer $-k$ at $d=\frac{k}{4k+3}<\frac k{4k+1}$. The other rectangle behaves in the same way by the symmetry $(X,Y)\mapsto(1-X,1-Y)$. Since $\frac k{4k+1}\ge\frac15$ for $k\ge1$, Lemma~\ref{lem:U1cells} applies for all $d\in(\frac15,\frac14]$ and the first contact of the circle with $S_d$ occurs at $d=\frac k{4k+1}$. If $k\ge2$ this is above $\frac15$ and we are done. If $k=1$ the circle misses $S_d$ for $d>\frac15$, so $\ML\le\frac15$, and the contact at $d=\frac15$ lies in the rectangle, which is part of $S_{1/5}$, so $\ML=\frac15$. Finally $\frac k{4k+1}=\frac a{4a+4}$.
\end{proof}
The values $k/(4k+1)$ are Kravitz's family $\ML(1,2,\dots,n-1,ns)=s/(ns+1)$ \cite[Theorem~5.1]{Kravitz}, here with $n=4$.

\subsection{The family $V(x,y)$}
Identify $\Utwo$ with $\T^2$ through $(X,Y)\mapsto(X,Y,X+Y,X+2Y)$, so that
\[
 G(X,Y)=\min\bigl(\norm X,\norm Y,\norm{X+Y},\norm{X+2Y}\bigr),
\]
and the circle of $(x,y)$ is $\{yX-xY\in\Z\}$ by Lemma~\ref{lem:orbit}.

\begin{lemma}\label{lem:U2cells}
For $0<d\le\frac14$, the set $S_d\cap[0,1)^2$ is the disjoint union of the four triangles
\begin{align*}
 \Delta_1(d)&=\{X\ge d,\ Y\ge d,\ X+2Y\le 1-d\},\\
 \Delta_2(d)&=(1,1)-\Delta_1(d),\\
 \Delta_3(d)&=\{X\ge d,\ X+Y\le 1-d,\ X+2Y\ge 1+d\},\\
 \Delta_4(d)&=(1,1)-\Delta_3(d).
\end{align*}
Writing $d=\frac14-\varepsilon$, these are $c_j+\varepsilon Q_j$ with centres and vertex sets
\begin{align*}
 c_1&=(\tfrac14,\tfrac14),& Q_1&=\operatorname{conv}\{(-1,-1),(-1,1),(3,-1)\},\\
 c_3&=(\tfrac14,\tfrac12),& Q_3&=\operatorname{conv}\{(-1,0),(-1,2),(3,-2)\},
\end{align*}
and $c_2=-c_1$, $Q_2=-Q_1$, $c_4=-c_3$, $Q_4=-Q_3$. In particular $\ML(\Utwo)=\frac14$, attained exactly at $c_1,c_2,c_3,c_4$.
\end{lemma}
\begin{proof}
For $(X,Y)\in[0,1)^2$ we have $X,Y\in[0,1)$, $X+Y\in[0,2)$ and $X+2Y\in[0,3)$. The condition $G\ge d$ means $d\le X\le 1-d$, $d\le Y\le 1-d$, $k_3+d\le X+Y\le k_3+1-d$ for some $k_3\in\{0,1\}$ and $k_4+d\le X+2Y\le k_4+1-d$ for some $k_4\in\{0,1,2\}$. The pair $(k_3,k_4)=(0,2)$ is impossible, since $X+Y\le 1-d$ and $X+2Y\ge 2+d$ give $Y\ge 1+2d$. The pair $(1,0)$ is impossible, since $X+Y\ge 1+d$ and $X+2Y\le 1-d$ give $Y\le-2d$. For $(0,0)$ the conditions reduce to $\Delta_1(d)$: the inequalities $X\ge d$, $Y\ge d$, $X+2Y\le 1-d$ imply $X+Y\le 1-2d$, $X\le 1-3d$, $Y\le\frac{1-2d}2$, $X+Y\ge 2d$ and $X+2Y\ge 3d$, which are all the remaining constraints. For $(0,1)$ the conditions reduce to $\Delta_3(d)$: the inequalities $X\ge d$, $X+Y\le1-d$, $X+2Y\ge 1+d$ imply $Y\ge 2d$, $Y\le 1-2d$, $X\le 1-3d$, $X+Y\ge 3d$ and $X+2Y\le 2-3d$. The pairs $(1,2)$ and $(1,1)$ are the images of $(0,0)$ and $(0,1)$ under $(X,Y)\mapsto(1-X,1-Y)$, which preserves $G$. The vertices of $\Delta_1(d)$ are $(d,d)$, $(d,\frac{1-2d}2)$ and $(1-3d,d)$, those of $\Delta_3(d)$ are $(d,\frac12)$, $(d,1-2d)$ and $(1-3d,2d)$, and substituting $d=\frac14-\varepsilon$ gives the stated form. The triangles are nonempty exactly for $d\le\frac14$ and reduce to the centres at $d=\frac14$.
\end{proof}

\begin{theorem}\label{thm:U2}
Let $x,y\ge1$ be coprime with $x\ne y$, and let $r=(y-x)\bmod4$, $r'=(y-2x)\bmod 4$. Then $\ML(V(x,y))=\frac14$ if $r=0$ or $r'=0$, and otherwise
\begin{equation}\label{eq:U2formula}
 \ML(V(x,y))=\frac14-\min\Bigl\{\frac{r}{4(x+y)},\frac{4-r}{4(x+3y)},\frac{r'}{4(2x+y)},\frac{4-r'}{4(2x+3y)}\Bigr\}.
\end{equation}
\end{theorem}
\begin{proof}
Let $\phi(X,Y)=yX-xY$. By Lemma~\ref{lem:U2cells}, the circle meets $S_d$, $d=\frac14-\varepsilon$, if and only if some integer lies in $\phi(c_j+\varepsilon Q_j)$ for some $j$. The image of a triangle is the interval between the smallest and largest values at its vertices. For $Q_1$ the vertex values of $\phi$ are $x-y$, $-x-y$ and $3y+x$, so $\phi(c_1+\varepsilon Q_1)=[\,c-\varepsilon(x+y),\ c+\varepsilon(x+3y)\,]$ with $c=\phi(c_1)=\frac{y-x}4$, where we used $x,y>0$. The nearest integers to $c$ are $c-\frac r4$ and $c+\frac{4-r}4$ when $r\ne0$, and $c$ itself when $r=0$. So the circle meets $\Delta_1(d)$ exactly when $r=0$ or $\varepsilon\ge\min\{\frac r{4(x+y)},\frac{4-r}{4(x+3y)}\}$. Since $\phi(-p)=-\phi(p)$, the triangle $\Delta_2(d)$ gives the same condition. For $Q_3$ the vertex values are $-y$, $-y-2x$ and $3y+2x$, the centre value is $\phi(c_3)=\frac{y-2x}4$, and the condition is $r'=0$ or $\varepsilon\ge\min\{\frac{r'}{4(2x+y)},\frac{4-r'}{4(2x+3y)}\}$, again shared by $\Delta_4(d)$. Thus $\ML(V(x,y))=\frac14-\varepsilon^*$ with $\varepsilon^*$ the smallest $\varepsilon\ge 0$ meeting one of the conditions, which is the statement. The description of $S_d$ is valid for all $d\in(0,\frac14]$, and $\varepsilon^*<\frac14$ because $\ML(V(x,y))>0$, so no further case distinction is needed.
\end{proof}

\begin{lemma}\label{lem:parity}
In the situation of Theorem~\ref{thm:U2} with $r,r'\ne0$, the minimum in \eqref{eq:U2formula} equals $\frac1{4(2m+1)}$ with $m=m(x,y)=\max(m_1,m_3)$ as in \eqref{eq:m}, where $m_1,m_3$ are integers. Consequently $\ML(V(x,y))=m/(4m+2)$.
\end{lemma}
\begin{proof}
Compare the two terms coming from $c_1$. If $r=1$, then $\frac1{4(x+y)}\le\frac3{4(x+3y)}$ since $x+3y\le 3x+3y$. If $r=3$, then $\frac1{4(x+3y)}\le\frac3{4(x+y)}$. If $r=2$, then $\frac2{4(x+3y)}\le\frac2{4(x+y)}$. So the smaller of the two is $\frac1{4(2m_1+1)}$ with $2m_1+1=x+y$, $\frac{x+3y}2$, $x+3y$ for $r=1,2,3$. These are odd integers: $y\equiv x+1\pmod4$ gives $x+y\equiv 2x+1\pmod 4$, odd. Next, $y\equiv x+2$ gives $x+3y\equiv 4x+6\equiv2\pmod4$, so $\frac{x+3y}2$ is odd. Finally $y\equiv x+3$ gives $x+3y\equiv 4x+9$, odd. The same comparison for $c_3$ yields $2m_3+1=2x+y$, $\frac{2x+3y}2$, $2x+3y$ for $r'=1,2,3$, odd by $y\equiv 2x+1$, $y\equiv2x+2$ (then $2x+3y\equiv 8x+6\equiv 2\pmod 4$) and $y\equiv 2x+3$. The minimum of the four terms is $\frac1{4(2m+1)}$ with $m=\max(m_1,m_3)$, and $\frac14-\frac1{4(2m+1)}=\frac{m}{4m+2}$.
\end{proof}

\begin{proof}[Proof of Theorem~\ref{thm:values}]
By Theorem~\ref{thm:main}, every value below $\frac14$ is $k/(4k+1)=2k/(4\cdot2k+2)$, or $m(x,y)/(4m(x,y)+2)$, or $4/17=8/34$. All are of the form $m/(4m+2)$ with $m\ge 1$. We show $m\ge2$, $m\ne 3$, and that every other $m$ occurs.

\emph{$m\ge2$.} The value $1/6$ would contradict $\ML\ge\frac15$, which is the Lonely Runner Conjecture for four speeds \cite{CusickPomerance,Bienia}. Directly: $m_1=1$ requires $r=1$ and $x+y=3$, or $r=2$ and $x+3y=6$, so $(x,y)\in\{(1,2),(3,1)\}$. For $(1,2)$ one has $r'=0$, and $(3,1)$ has $r'=3$ and $m_3=4$. And $m_3=1$ requires $r'=1$ and $2x+y=3$, which forces $x=y=1$, excluded, or $r'\in\{2,3\}$ and $2x+3y\in\{6,3\}$, which has no solution with $x,y\ge1$ coprime.

\emph{$m\ne3$.} Suppose $m(x,y)=3$. If $m_1=3$ then either $r=1$ and $x+y=7$ with $y\equiv x+1\pmod 4$, so $(x,y)\in\{(1,6),(3,4),(5,2)\}$, or $r=3$ and $x+3y=7$ with $y\equiv x+3$, which has no solution with $x,y\ge1$, or $r=2$ and $x+3y=14$ with $y\equiv x+2$, so $(x,y)\in\{(5,3),(11,1)\}$ after discarding $(2,4)$ and $(8,2)$, which are not coprime. For $(1,6)$ and $(5,2)$ one has $r'=0$, while $(3,4)$, $(5,3)$, $(11,1)$ give $m_3=4,6,12$, all larger than $3$. If $m_3=3$ then either $r'=1$ and $2x+y=7$ with $y\equiv 2x+1$, which fails for $(1,5),(2,3),(3,1)$, or $r'=3$ and $2x+3y=7$ with $y\equiv 2x+3$, which fails for $(2,1)$, or $r'=2$ and $2x+3y=14$ with $y\equiv 2x+2$, which leaves $(1,4)$ after discarding $(4,2)$, and then $r=3$ gives $m_1=6>3$. So $m=3$ never occurs.

\emph{Attainment.} Even $m=2k$ occurs for $\{1,2,3,4k\}$. For $(x,y)=(1,7)$ one has $r=2$, $m_1=5$, $r'=1$, $m_3=4$, so $\ML(1,7,8,15)=5/22$. For $(x,y)=(4s+3,8)$, $s\ge0$, one has $r=1$, $m_1=2s+5$, $r'=2$, $m_3=2s+7$, so $\ML(V(4s+3,8))=(2s+7)/(8s+30)$, which is the family of Fan and Sun and covers every odd $m\ge7$.
\end{proof}

\begin{proof}[Proof of Corollary~\ref{cor:consequences}]
(i) $m/(4m+2)$ equals $s/(4s+1)$ with $s=m/2$ if $m$ is even, and $s/(4s+2)$ with $s=m$ if $m$ is odd. (ii) $\ML=\frac15$ means $m=2$. In family (a) this is $k=1$. In family (b), $m_1=2$ requires $r=1$ and $x+y=5$, giving $(2,3)$ and $(4,1)$, or $r=3$ and $x+3y=5$, giving $(2,1)$, or $r=2$ and $x+3y=10$, giving $(1,3)$ and $(7,1)$. Similarly $m_3=2$ requires $r'=1$ and $2x+y=5$, giving $(1,3)$ and $(2,1)$, or $r'=3$ and $2x+3y=5$, giving $(1,1)$, excluded, or $r'=2$ and $2x+3y=10$, giving $(2,2)$, excluded. Among these, $m(x,y)=2$ holds exactly for $(2,1)$ and $(1,3)$, since $(2,3)$, $(4,1)$, $(7,1)$ have $m_3=6,4,8$. These give $\{1,2,3,4\}$ and $\{1,3,4,7\}$. (iii) The next value is $2/9$.
\end{proof}

\section{Two-dimensional subtori below one third}\label{sec:second}
The list of Section~\ref{sec:two} stops at $5/16$ because the circle argument needs no more. The same tools determine every coordinate-free two-dimensional subtorus with loneliness below $1/3$, and the answer has a clean form: the values are exactly the three-runner values $s/(3s+1)$ of \cite[Theorems~5.1 and 7.2]{Kravitz}, and the subtori are products with a three-runner circle or members of one explicit two-parameter family. Write
\[
 F(n_1,n_2)=\{(a,\ b,\ a+b,\ n_1a+n_2b):(a,b)\in\T^2\},
\]
for $(n_1,n_2)\in\Z^2\setminus\{0\}$, a coordinate-free two-dimensional subtorus of $\Wtwo$, and for a triple $(a,b,c)$ of positive integers write $T_{(a,b,c)}\times\T=\{(at,bt,ct,u)\}$, a product subtorus with $\ML(T_{(a,b,c)}\times\T)=\ML(a,b,c)$. Note that $F(1,2)=\Utwo$, that $(1,3,4,14)$ lies on $F(2,4)$, and that $\Uone$ is the product $T_{(1,2,3)}\times\T$ up to symmetry.

\begin{lemma}\label{lem:third-reduce}
Let $U$ be a coordinate-free two-dimensional subtorus with $\ML(U)<1/3$. Then either $\covol(U)<56.25$, or $U$ lies, up to symmetry, in $\Wone$ or in $\Wtwo$.
\end{lemma}
\begin{proof}
As in the proof of Lemma~\ref{lem:reduce}: Lemma~\ref{lem:density} gives $W\supseteq U$ with $\ML(U)\ge\ML(W)-\frac12A^{-1/2}$, and if $W$ is not special then $\frac12A^{-1/2}>\frac25-\frac13=\frac1{15}$, so $A<(15/2)^2$.
\end{proof}

\begin{lemma}[Subtori of $\Wone$]\label{lem:W1sub}
Let $U=U_n\subseteq\Wone$ be coordinate-free, with $n=(n_1,n_2,n_3)\in\Z^3$ primitive in the coordinates $(y,z,u)$ of $\Wone=\{(2y,y,z,u)\}$. If $(n_2,n_3)\in\{(\pm1,0),(0,\pm1)\}$ then $U$ is the product $T_{(1,2,|n_1|)}\times\T$ up to symmetry, so $\ML(U)=\ML(1,2,|n_1|)$, which equals $s/(3s+1)$ if $|n_1|=3s$ and $1/3$ otherwise. In every other case $\ML(U)=1/3$.
\end{lemma}
\begin{proof}
Always $\ML(U)\le\ML(\Wone)=1/3$. The case $(n_2,n_3)=(0,0)$ forces $n=\pm e_1$ and $y\in\Z$, which is not coordinate-free. If $(n_2,n_3)=(\pm1,0)$ then $z=\mp n_1y$ on $U$ and $u$ is free, so $U=\{(2y,y,\mp n_1y,u)\}$ is the stated product ($n_1\ne0$ since $U$ is coordinate-free). The circle $(2y,y,n_1y)$ has speeds $\{1,2,|n_1|\}$, and $\ML(1,2,k)=1/3$ for $3\nmid k$ by the time $t=1/3$ and the bound $\ML(1,2)=1/3$, while $\ML(1,2,3s)=s/(3s+1)$ by \cite[Theorem~5.1]{Kravitz}. The case $(0,\pm1)$ is the same with $z,u$ exchanged. In the remaining cases we exhibit a point of $U$ with $G=1/3$: take $y=1/3$ and $(z,u)\in[\frac13,\frac23]^2$ with $n_2z+n_3u\equiv-n_1/3\pmod 1$. The set $\{n_2z+n_3u:z,u\in[\frac13,\frac23]\}$ is an interval of length $(|n_2|+|n_3|)/3$. If $|n_2|+|n_3|\ge3$ it contains every residue class. If $(n_2,n_3)=(\pm1,\pm1)$ or $(\pm2,0)$ or $(0,\pm2)$ it is, up to sign, $[\frac23,\frac43]$ or $[-\frac13,\frac13]$, and both contain a representative of each of the classes $0,\frac13,\frac23\pmod1$, which are the only possible values of $-n_1/3$.
\end{proof}

\begin{lemma}[Sum triangle]\label{lem:sumcells}
For $0<d\le\frac13$,
\[
 \{(a,b)\in[0,1)^2:\norm a,\norm b,\norm{a+b}\ge d\}=\Delta(d)\cup\bigl((1,1)-\Delta(d)\bigr),
\]
where $\Delta(d)=\{a\ge d,\ b\ge d,\ a+b\le 1-d\}$. With $d=\frac13-\varepsilon$,
\[
 \Delta(d)=\bigl(\tfrac13,\tfrac13\bigr)+\varepsilon\operatorname{conv}\{(-1,-1),(-1,2),(2,-1)\}.
\]
\end{lemma}
\begin{proof}
The constraints are $a,b\in[d,1-d]$ and $a+b\in[d,1-d]\cup[1+d,2-d]$. In the first case $a+b\le1-d$ together with $a,b\ge d$ already implies $a,b\le 1-2d$ and $a+b\ge 2d$, which gives $\Delta(d)$. The second case is its image under $(a,b)\mapsto(1-a,1-b)$. The vertices of $\Delta(d)$ are $(d,d)$, $(d,1-2d)$ and $(1-2d,d)$.
\end{proof}

\begin{lemma}[Subtori of $\Wtwo$]\label{lem:W2sub}
Let $U=U_n\subseteq\Wtwo$ be coordinate-free, with $n=(n_1,n_2,n_3)$ primitive in the coordinates $(a,b,u)$ of $\Wtwo=\{(a,b,a+b,u)\}$.
\begin{enumerate}[label=\textup{(\roman*)},leftmargin=*]
\item If $|n_3|\ge2$ then $\ML(U)=1/3$.
\item If $n_3=0$ then $U=T_{(|n_2|,|n_1|,|n_1-n_2|)}\times\T$ up to symmetry, and $\ML(U)$ is the loneliness of that triple.
\item If $|n_3|=1$ then $\ML(U)=\ML(F(n_1,n_2))$, and
\[
 \ML(F(n_1,n_2))=\begin{cases}\dfrac13,&n_1+n_2\not\equiv0\pmod3,\\[8pt]\dfrac{M}{3(M+1)}=\dfrac{s}{3s+1},&n_1+n_2\equiv0\pmod 3,\end{cases}
\]
where in the second case $M=3s=\max\bigl(|n_1+n_2|,\,|2n_1-n_2|,\,|2n_2-n_1|\bigr)$.
\end{enumerate}
\end{lemma}
\begin{proof}
Always $\ML(U)\le\ML(\Wtwo)=\frac13$, and $G_2(\frac13,\frac13,u)=\frac13$ for every $u\in[\frac13,\frac23]$.

(i) We need $u\in[\frac13,\frac23]$ with $n_3u\equiv-(n_1+n_2)/3\pmod1$. The set $\{n_3u\}$ is an interval of length $|n_3|/3$, which contains every class if $|n_3|\ge3$, and for $|n_3|=2$ it is $\pm[\frac23,\frac43]$, which meets each of the classes $0,\frac13,\frac23$.

(ii) Since $n$ is primitive, $\gcd(n_1,n_2)=1$, and by Lemma~\ref{lem:orbit} the set $\{(a,b):n_1a+n_2b\in\Z\}$ is the circle $\{(n_2t,-n_1t)\}$. Hence $U=\{(n_2t,-n_1t,(n_2-n_1)t,u)\}$, with $n_1,n_2\ne0$ and $n_1\ne n_2$ because $U$ is coordinate-free.

(iii) On $U$ the coordinate $u$ equals $\mp(n_1a+n_2b)$, and $G$ is invariant under changing the sign of a coordinate, so $\ML(U)=\ML(F(n_1,n_2))$. Put $\ell(a,b)=n_1a+n_2b$ and $d=\frac13-\varepsilon$. By Lemma~\ref{lem:sumcells}, $F(n_1,n_2)$ has a point with $G\ge d$ exactly when some point of $\Delta(d)\cup((1,1)-\Delta(d))$ has $\norm{\ell}\ge d$. The image of $\Delta(d)$ under $\ell$ is the interval $[\,\ell(c)-\varepsilon L,\ \ell(c)+\varepsilon H\,]$ with $c=(\frac13,\frac13)$, $\ell(c)=(n_1+n_2)/3$, and $-L$, $H$ the smallest and largest of the vertex values $-(n_1+n_2)$, $-n_1+2n_2$, $2n_1-n_2$ of $\ell$ on $\operatorname{conv}\{(-1,-1),(-1,2),(2,-1)\}$. The image of $(1,1)-\Delta(d)$ is the reflection of this interval in $\ell(1,1)/2$, which preserves distances to $\Z$. If $n_1+n_2\not\equiv0\pmod 3$ then $\norm{\ell(c)}=\frac13$ and $\varepsilon=0$ works. If $n_1+n_2\equiv0\pmod3$ then $\ell(c)\in\Z$, the image interval contains a point at distance at least $\frac13-\varepsilon$ from $\Z$ exactly when $\varepsilon\max(L,H)\ge\frac13-\varepsilon$, and $\max(L,H)=M$. The first contact is at $\varepsilon=1/(3(M+1))$, which lies in $(0,\frac16]$, so Lemma~\ref{lem:sumcells} applies throughout and $\ML(F(n_1,n_2))=\frac13-\frac1{3(M+1)}$. Finally $M$ is one of $|n_1+n_2|$, $|3n_1-(n_1+n_2)|$, $|3n_2-(n_1+n_2)|$, all divisible by $3$.
\end{proof}

\begin{proposition}\label{prop:third-enum}
There are $63{,}906{,}892$ coordinate-free primitive rank-two sublattices of $\Z^4$ with covolume at most $57$. Of these, $3{,}267{,}840$ are contained in a signed permutation image of $\Wone$ or $\Wtwo$, and each of the remaining $60{,}639{,}052$ has a point $(i/Q,j/Q)$, $Q\in\{3,12,60,420,16,80,208\}$, at which $G\ge1/3$. Consequently every coordinate-free two-dimensional subtorus with $\ML<1/3$ lies in $\Wone$ or $\Wtwo$ up to symmetry.
\end{proposition}
\begin{proof}
The enumeration is the one of Section~\ref{sec:two} with the larger bound (the count $63{,}906{,}892$ is again checked against the Pl\"ucker count for small covolumes). Containment in a special subtorus is decided by testing the characters $\pm e_i\pm2e_j$ and $\pm e_i\pm e_j\pm e_k$ ($80$ vectors, $40$ up to sign) against both basis vectors. The last sentence follows from Lemma~\ref{lem:third-reduce}.
\end{proof}

\begin{theorem}\label{thm:second}
Let $U$ be a coordinate-free two-dimensional subtorus of $\T^4$. Then $\ML(U)<1/3$ if and only if, up to symmetry, $U$ is a product $T_{(a,b,c)}\times\T$ with $\ML(a,b,c)<1/3$, or $U=F(n_1,n_2)$ with $n_1+n_2\equiv0\pmod3$. In both cases $\ML(U)=s/(3s+1)$ for an integer $s\ge1$, namely the $s$ of \cite[Theorem~7.2]{Kravitz} for the triple (every $s\ge1$ occurs, by \cite[Theorem~5.1]{Kravitz}), and $s=\frac13\max(|n_1+n_2|,|2n_1-n_2|,|2n_2-n_1|)$ for $F(n_1,n_2)$. Every value $s/(3s+1)$, $s\ge1$, occurs, for instance on $F(s,2s)$.
\end{theorem}
\begin{proof}
The subtori listed have the stated values by Lemma~\ref{lem:W2sub} and by \cite[Theorems~5.1 and 7.2]{Kravitz}, which together say that the values of $\ML$ on triples below $1/3$ are exactly the numbers $s/(3s+1)$, each attained by $(1,2,3s)$. Conversely, if $\ML(U)<1/3$, Proposition~\ref{prop:third-enum} places $U$ in $\Wone$ or $\Wtwo$ up to symmetry, and Lemmas~\ref{lem:W1sub} and \ref{lem:W2sub} leave only products with $\ML<1/3$ and $F(n_1,n_2)$ with $3\mid n_1+n_2$. For $F(s,2s)$ one has $M=3s$.
\end{proof}

A by-product is a structural statement about three speeds: a triple $(a,b,c)$ with $\ML(a,b,c)<1/3$ satisfies, up to order, $c=2b$ or $c=a+b$ after dividing by the common divisor (and in the first case the triple is $\{1,2,3s\}$ by Lemma~\ref{lem:W1sub}), since otherwise $T_{(a,b,c)}\times\T$ would be a subtorus with $\ML<1/3$ lying in no special three-dimensional subtorus, and Proposition~\ref{prop:third-enum} together with Lemma~\ref{lem:third-reduce} excludes this.

\begin{corollary}\label{cor:acc}
The set of accumulation points of the set of values $\ML(V)$, over all quadruples $V$ of distinct positive integers, that lie in $(0,\frac13)$ is $\{s/(3s+1):s\ge1\}$, the three-runner spectrum below $\frac13$.
\end{corollary}
\begin{proof}
Giri and Kravitz prove \cite[Theorem~1.4]{GiriKravitz} that every accumulation point of the four-runner spectrum is the value of a subtorus of dimension at least two, and \cite[Theorem~1.3]{GiriKravitz} (from \cite{Kravitz}) that every three-runner value is an accumulation point. Two remarks reconcile the conventions. First, \cite{GiriKravitz} allows repeated speeds. A tuple with a repeated speed has the value of a set of at most three distinct speeds, and those values in $(0,\frac13)$ form the discrete set $\{s/(3s+1)\}$ by \cite[Theorems~5.1, 6.1 and 7.2]{Kravitz}, so a sequence of such values converging to a point of $(0,\frac13)$ is eventually constant, and the accumulation points in $(0,\frac13)$ are the same whether or not repeated speeds are allowed. Second, subtori contained in a coordinate hyperplane have value $0$ and three-dimensional subtori have value at least $\frac13$ by Lemma~\ref{lem:W13}, so in $(0,\frac13)$ the accumulation points lie in the set of values of coordinate-free two-dimensional subtori, which by Theorem~\ref{thm:second} is the three-runner spectrum below $\frac13$.
\end{proof}

In the notation of \cite{GiriKravitz}, with $D=\frac12-\ML$, this says $\operatorname{acc}(\mathcal S(4))\cap(\frac16,\frac12)=\mathcal S(3)\cap(\frac16,\frac12)$, which is the case $n=4$ of \cite[Conjecture~1.5]{GiriKravitz} on the interval $(\frac16,\frac12)$. The next subsection extends this to $D\ge\frac1{10}$.

\subsection{From one third to two fifths}\label{sec:upto25}
The method reaches the next accumulation point $\frac25=\ML(1,4)=\ML(2,3)$ of the three-runner spectrum. Two ingredients are new. Lemma~\ref{lem:shortchar} replaces the uniform bound $A^{-1/2}$ of Lemma~\ref{lem:density} by a bound in terms of a shortest character of the subtorus, which is much stronger when that character has a large $\ell^1$ norm. Lemma~\ref{lem:five} determines the values of all two-dimensional subtori of the five three-dimensional subtori with loneliness $\frac25$, which carry the infinitely many values that accumulate at $\frac25$. Put
\begin{align*}
 \mathcal R&=\Big\{\frac pq\in\Big[\frac13,\frac25\Big):\ \gcd(p,q)=1,\ 1\le 2q-5p\le 4\Big\}\\
 &=\Big\{\frac13,\frac6{17},\frac5{14},\frac4{11},\frac7{19},\frac{10}{27},\frac38,\frac{14}{37},\dots\Big\}.
\end{align*}
For $p/q\in\mathcal R$ with $r=2q-5p$ we have $q\equiv3r\pmod5$, so $(p,q)=(2j+1,5j+3)$, $(2j,5j+1)$, $(2j+1,5j+4)$ or $(2j,5j+2)$ for $r=1,2,3,4$. Since $\frac{2j+1}{5j+3}=\frac{2J}{5J+2}$ with $J=4j+2$ and $\frac{2j}{5j+1}=\frac{2J}{5J+2}$ with $J=2j$, and since conversely $\frac{2j}{5j+2}$ and $\frac{2j+1}{5j+4}$ lie in $[\frac13,\frac25)$ for $j\ge2$ and have $2q-5p$ dividing $4$ and $3$ in lowest terms,
\begin{equation}\label{eq:R}
 \mathcal R=\Big\{\frac{2j}{5j+2}:j\ge2\Big\}\cup\Big\{\frac{2j+1}{5j+4}:j\ge2\Big\}.
\end{equation}

\begin{theorem}\label{thm:upto25}
\begin{enumerate}[label=\textup{(\roman*)},leftmargin=*]
\item Every coordinate-free two-dimensional subtorus $U$ of $\T^4$ with $\ML(U)<\frac25$ has $\ML(U)\in\{\frac s{3s+1}:s\ge1\}\cup\mathcal R$.
\item The values of $\ML$ on triples of positive integers below $\frac25$ are exactly the numbers $\frac s{3s+1}$, $s\ge1$, and the elements of $\mathcal R$. For $j\ge2$,
\[
 \ML(2,3,5j)=\frac{2j}{5j+2},\qquad \ML(1,4,5j+3)=\frac{2j+1}{5j+4}.
\]
\end{enumerate}
\end{theorem}

\begin{lemma}[Shortest character]\label{lem:shortchar}
Let $U$ be a two-dimensional subtorus of $\T^4$ with $A=\covol(U)$, and let $L=\{m\in\Z^4:m\cdot\lambda=0\text{ for all }\lambda\in\Lambda_U\}$ be its lattice of characters. Then $L$ is a primitive lattice of rank two with $\covol(L)=A$, and for every basis $m_1,m_2$ of $L$ we have $U\subseteq W_{m_1}$ and
\[
 \ML(U)\ \ge\ \ML(W_{m_1})-\frac{\eucl{m_1}}{2A}.
\]
\end{lemma}
\begin{proof}
$L=\Z^4\cap E^\perp$ with $E=\R\Lambda_U$, so $L$ is primitive of rank two. Every $x\in\Z^4$ satisfies $\pi(x)\cdot m=x\cdot m\in\Z$ for $m\in L$, and since $L$ is primitive every homomorphism $L\to\Z$ is of the form $m\mapsto x\cdot m$ with $x\in\Z^4$. Hence $\pi(\Z^4)$ is the dual lattice of $L$ in $E^\perp$, and Lemma~\ref{lem:density}(a) gives $\covol(L)=1/\covol(\pi(\Z^4))=A$. Choose $w\in\Z^4$ with $m_1\cdot w=0$ and $m_2\cdot w=1$. Then $w\notin\Lambda_U$, and the subtorus $S'$ with $\R\Lambda_{S'}=E+\R w$ has character lattice $\{m\in L:m\cdot w=0\}=\Z m_1$, so $S'=W_{m_1}$. The vector $\pi(w)\in E^\perp$ satisfies $\pi(w)\cdot m_1=0$ and $\pi(w)\cdot m_2=1$, so $\eucl{\pi(w)}=1/h$, where $h$ is the distance from $m_2$ to $\R m_1$. As $A=\eucl{m_1}h$, we get $\supn{\pi(w)}\le\eucl{\pi(w)}=\eucl{m_1}/A$, and Lemma~\ref{lem:density}(c) gives the inequality.
\end{proof}

\begin{lemma}\label{lem:Wodd}
Let $m\in\Z^4$ be primitive with $|m|_1=\sum_i|m_i|\ge2$. Then $\ML(W_m)=\frac12$ if $|m|_1$ is even, and $\ML(W_m)=\frac s{2s+1}$ if $|m|_1=2s+1$. If $|m|_1=5$, then, up to symmetry, $W_m$ is one of
\begin{gather*}
 W_A=\{(y,4y,z,u)\},\quad W_B=\{(2y,3y,z,u)\},\\ W_C=\{(a,b,3a+b,u)\},\\
 W_D=\{(a,b,2a+2b,u)\},\quad W_E=\{(a,b,c,2a+b+c)\},
\end{gather*}
where the parameters run over $\T$.
\end{lemma}
\begin{proof}
Put $x_i=\operatorname{sgn}(m_i)\,t$ for $m_i\ne0$ and $x_i=\frac12$ for $m_i=0$, with $t=\lfloor\frac{|m|_1}2\rfloor/|m|_1$. Then $m\cdot x=\lfloor\frac{|m|_1}2\rfloor\in\Z$ and $G(x)=t$, which is $\frac12$ or $\frac s{2s+1}$. Conversely let $|m|_1=2s+1$, $c=\frac s{2s+1}$, and suppose $\norm{x_i}>c$ for all $i$ with $m_i\ne0$. Then $y_i=\operatorname{sgn}(m_i)x_i$ has a representative in $(c,1-c)$, and $m\cdot x=\sum_i|m_i|y_i$ lies in $((2s+1)c,(2s+1)(1-c))=(s,s+1)$, so $x\notin W_m$. Hence $\ML(W_m)\le c$. If $|m|_1=5$, then up to signed permutations the absolute values of the entries of $m$ are $(4,1,0,0)$, $(3,2,0,0)$, $(3,1,1,0)$, $(2,2,1,0)$ or $(2,1,1,1)$, and the five parametrizations are injective homomorphisms $\T^3\to\T^4$ onto $W_m$ for $m=(4,-1,0,0)$, $(3,-2,0,0)$, $(3,1,-1,0)$, $(2,2,-1,0)$, $(2,1,1,-1)$.
\end{proof}

\begin{lemma}\label{lem:generic25}
Let $U$ be a coordinate-free two-dimensional subtorus with $\ML(U)<\frac25$ that is contained in no $W_m$ with $|m|_1\in\{3,5\}$. Then $\covol(U)<\frac{35}2\sqrt{37}$, that is, $\covol(U)^2\le11{,}331$.
\end{lemma}
\begin{proof}
Let $A=\covol(U)$ and let $m_1$ be a shortest nonzero vector of $L$. It is part of a basis of $L$ and primitive in $\Z^4$. As $U$ is coordinate-free, $m_1\ne\pm e_i$, so $|m_1|_1\ge2$, and $|m_1|_1\notin\{3,5\}$ by hypothesis. Hermite's bound in rank two gives $\eucl{m_1}^2\le\frac2{\sqrt3}A$, and Lemmas~\ref{lem:shortchar} and~\ref{lem:Wodd} give $\frac25>\ML(U)\ge\ML(W_{m_1})-\eucl{m_1}/(2A)$. If $|m_1|_1$ is even, this gives $A<5\eucl{m_1}$, hence $\eucl{m_1}<10/\sqrt3$ and $A<50/\sqrt3$. If $|m_1|_1=7$, it gives $A<\frac{35}2\eucl{m_1}\le\frac{35}2\sqrt{37}$, since $\eucl m^2\le37$ for primitive $m$ with $|m|_1=7$. If $|m_1|_1=2s+1\ge9$, it gives $A<\frac{5(2s+1)}{2(s-2)}\eucl{m_1}$. For $s\le6$ use $\eucl{m_1}\le|m_1|_1$ to get $A<\frac{5(2s+1)^2}{2(s-2)}\le\frac{845}8$. For $s\ge7$, Hermite's bound gives $\eucl{m_1}<\frac{5(2s+1)}{\sqrt3(s-2)}$, while $\eucl{m_1}\ge\frac12|m_1|_1=\frac{2s+1}2$, so $s=7$ and $A<\frac{15}2\cdot\frac{15}{\sqrt3}<65$. In every case $A<\frac{35}2\sqrt{37}$, and $(\frac{35}2)^2\cdot37=11{,}331.25$.
\end{proof}

We now treat the five subtori of Lemma~\ref{lem:Wodd}. For $\tau\in\{A,B,C,D,E\}$ let $\varphi_\tau:\T^3\to W_\tau$ be the parametrization of Lemma~\ref{lem:Wodd} and $G_\tau=G\circ\varphi_\tau$. As in Section~\ref{sec:two}, the two-dimensional subtori of $W_\tau$ are the images of $U_n=\{y\in\T^3:n\cdot y\in\Z\}$ for primitive $n\in\Z^3$. Write $J(\varepsilon)=[\frac25-\varepsilon,\frac35+\varepsilon]$ and, for $0\le\varepsilon\le\frac1{15}$,
\begin{align*}
 P_A(\varepsilon)&=[\tfrac25-\varepsilon,\tfrac25+\tfrac\varepsilon4]\times J(\varepsilon)^2, & P_B(\varepsilon)&=[\tfrac15-\tfrac\varepsilon2,\tfrac15+\tfrac\varepsilon3]\times J(\varepsilon)^2,\\
 P_C(\varepsilon)&=\big((\tfrac25,\tfrac25)+\varepsilon\Delta_C\big)\times J(\varepsilon), & P_D(\varepsilon)&=\big((\tfrac25,\tfrac25)+\varepsilon\Delta_D\big)\times J(\varepsilon),\\
 P_E(\varepsilon)&=(\tfrac25,\tfrac25,\tfrac25)+\varepsilon\Delta_E,
\end{align*}
where
\begin{align*}
 \Delta_C&=\operatorname{conv}\{(-1,-1),(\tfrac23,-1),(-1,4)\},\\
 \Delta_D&=\operatorname{conv}\{(-1,-1),(\tfrac32,-1),(-1,\tfrac32)\},\\
 \Delta_E&=\operatorname{conv}\{(-1,-1,-1),(\tfrac32,-1,-1),(-1,4,-1),(-1,-1,4)\}.
\end{align*}

\begin{lemma}\label{lem:fivecells}
For $0\le\varepsilon\le\frac1{15}$ the set $\{y\in\T^3:G_\tau(y)\ge\frac25-\varepsilon\}$ is the image of $P_\tau(\varepsilon)\cup(-P_\tau(\varepsilon))$.
\end{lemma}
\begin{proof}
Put $d=\frac25-\varepsilon\in[\frac13,\frac25]$. The conditions $\norm z,\norm u\ge d$ mean $z,u\in[d,1-d]+\Z$, which is $J(\varepsilon)+\Z$, and $J(\varepsilon)=-J(\varepsilon)+1$. For $W_A$: if $y\in[0,\frac12]$, then $\norm y\ge d$ means $y\ge d$, and then $4y\in[\frac43,2]$, so $\norm{4y}\ge d$ means $4y\in[1+d,2-d]$. As $d\ge\frac13$, $\frac{1+d}4\le d$, so the solutions in $[0,\frac12]$ form $[d,\frac{2-d}4]=[\frac25-\varepsilon,\frac25+\frac\varepsilon4]$, and those in $[\frac12,1]$ are their negatives. For $W_B$: in $[0,\frac12]$, $\norm{2y}\ge d$ means $y\in[\frac d2,\frac{1-d}2]$, and $\norm{3y}\ge d$ means $y\in[\frac d3,\frac{1-d}3]\cup[\frac{1+d}3,\frac12]$. Since $\frac{1+d}3>\frac{1-d}2$ for $d>\frac15$, the intersection is $[\frac d2,\frac{1-d}3]=[\frac15-\frac\varepsilon2,\frac15+\frac\varepsilon3]$. For $W_C$, $W_D$, $W_E$ let $\ell$ be the last coordinate form, $3a+b$, $2a+2b$ or $2a+b+c$. Its coefficients are at least $1$ and sum to $4$. On the box where the other coordinates lie in $[d,1-d]$, $\ell$ takes values in $[4d,4-4d]\subseteq[\frac43,\frac83]$, so $\norm\ell\ge d$ means $\ell\in[1+d,2-d]\cup[2+d,3-d]$, and the bounds $1+d\le 4d$ and $4-4d\le 3-d$ are automatic. The set $\{\text{coordinates}\ge d,\ \ell\le 2-d\}$ already forces every coordinate to be at most $2-d-3d\le 1-d$, and the substitution $y\mapsto -y$ maps it onto $\{\ell\ge 2+d\}$ modulo $\Z^3$. Writing the coordinates as $\frac25+\varepsilon\delta_i$, the conditions read $\delta_i\ge-1$ and $3\delta_1+\delta_2\le1$, $\delta_1+\delta_2\le\frac12$, respectively $2\delta_1+\delta_2+\delta_3\le1$, which define $\Delta_C$, $\Delta_D$, $\Delta_E$.
\end{proof}

\begin{lemma}[Subtori of the five tori]\label{lem:five}
Let $\tau\in\{A,B,C,D,E\}$ and let $n\in\Z^3$ be primitive with $\varphi_\tau(U_n)$ coordinate-free. If $\frac13\le\ML(U_n)<\frac25$, then $\ML(U_n)\in\mathcal R$.
\end{lemma}
\begin{proof}
Write $V=\ML(U_n)=\frac25-\varepsilon^*$. Since $P_\tau(\varepsilon)$ is convex and $U_n$ lifts to $\{y\in\R^3:n\cdot y\in\Z\}$, Lemma~\ref{lem:fivecells} shows that for $\varepsilon\le\frac1{15}$ we have $V\ge\frac25-\varepsilon$ exactly when the interval $n\cdot P_\tau(\varepsilon)$ contains an integer (the set $-P_\tau$ gives the negative interval). Each $P_\tau(\varepsilon)$ is a product of intervals $[c-\varepsilon e,c'+\varepsilon e']$ and of a set $p+\varepsilon\Delta$, so
\[
 n\cdot P_\tau(\varepsilon)=[\alpha-\varepsilon a,\ \beta+\varepsilon b],
\]
where $\beta=n\cdot x_0$ and $b=n\cdot w$ for a vertex $x_0+\varepsilon w$ of $P_\tau(\varepsilon)$ at which $n\cdot y$ is maximal for every $\varepsilon\in(0,\frac1{15}]$, and similarly $\alpha$, $a$ at a minimal vertex. Here $x_0\in\frac15\Z^3$, so $\alpha,\beta\in\frac15\Z$. As $V<\frac25$, $[\alpha,\beta]$ contains no integer. Let $N=\lceil\beta\rceil$, $N'=\lfloor\alpha\rfloor$, $k=5(N-\beta)$ and $k'=5(\alpha-N')$, so $k,k'\in\{1,2,3,4\}$, and $\varepsilon^*=\min(\frac k{5b},\frac{k'}{5a})$ with $\frac k0=\infty$. Suppose first that the minimum is $\frac k{5b}$.

At the vertex $x_0+\varepsilon w$ three linearly independent coordinate forms $\phi_i$, $i\in I$, of $\varphi_\tau$ are active, that is, $\sigma_i\phi_i(x_0+\varepsilon w)=\kappa_i+\frac25-\varepsilon$ with $\sigma_i=\pm1$ and $\kappa_i\in\Z$. Write $n=\sum_{i\in I}\lambda_i\sigma_i\phi_i$, so that $\lambda_i\in\frac1D\Z$ with $D=|\det(\phi_i)_{i\in I}|$, and put $K=\sum\lambda_i\kappa_i$ and $\Lambda=\sum\lambda_i$. Evaluating at $\varepsilon=0$ and $\varepsilon=\varepsilon^*$ gives $\beta=K+\frac25\Lambda$ and $N=K+V\Lambda$. Hence $V=P/Q$ with $P=D(N-K)$ and $Q=D\Lambda$ integers and $2Q-5P=5D(\beta-N)=-Dk$. As $V<\frac25$, writing $V=p/q$ in lowest terms gives $1\le2q-5p\le Dk$. If the minimum is $\frac{k'}{5a}$, the same argument for $-n$, which exchanges the two ends, gives $1\le2q-5p\le D'k'$, where $D'$ is the determinant at the minimal vertex. So $V\in\mathcal R$ whenever the end attaining the minimum has determinant $1$.

The vertices with $D>1$ are the following. For $\tau=A$ the ends $y=\frac25+\frac\varepsilon4$ (form $4y$, $D=4$). For $\tau=B$ the ends $y=\frac15-\frac\varepsilon2$ ($2y$, $D=2$) and $y=\frac15+\frac\varepsilon3$ ($3y$, $D=3$). For $\tau=C$ the vertex $(\frac23,-1)$ ($D=3$). For $\tau=D$ the vertices $(\frac32,-1)$, $(-1,\frac32)$ ($D=2$). For $\tau=E$ the vertex $(\frac32,-1,-1)$ ($D=2$). We show that if an end at such a vertex attains the minimum and the product of its determinant and its $k$ exceeds $4$, then the other end attains the minimum as well, at a vertex with determinant $1$, or with determinant $2$ and $k\le2$. Since $V$ does not depend on which end is used, this proves the lemma.

Types $A$ and $B$. Using the symmetries $y\mapsto-y$, $z\mapsto-z$, $u\mapsto-u$ we may take $n_1,n_2,n_3\ge0$, and $c=n_2+n_3\ge1$ since $\varphi_\tau(U_n)$ is coordinate-free. If $n_1=0$, then $\varphi_\tau(U_n)$ is the product of the circle of $(1,4)$ or $(2,3)$ with a circle in the $(z,u)$ plane, so $V=\min(\frac25,\ML(n_2,n_3))$, and $V\in[\frac13,\frac25)$ forces $V=\frac13\in\mathcal R$. Let $n_1\ge1$. The interval $[\alpha,\beta]$ has length $\frac c5$, so $k+k'=5-c$. For $A$, $a=n_1+c$, $b=\frac{n_1}4+c$, and the upper end has $D=4$, while the lower end has $D=1$. If $k\ge2$, then $k'\le3-c$, and $k a\ge 2n_1+2c>(3-c)(\frac{n_1}4+c)\ge k'b$, so $\frac k{5b}>\frac{k'}{5a}$ and the upper end is not the minimum. For $B$, $a=\frac{n_1}2+c$ and $b=\frac{n_1}3+c$. If the upper end ($D=3$) gives the minimum with $k\ge2$, then $k'\le3-c$ and $ka\ge n_1+2c>(3-c)(\frac{n_1}3+c)\ge k'b$, a contradiction. If the lower end ($D=2$) gives the minimum with $k'\ge3$, then $c=1$, $k=1$, $k'=3$, and $k'b=n_1+3>\frac{n_1}2+1=ka$, again a contradiction.

Types $C$, $D$, $E$. Let $f$ be the linear form $n$ restricted to the $\delta$-coordinates of $\Delta_\tau$, and $M=\max_{\Delta_\tau}f$, $m=-\min_{\Delta_\tau}f$. As $0$ is an interior point of $\Delta_\tau$ and $f\ne0$ (otherwise $\varphi_\tau(U_n)$ lies in $\{u=0\}$), $M,m>0$. With $n_3\ge0$ for $C$ and $D$, we have $b=M+c$, $a=m+c$ with $c=n_3$ for $C$, $D$ and $c=0$ for $E$, and $k+k'=5-c$. The origin is the barycenter of the vertices of $\Delta_\tau$ with weights $(\frac15,\frac35,\frac15)$, $(\frac15,\frac25,\frac25)$ and $(\frac15,\frac25,\frac15,\frac15)$ in the order listed, so the weighted sum of the values of $f$ at the vertices vanishes. If $f$ is maximal at the vertex with $D>1$ and weight $\mu$, this gives $\mu M\le(1-\mu)m$, that is, $M\le\frac23m$ for $C$ and $M\le\frac32m$ for $D$ and $E$. For $C$ ($D=3$), suppose $k\ge2$ and $\frac kb\le\frac{k'}a$. Since $b\le a$, this needs $k\le k'$, so $k+k'\ge4$ and $c\le1$. If $c=1$ then $k=k'=2$ and $a\le b$ forces $m\le M\le\frac23m$, which is impossible. If $c=0$ then $(k,k')=(2,3)$ and $2m\le 3M\le2m$, so the two ends give the same $\varepsilon^*$, and the lower end is a vertex with $D=1$. For $D$ and $E$ ($D=2$), suppose $k\ge3$ and $k a\le k'b$. Then $k'=5-c-k\le2-c$. If $c=1$, then $k=3$, $k'=1$ and $3(m+1)\le M+1\le\frac32m+1$, which is impossible. So $c=0$, $k'=5-k\le2$, and $3m\le km\le k'M\le\frac32k'm\le3m$. Hence $k=3$, $k'=2$ and the two ends give the same $\varepsilon^*$. The lower end is a vertex with $D=1$, or for $\tau=D$ possibly $(-1,\frac32)$, with $D=2$ and $Dk'=4$. The case of a minimal vertex with $D>1$ follows by replacing $n$ with $-n$.
\end{proof}

\begin{proposition}\label{prop:enum25}
There are $789{,}460{,}684$ coordinate-free primitive rank-two sublattices of $\Z^4$ with squared covolume at most $11{,}331$. Of these, $21{,}301{,}056$ are contained in a signed permutation image of $\Wone$ or $\Wtwo$, and $36{,}832{,}896$ further ones in some $W_m$ with $|m|_1=5$. Each of the remaining ones except $18{,}960$ has a point with $G\ge\frac25$: a grid point $(i/Q,j/Q)$ with $Q\in\{2,10,20,30\}$, or else the point returned by the exact maximization of Section~\ref{sec:comp}, at which $G$ is evaluated exactly. The remaining $18{,}960$ subtori have $\ML<\frac25$. They have squared covolume at most $1{,}250$, form $141$ orbits under signed coordinate permutations, and their values are $\frac7{19}$, $\frac38$, $\frac8{21}$, $\frac5{13}$, $\frac{12}{31}$, $\frac9{23}$, $\frac{13}{33}$ and $\frac{15}{38}$, all of them in $\mathcal R$.
\end{proposition}

\begin{proof}[Proof of Theorem~\ref{thm:upto25}]
(i) Let $\ML(U)<\frac25$. If $\ML(U)<\frac13$, Theorem~\ref{thm:second} gives $\ML(U)=\frac s{3s+1}$. Let $\frac13\le\ML(U)<\frac25$. If $U$ lies in a signed permutation image of $\Wone$ or $\Wtwo$, then $\ML(U)\le\frac13$, so $\ML(U)=\frac13\in\mathcal R$. If $U$ lies in some $W_m$ with $|m|_1=5$, then up to symmetry $U=\varphi_\tau(U_n)$ by Lemma~\ref{lem:Wodd}, and Lemma~\ref{lem:five} applies. Otherwise $\covol(U)^2\le11{,}331$ by Lemma~\ref{lem:generic25}, and Proposition~\ref{prop:enum25} applies.

(ii) For a triple $v$ of distinct positive integers, $T_v\times\T$ is a coordinate-free two-dimensional subtorus with $\ML(T_v\times\T)=\ML(v)$, so by (i) every value below $\frac25$ is $\frac s{3s+1}$ or lies in $\mathcal R$. A triple with a repeated entry has the value of a pair, $\frac s{2s+1}$ or $\frac12$, and only $\frac13\in\mathcal R$ occurs below $\frac25$. Conversely, $\frac s{3s+1}=\ML(1,2,3s)$ by \cite[Theorem~5.1]{Kravitz}, and by \eqref{eq:R} it remains to prove the two formulas. The product $T_{(1,4,k)}\times\T$ is $\varphi_A(U_{(k,1,0)})$ up to symmetry, and $(k,1,0)\cdot P_A(\varepsilon)=[\frac{2k+2}5-\varepsilon(k+1),\ \frac{2k+3}5+\varepsilon(\frac k4+1)]$. For $k=5j+3$ this is $[2j+\frac85-\varepsilon(k+1),\ 2j+\frac95+\varepsilon(\frac k4+1)]$, which first contains an integer at $\varepsilon=\min(\frac4{5(k+4)},\frac3{5(k+1)})=\frac3{5(k+1)}$ for $k\ge8$, a value at most $\frac1{15}$. By Lemma~\ref{lem:fivecells}, $\ML(1,4,5j+3)=\frac25-\frac3{5(5j+4)}=\frac{2j+1}{5j+4}$. Similarly $T_{(2,3,k)}\times\T=\varphi_B(U_{(k,1,0)})$, and for $k=5j$ the interval $(k,1,0)\cdot P_B(\varepsilon)=[j+\frac25-\varepsilon(\frac k2+1),\ j+\frac35+\varepsilon(\frac k3+1)]$ first contains an integer at $\varepsilon=\frac4{5(k+2)}\le\frac1{15}$ for $j\ge2$, so $\ML(2,3,5j)=\frac25-\frac4{5(5j+2)}=\frac{2j}{5j+2}$.
\end{proof}

\begin{corollary}\label{cor:acc25}
\begin{enumerate}[label=\textup{(\roman*)},leftmargin=*]
\item The accumulation points of the four-runner spectrum in $(0,\frac25]$ are exactly the three-runner values there, $\{\frac s{3s+1}:s\ge1\}\cup\mathcal R\cup\{\frac25\}$. In the notation of \cite{GiriKravitz},
\[
 \operatorname{acc}(\mathcal S(4))\cap[\tfrac1{10},\tfrac12)=\mathcal S(3)\cap[\tfrac1{10},\tfrac12),
\]
the case $n=4$ of \cite[Conjecture~1.5]{GiriKravitz} on $[\frac1{10},\frac12)$.
\item $\frac7{50}\notin\mathcal S_2(4)$. Since $\frac7{50}\in\mathcal S^*(3)=\mathcal S^*_2(4)$ by \cite[Proposition~8.1 and Lemma~3.2]{GiriKravitz}, the spectrum of two-dimensional subtori of $\T^4$ is strictly smaller than the spectrum of two-dimensional closed subgroups.
\item In the notation of \cite{JainKravitz}, with $\operatorname{Prog}(\alpha,\beta)=\{\alpha s+\beta:s\ge0\}$ and $\frac c{5}\operatorname{Prog}(\alpha,\beta)$ read as $\{\frac c{5x}:x\in\operatorname{Prog}(\alpha,\beta)\}$,
\[
 \mathcal S_1(3)\cap(\tfrac1{10},\tfrac16]=\Big(\big(\tfrac1{10}+\tfrac45\operatorname{Prog}(5,7)\big)\cup\big(\tfrac1{10}+\tfrac35\operatorname{Prog}(5,9)\big)\Big)\cap(\tfrac1{10},\tfrac16].
\]
So the finite symmetric difference in \cite[Theorem~1.4]{JainKravitz} has no element in $(\frac1{10},\frac16]$. It consists of the single value $\frac3{14}$, the term $s=0$ of the first progression, which lies outside that interval.
\end{enumerate}
\end{corollary}
\begin{proof}
(i) Let $A$ be the set of accumulation points in $(0,\frac25]$ of the values of all quadruples, repeated speeds allowed, and $A'$ the same for quadruples of distinct integers. Then $A'\subseteq A$. Every point of $A$ is the value of a two-dimensional subtorus by \cite[Theorem~1.4]{GiriKravitz}, which is coordinate-free since the value is positive, so it is $\frac25$ or lies in the set of Theorem~\ref{thm:upto25}(i). Every three-runner value, in particular $\frac25=\ML(1,4,6)$, is a lower accumulation point of values of quadruples of distinct integers by \cite[Theorem~1.3]{GiriKravitz}, whose proof, from \cite{Kravitz}, adds one large fourth speed to a triple. With Theorem~\ref{thm:upto25}(ii) this gives $A=A'=\mathcal S(3)\cap(0,\frac25]$. (ii) $\frac9{25}\notin\mathcal R$ since $2\cdot25-5\cdot9=5$, and a two-dimensional subtorus contained in a coordinate hyperplane has value $0$. (iii) Since $D=\frac12-\ML$, we have $\frac12-\frac{2j}{5j+2}=\frac1{10}+\frac4{5(5j+2)}$ and $\frac12-\frac{2j+1}{5j+4}=\frac1{10}+\frac3{5(5j+4)}$, and $5j+2$, $5j+4$ run through $\operatorname{Prog}(5,7)$, $\operatorname{Prog}(5,9)$ for $j\ge1$. The terms with $j=1$ are $\frac3{14}$ and $\frac16$. Now apply Theorem~\ref{thm:upto25}(ii) and \eqref{eq:R}.
\end{proof}

The same two ingredients apply one level higher. By Lemmas~\ref{lem:shortchar} and~\ref{lem:Wodd}, a coordinate-free two-dimensional subtorus with $\ML<\frac37$ lies in some $W_m$ with $|m|_1\in\{3,5,7\}$ or has bounded covolume, and the ten types of $W_m$ with $|m|_1=7$ all have loneliness $\frac37$. The contact argument of Lemma~\ref{lem:five} gives $1\le3q-7p\le6$ at vertices with determinant $1$, but it no longer excludes the other vertices: the subtorus $U=\{(10X+Y,\,-49X-5Y,\,9X,\,Y)\}$ of $W_{(4,1,1,1)}$ has $\ML(U)=\frac{43}{103}$, and $3\cdot103-7\cdot43=8$. This value is again a three-runner value, $\frac{43}{103}=\ML(1,102,25)$, so the level $\frac37$ needs a finer identification of the three-runner values than \eqref{eq:R}. Sections~\ref{sec:full} and~\ref{sec:three} provide it by a different argument, without a classification of the subtori: the value of every coordinate-free two-dimensional subtorus is a three-runner value (Theorem~\ref{thm:full}), and the three-runner values between $\frac25$ and $\frac37$ are listed in Corollary~\ref{cor:windows}(c), with $1\le3q-7p\le9$.

\begin{remark}[Why slicing at a maximizer does not suffice]\label{rem:slicing}
Giri and Kravitz observe that proving $\mathcal S_2(4)=\mathcal S(3)$ would settle \cite[Conjecture~1.5]{GiriKravitz} for $n=4$, and that the slicing argument behind $\mathcal S_2(n)\subseteq\mathcal S^*(n-1)$ would have to be refined. Let $x^*$ be a maximizer of $G$ on $U$ and let $a,b$ be active coordinates, $\norm{x^*_a}=\norm{x^*_b}=\ML(U)$, with $x^*_a=\sigma x^*_b$. The subgroup $U\cap\{x_a=\sigma x_b\}$ contains $x^*$, and if $x^*$ lies on its identity component, a circle, then $\ML(U)$ is the loneliness of a circle with a repeated speed, a three-runner value. For all $916{,}492$ coordinate-free two-dimensional subtori of covolume at most $20$, every one of the $1{,}348{,}608$ maximizers at which two independent coordinate forms are active, over the $582{,}624$ subtori with $\frac13<\ML(U)<\frac12$, lies on such a circle. This fails in general. The subtorus
\[
 U=\{(y_1,y_2,y_3,u):4y_1+21y_2+5y_3=0,\ u=2y_1+3y_2+3y_3\}
\]
has $\ML(U)=\frac{13}{30}$, its only maximizers are $\pm x^*$ with $x^*=(\frac{13}{30},\frac{13}{30},\frac{13}{30},\frac7{15})$, and $x^*$ lies on a non-identity component of each of its three slices. Indeed, with $\lambda=(4,21,5)$ and $c$ the index other than $a,b$, the slice $U\cap\{y_a=y_b\}$ has $\gcd(\lambda_c,30)\in\{2,3,5\}$ components, the point $(\frac r{30},\frac r{30},\frac r{30},\frac{8r}{30})$ of $U$ lies on the component indexed by $r$ modulo $\gcd(\lambda_c,30)$, and $x^*$ is the point with $r=13$. Nevertheless $\frac{13}{30}=\ML(1,8,29)$, and the circle of $(1,29,8)$ at the times $r/30$ reproduces the values of $G$ at these thirty points, which form the finite subgroup $U\cap\{y_1=y_2=y_3\}$. Among the $19{,}000$ distinct subtori $\{(y,u):\lambda\cdot y=0,\ \mu u=\theta\cdot y\}$ with $\lambda\in\Z_{>0}^3$ primitive, $\lambda_1+\lambda_2+\lambda_3\le42$, $\gcd(\lambda_c,\lambda_1+\lambda_2+\lambda_3)>1$ for every $c$, $\theta\in\{-4,\dots,4\}^3$ and $1\le\mu\le3$, the $602$ with no maximizer on a slice circle take $14$ distinct values, all of them three-runner values. So a proof of $\mathcal S_2(4)=\mathcal S(3)$ needs a second mechanism besides slicing, such as the circle through this finite subgroup. Below $\frac25$ such a mechanism is the arithmetic of the contact vertex in Lemma~\ref{lem:five}, which identifies the values without producing a circle. In general it is the arithmetic of the contact edge in Section~\ref{sec:full}, which gives the triple $(2,11,19)$ for the subtorus above (Remark~\ref{rem:read}). Section~\ref{sec:obstr} shows that the failure is not isolated: it occurs in an infinite family with values tending to $\frac12$, with a loss of at least $N^{-3}$ for every circle with a repeated speed.
\end{remark}

\section{Slicing obstructions of unbounded order}\label{sec:obstr}
Remark~\ref{rem:slicing} exhibits one subtorus whose maximizers lie off every slice circle. This section gives an infinite family with the same property, in which the slices through the maximizers have arbitrarily many components and the values tend to $\frac12$, and it makes the loss quantitative. Call a coordinate-free circle contained in some $\{x_i=\pm x_j\}$, $i\ne j$, a \emph{repeated-speed circle}. For positive integers $A,B,C$ with $\gcd(A,B,C)=1$ put $\lambda=(A,B,C)$, $N=A+B+C$, $H=N/2$ when $N$ is even, and
\[
 U(A,B,C)=\{(y_1,y_2,y_3,u)\in\T^4:\ Ay_1+By_2+Cy_3=0,\ u=Hy_2\}.
\]
This is the graph of $y\mapsto Hy_2$ over the connected kernel of the primitive character $\lambda$ on $\T^3$, hence a two-dimensional subtorus, and it is coordinate-free since $A,B,C\ne0$.

\begin{theorem}\label{thm:obstr}
Let $4\mid N$ and $N\ge2(A+C)+2$, and put $r=\frac N2-1$ and $L=\frac rN=\frac12-\frac1N$.
\begin{enumerate}[label=\textup{(\roman*)},leftmargin=*]
\item $\ML(U(A,B,C))=L$, and the only maximizers are $\pm p$ with $p=(L,L,L,\frac12)$. Moreover $\ML(1,\frac N2,N-1)=L$.
\item For $\{i,j,k\}=\{1,2,3\}$ the slice $U(A,B,C)\cap\{y_i=y_j\}$ has $\gcd(\lambda_k,N)$ components, and $p$ lies on its identity component if and only if $\gcd(\lambda_k,N)=1$. A repeated-speed circle in $U(A,B,C)$ attains $L$ if and only if one of $\gcd(A,N)$, $\gcd(B,N)$, $\gcd(C,N)$ equals $1$.
\item If $\gcd(\lambda_k,N)>1$ for $k=1,2,3$, then $\ML(T)\le L-N^{-3}$ for every repeated-speed circle $T\subseteq U(A,B,C)$.
\end{enumerate}
\end{theorem}
\begin{proof}
(i) Since $\lambda\cdot(L,L,L)=r$ and $Hr/N=\frac N4-\frac12$, the point $p$ lies on $U$, and $G(p)=L$. Let $G(y,u)\ge L$ and write $y_i=\frac12+z_i$ with $|z_i|\le\frac1N$. As $\frac N2\in\Z$, the number $\lambda\cdot z$ is an integer in $[-1,1]$. If it is $\pm1$, positivity of $\lambda$ forces $z_i=\pm\frac1N$ for all $i$, that is $(y,u)=\mp p$. If it is $0$, then $B|z_2|=|Az_1+Cz_3|\le\frac{A+C}N$, and since $\frac H2=\frac N4\in\Z$ we get $u\equiv Hz_2$ and $\norm u\le\frac{A+C}{2B}$. With $S=A+C$ and $B=N-S$, the inequality $\frac S{2B}<L$ is equivalent to $N^2-2(S+1)N+2S>0$, which holds for $N\ge2S+2$. So this case is impossible. For the triple, $t=\frac rN$ gives the value $L$, since $\frac N2\cdot\frac rN\equiv\frac12$ and $\norm{(N-1)t}=\norm{t}$. Conversely, if all three distances exceed $L$ at some $t=\frac12-\delta$ with $0\le\delta<\frac1N$, then $\norm{(N-1)t}=\frac12-\norm{(N-1)\delta}$ forces $(N-1)\delta<\frac1N$, and then $\norm{\frac N2t}=\norm{\frac N2\delta}\le\frac N2\delta<\frac1{2(N-1)}<L$, a contradiction.

(ii) On $y_i=y_j=t$ the relation reads $(N-\lambda_k)t+\lambda_ky_k=0$, whose kernel in $\T^2$ has $\gcd(N-\lambda_k,\lambda_k)=\gcd(\lambda_k,N)=g$ components, labelled by the value of the primitive character $\frac{N-\lambda_k}g t+\frac{\lambda_k}g y_k$ in $\frac1g\Z/\Z$. At $p$ this value is $\frac rg$, and $\gcd(r,N)=1$ because $r=\frac N2-1$ is odd and prime to $\frac N2$. So $p$ is on the identity component exactly when $g=1$. A repeated-speed circle attaining $L$ contains $\pm p$ by (i). At $p$ the last coordinate has distance $\frac12\ne L$, and $y_i=-y_j$ would give $\frac{2r}N\in\Z$. So such a circle lies in some $\{y_i=y_j\}$ and is the identity component of the slice, which contains $p$ only if $g=1$. Conversely, if $g=1$ the slice is a circle through $p$, coordinate-free because no coordinate of $p$ vanishes.

(iii) None of the twelve equations $x_i=\pm x_j$ holds on all of $U$, because $\lambda$ has no zero entry and $H>1$. So a repeated-speed circle $T\subseteq U$ is the identity component of the one-dimensional group $U\cap\{x_i=\pm x_j\}$ for one of them. For $y_i=\sigma y_j$ the vector $w$ with $w_i=\lambda_k$, $w_j=\sigma\lambda_k$, $w_k=-(\lambda_i+\sigma\lambda_j)$, $w_4=Hw_2$ lies in $\Lambda_U$ on that slice. For $y_1=\sigma u$ and $y_3=\sigma u$ the vectors $(\sigma HC,C,-(\sigma HA+B),HC)$ and $(-(B+\sigma HC),A,\sigma HA,HA)$ do the same, and $y_2=\pm u$ forces $y_2=u=0$ on the identity component, which is not coordinate-free. So $T$ has speeds of absolute value at most $HN$, and by Lemma~\ref{lem:corner} its loneliness is $a/d$ with integers $a$ and $1\le d\le2HN=N^2$. By (ii), $a/d<L$, so $L-\frac ad=\frac{rd-aN}{Nd}\ge\frac1{N^3}$.
\end{proof}

For instance $A=m$, $C=m+1$, $N=4m(m+1)(2m+1)K$ and $B=N-2m-1$, with $m\ge2$ and $K\ge1$, satisfy the hypotheses with $\gcd(\lambda_k,N)$ equal to $m$, $2m+1$ and $m+1$. So the three slices through each maximizer have arbitrarily many components, none of their identity components meets a maximizer, and the values $L=\frac12-\frac1N$ tend to $\frac12$. All of them are three-runner values by (i), as Conjecture~1.5 of \cite{GiriKravitz} predicts. For example $U(2,55,3)$ has $L=\frac{29}{60}$ and component orders $2,5,3$.

The loss in (iii) is not an artifact of the bound. Inside the same subtori, circles with four distinct speeds come closer to $L$ than any repeated-speed circle can.

\begin{theorem}\label{thm:fourspeed}
Let $4\mid N$, $N\ge12$ and $\ell\ge1$. Then
\[
 \ML\Big(1,\ \frac N2,\ 2N\ell-3,\ 3N\ell-\frac N2-2\Big)=\frac{(N-2)\ell-1}{2N\ell-2}=L-\frac1{N(N\ell-1)}.
\]
If moreover $60\mid N$ and $\ell\ge N+1$, then this circle lies in $U(2,N-5,3)$ up to signs, and its loneliness exceeds that of every repeated-speed circle in $U(2,N-5,3)$.
\end{theorem}
\begin{proof}
Put $U_N=U(2,B,3)$ with $B=N-5$. The map $(s,t)\mapsto(Bs+3t,\,s,\,-Bs-2t,\,Hs)$ is an isomorphism $\T^2\to U_N$ with inverse $s=y_2$, $t=y_1+y_3$. For $k\ge0$ the circle $\{t=ks\}$ has the speeds $1$, $H$, $B+2k$, $B+3k$ up to sign, and $k=N\ell-H+1$ gives the four speeds of the theorem.

Let $0\le\varepsilon\le\varepsilon_0=\frac1{N(N-1)}$. We claim that the points of $U_N$ with $G\ge L-\varepsilon$ are the points $(y,Hy_2)$ with $y=\pm\big((L,L,L)+\eta\big)$ and $\eta$ in the triangle $\{2\eta_1+B\eta_2+3\eta_3=0,\ \eta_i\ge-\varepsilon\}$. Write $y_i=\frac12+z_i$ with $|z_i|\le D=\frac1N+\varepsilon\le\frac1{N-1}$. Then $2z_1+Bz_2+3z_3$ is an integer of absolute value at most $ND<2$. If it is $0$, then $\norm u\le H|z_2|\le\frac{5HD}B\le\frac{5N}{2B(N-1)}<\frac{N-3}{2(N-1)}=L-\varepsilon_0$, where the strict inequality is $N^2-13N+15>0$, true for $N\ge12$. If it is $-1$, put $z_i=-\frac1N+\eta_i$. Then $\lambda\cdot\eta=0$ and $\eta_i\ge-\varepsilon$, the vertices of the triangle are $\varepsilon(H-1,-1,-1)$, $\varepsilon(-1,\frac5B,-1)$ and $\varepsilon(-1,-1,\frac{N-3}3)$, so $|\eta_i|\le\frac1N$ and the upper bounds $z_i\le D$ hold, and $\norm u=\frac12-H|\eta_2|\ge\frac12-H\varepsilon\ge L-\varepsilon$. The level $+1$ is the negative. This proves the claim.

At $p$, where $(s,t)=(\frac rN,\frac{2r}N)$, the character $ks-t$ takes the value $\frac{(k-2)r}N$, and $(k-2)r\equiv(-H-1)(H-1)=1-H^2\equiv1\pmod N$ since $4\mid N$. On the triangle it changes by $k\eta_2-\eta_1-\eta_3$, whose values at the three vertices are $-(k+H-2)\varepsilon$, $\frac{5k+2B}B\varepsilon$ and $-(k+\frac N3-2)\varepsilon$. So the circle meets $\{G\ge L-\varepsilon\}$ exactly when $[\frac1N-(k+H-2)\varepsilon,\ \frac1N+\frac{5k+2B}B\varepsilon]$ contains an integer, and the negative triangle gives the same condition. The first contact is with $0$ at $\varepsilon=\frac1{N(k+H-2)}=\frac1{N(N\ell-1)}\le\varepsilon_0$, because reaching $1$ needs $\varepsilon\ge\frac{(N-1)B}{N(5k+2B)}$, which is larger. This proves the formula. If $60\mid N$, then $\gcd(2,N)$, $\gcd(N-5,N)$ and $\gcd(3,N)$ are $2$, $5$ and $3$, and for $\ell\ge N+1$ the deficit $\frac1{N(N\ell-1)}$ is smaller than $N^{-3}$, so Theorem~\ref{thm:obstr}(iii) gives the comparison.
\end{proof}

So in $U(2,N-5,3)$ with $60\mid N$ the loneliness of circles with distinct speeds approaches the maximum $L$, a three-runner value, while every repeated-speed circle stays at distance at least $N^{-3}$ below it. A proof of $\mathcal S_2(4)=\mathcal S(3)$ by passing from a subtorus to a circle with a repeated speed inside it is therefore impossible in general, and every circle with a repeated speed that attains the three-runner value $L=\ML(1,\frac N2,N-1)$ lies outside the subtorus. Theorem~\ref{thm:full} proves $\mathcal S_2(4)=\mathcal S(3)$ in another way, and for $U(A,B,C)$ it produces the triple $(1,N-1,\frac N2)$ (Remark~\ref{rem:read}). Theorem~\ref{thm:fourspeed} is an explicit instance of the relative spectra of Jain and Kravitz \cite{JainKravitz}.

The obstruction is one phase of a larger family whose values are all three-runner values. For positive integers $A,B,C,H$ with $\gcd(A,B,C)=1$ put $S=A+C$, $N=S+B$ and
\[
 U_H(A,B,C)=\{(y_1,y_2,y_3,u)\in\T^4:\ Ay_1+By_2+Cy_3=0,\ u=Hy_2\},
\]
so that $U(A,B,C)=U_{N/2}(A,B,C)$.

\begin{proposition}[Coalescence]\label{prop:coal}
$\ML(U_H(A,B,C))$ is the maximum of $\min(\norm y,\norm z,\norm{Hy})$ over the closed subgroup $\{(y,z)\in\T^2:By+Sz=0\}$. In particular it depends on $A$ and $C$ only through $S$, and if $\gcd(B,S)=1$ it equals $\ML(S,B,HS)$.
\end{proposition}
\begin{proof}
Let $0\le d\le\frac12$. If a point of $U_H$ has $G\ge d$, choose representatives $y_i\in[d,1-d]$ and put $y=y_2$ and $z=(Ay_1+Cy_3)/S$. Then $z\in[d,1-d]$, $By+Sz=Ay_1+By_2+Cy_3\in\Z$ and $\norm{Hy}=\norm u\ge d$. Conversely such $y,z$ give the point $(z,y,z,Hy)$ of $U_H$ with $G\ge d$. If $\gcd(B,S)=1$, the subgroup is the circle $(y,z)=(St,-Bt)$.
\end{proof}

\begin{theorem}[A phase transition]\label{thm:phase}
Let $4\mid N$, $B\ge S$ and $H\equiv\frac N2\pmod N$, and put $D^*=\frac B{2(HS+B)}$. Then
\[
 \ML(U_H(A,B,C))=\frac12-\min\Big(\frac1N,\ D^*\Big).
\]
If $D^*>\frac1N$, the only maximizers are $\pm p$ with $p=(\frac12-\frac1N,\frac12-\frac1N,\frac12-\frac1N,\frac12)$, and the argument of Theorem~\ref{thm:obstr}(ii) applies unchanged. If $D^*<\frac1N$, the only maximizers are $\pm p^*$ with $p^*=(\frac12-D^*,\frac12+\frac SBD^*,\frac12-D^*,\frac12-D^*)$. If moreover $g=\gcd(B,S)$ divides $\frac N2$, then the repeated-speed circle $\{(-\frac Bgt,\frac Sgt,-\frac Bgt,\frac{HS}gt)\}\subseteq U_H$ contains $p^*$, and $\ML(U_H(A,B,C))=\ML(\frac Bg,\frac Sg,\frac{HS}g)$.
\end{theorem}
\begin{proof}
At a level $\frac12-D$ write $y_i=\frac12+z_i$ with $|z_i|\le D$. As $\frac N2\in\Z$, $\lambda\cdot z$ is an integer, and $\frac H2\in\Z$ gives $u\equiv Hz_2$. If $D<\frac1N$, then $\lambda\cdot z=0$, so $B|z_2|\le SD$ and $\norm u\le\frac{HSD}B$, which is less than $\frac12-D$ when $D<D^*$. This proves $\ML\le\frac12-\min(\frac1N,D^*)$. The point $p$ lies on $U_H$ with $G(p)=\frac12-\frac1N$, because $\lambda\cdot(\frac12-\frac1N)(1,1,1)=\frac N2-1$ and $H(\frac12-\frac1N)\equiv\frac12$. The point $p^*$ lies on $U_H$ with $\lambda$-value $\frac N2$ and $u\equiv\frac{HSD^*}B=\frac12-D^*$, and $G(p^*)=\frac12-D^*$ since $S\le B$. This proves the formula. If $D^*>\frac1N$, the level $\lambda\cdot z=0$ is excluded at $D=\frac1N$, so a maximizer has $\lambda\cdot z=\pm1$ and is $\mp p$. If $D^*<\frac1N$, a maximizer has $\lambda\cdot z=0$ and equality in $B|z_2|\le SD^*$ and in $|Az_1+Cz_3|\le SD^*$, so it is $\pm p^*$. The displayed circle lies in $U_H$ because $-AB+BS-CB=0$. On the subgroup of Proposition~\ref{prop:coal} the primitive character $\frac Bgy+\frac Sgz$ takes the value $\frac N{2g}$ at the image of $p^*$, so $p^*$ lies on the identity component, which is the image of this circle, when $g\mid\frac N2$. The circle has the speeds $\frac Bg,\frac Sg,\frac{HS}g$, with $\frac Bg$ repeated.
\end{proof}

For $(A,B,C)=(2,55,3)$, with $N=60$, $S=5$ and $g=5$, the transition is at $H=\frac{B(N-2)}{2S}=319$. For $H=30,90,\dots,270$ the value is $\frac{29}{60}$ and no repeated-speed circle in $U_H$ attains it, while for $H=330$ and $390$ the values $\frac{15}{31}=\ML(1,11,330)$ and $\frac{195}{401}=\ML(1,11,390)$ are attained on $\{y_1=y_3\}$. So the obstruction of Theorem~\ref{thm:obstr} is one phase of the family $U_H$, and on both sides of the transition the values are three-runner values, $\ML(1,\frac N2,N-1)$ in the first phase and $\ML(\frac Bg,\frac Sg,\frac{HS}g)$ in the second when $g\mid\frac N2$.

\section{Near-tight tuples for every number of speeds}\label{sec:general}
In this section $n\ge3$, $\T^n=\R^n/\Z^n$ and $G(x)=\min_{i\le n}\norm{x_i}$, and subtori, $\Lambda_S$, $\covol$, coordinate-freeness, $T_v$ and $\ML$ are defined as in Section~2. The inequality \eqref{eq:lip} holds in $\T^n$ with the same proof. We say that LRC$(m)$ holds if every set of $m$ positive integers has loneliness at least $\frac1{m+1}$. Since $\ML(V)\ge\ML(V')$ for $V\subseteq V'$, LRC$(m)$ then gives the same bound for every set of at most $m$ positive integers. LRC$(m)$ is known for $m\le6$ \cite{CusickPomerance,Chen,Bienia,BHK,BarajasSerra}, and the computer assisted proofs \cite{Rosenfeld,Rosenfeld9,Trak,SungTrak,Allikvere} cover $m\le14$. A two-dimensional subtorus $W\subseteq\T^n$ with $\ML(W)=\frac1n$ is \emph{critical}, and a set of $n-1$ positive integers with loneliness $\frac1n$ is a \emph{tight instance} for $n-1$ speeds.

\begin{theorem}\label{thm:general}
Let $n\ge3$, assume LRC$(n-1)$ and LRC$(n-2)$, and put $\kappa_n=(n(n-1))^{n-2}$.
\begin{enumerate}[label=\textup{(\roman*)},leftmargin=*]
\item Every primitive $v\in\Z^n$ with all $v_i\ne0$ satisfies $\ML(T_v)>\frac1n-\frac{\kappa_n}{2\eucl v}$.
\item Every coordinate-free two-dimensional subtorus $W\subseteq\T^n$ has $\ML(W)\ge\frac1n$. If $\ML(W)<\frac1n+\frac1{2n\kappa_n}$, then $W$ is critical and $\covol(W)<\kappa_n$. Every critical subtorus contains the circle of a tight instance for $n-1$ speeds.
\item If $\ML(T_v)<\frac1n$ and $\eucl v\ge n\kappa_n^2$, then $T_v$ lies in a critical subtorus of covolume less than $\kappa_n$.
\end{enumerate}
\end{theorem}

The proof uses three lemmas. The first is a form of \cite[Lemma~3.3]{GiriKravitz}, which states in the notation $D=\frac12-\ML$ that the largest value of $D$ on $k$-dimensional subtori of $\T^n$ equals the largest value on $(k-\ell)$-dimensional subtori of $\T^{n-\ell}$, and its last statement is \cite[Lemma~3.3]{Allikvere}. We include a proof because we use the circle it produces.

\begin{lemma}[Coalescing two coordinates]\label{lem:coalesce}
Let $S\subseteq\T^n$ be a coordinate-free subtorus of dimension $d\ge2$, and let $\ell_i$ be the restriction of the $i$th coordinate to $E=\R\Lambda_S$. There are $a\ne b$ with $\ell_a,\ell_b$ not proportional and $\varepsilon\in\{\pm1\}$ such that $S$ contains a coordinate-free subtorus of dimension $d-1$ on which $x_a=\varepsilon x_b$. Consequently $S$ contains a coordinate-free circle whose speeds take at most $n-d+1$ distinct absolute values, and $\ML(S)\ge\frac1{n-d+2}$ if LRC$(n-d+1)$ holds.
\end{lemma}
\begin{proof}
The forms $\ell_i$ are nonzero, take integer values on $\Lambda_S$, and span the dual space $E^*$, which has dimension $d\ge2$. Let $P\subseteq E^*$ be the plane spanned by two non-proportional forms, and consider the rays $\R_{>0}\sigma\ell_k$ with $\ell_k\in P$ and $\sigma=\pm1$. They lie on at least two lines, and every line meets both open half-planes bounded by any other line. Hence two consecutive rays in the cyclic order around the origin, say $\R_{>0}\sigma\ell_a$ and $\R_{>0}\tau\ell_b$, lie on different lines and make an angle less than $\pi$. The form $\phi=\sigma\ell_a+\tau\ell_b$ lies in the open cone between them, which contains no ray, so $\phi$ is proportional to no $\ell_k$ (such an $\ell_k$ would lie in $P$). Put $\varepsilon=-\sigma\tau$, so that $\phi=\sigma(\ell_a-\varepsilon\ell_b)$, and let $H=\ker\phi$. It is a rational hyperplane of $E$, so the subtorus $S_H$ with $\R\Lambda_{S_H}=H$ has dimension $d-1$, it lies in $S$, and $x_a=\varepsilon x_b$ on $S_H$. A coordinate vanishes on $S_H$ exactly when $\ell_k$ vanishes on $H$, that is, when $\ell_k$ is proportional to $\phi$, so $S_H$ is coordinate-free.

Call two coordinates equivalent on a subtorus if they agree there up to sign. On $S_H$ the classes of $a$ and $b$ merge, so the number of classes drops by at least one. Starting from at most $n$ classes on $S$ and applying the first part $d-1$ times gives a coordinate-free circle $T\subseteq S$ with at most $n-d+1$ classes, that is, whose speeds take at most $n-d+1$ distinct absolute values. Then $\ML(S)\ge\ML(T)\ge\frac1{n-d+2}$ by LRC$(n-d+1)$.
\end{proof}

\begin{lemma}[Density in $\T^n$]\label{lem:densityn}
Let $S$ be a $k$-dimensional subtorus of $\T^n$ with $1\le k<n$, let $E=\R\Lambda_S$, let $\pi:\R^n\to E^\perp$ be the orthogonal projection, and let $\lambda_1\le\dots\le\lambda_{n-k}$ be the successive minima of the lattice $\pi(\Z^n)$ with respect to the maximum norm.
\begin{enumerate}[label=\textup{(\alph*)},leftmargin=*]
\item $\lambda_1\cdots\lambda_{n-k}\le1/\covol(S)$.
\item Let $w_1,\dots,w_j\in\Z^n$ have linearly independent projections $\pi(w_i)$. The subtorus $S'$ with $\R\Lambda_{S'}=E+\R w_1+\dots+\R w_j$ has dimension $k+j$, and every point of $S'$ lies within maximum norm distance $\frac12\sum_i\supn{\pi(w_i)}$ of a point of $S$. In particular $\ML(S)\ge\ML(S')-\frac12\sum_i\supn{\pi(w_i)}$.
\end{enumerate}
\end{lemma}
\begin{proof}
As in Lemma~\ref{lem:density}(a), $\pi(\Z^n)$ is a lattice of rank $n-k$ in $E^\perp$ with covolume $1/\covol(S)$. The body $K=\{z\in E^\perp:\supn z\le1\}$ has volume at least $2^{n-k}$ by Vaaler's theorem \cite{Vaaler}, and Minkowski's second theorem \cite[Chapter~VIII]{Cassels} gives $\lambda_1\cdots\lambda_{n-k}\operatorname{vol}(K)\le2^{n-k}\covol(\pi(\Z^n))$, which is (a). For (b), $E+\R w_1+\dots+\R w_j=E+\R\pi(w_1)+\dots+\R\pi(w_j)$ has dimension $k+j$. Write $w_i=\pi(w_i)+e_i$ with $e_i\in E$, represent a point of $S'$ by $y=s+\sum_ib_i\pi(w_i)$ with $s\in E$, and let $c_i$ be an integer nearest to $b_i$. Then $y-\sum_ic_iw_i=(s-\sum_ic_ie_i)+\sum_i(b_i-c_i)\pi(w_i)$ represents the same point, its first summand represents a point of $S$, and its second summand has maximum norm at most $\frac12\sum_i\supn{\pi(w_i)}$. The last claim follows from \eqref{eq:lip}.
\end{proof}

\begin{lemma}[Corner lemma for planes]\label{lem:corner2}
Let $W\subseteq\T^n$ be a coordinate-free two-dimensional subtorus. Then $\ML(W)=p/q$ with integers $p\ge1$ and $1\le q\le2\covol(W)$.
\end{lemma}
\begin{proof}
Let $b,b'$ be a basis of $\Lambda_W$ and $r_i=(b_i,b'_i)\in\Z^2$. Then $W$ is the image of $\R^2$ under $y\mapsto(r_i\cdot y)_i$, the $r_i$ span $\R^2$, and the minors $D_{ij}=\det(r_i,r_j)$ satisfy $\sum_{i<j}D_{ij}^2=\covol(W)^2$ by the Cauchy--Binet formula, so $\covol(W)\ge1$. Let $f(y)=\min_i\norm{r_i\cdot y}$, let $y^*$ be a maximizer and $\mu=f(y^*)$, which is positive because $W$ contains a coordinate-free circle. If $\mu=\frac12$ we are done. Otherwise let $I$ be the set of $i$ with $\norm{r_i\cdot y^*}=\mu$. For $i\in I$ there are $\sigma_i\in\{\pm1\}$ and $k_i\in\Z$ with $\norm{r_i\cdot y}=\sigma_i(r_i\cdot y-k_i)$ near $y^*$, and $f(y)=\min_{i\in I}\sigma_i(r_i\cdot y-k_i)$ near $y^*$. Hence $(y^*,\mu)$ is a local, and therefore global, maximum point of the linear function $\nu$ on the convex polyhedron $P=\{(y,\nu):\sigma_i(r_i\cdot y-k_i)\ge\nu,\ i\in I\}$. A line contained in $P$ has a direction $(\delta,\eta)$ with $\sigma_ir_i\cdot\delta=\eta$ for $i\in I$, and $\eta=0$ because $\nu$ is bounded on $P$.

If the $r_i$ with $i\in I$ span $\R^2$, then $P$ contains no line, so $\nu$ attains $\mu$ at a vertex of $P$, where three constraints $i,j,m$ with linearly independent rows $(\sigma_ir_i,-1)$ are active. By Cramer's rule $\mu\in\frac1\Delta\Z$, where $\Delta$ is the absolute value of the determinant of these rows, and expanding along the last column gives $\Delta\le|D_{ij}|+|D_{im}|+|D_{jm}|\le\sqrt3\covol(W)$. Otherwise $r_i=c_iw$ for $i\in I$, with a primitive $w\in\Z^2$ and integers $c_i\ne0$. Then the constraints depend on $y$ only through $z=w\cdot y$, and the maximum of $\nu$ on the polygon $\{(z,\nu):\sigma_i(c_iz-k_i)\ge\nu,\ i\in I\}$, which contains no line, is attained at a vertex, where two constraints $i,j$ with $\sigma_ic_i\ne\sigma_jc_j$ are active. So $\mu\in\frac1{|\sigma_ic_i-\sigma_jc_j|}\Z$. Some $r_m$ is not parallel to $w$, and $|c_i|\le|c_i\det(w,r_m)|=|D_{im}|$ and likewise $|c_j|\le|D_{jm}|$, so $|\sigma_ic_i-\sigma_jc_j|\le\sqrt2\covol(W)$.
\end{proof}

\begin{proof}[Proof of Theorem~\ref{thm:general}]
By Lemma~\ref{lem:coalesce}, coordinate-free subtori of dimension two have loneliness at least $\frac1n$, and those of dimension three have loneliness at least $\frac1{n-1}$.

(ii) Let $W$ be coordinate-free of dimension two with $\ML(W)<\frac1n+\frac1{2n\kappa_n}$, and let $A=\covol(W)$. By Lemma~\ref{lem:densityn}(a) with $k=2$ there is $w\in\Z^n$ with $0<\supn{\pi(w)}=\lambda_1\le A^{-1/(n-2)}$, and Lemma~\ref{lem:densityn}(b) gives a coordinate-free three-dimensional subtorus $S'\supseteq W$ with $\frac1{n-1}\le\ML(S')\le\ML(W)+\frac12\lambda_1$. Hence
\[
 A^{-1/(n-2)}\ \ge\ \lambda_1\ \ge\ \frac2{n-1}-2\ML(W)\ >\ \frac2{n(n-1)}-\frac1{n\kappa_n}\ \ge\ \frac1{n(n-1)},
\]
since $\kappa_n\ge n-1$, so $A<\kappa_n$. By Lemma~\ref{lem:corner2}, $\ML(W)=p/q$ with $q\le2A<2\kappa_n$. If $\ML(W)\ne\frac1n$, then $\ML(W)-\frac1n=\frac{np-q}{nq}\ge\frac1{nq}>\frac1{2n\kappa_n}$, a contradiction. If $W$ is critical, Lemma~\ref{lem:coalesce} gives a coordinate-free circle $T\subseteq W$ whose speeds take at most $n-1$ distinct absolute values, and $\ML(T)\le\ML(W)=\frac1n$. By LRC$(n-2)$ there are exactly $n-1$ values, and by LRC$(n-1)$ we get $\ML(T)=\frac1n$.

(i) Suppose that $\ML(T_v)<\frac1n$. Apply Lemma~\ref{lem:densityn} to $S=T_v$, for which $k=1$ and $\covol(S)=\eucl v$, and choose $w_1,w_2\in\Z^n$ with linearly independent projections and $\supn{\pi(w_i)}=\lambda_i$. Part (b) gives coordinate-free subtori $W\supseteq T_v$ of dimension two and $S_3\supseteq T_v$ of dimension three with
\begin{gather*}
 \ML(T_v)\ge\ML(W)-\tfrac12\lambda_1\ge\tfrac1n-\tfrac12\lambda_1,\\ \ML(T_v)\ge\ML(S_3)-\tfrac12(\lambda_1+\lambda_2)\ge\tfrac1{n-1}-\lambda_2.
\end{gather*}
The second inequality gives $\lambda_i\ge\lambda_2>\frac1{n-1}-\frac1n=\frac1{n(n-1)}$ for $i\ge2$, so part (a) gives $\lambda_1<\kappa_n/\eucl v$, and the first inequality gives (i).

(iii) For $W$ as in the proof of (i), $\ML(W)\le\ML(T_v)+\frac12\lambda_1<\frac1n+\frac{\kappa_n}{2\eucl v}\le\frac1n+\frac1{2n\kappa_n}$, and (ii) applies.
\end{proof}

The hypotheses of Theorem~\ref{thm:general} hold for $3\le n\le7$, and for $n\le15$ by the computer assisted proofs cited above. They are the hypotheses of Cordella's reduction lemma \cite[Lemma~2.4]{Cordella}, and part (iii) is its effective form: a near-tight $n$-tuple on no critical subtorus has $\eucl v<n\kappa_n^2$. This answers \cite[Question~6.7]{Cordella} for every $n$. The step of \cite{Cordella} that is not effective is a gap above $\frac1n$ in the values of two-dimensional subtori, which comes from the fact that, in the notation $D=\frac12-\ML$, these values have only upper accumulation points \cite[Theorem~1.4]{GiriKravitz}. Part (ii) makes the gap explicit and bounds the covolume of the critical subtori. Its last statement explains why tight instances for $n-1$ speeds govern the near-tight $n$-tuples. For $n=4$ the circle of $\{1,2,3\}$ lies on $\Uone$ (at $a=b$) and on $\Utwo$ (at $x=y$), and for $n=5$ Cordella combines his reduction with Chen's two tight quadruples $\{1,2,3,4\}$ and $\{1,3,4,7\}$. In contrast with Section~\ref{sec:obstr}, at the lowest value $\frac1n$ a circle with a repeated speed always attains $\ML(W)$.

Part (i) says that the values of $\ML$ below $\frac1n$ approach $\frac1n$ at rate $1/\eucl v$. For $\eta>0$ the values at most $\frac1n-\eta$ come from tuples with $\eucl v<\kappa_n/(2\eta)$, so they form an explicit finite set, and the exponent of $\eucl v$ is best possible because $\ML(1,2,3,4k)=\frac14-\frac1{16k+4}$ by Theorem~\ref{thm:main}. Giri and Kravitz \cite[Section~7]{GiriKravitz} showed effectively that a circle of length $\eucl v$ is $\varepsilon$-dense in a two-dimensional subtorus once $\eucl v\ge c_n\varepsilon^{1-n}$ with an explicit $c_n$, which gives a deficit of order $\eucl v^{-1/(n-1)}$. The exponent $1$ comes from the second minimum $\lambda_2$, which LRC$(n-2)$ controls. For the conjecture itself, (i) reduces the case of $n$ speeds to $\eucl v<n(n+1)\kappa_n/2$. This does not improve the bound $\binom{n+1}2^{n-1}$ of Malikiosis, Santos and Schymura \cite{MSS}, and Allikvere \cite[Theorem~3.8]{Allikvere} bounds counterexamples much more sharply by applying LRC$(m)$ for all $m<n$ along a whole flag of subtori containing $T_v$. The content of (i) is the range between $\frac1{n+1}$ and $\frac1n$, where counterexample bounds say nothing. The constants are far from what an exhaustive search can handle: $4\kappa_4^2=82{,}944$ against the bound $512$ of Proposition~\ref{prop:large}, and $5\kappa_5^2=3.2\cdot10^8$, $6\kappa_6^2\approx3.9\cdot10^{12}$. So for five and six speeds the theorems of \cite{Cordella} now hold for all near-tight tuples outside an explicit finite set, and checking that set needs sharper constants, as Propositions~\ref{prop:list} and~\ref{prop:large} provide for $n=4$. Section~\ref{sec:tree} determines this set for five and six speeds by a different method.

\subsection{All successive minima}
The proof of Theorem~\ref{thm:general}(i) uses only $\lambda_1$ and $\lambda_2$. The whole flag gives a better constant, through the following inequality.

\begin{lemma}[From partial sums to a product]\label{lem:prodsum}
Let $0<\lambda_1\le\dots\le\lambda_r$ and $d_1,\dots,d_r>0$. If $\sum_{i\le j}\lambda_i\ge\sum_{i\le j}d_i$ for $1\le j\le r$, then $\prod_i\lambda_i\ge\prod_id_i$, and the inequality is strict if the last of the hypotheses is strict.
\end{lemma}
\begin{proof}
Put $S_j=\sum_{i\le j}(\lambda_i-d_i)\ge0$. The inequality $\log u\le u-1$ for $u=d_i/\lambda_i$ gives $\log(\lambda_i/d_i)\ge(\lambda_i-d_i)/\lambda_i$, and summation by parts gives
\[
 \sum_{i=1}^r\frac{\lambda_i-d_i}{\lambda_i}=\frac{S_r}{\lambda_r}+\sum_{i=1}^{r-1}S_i\Bigl(\frac1{\lambda_i}-\frac1{\lambda_{i+1}}\Bigr)\ge\frac{S_r}{\lambda_r}\ge0.
\]
So $\sum_i\log(\lambda_i/d_i)\ge S_r/\lambda_r$, which is positive if $S_r>0$.
\end{proof}

\begin{proposition}\label{prop:flag}
Let $n\ge3$ and assume LRC$(m)$ for $1\le m\le n-1$. Let $v\in\Z^n$ be primitive with all $v_i\ne0$.
\begin{enumerate}[label=\textup{(\alph*)},leftmargin=*]
\item $\displaystyle\ML(T_v)\ge\frac1n-\frac{n!\,(n-1)!}{2^n\eucl v}$.
\item Suppose that $\mu=\ML(T_v)<\frac1n$ and that every coordinate-free two-dimensional subtorus containing $T_v$ has loneliness at least $\gamma>\frac1n$. Put $B_0=0$, $B_1=2(\gamma-\frac1n)$ and $B_j=2(\frac1{n-j+1}-\frac1n)$ for $2\le j\le n-1$. If $d_j=B_j-B_{j-1}>0$ for $1\le j\le n-1$, then $\eucl v<(d_1\cdots d_{n-1})^{-1}$.
\item If $\ML(T_v)=\frac pq<\frac1n$ in lowest terms, then
\[
 \max_i|v_i|\le\eucl v\le\frac{(n!)^2}{2^n}\cdot\frac q{q-np}.
\]
\end{enumerate}
\end{proposition}
\begin{proof}
(a) There is nothing to prove if $\mu=\ML(T_v)\ge\frac1n$. Otherwise apply Lemma~\ref{lem:densityn} to $S=T_v$, for which $k=1$ and $\covol(S)=\eucl v$, and choose $w_1,\dots,w_{n-1}\in\Z^n$ with linearly independent projections and $\supn{\pi(w_i)}=\lambda_i$. For $1\le j\le n-1$ let $S_{j+1}$ be the subtorus of dimension $j+1$ that Lemma~\ref{lem:densityn}(b) gives for $w_1,\dots,w_j$. It contains $T_v$, so it is coordinate-free, and Lemma~\ref{lem:coalesce} and LRC$(n-j)$ give $\ML(S_{j+1})\ge\frac1{n-j+1}$. Hence
\[
 \sum_{i\le j}\lambda_i\ \ge\ 2\Bigl(\frac1{n-j+1}-\mu\Bigr)\quad(1\le j\le n-1),\qquad \prod_{i\le n-1}\lambda_i\le\frac1{\eucl v}.
\]
The successive differences of the right-hand sides of the first inequalities, starting from $0$, are $2(\frac1n-\mu)$ and $\frac2{(n-j+1)(n-j+2)}$ for $2\le j\le n-1$, and their product is $2^n(\frac1n-\mu)/(n!\,(n-1)!)$. Lemma~\ref{lem:prodsum} gives $2^n(\frac1n-\mu)/(n!\,(n-1)!)\le1/\eucl v$, which is (a).

(b) Now $\ML(S_2)\ge\gamma$, so the first inequality for $j=1$ reads $\lambda_1\ge2(\gamma-\mu)>B_1$, and for $j\ge2$ the inequality $\mu<\frac1n$ gives $\sum_{i\le j}\lambda_i>B_j$. The strict case of Lemma~\ref{lem:prodsum} gives $\prod_i\lambda_i>\prod_jd_j$, and $\prod_i\lambda_i\le1/\eucl v$.

(c) Substitute $\frac1n-\frac pq=\frac{q-np}{nq}$ in (a).
\end{proof}

For $n=4,5,6$ the constant $n!\,(n-1)!/2^n$ is $9$, $90$ and $1350$, against $\kappa_n/2=72$, $4000$ and $405{,}000$ in Theorem~\ref{thm:general}(i). Part (c) gives an explicit constant in the comparison of the height of a near-tight tuple with the denominator of its loneliness \cite[Question~6.6]{Cordella}, for all such tuples, including those on no critical subtorus. Section~\ref{sec:seven} gives the optimal eventual constant for seven speeds.

\subsection{Five speeds}\label{sec:five}
For five speeds the general constants can be replaced by an exact statement up to $\frac29$, because Theorem~\ref{thm:main} identifies the circles with a repeated speed that a subtorus of small loneliness must contain. For a quadruple $u$ write $T_u\times\T=\{(u_1t,u_2t,u_3t,u_4t,s)\}\subseteq\T^5$. Its loneliness is $\ML(u)$, since the last coordinate can be put at $\frac12$.

\begin{theorem}\label{thm:five}
Let $W\subseteq\T^5$ be a coordinate-free two-dimensional subtorus with $\ML(W)<\frac29$. Then, up to signed coordinate permutations, $W=T_{(1,2,3,4)}\times\T$ or $W=T_{(1,3,4,7)}\times\T$, and $\ML(W)=\frac15$.
\end{theorem}
\begin{proof}
By Lemma~\ref{lem:coalesce}, $W$ contains a coordinate-free circle $T$ whose speeds take at most four distinct absolute values, and $\ML(T)\le\ML(W)<\frac29$. At most three values would give $\ML(T)\ge\frac14$ by the conjecture for three speeds, so there are exactly four, and by Corollary~\ref{cor:consequences}(iii) they are a multiple of $\{1,2,3,4\}$ or of $\{1,3,4,7\}$. The same holds for every coordinate-free circle in $W$ with a repeated absolute speed. After a signed coordinate permutation, $T$ is the circle of one of the eight vectors
\begin{gather*}
 (1,1,2,3,4),\quad(1,2,2,3,4),\quad(1,2,3,3,4),\quad(1,2,3,4,4),\\
 (1,1,3,4,7),\quad(1,3,3,4,7),\quad(1,3,4,4,7),\quad(1,3,4,7,7),
\end{gather*}
call it $q$.
Since $q$ is primitive and lies in the primitive lattice $\Lambda_W$, we have $\Lambda_W=\Z q+\Z c$ for some $c\in\Z^5$. Replacing $c$ by $\pm c+mq$ we may assume $c_1=0$, as $q_1=1$, and that the first nonzero entry of $c$ is positive. The minors $c_k=q_1c_k-q_kc_1$ show that $\Z q+\Z c$ is primitive exactly when $\gcd(c)=1$.

Next we bound $\covol(W)$. Apply Lemma~\ref{lem:densityn} to $S=W$, with $k=2$ and successive minima $\lambda_1\le\lambda_2\le\lambda_3$. Part (b) gives coordinate-free subtori containing $W$ of dimensions three, four and five, with loneliness at least $\frac14$, $\frac13$ and $\frac12$ by Lemma~\ref{lem:coalesce} and the conjecture for three and two speeds. Hence $\lambda_1>2(\frac14-\frac29)=\frac1{18}$, $\lambda_1+\lambda_2>2(\frac13-\frac29)=\frac29$ and $\lambda_1+\lambda_2+\lambda_3>2(\frac12-\frac29)=\frac59$. So $\lambda_2>\frac19$ and $\lambda_3>\frac5{27}$, and part (a) gives $\covol(W)<\frac{4374}5$, that is
\[
 \covol(W)^2=\sum_{i<j}(q_ic_j-q_jc_i)^2\le 765{,}275 .
\]
A circle of $W$ with a repeated absolute speed is the image of a line in the direction orthogonal to $r_i-\varepsilon r_j$ for some $i<j$ and $\varepsilon=\pm1$, where $r_k=(q_k,c_k)$ are the coordinate rows of the parametrization $(t,s)\mapsto tq+sc$. For each of the eight vectors $q$, the program \texttt{g5enum2} (Section~\ref{sec:comp}) enumerates every normalized $c$ in this range and tests all of these directions. It finds that $W=\Z q+\Z c$ satisfies the condition of the first paragraph only for $16$ pairs $(q,c)$, two for each $q$, and in each of them one of the two coordinates where $q$ has a repeated entry is a free coordinate of $W$ while the other four coordinates are the circle of $\{1,2,3,4\}$ or $\{1,3,4,7\}$. So $W$ is one of the two products, and both have loneliness $\frac15$.
\end{proof}

\begin{corollary}\label{cor:five}
\begin{enumerate}[label=\textup{(\roman*)},leftmargin=*]
\item The accumulation points of the set of values $\ML(V)$, over all sets $V$ of five positive integers, that lie in $(0,\frac29)$ consist of $\frac15$ alone. In the notation $D=\frac12-\ML$ this is the case $n=5$ of \cite[Conjecture~1.5]{GiriKravitz} on $(\frac5{18},\frac12)$.
\item Every primitive five-speed tuple $v$ with $\ML(v)<\frac15$ and $\eucl v\ge7{,}290$ is, up to order, $\{a,2a,3a,4a,b\}$ or $\{a,3a,4a,7a,b\}$ with $\gcd(a,b)=1$. In particular the critical subtori for five speeds are the two products of Theorem~\ref{thm:five}.
\end{enumerate}
\end{corollary}
\begin{proof}
(i) By \cite[Theorem~1.4]{GiriKravitz}, $\mathcal S(4)\subseteq\operatorname{acc}(\mathcal S(5))\subseteq\mathcal S_2(5)$. The two conventions are reconciled as in the proof of Corollary~\ref{cor:acc}: a tuple with a repeated speed has the value of a set of at most four speeds, and by Corollary~\ref{cor:consequences} the only such value below $\frac29$ is $\frac15$. Theorem~\ref{thm:five} gives $\mathcal S_2(5)\cap(0,\frac29)=\{\frac15\}$, and $\frac15=\ML(1,2,3,4)$.

(ii) Suppose that $T_v$ lies on neither product. Then every coordinate-free two-dimensional subtorus containing $T_v$ has loneliness at least $\frac29$ by Theorem~\ref{thm:five}, and Proposition~\ref{prop:flag}(b) applies with $n=5$ and $\gamma=\frac29$. The subtori of the flag of dimensions $2,3,4,5$ have loneliness at least $\frac29,\frac14,\frac13,\frac12$, so $(B_1,B_2,B_3,B_4)=(\frac2{45},\frac1{10},\frac4{15},\frac35)$ and $(d_1,d_2,d_3,d_4)=(\frac2{45},\frac1{18},\frac16,\frac13)$, with $d_1d_2d_3d_4=\frac1{7290}$. Hence $\eucl v<7{,}290$. So for $\eucl v\ge7{,}290$ the circle $T_v$ lies on one of the two products $W$, and $v\in\Lambda_W$ gives the two forms. The last statement follows because a critical subtorus has loneliness $\frac15<\frac29$.
\end{proof}

Part (ii) improves the five-speed bound $n\kappa_n^2=3.2\cdot10^8$ of Theorem~\ref{thm:general}(iii) to $7{,}290$ in Cordella's reduction \cite{Cordella}. For four speeds the second critical subtorus in Proposition~\ref{prop:large} is not a product. For five speeds all coordinate-free two-dimensional subtori below $\frac29$ are products. Cordella proved that on the two critical products every quintuple with $\ML<\frac15$ has $\ML=\frac s{5s+1}$ for some integer $s\ge1$ \cite[Theorem~4.2]{Cordella}. He also verified the same conclusion for all quintuples with largest speed at most $130$ \cite[Theorem~6.4]{Cordella}. Thus any value below $\frac15$ outside this sequence would have to come from a primitive quintuple on neither product, with $\max_iv_i>130$ and $\eucl v<7{,}290$. Theorem~\ref{thm:fivefull} closes this range: the primitive quintuples with $\ML<\frac15$ on neither product are twelve sets with largest speed at most $27$, and their values are in the sequence.

\section{Subtori of codimension two and three runners}\label{sec:full}
We keep the notation of Section~\ref{sec:general}. Throughout this section $n\ge3$ and $U\subseteq\T^n$ is a coordinate-free subtorus of dimension $n-2$. For $n=4$ these are the subtori of Theorems~\ref{thm:second-intro} and~\ref{thm:upto25-intro}. We show that $\ML(U)$ is always a three-runner value, and we exhibit a triple. By Remark~\ref{rem:slicing} and Section~\ref{sec:obstr}, no repeated-speed circle in $U$ need attain $\ML(U)$. The triple below is produced from the arithmetic of the contact at a maximizer instead, and it is unrelated to the slices of~$U$.

\begin{theorem}\label{thm:full}
Let $n\ge 3$ and let $U\subseteq\T^n$ be a coordinate-free subtorus of dimension $n-2$. Then $\ML(U)=\ML(a,b,c)$ for some positive integers $a,b,c$. More precisely, either $\ML(U)=\frac12=\ML(1,3,5)$, or $\ML(U)=\ML(g,N-g,d)$ with $N$, $g$, $d$ as in Proposition~\ref{prop:contact} and Lemma~\ref{lem:three}.
\end{theorem}

\begin{remark}\label{rem:distinct}
The triple $(g,N-g,d)$ of Lemma~\ref{lem:three} has a repeated entry only if it is $(1,N-1,1)$ up to order. Indeed $g=N-g$ would give $N=2g$ and $\gcd(g,N)=g=1$, against $N\ge3$, and since $d$ divides $N$ and $\gcd(g,N)=1$, each of $d=g$ and $d=N-g$ forces $d=1$. In that case $N=2s+1$ is odd, because the value is below $\frac12$, and the value is $\ML(1,2s)=\frac s{2s+1}=\ML(1,2s,2s+2)$, as the time $t=\frac s{2s+1}$ shows. So $a,b,c$ in Theorem~\ref{thm:full} can be taken distinct, and the set of values of $\ML$ on triples is the same whether or not repeated speeds are admitted.
\end{remark}

\begin{corollary}\label{cor:full}
\begin{enumerate}[label=\textup{(\roman*)},leftmargin=*]
\item For every $n\ge3$ the set of values $\ML(U)$, over the coordinate-free subtori $U\subseteq\T^n$ of dimension $n-2$, is the set of values of $\ML$ on triples of distinct positive integers. In the notation of \cite{GiriKravitz} this is $\mathcal S_{n-2}(n)=\mathcal S(3)$, and in particular $\mathcal S_2(4)=\mathcal S(3)$.
\item The accumulation points of the four-runner spectrum are exactly the three-runner values, that is, $\operatorname{acc}(\mathcal S(4))=\mathcal S(3)$. This is \cite[Conjecture~1.5]{GiriKravitz} for $n=4$. The same holds for the set of values of quadruples of distinct positive integers.
\end{enumerate}
\end{corollary}
\begin{proof}
(i) If $v$ is a triple of positive integers, then $T_v\times\T^{n-3}$ is a coordinate-free subtorus of dimension $n-2$ with loneliness $\ML(v)$, since the free coordinates can be put at $\frac12$. The other inclusion is Theorem~\ref{thm:full} together with Remark~\ref{rem:distinct}.

(ii) By \cite[Theorem~1.4]{GiriKravitz}, $\mathcal S(3)\subseteq\operatorname{acc}(\mathcal S(4))\subseteq\mathcal S_2(4)$, and (i) closes the chain. For quadruples of distinct integers we argue as in the proof of Corollary~\ref{cor:acc25}. Their values form a subset of $\mathcal S(4)$, so their accumulation points are three-runner values. Conversely, by Remark~\ref{rem:distinct} every three-runner value is the value of a triple of distinct integers, and it is a lower accumulation point of values of quadruples of distinct integers by \cite[Theorem~1.3]{GiriKravitz}, whose proof, from \cite{Kravitz}, adds one large fourth speed to the triple.
\end{proof}

Theorem~\ref{thm:full} has no restriction on the value, so it covers the level $\frac37$ at which the method of Section~\ref{sec:upto25} stopped. What Theorem~\ref{thm:second} adds is the classification of the subtori below $\frac13$. Section~\ref{sec:three} lists all three-runner values. Part (i) of Corollary~\ref{cor:full} is the case of codimension two of \cite[Problem~1.6]{GiriKravitz}, which asks whether $\mathcal S_k(n)=\mathcal S_{k-\ell}(n-\ell)$ for $1\le\ell<k\le n$.

\subsection{Zonotopal form}
Let $L\subseteq\Z^n$ be the character lattice of $U$, that is, the set of all $m\in\Z^n$ with $m\cdot x\in\Z$ for every $x\in U$. It is a primitive lattice of rank~$2$, and $U=\{x\in\T^n: m\cdot x\in\Z\ \text{for all}\ m\in L\}$. Fix a basis $m,m'$ of $L$ and put
\[
 c_i=(m_i,m'_i)\in\Z^2\quad(1\le i\le n),\qquad Z=\sum_{i=1}^n[-c_i,c_i],\qquad \omega=\sum_{i=1}^n c_i ,
\]
so that $x\in U$ if and only if $\sum_i x_ic_i\in\Z^2$. Here $Z$ is a Minkowski sum of segments, a zonotope. For $t\in\R^2$ let $|t|_Z=\min\{\rho\ge0: t\in\rho Z\}$. A change of basis of $L$ applies an element of $\mathrm{GL}_2(\Z)$ to all the $c_i$, to $Z$ and to $\omega$. The symmetry $x_i\mapsto -x_i$ replaces $c_i$ by $-c_i$, which changes neither $Z$ nor the coset $\omega+2\Z^2$.

\begin{lemma}\label{lem:zon}
\begin{enumerate}[label=\textup{(\roman*)},leftmargin=*]
\item For every $i$, the vectors $c_j$ with $j\ne i$ span $\R^2$. In particular $Z$ is a centrally symmetric convex polygon with nonempty interior.
\item $\displaystyle \tfrac12-\ML(U)=\tfrac12\min\bigl\{|t|_Z : t\in\omega+2\Z^2\bigr\}$.
\end{enumerate}
\end{lemma}
\begin{proof}
(i) Otherwise some nonzero $w\in\Z^2$ is orthogonal to all $c_j$ with $j\ne i$. Then $w_1m+w_2m'\in L$ is a nonzero multiple of the $i$th unit vector. As $L$ is primitive, the $i$th unit vector lies in $L$, so $U\subseteq\{x_i=0\}$, a contradiction.

(ii) A point $x\in\T^n$ satisfies $G(x)\ge\delta$ if and only if it has a representative $x=\frac12(1,\dots,1)+z$ with $\supn z\le\frac12-\delta$, and then $x\in U$ if and only if $\sum_i z_ic_i\in\Z^2-\frac12\omega$. Since $\{\sum_i z_ic_i:\supn z\le\rho\}=\rho Z$, the number $\frac12-\ML(U)$ is the least $\rho$ for which $\rho Z$ meets $\Z^2-\frac12\omega=\frac12(\omega+2\Z^2)$.
\end{proof}

Zonotopes also underlie the restatements of the Lonely Runner Conjecture in \cite{MSS}.

\subsection{The arithmetic of the contact}

\begin{proposition}\label{prop:contact}
Suppose $\ML(U)<\frac12$. Then there are integers $N\ge3$, $k\ge1$ and $Q$, with $k<N$ and $k\equiv N\pmod 2$, such that, with $k'=(N-k)/2$,
\begin{enumerate}[label=\textup{(\alph*)},leftmargin=*]
\item $\norm{jQ/N}<k'/N$ for every integer $j$ with $k'<j<N-k'$,
\item $\norm{k'Q/N}\ge k'/N$,
\item $\ML(U)=k'/N$.
\end{enumerate}
\end{proposition}
\begin{proof}
Let $t^*$ minimize $|t|_Z$ over $\omega+2\Z^2$ and put $\rho=|t^*|_Z$, so that $\ML(U)=\frac12-\frac12\rho$ by Lemma~\ref{lem:zon}(ii). Then $\rho>0$, and $\rho<1$ because $\ML(U)>0$ (a coordinate-free subtorus contains a point with no coordinate in $\Z$). The point $t^*$ lies on an edge of the polygon $\rho Z$, and the edges of $Z$ are parallel to the nonzero $c_i$. Let $u$ be a primitive vector parallel to such an edge. Changing the basis of $L$ we may assume $u=(0,1)$.

\emph{Normal form.} The nonzero $c_i$ parallel to $u$ are of the form $(0,y_i)$. Let $e\ge1$ be the sum of the corresponding $|y_i|$. The other nonzero $c_i$ have nonzero first coordinate, and after changing the signs of some coordinates of $\T^n$ they are $(A_i,q_i)$ with $A_i\ge1$. There is at least one of them by Lemma~\ref{lem:zon}(i). Let $N$ be the sum of these $A_i$, let $Q_0$ be the sum of these $q_i$, let $R=(N,Q_0)$, and let $\Pi$ be the sum of the segments $[-(A_i,q_i),(A_i,q_i)]$. Then $Z=\Pi+[-e,e]\,u$, and since all $A_i$ are positive, the first coordinate is at most $N$ on $\Pi$, with equality only at $R$. Hence the edges of $Z$ parallel to $u$ are $\pm\bigl(\{N\}\times[Q_0-e,Q_0+e]\bigr)$, and replacing $t^*$ by $-t^*$ we may assume $t^*\in\rho\,(\{N\}\times[Q_0-e,Q_0+e])$. Writing $t^*=(k,y^*)$, this says
\begin{equation}\label{eq:edge}
 \rho=k/N,\qquad |Ny^*-kQ_0|\le ke .
\end{equation}
Moreover $\omega\equiv(N,Q_0+e)\pmod 2$, so every $(x,y)\in\omega+2\Z^2$ has $x\equiv N$ and $y\equiv Q_0+e\pmod 2$. In particular $k\equiv N\pmod2$, and $0<k<N$ since $0<\rho<1$. Thus $k'\ge1$, $N\ge3$ and $\ML(U)=\frac12-\frac{k}{2N}=k'/N$, which is~(c).

\emph{Choice of $Q$.} The integers $Q\equiv Q_0+e-1\pmod2$ with $|Q-Q_0|\le e-1$ are $Q_0-e+1,\,Q_0-e+3,\,\dots,\,Q_0+e-1$. Since $Ny^*/k\in[Q_0-e,Q_0+e]$ by \eqref{eq:edge}, one of them satisfies
\begin{equation}\label{eq:Q}
 |Ny^*-kQ|\le k .
\end{equation}
We fix such a $Q$ (if $e=1$, then $Q=Q_0$). Let $P=\{(x,y):|x|\le N,\ |y-xQ/N|\le1\}$. We claim that
\begin{equation}\label{eq:int}
 P\cap\{|x|<N\}\subseteq\operatorname{int}Z .
\end{equation}
Since $\pm R\in\Pi$, the polygon $Z$ contains $B=\{(x,y):|x|\le N,\ |y-xQ_0/N|\le e\}$. If $e\ge2$ and $(x,y)\in P$ with $|x|<N$, then $|y-xQ_0/N|\le1+|x|(e-1)/N<e$, so $(x,y)$ is an interior point of $B$. If $e=1$, then exactly one $c_i$ is parallel to $u$, so by Lemma~\ref{lem:zon}(i) the remaining ones span $\R^2$ and $\Pi$ has nonempty interior. Since $0\in\operatorname{int}\Pi$ and $\pm R\in\Pi$, the open segment $(-R,R)$ lies in $\operatorname{int}\Pi$, and $P\cap\{|x|<N\}=(-R,R)+[-1,1]\,u\subseteq\operatorname{int}\Pi+[-1,1]\,u\subseteq\operatorname{int}Z$.

\emph{Reading off {\rm(a)} and {\rm(b)}.} We have $\rho P=\{(x,y):|x|\le k,\ |y-xQ/N|\le k/N\}$. Let $0\le i\le k-1$. A point of $\omega+2\Z^2$ with first coordinate $k-2i$ is of the form $(k-2i,\,Q+1-2\nu)$ with $\nu\in\Z$, and since $N-(k-2i)=2(k'+i)$,
\begin{equation}\label{eq:test}
 (k-2i,\,Q+1-2\nu)\in\rho P\iff\Bigl|\frac{(k'+i)Q}{N}+\frac12-\nu\Bigr|\le\frac{k}{2N}.
\end{equation}
For $i=0$ the point $t^*$ is of this form and lies in $\rho P$ by \eqref{eq:Q}. So $k'Q/N$ is within $k/(2N)$ of $\frac12+\Z$, and $\norm{k'Q/N}\ge\frac12-\frac{k}{2N}=k'/N$. This is (b). For $1\le i\le k-1$ we have $|k-2i|<k$, so by \eqref{eq:int} a point as in \eqref{eq:test} would lie in the interior of $\rho Z$, which contains no point of $\omega+2\Z^2$ by the minimality of $\rho$. Hence \eqref{eq:test} fails for every $\nu$, that is, $(k'+i)Q/N$ is at distance more than $k/(2N)$ from $\frac12+\Z$, and $\norm{(k'+i)Q/N}<k'/N$. As $i$ runs through $1,\dots,k-1$, the integer $j=k'+i$ runs through $k'<j<N-k'$. This is (a).
\end{proof}

\subsection{Three runners}

\begin{lemma}\label{lem:three}
Let $N\ge3$ and $k,k'\ge1$ be integers with $N=k+2k'$, and let $Q\in\Z$ satisfy {\rm(a)} and {\rm(b)} of Proposition~\ref{prop:contact}. Put $d=\gcd(Q,N)$. Then $kd\le N$, there is an integer $g$ with $0<g<N$, $\gcd(g,N)=1$ and $gQ\equiv-d\pmod N$, and for every such $g$
\[
 \ML(g,\,N-g,\,d)=\frac{k'}{N}.
\]
\end{lemma}
\begin{proof}
Write $\mu=k'/N$, $N=dN_1$ and $Q=dQ_1$. If $kd>N$, then $N_1\le k-1$, so $j=k'+N_1$ satisfies $k'<j<N-k'$ and $jQ\equiv k'Q\pmod N$, and (a) contradicts (b). Hence $kd\le N$. As $\gcd(Q_1,N_1)=1$, the congruence $gQ_1\equiv-1\pmod{N_1}$ is solvable, and since $(\Z/N)^\times\to(\Z/N_1)^\times$ is onto we may take $0<g<N$ with $\gcd(g,N)=1$. Then $gQ\equiv-d\pmod N$. Fix such a $g$.

\emph{Lower bound.} Choose $m\in\Z$ with $gm\equiv k'\pmod N$ and let $t=m/N$. Then $\norm{gt}=\norm{(N-g)t}=\mu$, and $dm\equiv-gQm\equiv-k'Q\pmod N$, so $\norm{dt}=\norm{k'Q/N}\ge\mu$ by~(b).

\emph{Upper bound.} Let $t\in\R$ with $\norm{gt}>\mu$ and $\norm{(N-g)t}>\mu$. We show that $\norm{dt}<\mu$. Choose $x,x'\in(\mu,1-\mu)$ with $x\equiv gt$ and $x'\equiv-(N-g)t\pmod1$, and put $\theta=x-x'$. Then $\theta\equiv Nt\pmod 1$ and $|\theta|<1-2\mu=k/N$, so $m=Nt-\theta$ is an integer. Now
\[
 \frac{gm}{N}=gt-\frac gN\,\theta\equiv x-\frac gN\,\theta=\Bigl(1-\frac gN\Bigr)x+\frac gN\,x'\pmod 1,
\]
and the right-hand side lies in $(\mu,1-\mu)$. Hence $gm\equiv j\pmod N$ for an integer $j$ with $k'<j<N-k'$. Since $d\equiv-gQ\pmod N$, we get $dm\equiv-jQ\pmod N$, so $\norm{dm/N}=\norm{jQ/N}\le(k'-1)/N$ by (a), both sides being multiples of $1/N$. Finally $dt=dm/N+d\theta/N$ with $|d\theta/N|<dk/N^2\le1/N$, and therefore $\norm{dt}<(k'-1)/N+1/N=\mu$.
\end{proof}

\begin{proof}[Proof of Theorem~\ref{thm:full}]
If $\ML(U)<\frac12$, apply Lemma~\ref{lem:three} to the integers $N,k,Q$ of Proposition~\ref{prop:contact}.
\end{proof}

\begin{corollary}\label{cor:arith}
For $\mu\in(0,\frac12)$ the following are equivalent.
\begin{enumerate}[label=\textup{(\roman*)},leftmargin=*]
\item $\mu=\ML(a,b,c)$ for some positive integers $a,b,c$.
\item $\mu=\ML(U)$ for some coordinate-free two-dimensional subtorus $U\subseteq\T^4$.
\item $\mu=k'/N$ for integers $N\ge3$ and $1\le k'<N/2$ such that some $Q\in\Z$ satisfies {\rm(a)} and {\rm(b)} of Proposition~\ref{prop:contact} with $k=N-2k'$.
\end{enumerate}
\end{corollary}
\begin{proof}
For (i)$\Rightarrow$(ii) take $U=T_{(a,b,c)}\times\T$. The implication (ii)$\Rightarrow$(iii) is Proposition~\ref{prop:contact}, and (iii)$\Rightarrow$(i) is Lemma~\ref{lem:three}.
\end{proof}

So the three-runner spectrum has a description by a finite test for each fraction $k'/N$, with no reference to speeds. Section~\ref{sec:three} solves this test in closed form.

\begin{remark}[Reading the data off $U$]\label{rem:read}
Let $n=4$ and $e=1$ in the proof of Proposition~\ref{prop:contact}, and suppose that no $c_i$ vanishes. The normal form then says that, up to symmetry, $U=\{(\chi(y),y):y\in W_A\}$, where $A=(A_1,A_2,A_3)$ has positive entries with sum $N$, $W_A=\{y\in\T^3:A\cdot y\in\Z\}$ and $\chi$ is a character of $W_A$. The maximum $\mu$ of $G$ on $U$ is attained at $(\chi(\mu,\mu,\mu),\mu,\mu,\mu)$, and $Q\equiv\pm N\chi(\frac1N,\frac1N,\frac1N)\pmod N$. The sign is immaterial, since replacing $Q$ by $-Q$ exchanges $g$ and $N-g$. For $U(A,B,C)$ as in Theorem~\ref{thm:obstr} this gives $(N,k,Q)=(A+B+C,\,2,\,N/2)$, hence $d=N/2$, $g=1$ and the triple $(1,N-1,N/2)$ of Theorem~\ref{thm:obstr}(i). For the subtorus of Remark~\ref{rem:slicing} with loneliness $\frac{13}{30}$ it gives $(N,k,Q)=(30,4,\pm8)$ and $\frac{13}{30}=\ML(2,11,19)$. For the subtorus $\{(10X+Y,-49X-5Y,9X,Y)\}$ of Section~\ref{sec:upto25} one finds $(N,k,Q)=(103,17,\pm25)$ and $\frac{43}{103}=\ML(1,33,70)$.
\end{remark}

\begin{remark}\label{rem:strict}
The strict inequality in (a) cannot be relaxed. The integers $N=25$, $k'=9$, $Q=7$ satisfy (b), and they satisfy (a) with $\le$ in place of $<$, with equality for $j=12,13$. But $\frac9{25}$ is not a three-runner value by Theorem~\ref{thm:upto25}. In the normalization of \cite{GiriKravitz} this is the value $\frac7{50}$ of \cite[Proposition~8.1]{GiriKravitz}.
\end{remark}

\begin{remark}\label{rem:subgroups}
Connectedness of $U$ is not used. Let $\Delta\subseteq\T^n$ be a closed subgroup of dimension $n-2$ whose identity component is coordinate-free, and define $c_1,\dots,c_n$ from a basis of its character lattice, which need not be primitive. Lemma~\ref{lem:zon} still holds (in (i), a nonzero multiple of the $i$th unit vector in the character lattice forces $x_i=0$ on the identity component), and the proof of Proposition~\ref{prop:contact} goes through verbatim, so $\max_{x\in\Delta}G(x)$ is again a three-runner value. In particular the elements of $\mathcal S^*(3)\setminus\mathcal S(3)$ in \cite{GiriKravitz}, such as $\frac7{50}$, can only come from subgroups whose identity component lies in a coordinate hyperplane.
\end{remark}

\section{The three-runner spectrum}\label{sec:three}
Let $\mathcal V_3\subseteq(0,\frac12]$ be the set of values of $\ML$ on triples of positive integers, so that $\mathcal S(3)=\frac12-\mathcal V_3$ in the notation of \cite{GiriKravitz}. By Corollaries~\ref{cor:full} and~\ref{cor:arith}, $\mathcal V_3$ is also the set of values of $\ML$ on coordinate-free two-dimensional subtori of $\T^4$ and the set of accumulation points of the four-runner values, and $\mathcal V_3\setminus\{\frac12\}$ consists of the numbers $\frac12-\frac k{2N}$, where $(N,k)$ runs through the pairs of integers with $N>k\ge1$ and $N\equiv k\pmod 2$ for which some $Q\in\Z$ satisfies (a) and (b) of Proposition~\ref{prop:contact} with $k'=(N-k)/2$. We call such a pair \emph{admissible}. The set $\mathcal V_3$ was known below $\frac25$. Below $\frac13$ it is due to Kravitz \cite{Kravitz}. Between $\frac13$ and $\frac25$ Jain and Kravitz \cite{JainKravitz} described it up to finitely many exceptions, and Theorem~\ref{thm:upto25} gives it exactly. Here we determine it completely.

\begin{theorem}\label{thm:three}
A pair $(N,k)$ is admissible if and only if
\begin{equation}\label{eq:star}
 N=kH-2ja
\end{equation}
for some integers $H,a,j$ with $H\ge3$ odd, $0\le a<k$, $\gcd(a,k)=1$, $0\le j\le k/2$ and $a+j<H$. Consequently
\[
 \mathcal V_3\setminus\{\tfrac12\}=\Bigl\{\frac{x-1}{2x}:x\in X\Bigr\},\qquad X=\Bigl\{H-\frac{2ja}{k}\ :\ H,a,j,k\ \text{as above}\Bigr\},
\]
and $\mathcal S(3)=\{0\}\cup\{1/(2x):x\in X\}$.
\end{theorem}

For $k=1$ the conditions force $a=j=0$, and \eqref{eq:star} gives the two-runner values $(H-1)/(2H)$. Every pair as in \eqref{eq:star} has $N\equiv k\pmod2$ and $N>k$, because $2ja\le ka\le k(H-1-j)$ gives $N-k\ge kj$, and $N-k=k(H-1)$ if $j=0$. Membership in $\mathcal V_3$ is decided by a finite check. If $x=1/(1-2\mu)$ lies in $X$, then $x\le H\le3x-2$ (see the proof of Corollary~\ref{cor:acc3}), and $k=2ja/(H-x)$ is determined by $H,a,j$ unless $x=H$.

\begin{proposition}\label{prop:witness}
Let $N=kH-2ja$ be as in Theorem~\ref{thm:three}, and put $c=k-2j$, $w=H-a$ and $\mu=\frac12-\frac{k}{2N}$.
\begin{enumerate}[label=\textup{(\roman*)},leftmargin=*]
\item If $a\ge1$ and $\gcd(a,H)=1$, then $\ML(a,\,H-a,\,N-a)=\mu$.
\item In general let $d=\gcd(c,w)$. There are integers $\alpha,\beta$ with $\alpha w-\beta c=d$ and $\gcd(\alpha a+\beta k,N)=1$, and then $\ML(g,\,N-g,\,d)=\mu$ for the residue $g$ of $\alpha a+\beta k$ modulo $N$.
\end{enumerate}
\end{proposition}

Writing $b=H-a$, part (i) says that
\begin{equation}\label{eq:family}
 \ML\bigl(a,\ b,\ k(a+b)-(2j+1)a\bigr)=\frac12-\frac{k}{2\bigl(k(a+b)-2ja\bigr)}
\end{equation}
whenever $a,b\ge1$ are coprime with $a+b$ odd, $k>a$ is coprime to $a$, and $0\le j\le\min(k/2,\,b-1)$. For $(a,b,j)=(1,2,1)$ this is $\ML(1,2,3s)=s/(3s+1)$ with $s=k-1$. For $(a,b,j)=(2,3,2)$ and $(1,4,3)$ it gives $\ML(2,3,5i)=2i/(5i+2)$ for odd $i\ge3$ and $\ML(1,4,5i+3)=(2i+1)/(5i+4)$ for $i\ge4$, with $i=k-2$. These triples belong to the two families of Theorem~\ref{thm:upto25}.

\subsection{Proofs}
For $u\in\R$ let $|u|_N$ be the distance from $u$ to $N\Z$. If $m\in\Z$, then $\norm{m/N}=|m|_N/N$ and $|m|_N+|m-\frac N2|_N=\frac N2$. Since $k'=\frac N2-\frac k2$, this gives
\begin{equation}\label{eq:shift}
\begin{aligned}
 \norm{m/N}\ge k'/N&\iff\bigl|m-\tfrac N2\bigr|_N\le\tfrac k2,\\
 \norm{m/N}<k'/N&\iff\bigl|m-\tfrac N2\bigr|_N>\tfrac k2 .
\end{aligned}
\end{equation}

\begin{proof}[Proof of Theorem~\ref{thm:three}]
Fix $N>k\ge1$ with $N\equiv k\pmod2$, let $Q\in\Z$, and put $u_i=(k'+i)Q-\frac N2$. By \eqref{eq:shift}, (b) says that $|u_0|_N\le\frac k2$, and (a) says that $|u_i|_N>\frac k2$ for $0<i<k$. If $k=1$, then (a) is void and $Q=1$ satisfies (b). So $(N,1)$ is admissible for every odd $N\ge3$, in agreement with \eqref{eq:star}. Let $k\ge2$.

\emph{Step 1: the meaning of {\rm(b)}.} Condition (b) holds if and only if there is $m_1\in\Z$ with $|k'Q-\frac N2-Nm_1|\le\frac k2$. Put $c=N(2m_1+1)-2k'Q$ and $m_0=Q-2m_1-1$. Since $2k'=N-k$, this means
\begin{equation}\label{eq:cm}
 kQ=c+Nm_0,\qquad |c|\le k,\qquad m_0\equiv Q+1\pmod 2,
\end{equation}
and conversely \eqref{eq:cm} gives (b) with $m_1=(Q-m_0-1)/2$. Moreover $u_0\equiv-\frac c2\pmod N$, hence
\begin{equation}\label{eq:ui}
 u_i=u_0+iQ\equiv c\Bigl(\frac ik-\frac12\Bigr)+N\,\frac{im_0}{k}\pmod N\qquad(0\le i\le k).
\end{equation}

\emph{Step 2: the meaning of {\rm(a)}.} Assume \eqref{eq:cm}. Conditions (a), (b) and \eqref{eq:cm} are preserved when $(Q,c,m_0)$ is replaced by $(-Q,-c,-m_0)$, so we may assume $c\ge0$. If $\gcd(m_0,k)>1$, then $i=k/\gcd(m_0,k)$ satisfies $0<i<k$ and $im_0/k\in\Z$, so $|u_i|_N\le c\,|\frac ik-\frac12|<\frac k2$ by \eqref{eq:ui}, and (a) fails. Let $\gcd(m_0,k)=1$, and let $a\in\{1,\dots,k-1\}$ with $am_0\equiv-1\pmod k$. Multiplying $c+Nm_0\equiv0\pmod k$ by $a$ we get $N\equiv ca\pmod k$. Put $w=(N-ca)/k$. For $1\le t\le k-1$ let $i_t\in\{1,\dots,k-1\}$ with $i_t\equiv-ta\pmod k$ and $z_t=(i_t+ta)/k$. Thus $z_t\ge1$ is an integer, $z_1=1$, and $t\mapsto i_t$ is a permutation with $i_tm_0\equiv t\pmod k$ and $i_{k-t}=k-i_t$. By \eqref{eq:ui},
\[
 u_{i_t}\equiv v_t:=\frac{Nt}{k}+c\Bigl(\frac{i_t}{k}-\frac12\Bigr)\pmod N ,\qquad -\tfrac k2<v_t<N+\tfrac k2 ,
\]
so $|u_{i_t}|_N>\frac k2$ if and only if $\frac k2<v_t<N-\frac k2$. As $v_{k-t}=N-v_t$, condition (a) holds if and only if $v_t>\frac k2$ for all $t$. Since $N=ca+kw$ and $i_t=kz_t-ta$,
\[
 v_t-\frac k2=\frac{2Nt+2ci_t-ck-k^2}{2k}=\frac{2(tw+cz_t)-(k+c)}{2}.
\]
For $t=1$ the condition $v_1>\frac k2$ reads $2w>k-c$. Conversely, if $2w>k-c$, then $w\ge1$ and $tw+cz_t\ge w+c$ for all $t$. Hence, under \eqref{eq:cm} with $c\ge0$,
\begin{equation}\label{eq:aw}
 \text{(a) holds}\iff\gcd(m_0,k)=1\ \text{ and }\ 2w>k-c .
\end{equation}

\emph{Step 3: parity.} Let $\gcd(m_0,k)=1$, let $a,w$ be as in Step~2, and choose $z_0\in\Z$ with $am_0=kz_0-1$. From $kQ=c+(ca+kw)m_0=c(1+am_0)+kwm_0$ we get
\begin{equation}\label{eq:Qz}
 Q=cz_0+wm_0 ,
\end{equation}
so the congruence in \eqref{eq:cm} says that $cz_0+(w-1)m_0$ is odd. If $k$ is even, then $a$ and $m_0$ are odd and $ca=N-kw$ is even. So $c$ is even, and the congruence says that $w$ is even. If $k$ is odd, then $N=ca+kw$ is odd and $z_0\equiv am_0+1\pmod2$, so $cz_0+(w-1)m_0\equiv c+m_0(ca+w-1)\equiv c\pmod2$. The congruence then says that $c$ is odd, and in that case $a+w$ is odd. In both cases
\begin{equation}\label{eq:par}
 m_0\equiv Q+1\pmod2\iff c\equiv k\pmod 2\ \text{ and }\ a+w\ \text{is odd}.
\end{equation}

\emph{Step 4: conclusion.} Let $(N,k)$ be admissible, $k\ge2$, and let $Q$ satisfy (a) and (b). Steps 1 to 3 give $c,a,w$ with $0\le c\le k$, $c\equiv k\pmod2$, $1\le a<k$, $\gcd(a,k)=1$, $N=ca+kw$, $2w>k-c$ and $a+w$ odd. Put $j=(k-c)/2$ and $H=a+w$. Then $0\le j\le k/2$, $H$ is odd, $N=(k-2j)a+k(H-a)=kH-2ja$, and $2w>k-c$ says that $a+j<H$. In particular $H\ge3$. Conversely let $N=kH-2ja$ be as in the theorem, with $k\ge2$. Then $a\ge1$. Put $c=k-2j$ and $w=H-a$, choose $m_0,z_0\in\Z$ with $am_0-kz_0=-1$, and define $Q$ by \eqref{eq:Qz}. Then $N=ca+kw$ and $kQ=c+Nm_0$, we have $0\le c\le k$ and $\gcd(m_0,k)=1$, the parity condition holds by \eqref{eq:par}, and $2w>2j=k-c$. So \eqref{eq:cm} holds, which gives (b), and (a) holds by \eqref{eq:aw}. This proves the first assertion. The second one is its translation by Corollary~\ref{cor:arith}, with $x=N/k$.
\end{proof}

\begin{proof}[Proof of Proposition~\ref{prop:witness}]
If $k=1$, then $a=j=0$, $c=d=1$ and $w=N$, and (ii) holds with $(\alpha,\beta)=(0,-1)$, since $\ML(N-1,1,1)=\ML(1,N-1)=\mu$. Let $k\ge2$, and let $m_0,z_0,Q$ be as in Step~4, so that $Q$ satisfies (a) and (b). The vectors $e=(a,z_0)$ and $e'=(k,m_0)$ form a basis of $\Z^2$ with $\det(e,e')=-1$, and $(N,Q)=ce+we'$. Hence $\gcd(Q,N)=\gcd(c,w)=d$, and for $(g_0,q)=\alpha e+\beta e'$ we have $g_0=\alpha a+\beta k$ and
\begin{equation}\label{eq:gQ}
 g_0Q-Nq=\det\bigl((g_0,q),(N,Q)\bigr)=-(\alpha w-\beta c).
\end{equation}

(ii) Lemma~\ref{lem:three} provides $g$ with $0<g<N$, $\gcd(g,N)=1$, $gQ\equiv-d\pmod N$ and $\ML(g,N-g,d)=\mu$. Writing $gQ-Nq=-d$ and $(g,q)=\alpha e+\beta e'$, we get $\alpha w-\beta c=d$ from \eqref{eq:gQ}, and $g=\alpha a+\beta k$. Conversely, if $\alpha w-\beta c=d$ and $\alpha a+\beta k$ is coprime to $N$, then its residue $g$ satisfies $gQ\equiv-d\pmod N$ by \eqref{eq:gQ}, and Lemma~\ref{lem:three} applies to it.

(i) By \eqref{eq:gQ} with $(\alpha,\beta)=(1,0)$ we have $aQ\equiv-w\pmod N$. Moreover $\gcd(a,N)=\gcd(a,kH)=\gcd(a,H)=1$, and $1\le w\le N/k$ because $N/k=H-2ja/k\ge H-a$. The lower and the upper bound in the proof of Lemma~\ref{lem:three} use only that $\gcd(g,N)=1$, that $gQ\equiv-d\pmod N$ and that $kd\le N$, so they apply verbatim to $(g,d)=(a,w)$. Hence $\ML(a,N-a,w)=k'/N=\mu$.
\end{proof}

\subsection{Consequences}
For a set $S\subseteq\R$ let $\operatorname{acc}^{(0)}(S)=S$ and $\operatorname{acc}^{(m+1)}(S)=\operatorname{acc}(\operatorname{acc}^{(m)}(S))$.

\begin{corollary}\label{cor:acc3}
\begin{enumerate}[label=\textup{(\roman*)},leftmargin=*]
\item The accumulation points of the four-runner values are exactly $\frac12$ and the numbers $\frac{x-1}{2x}$, $x\in X$. The same numbers are the values of $\ML$ on the coordinate-free subtori of dimension $n-2$ of $\T^n$, for every $n\ge3$.
\item For every $n\ge3$ the accumulation points of order $n-3$ of the Lonely Runner spectrum $\mathcal S(n)$ form the set $\mathcal S(3)$:
\begin{gather*}
 \operatorname{acc}^{(n-3)}(\mathcal S(n))=\mathcal S(3),\qquad \operatorname{acc}^{(n-2)}(\mathcal S(n))=\mathcal S(2),\\
 \operatorname{acc}^{(n-1)}(\mathcal S(n))=\{0\}.
\end{gather*}
In particular $\operatorname{acc}(\operatorname{acc}(\mathcal S(5)))=\mathcal S(3)$.
\end{enumerate}
\end{corollary}
\begin{proof}
(i) is Theorem~\ref{thm:three} combined with Corollary~\ref{cor:full}.

(ii) By \cite[Theorem~1.4]{GiriKravitz}, $\operatorname{acc}(\mathcal S_k(n))\subseteq\mathcal S_{k+1}(n)$ for $1\le k<n$ and $\mathcal S(m-1)\subseteq\operatorname{acc}(\mathcal S(m))$ for $m\ge2$. Since $\operatorname{acc}$ is monotone and $\mathcal S(n)=\mathcal S_1(n)$, induction on $m$ gives $\operatorname{acc}^{(m)}(\mathcal S(n))\subseteq\mathcal S_{m+1}(n)$ and $\mathcal S(n-m)\subseteq\operatorname{acc}^{(m)}(\mathcal S(n))$ for $0\le m<n$. For $m=n-3$ the two sides are equal by Corollary~\ref{cor:full}(i). Finally let $x=H-2ja/k\in X$ and $s=H-x$. As $2j\le k$ and $a<k$, we have $s\le a$ and $s\le2j$, so $H-1\ge a+j\ge\frac32s$, that is, $H\le3x-2$. Hence a bounded set of elements of $X$ involves only finitely many $(H,a,j)$. A convergent sequence of distinct elements of $X$ therefore has $k\to\infty$, and its limit is an odd integer $H\ge3$, approached from below. Conversely $H-2/k\in X$ for $k\ge2$. Thus $\operatorname{acc}(\mathcal S(3))=\{0\}\cup\{1/(2H):H\ge3\text{ odd}\}=\mathcal S(2)$, and $\operatorname{acc}(\mathcal S(2))=\{0\}$.
\end{proof}

Conjecture~1.5 of \cite{GiriKravitz}, for all $n$, is equivalent to the equalities $\operatorname{acc}^{(m)}(\mathcal S(n))=\mathcal S(n-m)$ for all $n$ and all $m<n$. Part (ii) gives these equalities for $m\ge n-3$. The equality $\operatorname{acc}(\mathcal S(3))=\mathcal S(2)$ is already in \cite{GiriKravitz}.

\begin{corollary}\label{cor:windows}
Let $h\ge3$ be odd. The elements of $X$ in the interval $(h-2,h)$ are
\begin{enumerate}[label=\textup{(\roman*)},leftmargin=*]
\item the numbers $h-2\rho/k$, where $\rho=ja<k$ for integers $1\le a<k$ and $1\le j\le k/2$ with $\gcd(a,k)=1$ and $a+j\le h-1$,
\item finitely many further numbers, which have $H>h$, hence $k<ja$, $a+j<H$ and $H<3h-2$ in \eqref{eq:star}.
\end{enumerate}
There are no further numbers for $h=3$ and $h=5$, and for $h=7$ they are $\frac{21}4$ and $\frac{49}9$. In particular:
\begin{enumerate}[label=\textup{(\alph*)},leftmargin=*]
\item $\mathcal V_3\cap(0,\frac13)=\{s/(3s+1):s\ge1\}$.
\item $\mathcal V_3\cap(\frac13,\frac25)$ consists of the numbers $(2k-\rho)/(5k-2\rho)$ with $\rho=1$ and $k\ge2$, or $\rho=2$ and $k\ge3$, or $\rho=3$ and $k\ge4$, or $\rho=4$ and $k\ge5$ odd. Together with $\frac13$ this is the set $\mathcal R$ of Theorem~\ref{thm:upto25}.
\item $\mathcal V_3\cap(\frac25,\frac37)$ consists of the numbers $(3k-\rho)/(7k-2\rho)$, where $(\rho,k)$ satisfies one of the following: $\rho=1$ and $k\ge2$, $\rho=2$ and $k\ge3$, $\rho=3$ and $k\ge4$, $\rho=4$ and $k\ge5$ with $k$ odd or $k\ge8$, $\rho=5$ and $k\ge6$, $\rho=6$ and $k\ge7$ with $6\nmid k$, $\rho=7$ and $k\in\{8,9\}$, $\rho=8$ and $k\ge9$ odd, $\rho=9$ and $k\ge10$ with $3\nmid k$. If $p/q$ is such a number in lowest terms, then $3q-7p$ divides $\rho$, and all of $1,\dots,9$ occur as $3q-7p$.
\end{enumerate}
\end{corollary}
\begin{proof}
Let $x=H-2ja/k\in X$ with $h-2<x<h$, and put $s=H-x=2ja/k$. Since $x\le H$ and $H$ is odd, $H\ge h$. If $H=h$, then $0<s<2$, that is, $0<ja<k$, and we are in case~(i). Let $H>h$. Then $s>2$, so $k<ja$. As $2j\le k$ and $a<k$, we have $s\le a$ and $s<2j$, hence $H-1\ge a+j>\frac32s=\frac32(H-x)$ and $H<3x-2<3h-2$. So only finitely many $(H,a,j)$ occur, each with finitely many $k<ja$. For $h\le7$ this finite list is short. There is no quadruple $(H,a,j,k)$ for $h=3$, there are $8$ for $h=5$ and $50$ for $h=7$, and among the numbers they give only those of $(H,a,j,k)=(9,5,3,8)$ and $(9,4,4,9)$ are not already in~(i) (Section~\ref{sec:comp}).

Now (a) is the case $h=3$, where $a=j=1$ and $x=3-2/k$. For (b) take $h=5$. The pairs $(j,a)$ with $a+j\le4$ have $\rho=ja\in\{1,2,3,4\}$, and collecting, for each $\rho$, the conditions $k>\rho$, $k\ge2j$, $k>a$, $\gcd(a,k)=1$ over the factorizations $\rho=ja$ gives the list. Since $2(5k-2\rho)-5(2k-\rho)=\rho$, these numbers lie in $\mathcal R$. Conversely let $p/q\in\mathcal R\setminus\{\frac13\}$ be in lowest terms, $r=2q-5p$ and $k=q-2p$. Then $p=2k-r$, $q=5k-2r$, and $p/q>\frac13$ gives $k>r$. Moreover $k$ is odd if $r=4$, because $\gcd(p,q)=\gcd(4,k)$ in that case. So $(\rho,k)=(r,k)$ is in the list. For (c) take $h=7$. The pairs $(j,a)$ with $a+j\le6$ have $\rho=ja\in\{1,2,3,4,5,6,8,9\}$, and collecting, for each $\rho$, the same conditions over the factorizations $\rho=ja$ gives the list. The case $\rho=7$ consists of the two numbers of~(ii), which give $\frac{17}{42}$ and $\frac{20}{49}$. Finally $3(7k-2\rho)-7(3k-\rho)=\rho$, so $3q-7p$ divides $\rho$. For $1\le\rho\le6$ the pair $(\rho,\rho+1)$ is in the list and $(2\rho+3)/(5\rho+7)$ is in lowest terms, so $3q-7p=\rho$ occurs. The values $\frac{17}{42}$, $\frac{43}{103}$ and $\frac{21}{52}$, which belong to $(\rho,k)=(7,8)$, $(8,17)$ and $(9,10)$, have $3q-7p=7,8,9$.
\end{proof}

\begin{remark}\label{rem:level37}
Part (c) completes Section~\ref{sec:upto25}. The contact argument there gives $1\le3q-7p\le6$ at vertices of determinant one. The true range is $1\le3q-7p\le9$. The value $\frac{43}{103}$ of the subtorus found there is the case $(\rho,k)=(8,17)$, and for instance $\frac{21}{52}=\ML(4,23,29)$ has $3q-7p=9$. The two values with $\rho=7$ are $\frac{17}{42}=\ML(2,13,29)$ and $\frac{20}{49}=\ML(3,22,27)$. In general the window $(h-2,h)$ carries the regular families (i), one for each pair $(j,a)$ with $a+j\le h-1$, each converging to the two-runner value $(h-1)/(2h)$, and the finite set (ii), which has $2,3,20,36$ elements for $h=7,9,11,13$. For $h=9$ it gives the values $\frac{37}{86}$, $\frac{41}{95}$ and $\frac{45}{104}$. Also, the statements of Theorems~\ref{thm:second-intro} and~\ref{thm:upto25-intro} on the values of the subtori and on the accumulation points now follow from Theorem~\ref{thm:full} and Corollary~\ref{cor:windows}, without the enumerations of Propositions~\ref{prop:third-enum} and~\ref{prop:enum25}. The classification of the subtori in Theorem~\ref{thm:second} still rests on the first of these.
\end{remark}

The parametrization also shows which distances from the upper end of a window recur. The upper endpoint of the window $h-2<(1-2\mu)^{-1}<h$ is $\mu=\frac mh$ with $h=2m+1$.

\begin{proposition}\label{prop:defect}
Let $h=2m+1\ge3$, and let $\mu=\frac pq$ in lowest terms be a value of $\ML$ on triples with $h-2<(1-2\mu)^{-1}<h$. Put $D=mq-hp$, so that $\frac mh-\mu=\frac D{hq}$. Then $1\le D\le m^2$. A positive integer $D$ occurs for infinitely many such values $\mu$ if and only if $D=uv$ for positive integers $u,v$ with $u+v\le2m$. In particular $D=m^2$ occurs infinitely often.
\end{proposition}
\begin{proof}
By Theorem~\ref{thm:three} write $x=(1-2\mu)^{-1}=H-2ja/k$ as in \eqref{eq:star}, with $H\ge h$ odd as in the proof of Corollary~\ref{cor:windows}, $H=h+2e$, $0\le a<k$, $\gcd(a,k)=1$, $0\le2j\le k$ and $a+j\le H-1$. The numerator and denominator of $\mu=\frac12-\frac1{2x}$ before reduction are $P=\frac{H-1}2k-ja$ and $Q=Hk-2ja$. Since $Q-2P=k$, their greatest common divisor is $g=\gcd(k,ja)=\gcd(k,j)$, and $gD=mQ-hP=ja-ek$. With $b=a-2e$ and $t=k-2j\ge0$ this reads $gD=jb-et$. As $\mu<\frac mh$, $D>0$, so $jb>et\ge0$, hence $j,b\ge1$, and $j+b\le H-1-2e=2m$. Thus $D\le jb\le m^2$.

By Corollary~\ref{cor:windows}(ii), only finitely many values in the window need $H>h$. If $H=h$, then $e=0$ and $D=a\cdot(j/g)$ with $a+j/g\le a+j\le2m$. Conversely, for $u,v\ge1$ with $u+v\le2m$ take $H=h$, $a=u$, $j=v$ and any $k>\max(uv,2v)$ prime to $uv$. All conditions hold, $g=1$, $D=uv$, and the numbers $x=h-2uv/k$ lie in $(h-2,h)$ and are distinct for distinct $k$. For $u=v=m$ this gives $D=m^2$.
\end{proof}

For $h=5$ and $h=7$ the number $D$ is $2q-5p$ of Theorem~\ref{thm:upto25} and $3q-7p$ of Corollary~\ref{cor:windows}(c). In the window $(5,7)$ the values $D=1,2,3,4,5,6,8,9$ recur, and $D=7$ occurs only for the two values $\frac{17}{42}$ and $\frac{20}{49}$ of Remark~\ref{rem:level37}.

\begin{remark}[Five runners]\label{rem:five}
The proof of Theorem~\ref{thm:full} does not extend to two-dimensional subtori of $\T^5$, and the obstruction is the value $\frac3{14}$ of Theorem~\ref{thm:values}. The cyclic group generated by $(\frac1{14},\frac2{14},\frac3{14})$ in $\T^3$ has $\max G=\frac3{14}$, attained only at $\pm3$ times the generator. This is the exact analogue of the data $(N,k',Q)$ of Proposition~\ref{prop:contact}, with $(N,k')=(14,3)$ and $Q=(2,3)\in\Z^2$, but $\frac3{14}$ is not a four-runner value. Thus the analogue of Lemma~\ref{lem:three} is false, and so is the analogue of Remark~\ref{rem:subgroups}: the closed subgroup $\Delta=\langle(\frac18,\frac18,0,\frac14)\rangle+T_{(1,-3,-4,-2)}$ of $\T^4$ has two components, its identity component is the circle of $\{1,2,3,4\}$, and $\max_\Delta G=\frac3{14}$, attained on the other component. In the notation of \cite{GiriKravitz} this gives $\frac27\in\mathcal S^*(4)\setminus\mathcal S(4)$ and, by Theorem~\ref{thm:five}, $\frac27\in\mathcal S_2^*(5)\setminus\mathcal S_2(5)$, the analogue for five runners of \cite[Proposition~8.1]{GiriKravitz} and Corollary~\ref{cor:acc25}(ii). A proof of $\mathcal S_2(5)=\mathcal S(4)$ therefore has to use more of the subtorus than the finite group at the contact or a slice through a maximizer. For two-dimensional subtori of $\T^n$ with $n\ge5$ the polygon $Z$ of Section~\ref{sec:full} is replaced by a zonotope in $\R^{n-2}$, and the face of first contact need not be a segment.
\end{remark}

\section{Five and six speeds: a search tree over flags}\label{sec:tree}
For five and six speeds the bounds of Theorem~\ref{thm:general}(iii) and Corollary~\ref{cor:five}(ii) are out of reach of an exhaustive search. Bonfioli \cite{BonfioliFour,BonfioliFive} avoided the enumeration of tuples. He searches a tree whose nodes are the possible initial segments of a reduced basis of the lattice of relations of the speeds, prunes it with a nearest-plane estimate at a witness point of each subtorus, and obtains for five speeds a certificate with $693{,}222$ leaves that is checked by the Lean kernel \cite[Sections~4 and~5]{BonfioliFive}. This section uses the same tree with a sharper pruning test. The first step of the nearest-plane rounding is computed exactly instead of being bounded by $\frac12$, the displacement of each coordinate is bounded with a weight that depends on the coordinate, and every candidate is tested against several balls of its parent. The tree for five speeds then has $28{,}599$ leaves besides the critical ones, and the tree for six speeds can be searched with exact integer arithmetic in a few minutes. The outcome for five speeds agrees with \cite{BonfioliFive}, and for six speeds it completes the results of Cordella \cite{Cordella}.

In this section $n\ge4$, and $\T^n$, $G$, subtori, $\Lambda_S$, $\covol$, $T_v$ and $\ML$ are as in Section~\ref{sec:general}. Put $C_t=\sum_{s=0}^{t-1}(\frac43)^s$, so that $C_0=0$, $C_1=1$ and $C_{t+1}=1+\frac43C_t$.

\subsection{Flags}\label{subsec:flags}
For a primitive sublattice $L\subseteq\Z^n$ of rank $k$ let $T(L)=\{x\in\T^n:h\cdot x\in\Z\text{ for all }h\in L\}$. It is a subtorus of dimension $n-k$ with $\R\Lambda_{T(L)}=E(L)$, the orthogonal complement of $\R L$, and $\covol(T(L))=\covol(L)$, as in the proof of Lemma~\ref{lem:shortchar}. Let $\pi_L:\R^n\to E(L)$ be the orthogonal projection. Let $v\in\Z^n$ be primitive with all $v_i\ne0$. Its lattice of relations $L(v)=\{h\in\Z^n:h\cdot v=0\}$ is primitive of rank $n-1$, and $T(L(v))=T_v$.

A \emph{flag} of $v$ is a basis $h_1,\dots,h_{n-1}$ of $L(v)$ such that, with $L_k=\Z h_1+\dots+\Z h_k$ and $L_0=0$, the vector $\pi_{L_{k-1}}(h_k)$ is a shortest nonzero vector of the lattice $\pi_{L_{k-1}}(L(v))$ for $1\le k\le n-1$. These are the bases that are reduced in the sense of Korkine and Zolotarev \cite{LLS}. They exist: a shortest nonzero vector of the lattice $\pi_{L_{k-1}}(L(v))$ is primitive in it, so any preimage $h_k$ extends $h_1,\dots,h_{k-1}$ to a basis of a primitive sublattice of $L(v)$, and one continues. Since $L(v)$ is primitive in $\Z^n$, so is every $L_k$. We write
\[
 m_k=\eucl{\pi_{L_{k-1}}(h_k)}^2,\qquad T_k=T(L_k),\qquad E_k=E(L_k),
\]
so that $\T^n=T_0\supseteq T_1\supseteq\dots\supseteq T_{n-1}=T_v$ and $\covol(L_k)^2=m_1\cdots m_k$.

\begin{lemma}\label{lem:flag}
Let $h_1,\dots,h_{n-1}$ be a flag of $v$.
\begin{enumerate}[label=\textup{(\alph*)},leftmargin=*]
\item $m_{k+1}\ge\frac34m_k$ for $1\le k\le n-2$. Consequently $\sum_{i=k+2}^{n-1}m_i^{-1}\le\frac43C_{n-k-2}\,m_{k+1}^{-1}$ for $0\le k\le n-2$.
\item Let $0\le k\le n-2$ and $x\in T_k$. There are $y\in T_v$ and $u\in E_k$ with $y=x+u$ in $\T^n$ and
\[
 \eucl u^2\le\frac{\norm{h_{k+1}\cdot x}^2}{m_{k+1}}+\frac14\sum_{i=k+2}^{n-1}\frac1{m_i}.
\]
\end{enumerate}
\end{lemma}
\begin{proof}
(a) Let $g=\pi_{L_{k-1}}(h_k)$, so $\eucl g^2=m_k$, and write $\pi_{L_{k-1}}(h_{k+1})=\pi_{L_k}(h_{k+1})+\lambda g$. For an integer $a$ nearest to $\lambda$, the vector $\pi_{L_{k-1}}(h_{k+1}-ah_k)=\pi_{L_k}(h_{k+1})+(\lambda-a)g$ is a nonzero vector of $\pi_{L_{k-1}}(L(v))$, so its squared length $m_{k+1}+(\lambda-a)^2m_k$ is at least $m_k$. This gives $m_{k+1}\ge\frac34m_k$. Hence $m_i\ge(\frac34)^{i-k-1}m_{k+1}$ for $i\ge k+1$, and $\sum_{i=k+2}^{n-1}m_i^{-1}\le m_{k+1}^{-1}\sum_{s=1}^{n-k-2}(\frac43)^s$.

(b) Let $k\le j\le n-2$ and $x_j\in T_j$. The vector $g_{j+1}=\pi_{L_j}(h_{j+1})\in E_j$ satisfies $h_{j+1}\cdot g_{j+1}=m_{j+1}$ and $h\cdot g_{j+1}=0$ for $h\in L_j$. Choose a real $s$ with $|s|=\norm{h_{j+1}\cdot x_j}/m_{j+1}$ and $h_{j+1}\cdot x_j+sm_{j+1}\in\Z$. Then $x_{j+1}=x_j+sg_{j+1}$ lies in $T_{j+1}$, and $\eucl{sg_{j+1}}^2=\norm{h_{j+1}\cdot x_j}^2/m_{j+1}\le\frac1{4m_{j+1}}$. Starting from $x_k=x$ we arrive at $x_{n-1}\in T_{n-1}=T_v$. The displacements are multiples of $g_{k+1},\dots,g_{n-1}$, which lie in $E_k$ and are pairwise orthogonal, and the first one has squared length $\norm{h_{k+1}\cdot x}^2/m_{k+1}$.
\end{proof}

Part (b) is the nearest-plane rounding along the flag. Bonfioli uses it in the same way \cite[Section~4]{BonfioliFive}, and Allikvere bounds counterexamples to the conjecture along flags of subtori \cite{Allikvere}.

\subsection{Alive balls and the pruning rules}
For a subtorus $S\subseteq\T^n$ with $E=\R\Lambda_S$, a point $x\in S$ and $\rho>0$ let
\[
 B_S(x,\rho)=\{x+u:u\in E,\ \eucl u\le\rho\}\subseteq S,
\]
and let $w_j(S)$ be the Euclidean length of the orthogonal projection of the $j$th unit vector to $E$. Call the ball $B_S(x,\rho)$ \emph{alive} if
\begin{equation}\label{eq:alive}
 \norm{x_j}-\frac1n\ \ge\ \rho\,w_j(S)\qquad(1\le j\le n).
\end{equation}
Then $G\ge\frac1n$ on $B_S(x,\rho)$, because the $j$th coordinate of $u\in E$ is the inner product of $u$ with the projection of the $j$th unit vector, of absolute value at most $\eucl u\,w_j(S)$. Note that $w_j(S)\le1$ and $\sum_jw_j(S)^2=\dim S$, so on a subtorus of small dimension the weights allow much larger balls than the maximum norm would.

\begin{lemma}[Pruning]\label{lem:prune}
Let $v\in\Z^n$ be primitive with all $v_i\ne0$ and $\ML(v)<\frac1n$, let $h_1,\dots,h_{n-1}$ be a flag of $v$, let $0\le k\le n-2$, and let $B_{T_k}(x,\rho)$ be an alive ball. Then
\begin{equation}\label{eq:prune}
 \rho^2-\frac{\norm{h_{k+1}\cdot x}^2}{m_{k+1}}\ <\ \frac{C_{n-k-2}}{3\,m_{k+1}} .
\end{equation}
In particular $m_{k+1}<C_{n-k-1}/(4\rho^2)$.
\end{lemma}
\begin{proof}
Lemma~\ref{lem:flag} gives $y\in T_v$ and $u\in E_k$ with $y=x+u$ and $\eucl u^2\le\norm{h_{k+1}\cdot x}^2/m_{k+1}+C_{n-k-2}/(3m_{k+1})$. Since $G(y)\le\ML(v)<\frac1n$, the point $y$ is not in the ball, so $\eucl u>\rho$. The last statement follows from $\norm{\cdot}\le\frac12$ and $\frac14+\frac13C_t=\frac14C_{t+1}$.
\end{proof}

For $k=n-2$ we have $C_0=0$, and \eqref{eq:prune} says that the circle $T_v$, whose lifts to $E_{n-2}$ are parallel lines at distance $m_{n-1}^{-1/2}$ from each other, passes at distance more than $\rho$ from $x$.

\subsection{The tree}
Call a subtorus $S\subseteq\T^n$ \emph{admissible} if no coordinate vanishes identically on $\Lambda_S$ and no two coordinates agree up to sign on $\Lambda_S$. The tree is defined by recursion on the depth. A \emph{node} of depth $k$, $0\le k\le n-2$, is a chain $0=L_0\subset L_1\subset\dots\subset L_k$ of primitive sublattices of $\Z^n$, with $L_i$ of rank $i$ and $T(L_k)$ admissible, in which each $L_{i+1}$ is a child, in the sense defined below, of the node $L_0\subset\dots\subset L_i$. We denote a node by its last lattice $L=L_k$ and write $L^-=L_{k-1}$. The root is $L=0$. To every node we attach a finite set $\mathcal B(L)$ of alive balls of $T(L)$, for the root the single ball with centre $(\frac12,\dots,\frac12)$ and radius $\frac12-\frac1n$. How the balls are found plays no role in the argument. A \emph{child} of a node $L$ of depth $k$ is a primitive lattice $L'=L+\Z h$ of rank $k+1$ such that, with $m=\eucl{\pi_L(h)}^2$,
\begin{itemize}[leftmargin=*]
\item[(R)] $m\ge\frac34\covol(L)^2/\covol(L^-)^2$ if $k\ge1$,
\item[(S)] $\rho^2-\norm{h\cdot x}^2/m<C_{n-k-2}/(3m)$ for every ball $B_{T(L)}(x,\rho)\in\mathcal B(L)$.
\end{itemize}
Both conditions depend only on the node and on $L'$, since $h$ is determined up to sign modulo $L$. If $k\le n-3$, the children with $T(L')$ admissible are the nodes of depth $k+1$ below $L$. A node of depth $n-2$ is a \emph{leaf}. A child of a leaf $L$ is $L'=L(v)$ for a primitive $v\in\Lambda_{T(L)}$, unique up to sign, and we call $v$ a \emph{line} of the leaf. A leaf with $\mathcal B(L)=\emptyset$ is a \emph{critical leaf}. For such a leaf condition (S) is empty, and its lines are not listed.

\begin{proposition}\label{prop:tree}
Let $v\in\Z^n$ be primitive, with nonzero entries of pairwise distinct absolute values and $\ML(v)<\frac1n$, and let $h_1,\dots,h_{n-1}$ be a flag of $v$. Then the chains $L_0\subset\dots\subset L_k$, $0\le k\le n-2$, are nodes, and either $L_{n-2}$ is a critical leaf, so that $T_v$ lies in the two-dimensional subtorus $T(L_{n-2})$, or $v$ is a line of the leaf $L_{n-2}$. If every node of depth at most $n-3$ carries an alive ball, then the tree is finite, and every leaf that is not critical has finitely many lines.
\end{proposition}
\begin{proof}
The vector $v$ lies in $\Lambda_{T_k}$ for every $k$, so $T_k$ is admissible. Condition (R) is Lemma~\ref{lem:flag}(a), because $m_k=\covol(L_k)^2/\covol(L_{k-1})^2$, and (S) is Lemma~\ref{lem:prune}. A child of a node with an alive ball of radius $\rho$ has $m<C_{n-k-1}/(4\rho^2)$ by (S), and only finitely many classes $h+L$ have $\eucl{\pi_L(h)}^2$ below a given bound, because $\pi_L(\Z^n)$ is a lattice in $E(L)$.
\end{proof}

A signed coordinate permutation maps a flag of $v$ to a flag of the image of $v$. After such a permutation we may therefore assume that $h_1$ is a fixed representative of its orbit, and after a further permutation that fixes $\pm h_1$, that $L_2$ is a fixed representative of its orbit under the stabilizer of $\pm h_1$. So it suffices to build the tree below one vector $h_1$ from each orbit of primitive vectors, and at depth two below one lattice $L_2$ from each orbit of that stabilizer. The conclusion of Proposition~\ref{prop:tree} then holds up to signed permutations of $v$.

\subsection{The search}\label{subsec:search}
The search was carried out for $n=4,5,6$ by a program that uses exact integer arithmetic for every decision. Section~\ref{sec:comp} describes it. For six speeds put
\begin{gather*}
 U_A=T_{(1,2,3,4,5)}\times\T,\qquad U_B=T_{(1,3,4,5,9)}\times\T,\\
 U_C=\{(X,\,Y,\,X+Y,\,2X+Y,\,3X+Y,\,3X+2Y):X,Y\in\T\}.
\end{gather*}
These are the three subtori that Cordella found and proved to be, up to symmetry, the only coordinate-free two-dimensional subtori of $\T^6$ with loneliness $\frac16$ \cite[Theorem~3.7]{Cordella}, and $U_C$ is the subtorus $U^7$ of Jain and Kravitz \cite[Section~6]{JainKravitz}. Here they appear as the critical leaves, and the theorem of Cordella is not used.

\begin{proposition}\label{prop:search}
For $n=4,5,6$ the tree is finite, with the numbers of nodes given in Table~\ref{tab:tree}, and the following holds.
\begin{enumerate}[label=\textup{(\roman*)},leftmargin=*]
\item For $n=4$ the subtori of the critical leaves are, up to symmetry, $\Uone$ and $\Utwo$, and the lines of the other leaves with $\ML<\frac14$ are, up to symmetry, $(1,2,3,8)$ and $(1,3,4,14)$.
\item For $n=5$ the subtori of the critical leaves are, up to symmetry, $T_{(1,2,3,4)}\times\T$ and $T_{(1,3,4,7)}\times\T$, and the lines of the other leaves with $\ML<\frac15$ are, up to symmetry, the seven vectors $(1,2,3,4,5)$, $(1,2,3,4,10)$, $(1,3,4,7,5)$, $(1,3,4,7,10)$, $(1,3,4,7,15)$, $(1,3,4,7,20)$, $(1,3,4,7,25)$ and the twelve vectors of Theorem~\ref{thm:fivefull}(c).
\item For $n=6$ the subtori of the critical leaves are, up to symmetry, $U_A$, $U_B$ and $U_C$, and the lines of the other leaves with $\ML<\frac16$ are, up to symmetry, the five vectors $(1,2,3,4,5,6)$, $(1,2,3,4,5,12)$, $(1,3,4,5,9,6)$, $(1,3,4,5,9,12)$, $(1,3,4,5,9,18)$ and the nine vectors of Theorem~\ref{thm:six}(d).
\end{enumerate}
\end{proposition}

\begin{table}[ht]
\centering\footnotesize\setlength{\tabcolsep}{4.5pt}
\begin{tabular}{r r r r r r r r r r}
\toprule
 & & \multicolumn{4}{c}{nodes with a ball at depth} & critical & & & \\
$n$ & $h_1$ & $1$ & $2$ & $3$ & $4$ & leaves & candidates & lines & found\\\midrule
$4$ & $15$ & $13$ & $255$ & & & $2$ & $4{,}810$ & $614$ & $2$\\
$5$ & $26$ & $22$ & $1{,}011$ & $28{,}599$ & & $8$ & $471{,}556$ & $18{,}206$ & $19$\\
$6$ & $44$ & $40$ & $5{,}996$ & $1{,}233{,}501$ & $11{,}833{,}425$ & $328$ & $82{,}538{,}200$ & $603{,}184$ & $14$\\
\bottomrule
\end{tabular}
\caption{The search of Proposition~\ref{prop:search}. The column $h_1$ counts the orbits of primitive vectors $h_1$ with at least two nonzero entries and $\eucl{h_1}^2<C_{n-1}/(4(\frac12-\frac1n)^2)$, the bound of Lemma~\ref{lem:prune} at the root. The column ``candidates'' counts, over the leaves with a ball, the primitive vectors $v$ of $\Lambda_{T(L)}$ up to sign that satisfy (R) and the bound $m<1/(4\rho^2)$ for the largest ball of the leaf. The column ``lines'' counts those among them that satisfy (S) for every ball, that is, the lines of these leaves. For each line with nonzero entries of distinct absolute values, $\ML<\frac1n$ is decided by Lemma~\ref{lem:corner}, and ``found'' is the number of distinct lines with $\ML<\frac1n$ up to symmetry.}
\label{tab:tree}
\end{table}

By Proposition~\ref{prop:tree} this gives the reduction for five and six speeds: a primitive vector with nonzero entries of distinct absolute values and $\ML<\frac1n$ lies, up to symmetry, in the lattice of a critical subtorus or is one of the listed lines. For $n=4$ it is a second proof of the reduction behind Theorem~\ref{thm:main}, which does not use Sections~3 to~7: a near-tight quadruple lies on $\Uone$ or $\Utwo$, or is $(1,3,4,14)$, since $(1,2,3,8)$ lies on $\Uone$.

\subsection{Lines of $U_A$ and $U_B$}
Let $V_A=(1,2,3,4,5)$ and $V_B=(1,3,4,5,9)$, the two tight instances for five speeds \cite{BHK}, and for a tuple $V$ put $g_V(Y)=\min_{v\in V}\norm{vY}$.

\begin{lemma}\label{lem:tight5}
Let $V\in\{V_A,V_B\}$ and $I_d=[d,\frac{1-d}5]$, which is empty for $d>\frac16$. For $\frac18\le d\le\frac16$,
\[
 I_d\cup(1-I_d)\subseteq\{Y\in[0,1):g_V(Y)\ge d\},
\]
and the two sets are equal if $V=V_A$ and $d>\frac17$, and if $V=V_B$ and $d>\frac2{13}$. In particular $\ML(V_A)=\ML(V_B)=\frac16$, attained exactly at $Y=\frac16$ and $Y=\frac56$.
\end{lemma}
\begin{proof}
By the symmetry $Y\mapsto1-Y$ it suffices to consider $Y\in[0,\frac12]$. For a positive integer $v$ the condition $\norm{vY}\ge d$ says that $vY\in[j+d,j+1-d]$ for an integer $j\ge0$. Let $Y\in I_d$. Then $Y\ge d$ and $5Y\le1-d$, so $vY\in[d,1-d]$ for $1\le v\le5$, and $9Y\in[9d,\frac95(1-d)]\subseteq[1+d,2-d]$ because $d\ge\frac18$. This proves the inclusion.

Conversely let $d>\frac17$ and let $Y\in[0,\frac12]$ with $g_V(Y)\ge d$ and $Y\notin I_d$. Then $Y\ge d$, so $5Y>1-d$ and therefore $5Y\ge1+d$. Since $\frac{1+d}5>\frac{1-d}4$, the condition for the speed $4$ gives $4Y\ge1+d$, and since $\frac{1+d}4>\frac{1-d}3$, the condition for the speed $3$ gives $3Y\ge1+d$. Now $\frac{1+d}3>\frac{2-d}5$, so $5Y\ge2+d$. If $V=V_A$, this contradicts $2Y\le1-d$, because $\frac{2+d}5>\frac{1-d}2$. If $V=V_B$, then $9Y\ge\frac95(2+d)>4-d$, so $9Y\ge4+d$, while $4Y\le2-d$. This is impossible because $\frac{4+d}9>\frac{2-d}4$ for $d>\frac2{13}$.
\end{proof}

\begin{theorem}\label{thm:linesAB}
Let $V\in\{V_A,V_B\}$ and let $a,b\ge1$ be coprime. Then $\ML(bV\cup\{a\})=\frac16$ unless $b=1$ and $6\mid a$. Moreover
\begin{gather*}
 \ML(V_A\cup\{6s\})=\frac s{6s+1}\quad(s\ge1),\\
 \ML(V_B\cup\{6s\})=\frac s{6s+1}\quad(s\ge2),\qquad \ML(V_B\cup\{6\})=\frac2{13}.
\end{gather*}
\end{theorem}
\begin{proof}
Identify $T_V\times\T$ with $\T^2$ through $(X,Y)\mapsto(YV,X)$, so that $G=\min(\norm X,g_V(Y))$, and the circle of $(bV,a)$ is $\{bX-aY\in\Z\}$ by Lemma~\ref{lem:orbit}. By Lemma~\ref{lem:tight5}, $G\le\frac16$, with equality exactly on the segments $Y\in\{\frac16,\frac56\}$, $X\in[\frac16,\frac56]$. On the first of them $\phi=bX-aY$ runs through an interval of length $\frac{2b}3$. If $b\ge2$ it contains an integer and $\ML=\frac16$. Let $b=1$. The interval is $[\frac{1-a}6,\frac{5-a}6]$, which contains an integer exactly when $6\nmid a$. The second segment is the image of the first under $(X,Y)\mapsto(1-X,1-Y)$, which replaces $\phi$ by $b-a-\phi$, so it gives the same condition. Let $a=6s$ and $R_d=[d,1-d]\times I_d$. With $d=\frac16-\varepsilon$ the image of $R_d$ under $\phi=X-6sY$ is
\[
 \bigl[\tfrac16-s-\varepsilon(1+\tfrac{6s}5),\ \tfrac56-s+\varepsilon(6s+1)\bigr].
\]
As $\varepsilon$ grows, its upper end reaches the integer $1-s$ at $\varepsilon=\frac1{6(6s+1)}$, that is, at $d=d_s=\frac s{6s+1}$, and its lower end reaches $-s$ only at $\varepsilon=\frac5{6(6s+5)}$, which is larger. The rectangle $(1,1)-R_d$ gives the same condition, by the same symmetry. So the circle meets $R_d\cup((1,1)-R_d)$ exactly for $d\le d_s$, and $\ML\ge d_s$ by the inclusion in Lemma~\ref{lem:tight5}, as $d_s\ge\frac17$. For $V=V_A$ the two rectangles are the whole set $\{G\ge d\}$ when $d>\frac17$, so $\ML=d_s$. For $V=V_B$ this holds when $d>\frac2{13}$, and $d_s\ge\frac2{13}$ for $s\ge2$. For $V=V_B$ and $s=1$ it shows that the circle misses $\{G\ge d\}$ for every $d>\frac2{13}$, and the time $t=\frac6{13}$ puts the speeds $1,3,4,5,9,6$ at the distances $\frac1{13}(6,5,2,4,2,3)$.
\end{proof}

This is \cite[Proposition~4.1 and Theorem~4.2]{Cordella} for six speeds. The proof there uses a lemma of Kravitz on one fast runner for large speeds and an exact computation for the others, and the proof above is by hand, by the method of Theorem~\ref{thm:U1}. The corresponding statement for five speeds, with $(1,2,3,4)$ and $(1,3,4,7)$ in place of $V_A$ and $V_B$ and the exception $\ML(1,3,4,7,5)=\frac2{11}$, is \cite[Theorem~4.2]{Cordella}, see also \cite[Section~3]{BonfioliFive}.

\subsection{Lines of $U_C$}
Identify $U_C$ with $\T^2$ through $(X,Y)\mapsto(X,Y,X+Y,2X+Y,3X+Y,3X+2Y)$, so that
\[
 G(X,Y)=\min\bigl(\norm X,\norm Y,\norm{X+Y},\norm{2X+Y},\norm{3X+Y},\norm{3X+2Y}\bigr).
\]
The lattice of $U_C$ is $\Z(1,0,1,2,3,3)+\Z(0,1,1,1,1,2)$, so the primitive vectors in it are $w(A,B)=(A,\,B,\,A+B,\,2A+B,\,3A+B,\,3A+2B)$ with $A,B$ coprime, and the circle of $w(A,B)$ is $\{BX-AY\in\Z\}$ by Lemma~\ref{lem:orbit}.

\begin{lemma}\label{lem:cellsC}
For $\frac17\le d\le\frac16$ and $\varepsilon=\frac16-d$, the set $\{(X,Y)\in\T^2:G(X,Y)\ge d\}$ is the disjoint union of the twelve triangles $\pm(c_j+\varepsilon Q_j)$, $1\le j\le6$, where
\begin{align*}
 c_1&=(\tfrac16,\tfrac16),& Q_1&=\operatorname{conv}\{(-1,-1),(-1,2),(1,-1)\},\\
 c_2&=(\tfrac16,\tfrac13),& Q_2&=\operatorname{conv}\{(-1,1),(-1,4),(1,-2)\},\\
 c_3&=(\tfrac13,\tfrac16),& Q_3&=\operatorname{conv}\{(-2,5),(0,-1),(1,-1)\},\\
 c_4&=(\tfrac12,\tfrac16),& Q_4&=\operatorname{conv}\{(-3,5),(0,-1),(1,-1)\},\\
 c_5&=(\tfrac12,\tfrac13),& Q_5&=\operatorname{conv}\{(-3,4),(0,1),(1,-2)\},\\
 c_6&=(\tfrac23,\tfrac16),& Q_6&=\operatorname{conv}\{(-1,2),(0,-1),(2,-1)\}.
\end{align*}
In particular $\ML(U_C)=\frac16$, attained exactly at the twelve points $\pm c_j$.
\end{lemma}
\begin{proof}
For $(X,Y)\in[0,1)^2$ the forms $X+Y$, $2X+Y$, $3X+Y$, $3X+2Y$ lie in $[0,2)$, $[0,3)$, $[0,4)$, $[0,5)$. The condition $G\ge d$ says that $X,Y\in[d,1-d]$ and that each of the four forms lies in $[\kappa+d,\kappa+1-d]$ for an integer $\kappa$ in its range. So the set is the union of $120$ convex polygons $P_\kappa(d)$, one for each choice $\kappa=(\kappa_3,\dots,\kappa_6)$ of the four integers, each of them cut out by twelve inequalities that are linear in $(X,Y,d)$, and $P_\kappa(d)$ shrinks as $d$ grows. The polygons for different $\kappa$ are disjoint. These $120$ systems are solved exactly by the script of Section~\ref{sec:comp}. For $108$ choices of $\kappa$ the polygon $P_\kappa(\frac17)$ is empty. For the other twelve, $P_\kappa(\frac16)$ is a single point $c$, three of the inequalities are equalities at $c$, these three define the triangle $c+\varepsilon Q$ at the level $d=\frac16-\varepsilon$, and $P_\kappa(\frac17)=c+\frac1{42}Q$, with $c$ and $Q$ as listed. So $P_\kappa(d)\subseteq c+\varepsilon Q$, and the nine remaining inequalities hold at every vertex $c+\varepsilon q$ of the triangle for $\varepsilon=0$ and for $\varepsilon=\frac1{42}$, hence for all $\varepsilon\in[0,\frac1{42}]$, since they are linear in $\varepsilon$.
\end{proof}

\begin{theorem}[Lines of $U_C$]\label{thm:linesC}
Let $A,B$ be coprime integers such that the six entries of $w=w(A,B)$ are nonzero.
\begin{enumerate}[label=\textup{(\roman*)},leftmargin=*]
\item $\ML(w)<\frac16$ if and only if $6$ divides none of the six integers
\[
 A-B,\quad 2A-B,\quad A-2B,\quad A-3B,\quad 2A-3B,\quad A-4B .
\]
\item In that case the entries of $w$ have pairwise distinct absolute values, the pair $(A,B)$ or the pair $(-A,-B)$ is congruent modulo $6$ to one of the six pairs of Table~\ref{tab:UC}, and
\[
 \ML(w)=\frac{P-1}{6P},
\]
where $P=P(A,B)$ is the largest of the four integers in the corresponding row of the table, evaluated at that pair. Moreover $P\ge11$ and $P\equiv\pm1\pmod6$.
\item In that case $\max_i|w_i|<3q$, where $q$ is the denominator of $\ML(w)$ in lowest terms, that is, $q=P$ if $P\equiv1$ and $q=3P$ if $P\equiv5\pmod6$.
\end{enumerate}
\end{theorem}

\begin{table}[ht]
\centering\small
\begin{tabular}{c l}
\toprule
$(A,B)\bmod 6$ & $P(A,B)$ is the largest of\\\midrule
$(0,1)$ & $4A+B,\quad -A+B,\quad \mathbf{-2A-B},\quad \mathbf{-A-B}$\\
$(1,0)$ & $A-B,\quad A+2B,\quad -5A-2B,\quad -5A-3B$\\
$(1,3)$ & $4A+B,\quad 4A+3B,\quad A+2B,\quad -5A-2B$\\
$(1,4)$ & $5A+2B,\quad \mathbf{A+B},\quad \mathbf{-A},\quad -5A-3B$\\
$(1,5)$ & $4A+3B,\quad \mathbf{-A},\quad \mathbf{-2A-B},\quad -A-2B$\\
$(2,3)$ & $5A+3B,\quad -A+B,\quad -4A-B,\quad -4A-3B$\\
\bottomrule
\end{tabular}
\caption{The integer $P$ of Theorem~\ref{thm:linesC}. The forms in bold take values congruent to $5$ modulo $6$ on their class, the others values congruent to $1$.}
\label{tab:UC}
\end{table}

\begin{proof}
Let $\phi(X,Y)=BX-AY$. By Lemma~\ref{lem:cellsC}, for $0\le\varepsilon\le\frac1{42}$ the circle of $w$ meets the set $\{G\ge\frac16-\varepsilon\}$ if and only if the interval $\phi(c_j)+\varepsilon\phi(Q_j)$ contains an integer for some $j$. The triangles $-(c_j+\varepsilon Q_j)$ give the negatives of these intervals. For $\varepsilon=0$ this says that $\ML(w)=\frac16$ exactly when some $\phi(c_j)$ is an integer, and $6\phi(c_j)$ is $B-A$, $B-2A$, $2B-A$, $3B-A$, $3B-2A$, $4B-A$ for $j=1,\dots,6$. This is (i).

Assume now that no $\phi(c_j)$ is an integer, and let $r_j^+=\lceil\phi(c_j)\rceil-\phi(c_j)$ and $r_j^-=\phi(c_j)-\lfloor\phi(c_j)\rfloor$, which are positive multiples of $\frac16$ and depend only on $(A,B)$ modulo $6$. Put
\begin{equation}\label{eq:epsC}
 \varepsilon^*=\min_{1\le j\le6}\ \min\Bigl(\frac{r_j^+}{\max_{Q_j}\phi},\ \frac{r_j^-}{-\min_{Q_j}\phi}\Bigr),
\end{equation}
where the maxima and minima are attained at the vertices and all denominators are positive, since $0$ is an interior point of $Q_j$. This is the least $\varepsilon\ge0$ for which one of the intervals $\phi(c_j)+\varepsilon\phi(Q_j)$ contains an integer. The number $1/(6\varepsilon^*)$ is the largest of the $36$ numbers $\phi(q)/(6r_j^+)$ and $-\phi(q)/(6r_j^-)$, over $j$ and the vertices $q$ of $Q_j$, and for a fixed residue class of $(A,B)$ modulo $6$ these are $36$ linear forms in $(A,B)$ with rational coefficients. The rays on which two of them agree divide the plane into finitely many sectors, and on each sector one of the forms is the largest. There are $24$ classes of coprime pairs modulo $6$. Twelve of them are excluded by (i). For each of the other twelve the script of Section~\ref{sec:comp} determines the sectors exactly, and it finds that on every sector the largest form is one of the four forms of Table~\ref{tab:UC} for the class of $(A,B)$ or, after the substitution $(A,B)\mapsto(-A,-B)$, for the class of $(-A,-B)$. These forms have integer coefficients, and their values on the class are congruent to $1$ or $5$ modulo $6$ as indicated. So $\varepsilon^*=\frac1{6P}$ for a positive integer $P\equiv\pm1\pmod6$.

Put $q=P$ if $P\equiv1$ and $q=3P$ if $P\equiv5\pmod6$. On each sector the script checks that the linear forms $3q\mp w_i$ are nonnegative on the two bounding rays. So they are nonnegative on the sector, and they vanish there only on the rays spanned by $\pm(1,0)$, $\pm(1,-3)$ and $\pm(2,-3)$, on which $w$ has a zero entry. Hence $\max_i|w_i|<3q$. If $P\le7$, then $3q\le45$, so $|A|,|B|<45$, and a direct evaluation of $P$ for these pairs, again by the script, gives $P\ge11$. So $P\ge11$ in all cases, and $\varepsilon^*<\frac1{42}$. By the first paragraph the circle meets $\{G\ge\frac16-\varepsilon\}$ for $\varepsilon=\varepsilon^*$ and for no smaller $\varepsilon\ge0$, so $\ML(w)=\frac16-\varepsilon^*=\frac{P-1}{6P}$. If two entries of $w$ had the same absolute value, the conjecture for at most five speeds \cite{CusickPomerance,Bienia,BHK} would give $\ML(w)\ge\frac16$. This proves (ii).

(iii) If $P\equiv1\pmod 6$, then $\frac{P-1}6$ and $P$ are coprime, and the denominator of $\ML(w)$ is $P$. If $P\equiv5\pmod6$, then $\frac{P-1}2$ is prime to $3P$, and the denominator is $3P$. So $q$ is the denominator of $\ML(w)$, and $\max_i|w_i|<3q$ was shown above.
\end{proof}

Jain and Kravitz analyzed one sector of this subtorus \cite[Section~6]{JainKravitz}. Cordella proved, with a decomposition of the plane into cones and with computer assistance, that $\ML(w)=\frac{P-1}{6P}$ for an integer $P\equiv\pm1\pmod6$, and the bound (iii) \cite[Theorem~5.3]{Cordella}. Theorem~\ref{thm:linesC} adds the criterion (i) and the closed formula for $P$. Table~\ref{tab:UC} reproduces the four cases of \cite[Theorem~B]{Cordella} in which $q$ is divisible by $3$, and the family $(A,B)=(18j+31,6)$ with $\max_i|w_i|=3q-24$, by which he shows that the constant $3$ in (iii) is best possible. The formula \eqref{eq:epsC} is the analogue for $U_C$ of \eqref{eq:U2formula}.

\begin{proposition}[Jain and Kravitz \cite{JainKravitz}]\label{prop:k3}
For every integer $P\ge11$ with $P\equiv5\pmod6$,
\[
 \ML\bigl(6,\ |P-12|,\ P-6,\ P,\ P+6,\ 2P-6\bigr)=\frac{P-1}{6P},
\]
and the six speeds are distinct.
\end{proposition}
\begin{proof}
The six speeds are the absolute values of the entries of $w(-6,12-P)$, the pair $(-6,12-P)$ is coprime and congruent to $(0,1)$ modulo $6$, and the four forms of the first row of Table~\ref{tab:UC} take the values $-P-12$, $18-P$, $P$ and $P-6$ at it. The largest is $P$. The entries are nonzero, and they are distinct by Theorem~\ref{thm:linesC}(ii).
\end{proof}
For $P=6s+17$ this is the family $(6,\,6s+5,\,6s+11,\,6s+17,\,6s+23,\,12s+28)$ with which Jain and Kravitz proved that these values occur, and for $P=11$ it is their example $\{1,5,6,11,16,17\}$ with value $\frac5{33}$ \cite[Theorem~1.5 and Section~6]{JainKravitz}. The proof above derives it from Theorem~\ref{thm:linesC}. The maximum is attained at a time $t=\frac m{3P}$ with $2m\equiv\frac{P+1}6\pmod P$ and $m\equiv2\pmod3$, at which the speeds $P+6$ and $2P-6$ are at distance $\frac{P-1}{6P}$.

\subsection{The classification for five and six speeds}

\begin{theorem}[Five speeds, Bonfioli \cite{BonfioliFive}]\label{thm:fivefull}
Let $V$ be a primitive set of five distinct positive integers. Then $\ML(V)<\frac15$ if and only if one of the following holds.
\begin{enumerate}[label=\textup{(\alph*)},leftmargin=*]
\item $V=\{1,2,3,4,5s\}$ for some $s\ge1$, and then $\ML(V)=\frac s{5s+1}$.
\item $V=\{1,3,4,7,5s\}$ for some $s\ge1$, and then $\ML(V)=\frac s{5s+1}$ if $s\ge2$ and $\ML(V)=\frac2{11}$ if $s=1$.
\item $V$ is one of the twelve sets
\begin{gather*}
 \{1,3,4,5,9\}\ (\tfrac16),\qquad \{1,2,3,5,8\},\ \{1,3,4,5,18\},\ \{1,4,5,6,7\}\ (\tfrac2{11}),\\
 \{1,3,4,5,11\},\ \{1,7,8,9,15\},\ \{2,5,7,8,9\}\ (\tfrac3{16}),\\
 \{2,3,5,8,19\},\ \{3,5,8,9,13\},\ \{3,5,8,11,13\}\ (\tfrac4{21}),\\
 \{1,3,4,5,27\},\ \{1,4,5,6,27\}\ (\tfrac6{31}),
\end{gather*}
with the value of $\ML$ in parentheses.
\end{enumerate}
Consequently the values of $\ML$ below $\frac15$ on sets of five distinct positive integers are exactly the numbers $\frac s{5s+1}$, $s\ge1$.
\end{theorem}
\begin{proof}
Let $\ML(V)<\frac15$. By Propositions~\ref{prop:tree} and~\ref{prop:search}(ii), $V$ is one of the nineteen sets listed there, which are of the forms (a), (b), (c), or $V$ is the set of absolute values of a primitive vector in the lattice of $T_u\times\T$ for $u=(1,2,3,4)$ or $u=(1,3,4,7)$, that is, $V=bu\cup\{a\}$ with $a,b$ coprime. In the second case \cite[Proposition~4.1 and Theorem~4.2]{Cordella} give $b=1$ and $5\mid a$. The values in (a) and (b) are given by the same theorem, see also \cite[Section~3]{BonfioliFive}, and those in (c) are computed by Lemma~\ref{lem:corner}.
\end{proof}

Theorem~\ref{thm:fivefull} is due to Bonfioli. The statement on the values is his main theorem, which is formalized in Lean~4, and the list of sets is the one that his proof produces \cite[Sections~1 and~6]{BonfioliFive}. It confirms the original conjecture of Kravitz \cite{Kravitz} for five speeds. The search above is independent of his certificate, with a different pruning test and a different program, and it arrives at the same list. The twelve sets of (c) lie on neither critical subtorus, and they are the twelve quintuples that the search of Cordella up to speed $130$ found \cite[Section~6.2]{Cordella}. Theorem~\ref{thm:fivefull} determines the finite set that Corollary~\ref{cor:five}(ii) left below $7{,}290$.

\begin{theorem}[Six speeds]\label{thm:six}
Let $V$ be a primitive set of six distinct positive integers. Then $\ML(V)<\frac16$ if and only if one of the following holds.
\begin{enumerate}[label=\textup{(\alph*)},leftmargin=*]
\item $V=\{1,2,3,4,5,6s\}$ for some $s\ge1$, and then $\ML(V)=\frac s{6s+1}$.
\item $V=\{1,3,4,5,9,6s\}$ for some $s\ge1$, and then $\ML(V)=\frac s{6s+1}$ if $s\ge2$ and $\ML(V)=\frac2{13}$ if $s=1$.
\item $V=\{|A|,|B|,|A+B|,|2A+B|,|3A+B|,|3A+2B|\}$ for coprime integers $A,B$ such that these six numbers are nonzero and $6$ divides none of $A-B$, $2A-B$, $A-2B$, $A-3B$, $2A-3B$, $A-4B$, and then $\ML(V)=\frac{P-1}{6P}$ with $P=P(A,B)$ as in Theorem~\ref{thm:linesC}.
\item $V$ is one of the nine sets
\begin{gather*}
 \{1,2,5,6,7,8\},\ \{2,5,6,8,10,11\},\ \{1,4,5,6,7,22\}\ (\tfrac2{13}),\\
 \{1,2,3,5,8,18\},\ \{2,3,5,6,8,11\},\ \{2,5,6,7,8,11\}\ (\tfrac3{19}),\\
 \{1,3,4,5,18,46\}\ (\tfrac4{25}),\quad \{1,3,4,5,7,24\}\ (\tfrac5{31}),\quad \{1,4,5,6,7,33\}\ (\tfrac6{37}),
\end{gather*}
with the value of $\ML$ in parentheses.
\end{enumerate}
\end{theorem}
\begin{proof}
Let $\ML(V)<\frac16$ and let $v$ be a vector with the entries of $V$. By Propositions~\ref{prop:tree} and~\ref{prop:search}(iii), either $V$ is one of the fourteen sets listed there, which are of the forms (a), (b) and (d), or, after a signed permutation, $v$ lies in the lattice of $U_A$, $U_B$ or $U_C$. In the first two cases $V=bV_A\cup\{a\}$ or $V=bV_B\cup\{a\}$ with $a,b$ coprime, and Theorem~\ref{thm:linesAB} gives (a) or (b). In the third case $v=\pm w(A,B)$ up to a signed permutation, with $A,B$ coprime, and Theorem~\ref{thm:linesC} gives (c). Conversely the sets in (a), (b), (c) have the stated values by Theorems~\ref{thm:linesAB} and~\ref{thm:linesC}, and the values in (d) are computed by Lemma~\ref{lem:corner}.
\end{proof}

The nine sets of (d) lie on none of the three critical subtori. They are the nine sextuples that the search of Cordella up to speed $110$ found \cite[Table~2]{Cordella}, so the finite exceptional set of his reduction \cite[Lemma~2.4 and Theorem~6.1]{Cordella} can be taken to consist of these nine. Their values are also certified in Lean by Jones \cite{Jones}. The families overlap: $\{1,3,4,5,6,9\}$ belongs to (b) and to (c), with $(A,B)=(1,3)$.

\begin{corollary}\label{cor:six}
\begin{enumerate}[label=\textup{(\roman*)},leftmargin=*]
\item The set of values of $\ML$ below $\frac16$ on sets of six distinct positive integers is exactly
\begin{multline*}
 \Bigl\{\frac{P-1}{6P}:P\ge7,\ \gcd(P,6)=1\Bigr\}\\
 =\Bigl\{\frac s{6s+1}:s\ge1\Bigr\}\cup\Bigl\{\frac s{6s+3}:s\ge5,\ s\equiv2\!\!\pmod3\Bigr\}.
\end{multline*}
\item The denominator of every such value is odd, as Cordella conjectured \cite[Conjecture~6.2]{Cordella}. Every such value is $\frac s{6s+k}$ with $k\in\{1,3\}$, and both cases occur. In particular the values have the form $\frac s{6s+k}$, $1\le k\le6$, of the amended spectrum conjecture of Fan and Sun \cite{FanSun}.
\item \textup{(Marshall \cite{Marshall})} Up to scaling, $\{1,2,3,4,5,6\}$ is the only set of six positive integers with $\ML=\frac17$.
\item Every primitive set $V$ of six distinct positive integers with $\ML(V)=\frac pq<\frac16$ in lowest terms satisfies $\max V<3q$. This answers the first part of \cite[Question~6.6]{Cordella}.
\item The only accumulation point of the six-speed spectrum in $(0,\frac16]$ is $\frac16$.
\end{enumerate}
\end{corollary}
\begin{proof}
(i) The values in (a), (b) and (d) of Theorem~\ref{thm:six} are $\frac s{6s+1}=\frac{P-1}{6P}$ with $P=6s+1\ge7$, and in (c) we have $P\ge11$ and $P\equiv\pm1\pmod6$. Every $P\equiv1$ occurs in (a) and every $P\equiv5$, $P\ge11$, by Proposition~\ref{prop:k3}. For $P=6j+5$ the value is $\frac s{6s+3}$ with $s=3j+2$. (ii) follows from (i). (iii) A set with $\ML=\frac17$ has $P=7$. In (a) this is $s=1$, the values in (b) and (d) are at least $\frac2{13}$, and in (c) we have $P\ge11$. (iv) In (a) and (b), $\max V\le\max(9,6s)<3(6s+1)$, and $9<39$ for the exceptional set. In (c) this is Theorem~\ref{thm:linesC}(iii), and in (d) the largest elements $8,11,22,18,11,11,46,24,33$ are below $3q=39,39,39,57,57,57,75,93,111$. (v) By (i) the values below $\frac16$ are discrete in $(0,\frac16)$ and tend to $\frac16$.
\end{proof}

Part (iii) was proved before by Marshall \cite{Marshall} with a different computer-assisted argument, a strict form of the bound of \cite{MSS} combined with a sieve over primes, and Theorem~\ref{thm:six} gives a second proof. Cordella found only $\{1,\dots,6\}$ in a search up to speed $130$ and names the missing classification of the tight sextuples as the obstacle to determining the critical subtori for seven speeds by his method \cite[Section~6.3]{Cordella}. For five speeds the corresponding statements are known: the tight instances are $V_A$ and $V_B$ \cite{BHK}, and the values are those of Theorem~\ref{thm:fivefull}. In the notation $D=\frac12-\ML$, part (v) is \cite[Conjecture~1.5]{GiriKravitz} for $n=6$ on the range $[\frac13,\frac12)$, since $\frac16$ is the smallest loneliness of a set of five speeds.

\begin{remark}\label{rem:lrc6}
Neither the search nor the proofs of Theorems~\ref{thm:linesAB} and~\ref{thm:linesC} use the Lonely Runner Conjecture for six speeds. All values in Theorem~\ref{thm:six} are at least $\frac17$, and a tuple of six positive integers with a repeated entry has loneliness at least $\frac16$ by the conjecture for at most five speeds. So Theorem~\ref{thm:six} contains a computer-assisted proof of the conjecture for six speeds, that is, for seven runners, which is a theorem of Barajas and Serra \cite{BarajasSerra}. Marshall notes the same for his classification of the tight sets \cite{Marshall}.
\end{remark}

\subsection{Other levels and the first values above $\frac1n$}\label{sec:levels}
The level $\frac1n$ plays no role in Lemma~\ref{lem:prune} and Proposition~\ref{prop:tree}. Let $0<\lambda<\frac12$. Call a ball $B_S(x,\rho)$ \emph{alive at level $\lambda$} if $\norm{x_j}-\lambda\ge\rho\,w_j(S)$ for $1\le j\le n$, so that $G\ge\lambda$ on it. The \emph{tree at level $\lambda$} is defined as the tree above, with balls that are alive at level $\lambda$ and with the ball of centre $(\frac12,\dots,\frac12)$ and radius $\frac12-\lambda$ at the root.

\begin{proposition}\label{prop:level}
Let $v\in\Z^n$ be primitive, with nonzero entries of pairwise distinct absolute values and $\ML(v)<\lambda$, and let $h_1,\dots,h_{n-1}$ be a flag of $v$. Then the chains $L_0\subset\dots\subset L_k$, $0\le k\le n-2$, are nodes of the tree at level $\lambda$, and either $L_{n-2}$ is a critical leaf or $v$ is a line of the leaf $L_{n-2}$. If $L$ is a leaf and $v\in\Lambda_{T(L)}$, then $\ML(v)\le\ML(T(L))$.
\end{proposition}
\begin{proof}
A point $y\in T_v$ has $G(y)\le\ML(v)<\lambda$, so it lies in no ball that is alive at level $\lambda$. With this the proofs of Lemma~\ref{lem:prune} and Proposition~\ref{prop:tree} apply without change. The last statement holds because $T_v\subseteq T(L)$.
\end{proof}

So if the critical leaves at a level $\lambda>\frac1n$ are the critical subtori of Proposition~\ref{prop:search}, whose loneliness is $\frac1n$, then every set of $n$ speeds with $\frac1n<\ML<\lambda$ is among the finitely many lines of the other leaves. This is the case for the following three levels.

\begin{theorem}\label{thm:above}
\begin{enumerate}[label=\textup{(\alph*)},leftmargin=*]
\item Exactly $1{,}269$ primitive sets $V$ of four distinct positive integers have $\frac14<\ML(V)<\frac7{25}$. They take $62$ values, and their largest elements are at most $213$. The least value is $\frac{11}{43}$, and it is attained exactly by
\begin{gather*}
 \{1,10,11,33\},\quad\{1,13,14,42\},\quad\{3,10,13,30\},\\
 \{4,9,13,39\},\quad\{5,7,12,36\}.
\end{gather*}
A primitive quadruple has $\ML=\frac14$ if and only if its circle lies on a signed coordinate permutation of $\Uone$ or $\Utwo$ and it is not one of the quadruples of Theorem~\ref{thm:main}, or it is one of $29$ further quadruples, whose largest elements are at most $33$.
\item Exactly $267$ primitive sets $V$ of five distinct positive integers have $\frac15<\ML(V)<\frac3{14}$. They take $35$ values, and their largest elements are at most $180$. The least value is $\frac{18}{89}$, and it is attained only by $\{1,17,18,19,72\}$. A primitive set of five speeds has $\ML=\frac15$ if and only if it is $\{a,2a,3a,4a,b\}$ or $\{a,3a,4a,7a,b\}$ and not one of the sets of Theorem~\ref{thm:fivefull}, or it is one of $64$ further sets, whose largest elements are at most $52$.
\item Exactly $98$ primitive sets $V$ of six distinct positive integers have $\frac16<\ML(V)<\frac7{40}$. They take the $15$ values
\[
 \tfrac{10}{59},\ \tfrac9{53},\ \tfrac8{47},\ \tfrac7{41},\ \tfrac{13}{76},\ \tfrac6{35},\ \tfrac{11}{64},\ \tfrac5{29},\ \tfrac{14}{81},\ \tfrac9{52},\ \tfrac{13}{75},\ \tfrac4{23},\ \tfrac{15}{86},\ \tfrac{26}{149},\ \tfrac{18}{103},
\]
and their largest elements are at most $145$. The least value $\frac{10}{59}$ is attained exactly by
\begin{gather*}
 \{1,2,3,5,8,51\},\quad\{1,3,4,5,7,55\},\\
 \{1,3,4,5,18,55\},\quad\{1,4,5,6,7,55\}.
\end{gather*}
A primitive set of six speeds has $\ML=\frac16$ if and only if its circle lies on a signed coordinate permutation of $U_A$, $U_B$ or $U_C$ and it is not one of the sets of Theorem~\ref{thm:six}, or it is one of $54$ further sets, whose largest elements are at most $55$.
\end{enumerate}
In particular, for $n=4,5,6$ no set of $n$ speeds has loneliness in the interval $(\frac14,\frac{11}{43})$, $(\frac15,\frac{18}{89})$, $(\frac16,\frac{10}{59})$, respectively.
\end{theorem}
\begin{proof}
The program \texttt{lrlevel} (Section~\ref{sec:comp}) builds the tree at the level $\frac7{25}$ for $n=4$, $\frac3{14}$ for $n=5$ and $\frac7{40}$ for $n=6$. In each case every node of depth at most $n-3$ carries a ball, and the critical leaves are, up to symmetry, the critical subtori of Proposition~\ref{prop:search}. By Proposition~\ref{prop:level}, a primitive vector $v$ with nonzero entries of distinct absolute values and $\ML(v)<\lambda$ lies in the lattice of one of these subtori, and then $\ML(v)\le\frac1n$, or it is one of the lines of the other leaves. The program lists the lines with $\ML<\lambda$, and their values are computed by Lemma~\ref{lem:corner}. A set on a critical subtorus has $\ML<\frac1n$ exactly when it is one of the sets of Theorems~\ref{thm:main}, \ref{thm:fivefull} and~\ref{thm:six}, which gives the statements on $\ML=\frac1n$.
\end{proof}

The sets of Theorem~\ref{thm:above} are listed in the ancillary files, and an exhaustive search over all sets with largest element at most $260$, $200$ and $150$ for $n=4,5,6$ finds the same sets (Section~\ref{sec:comp}). The lists grow quickly with the level: for five speeds the tree at level $\frac{11}{50}$ lists $3{,}297$ sets with $\ML<\frac{11}{50}$, and $\frac29$ is an accumulation point of the spectrum. The $35$ values of (b) are
\begin{gather*}
 \tfrac{18}{89},\ \tfrac{14}{69},\ \tfrac{12}{59},\ \tfrac{10}{49},\ \tfrac9{44},\ \tfrac{17}{83},\ \tfrac8{39},\ \tfrac7{34},\ \tfrac{13}{63},\ \tfrac6{29},\ \tfrac{17}{82},\ \tfrac{11}{53},\\
 \tfrac5{24},\ \tfrac{14}{67},\ \tfrac9{43},\ \tfrac{22}{105},\ \tfrac{13}{62},\ \tfrac{30}{143},\ \tfrac4{19},\ \tfrac{23}{109},\ \tfrac{15}{71},\ \tfrac{26}{123},\ \tfrac{11}{52},\ \tfrac{18}{85},\\
 \tfrac7{33},\ \tfrac{38}{179},\ \tfrac{17}{80},\ \tfrac{10}{47},\ \tfrac{13}{61},\ \tfrac{29}{136},\ \tfrac{16}{75},\ \tfrac{19}{89},\ \tfrac{22}{103},\ \tfrac{34}{159},\ \tfrac{40}{187}.
\end{gather*}
The values $\frac{18}{89}$, $\frac{14}{69}$, $\frac{10}{49}$ and $\frac6{29}$ are attained by the sets $\{1,4k+1,4k+2,4k+3,16k+8\}$ with $k=4,3,2,1$. Their loneliness is $\frac{4k+2}{20k+9}$ for $1\le k\le4$, and it is $\frac5{24}$ for $k=5$.

\section{Seven and eight speeds}\label{sec:seven}
The six-speed classification determines the critical subtori for seven speeds, and a classification of the tight sets of seven speeds would do the same for eight speeds. We keep the notation of Section~\ref{sec:general}. For a vector $u$ of $n-1$ positive integers, $T_u\times\T=\{(u_1t,\dots,u_{n-1}t,s)\}\subseteq\T^n$ as in Section~\ref{sec:five}, and its loneliness is $\ML(u)$. For a coordinate-free two-dimensional subtorus $W\subseteq\T^n$ let $\ell_1,\dots,\ell_n$ be the restrictions of the coordinates to $\R\Lambda_W$. A \emph{repeat line} of $W$ is a coordinate-free circle $T\subseteq W$ on which $x_a=\pm x_b$ for two coordinates $a\ne b$ such that $\ell_a$ and $\ell_b$ are not proportional.

\begin{lemma}\label{lem:repeat}
Every coordinate-free two-dimensional subtorus $W\subseteq\T^n$ has at least two distinct repeat lines.
\end{lemma}
\begin{proof}
The forms $\ell_i$ span the plane $E^*$ dual to $E=\R\Lambda_W$, so their directions are $h\ge2$ distinct points of the projective line of $E^*$, which cut it into $h$ open arcs. Let an arc have endpoints the directions of $\ell_a$ and $\ell_b$, and choose signs $\sigma,\tau$ such that the open cone spanned by $\sigma\ell_a$ and $\tau\ell_b$ maps onto this arc. As in the proof of Lemma~\ref{lem:coalesce}, the form $\phi=\sigma\ell_a+\tau\ell_b$ is proportional to no $\ell_k$, so the circle $T$ with $\R\Lambda_T=\ker\phi$ is coordinate-free, and $x_a=-\sigma\tau x_b$ on $T$. The direction of $\phi$ lies in the arc, so distinct arcs give distinct circles. This includes the case $h=2$, in which the two arcs have the same endpoints.
\end{proof}

This is the consecutive-ray argument of Lemma~\ref{lem:coalesce} applied in every gap, and a special case of the repeat-line theorem of Bonfioli \cite[Theorem~2]{BonfioliRoot}.

\subsection{Seven speeds}

\begin{theorem}\label{thm:seven}
Let $u=(1,2,3,4,5,6)$. Up to signed coordinate permutations, $T_u\times\T$ is the only coordinate-free two-dimensional subtorus $W\subseteq\T^7$ with $\ML(W)<\frac5{33}$, and $\ML(T_u\times\T)=\frac17$. The product $T_{(1,5,6,11,16,17)}\times\T$ has loneliness $\frac5{33}$. The only accumulation point of the seven-speed spectrum in $(0,\frac5{33})$ is $\frac17$.
\end{theorem}
\begin{proof}
Let $\ML(W)<\frac5{33}$. If $\ell_a=\pm\ell_b$ for some $a\ne b$, deleting the coordinate $b$ gives a coordinate-free two-dimensional subtorus of $\T^6$ with the same loneliness, which is at least $\frac16$ by Lemma~\ref{lem:coalesce} and the conjecture for five speeds. So the forms $\ell_i$ are pairwise distinct up to sign. Let $T$ be a repeat line of $W$. Its speeds take at most six distinct absolute values, and $\ML(T)\le\ML(W)<\frac5{33}<\frac16$. Five or fewer values would give $\ML(T)\ge\frac16$, so there are exactly six, and by Corollary~\ref{cor:six}(i) and (iii) they have loneliness $\frac17$ and are a multiple of $\{1,2,3,4,5,6\}$. As the speed vector of $T$ is primitive, after a signed coordinate permutation it is $(1,2,3,4,5,6,r)$ with $1\le r\le6$.

By Lemma~\ref{lem:repeat}, $W$ has two distinct repeat lines. After a signed coordinate permutation the speed vector of the first is $q=(1,2,3,4,5,6,r)$, and the speed vector $v$ of the second is a signed permutation of $(1,2,3,4,5,6,r')$ with $1\le r'\le6$, where we may assume $v_1>0$. Then $q$ and $v$ span $E=\R\Lambda_W$, and in the coordinates $(\alpha,\beta)\mapsto\alpha q+\beta v$ the forms $\ell_k$ are given by the rows $p_k=(q_k,v_k)\in\Z^2$. These rows are pairwise distinct, which also excludes $v=q$. The repeat lines of $W$ are the coordinate-free kernels of the forms $p_i+\varepsilon p_j$ with $i<j$ and $\varepsilon=\pm1$, and the speed vector of $\ker(p_i+\varepsilon p_j)$ is proportional to $(d_1,\dots,d_7)$ with $d_k=\det(p_k,p_i+\varepsilon p_j)$. So whenever no $d_k$ vanishes, the distinct values $|d_k|$ must form the set $\{d,2d,\dots,6d\}$, where $d=\min_k|d_k|$. The program \texttt{plane7} (Section~\ref{sec:comp}) checks all candidates. For each $r$ there are $967{,}680$ vectors $v$, of which $944{,}640$ give pairwise distinct rows, and $22$ pass all the tests. In each of the $22$ cases six of the rows are proportional, with first entries $1,\dots,6$, and the seventh row is independent of them. So $E$ is the tangent space of a signed coordinate permutation of $T_u\times\T$, and $W$ is that product.

Conversely $\ML(T_u\times\T)=\ML(u)=\frac17$, and $\ML(T_{(1,5,6,11,16,17)}\times\T)=\ML(1,5,6,11,16,17)=\frac5{33}$ by Proposition~\ref{prop:k3}. For the last statement, $\operatorname{acc}(\mathcal S(7))\subseteq\mathcal S_2(7)$ by \cite[Theorem~1.4]{GiriKravitz}, and the two conventions are reconciled as in the proof of Corollary~\ref{cor:five}(i): a tuple with a repeated speed has the value of a set of at most six speeds, and by Corollary~\ref{cor:six} the only such value below $\frac5{33}$ is $\frac17$. Finally $\frac17$ is the limit of the values $\frac s{7s+1}=\ML(1,2,3,4,5,6,7s)$ \cite[Theorem~5.1]{Kravitz}.
\end{proof}

\begin{corollary}\label{cor:seven}
\begin{enumerate}[label=\textup{(\roman*)},leftmargin=*]
\item Every primitive set $V$ of seven distinct positive integers with $\ML(V)<\frac17$ and $\eucl V\ge5{,}145{,}525$ is $\{1,2,3,4,5,6,7s\}$ for some $s\ge1$. For all $s\ge1$, $\ML(1,2,3,4,5,6,7s)=\frac s{7s+1}$.
\item The set of values of $\ML$ below $\frac17$ on sets of seven positive integers contains $\{\frac s{7s+1}:s\ge1\}$ and has finite symmetric difference with it. Every other value $\frac pq$ in lowest terms has $q\le7{,}276{,}871$ and $q-7p\ge2$. In particular every value in $\bigl(\frac17-\frac2{7\cdot7{,}276{,}871},\frac17\bigr)$ is of the form $\frac s{7s+1}$.
\item If $V$ is as in \textup{(i)} and $\ML(V)=\frac pq$ in lowest terms, then $\max V=q-1$.
\end{enumerate}
\end{corollary}
\begin{proof}
(i) Let $v$ be a vector with the entries of $V$. If $T_v$ lies on no signed coordinate permutation of $T_u\times\T$, every coordinate-free two-dimensional subtorus containing $T_v$ has loneliness at least $\frac5{33}$ by Theorem~\ref{thm:seven}, and Proposition~\ref{prop:flag}(b) applies with $n=7$ and $\gamma=\frac5{33}$. The numbers $B_1,\dots,B_6$ are $\frac4{231},\frac1{21},\frac4{35},\frac3{14},\frac8{21},\frac57$, so
\begin{gather*}
 (d_1,\dots,d_6)=\Bigl(\frac4{231},\frac1{33},\frac1{15},\frac1{10},\frac16,\frac13\Bigr),\\
 \frac1{d_1\cdots d_6}=\frac{231\cdot33\cdot15\cdot10\cdot6\cdot3}4=5{,}145{,}525,
\end{gather*}
and $\eucl v<5{,}145{,}525$. Otherwise $V=\{a,2a,3a,4a,5a,6a,b\}$ with $a,b\ge1$ coprime. If $a\ge2$, then at the $a$ times $t=(\frac17+h)/a$, $0\le h<a$, the speeds $ia$ are at distance at least $\frac17$ from the origin, and the numbers $bt$ run through a coset of $\frac1a\Z$, so at one of these times $\norm{bt}\ge\frac12-\frac1{2a}\ge\frac14$. This contradicts $\ML(V)<\frac17$, so $a=1$. If $7\nmid b$, the time $\frac17$ gives $\ML(V)\ge\frac17$. So $b=7s$. The value is \cite[Theorem~5.1]{Kravitz}.

(ii) The first two statements follow from (i). If $\frac pq=\ML(V)<\frac17$ is not of the form $\frac s{7s+1}$, then $q-7p\ge2$, and $V$ is a dilate of a primitive set $V'$ with $\eucl{V'}<5{,}145{,}525$. By Lemma~\ref{lem:corner} the maximum for $V'$ is attained at a time $\frac m{u+w}$ with $u,w\in V'$, where $u\ne w$ because the value is below $\frac12$. So $q$ divides $u+w\le\sqrt2\,\eucl{V'}<\sqrt2\cdot5{,}145{,}525<7{,}276{,}872$. For the last statement, $\frac17-\frac pq=\frac{q-7p}{7q}\ge\frac2{7q}$.

(iii) Here $\max V=7s$ and $q=7s+1$.
\end{proof}

Part (iii) shows that the optimal eventual constant $C$ in $\max V\le Cq$, the comparison of \cite[Question~6.6]{Cordella}, is $1$ for seven speeds, and Proposition~\ref{prop:flag}(c) gives $\max V\le198{,}450\,q/(q-7p)$ for every primitive near-tight set of seven speeds. The finite symmetric difference in (ii) is not empty. Fan and Sun found the values $\frac3{23}$ and $\frac5{37}$ for seven speeds \cite[Table~1 and Remark~30]{FanSun}, and Cordella listed the three primitive exceptions with largest element at most $50$ \cite[Section~6.3]{Cordella}. A further value is
\[
 \ML(2,7,9,11,13,20,54)=\frac9{65}.
\]
At $t=\frac{11}{65}$ the seven distances from the nearest integer are $\frac{22}{65},\frac{12}{65},\frac{31}{65},\frac9{65},\frac{13}{65},\frac{25}{65},\frac9{65}$. Conversely put $L=\frac9{65}$. On $[0,\frac12]$ the set of $t$ with $\norm{vt}>L$ for the first six speeds consists of the four open intervals
\[
 \Bigl(\frac{139}{845},\frac{11}{65}\Bigr),\quad\Bigl(\frac{204}{845},\frac{79}{325}\Bigr),\quad\Bigl(\frac{167}{650},\frac{186}{715}\Bigr),\quad\Bigl(\frac{529}{1300},\frac{186}{455}\Bigr),
\]
and on them $54t$ lies within $L$ of $9$, $13$, $14$ and $22$, respectively. So no $t$ has all seven distances above $L$. In the three fractions $\frac3{23}$, $\frac5{37}$, $\frac9{65}$ one has $q-7p=2$. The census in Section~\ref{sec:schur} finds four primitive sets with these values up to $256$ and no other value with $q-7p\ne1$.

\subsection{Eight speeds}
Put $U_1=\{1,\dots,7\}$, $U_2=\{1,2,3,4,5,7,12\}$ and $U_3=\{1,4,5,6,7,11,13\}$. These are tight sets of seven speeds, and Marshall announced a computer-assisted proof that they are the only ones up to scaling \cite[Theorem~B2]{Marshall}. The results below that use this classification, which this paper does not prove, assume the following hypothesis.

\smallskip
\noindent\textup{(M)} Every set of seven positive integers with loneliness $\frac18$ is a dilate of $U_1$, $U_2$ or $U_3$.
\smallskip

\begin{theorem}\label{thm:eight}
Assume \textup{(M)}. Up to signed coordinate permutations, the coordinate-free two-dimensional subtori $W\subseteq\T^8$ with $\ML(W)=\frac18$ are the three products $T_{U_i}\times\T$, with the elements of $U_i$ in increasing order.
\end{theorem}
\begin{proof}
If $\ell_a=\pm\ell_b$ for some $a\ne b$, deleting $b$ gives a coordinate-free two-dimensional subtorus of $\T^7$, whose loneliness is at least $\frac17$ by Lemma~\ref{lem:coalesce} and the conjecture for six speeds. So the forms are pairwise distinct up to sign. The speeds of a repeat line $T$ take at most seven distinct absolute values, six or fewer would give $\ML(T)\ge\frac17$, and $\ML(T)\le\frac18$. So there are exactly seven, they have loneliness $\frac18$ by the conjecture for seven speeds, and by (M) they are a multiple of some $U_i$. Take two repeat lines (Lemma~\ref{lem:repeat}) and normalize as in the proof of Theorem~\ref{thm:seven}. The first speed vector is $q=(U_i,r)$ with $U_i$ in increasing order and $r\in U_i$, and the second is a signed permutation $v$ of $(U_j,r')$ with $r'\in U_j$ and $v_1>0$. For each of the $21$ vectors $q$ there are $21\cdot\frac{8!}2\cdot2^7=54{,}190{,}080$ vectors $v$. The program \texttt{plane8} applies the test of the proof of Theorem~\ref{thm:seven} to the $56$ signed pairs of rows, now requiring that the distinct values $|d_k|$, divided by their minimum, form one of $U_1$, $U_2$, $U_3$. For each $q$, $53{,}222{,}400$ vectors give pairwise distinct rows and $26$ pass. In each survivor seven of the rows are proportional, their first entries form $U_i$, and the eighth row is independent of them. So $W$ is a signed coordinate permutation of $T_{U_i}\times\T$, and $\ML(T_{U_i}\times\T)=\ML(U_i)=\frac18$.
\end{proof}

\begin{proposition}\label{prop:lines8}
Let $i\in\{1,2,3\}$. A primitive set of eight distinct positive integers whose circle lies on a signed coordinate permutation of $T_{U_i}\times\T$ has $\ML<\frac18$ if and only if it is $U_i\cup\{8s\}$ for some $s\ge1$. Then $\ML(U_i\cup\{8s\})=\frac s{8s+1}$ for $s\ge2$, and $\ML(U_i\cup\{8\})=\frac19,\frac2{17},\frac2{17}$ for $i=1,2,3$.
\end{proposition}
\begin{proof}
The set is $aU_i\cup\{b\}$ with $a,b\ge1$ coprime. No element of $U_i$ is divisible by $8$. If $a\ge2$, the times $(\frac18+h)/a$, $0\le h<a$, keep the speeds $au$, $u\in U_i$, at distance at least $\frac18$, and at one of them $\norm{bt}\ge\frac12-\frac1{2a}$, as in the proof of Corollary~\ref{cor:seven}. If $a=1$ and $8\nmid b$, take $t=\frac18$. It remains to treat $a=1$ and $b=8s$. Let $D(t)$ be the distance from $t$ to $\frac18\Z$ and $g_i(t)=\min_{u\in U_i}\norm{ut}$. For each $i$,
\begin{equation}\label{eq:envelope}
 g_i(t)\le\max\Bigl(\frac2{17},\ \frac18-D(t)\Bigr)\qquad(t\in\R).
\end{equation}
Both sides are even, $1$-periodic and piecewise linear. On $[0,\frac12]$ their breakpoints lie among the fractions $\frac kd$ with $d=2u$, $u+w$ or $|u-w|$ for $u,w\in U_i$, $u\ne w$ in the last two cases, the multiples of $\frac1{16}$ and the points $\frac k8\pm\frac1{136}$, and the script of Section~\ref{sec:comp} verifies \eqref{eq:envelope} at all of them, which are $45$, $69$ and $87$ points for $i=1,2,3$. The values for $s=1$ follow from Lemma~\ref{lem:corner}. Let $s\ge2$ and $L=\frac s{8s+1}=\frac18-\frac1{8(8s+1)}\ge\frac2{17}$. Since $\norm{8st}\le8sD(t)$, we have $\min(g_i(t),\norm{8st})\le L$ when $D(t)\le\frac1{8(8s+1)}$, and by \eqref{eq:envelope} also when $D(t)\ge\frac1{8(8s+1)}$. At $t=L$ the speeds $1$ and $8s$ are at distance $L$ from the origin. The speed $7$, the only element of the $U_i$ congruent to $-1$ modulo $8$, is at distance $\frac18+\frac7{8(8s+1)}>L$. Every other speed $u\ne1$ in $U_i$ satisfies $\norm{u/8}\ge\frac14$ and moves by at most $\frac{13}{8(8s+1)}<\frac18$ between $\frac18$ and $L$. So $\ML(U_i\cup\{8s\})=L$.
\end{proof}

\begin{corollary}\label{cor:eight}
Put $Q=1{,}900{,}000$ and $\Delta=\frac1{8Q}$.
\begin{enumerate}[label=\textup{(\roman*)},leftmargin=*]
\item Every value of $\ML$ on sets of seven positive integers that is larger than $\frac18$ is at least $\frac18+\Delta$.
\item Assume \textup{(M)}. Every coordinate-free two-dimensional subtorus $W\subseteq\T^8$ with $\ML(W)<\frac18+\Delta$ is one of the products of Theorem~\ref{thm:eight}, up to signed coordinate permutations.
\item Assume \textup{(M)}. Every primitive set of eight distinct positive integers with $\ML<\frac18$ and Euclidean length at least $12{,}065{,}804{,}452{,}964$ is $U_i\cup\{8s\}$ for some $i\in\{1,2,3\}$ and $s\ge1$.
\end{enumerate}
\end{corollary}
\begin{proof}
(i) Let $V$ be a set of seven positive integers with $\frac18<x=\ML(V)<\frac18+\Delta$. We may assume that $V$ is primitive, and its elements are distinct because $x<\frac17$. If its circle lies on a signed coordinate permutation of $T_u\times\T$, then $V=\{1,2,3,4,5,6,7s\}$ by the proof of Corollary~\ref{cor:seven}(i), and $x=\frac s{7s+1}$ with $s\ge2$, so $x\ge\frac2{15}$, which is impossible. Otherwise, for a vector $v$ with the entries of $V$, the subtorus $S_2$ of the flag in the proof of Proposition~\ref{prop:flag} has $\ML(S_2)\ge\frac5{33}$ by Theorem~\ref{thm:seven}, so
\begin{gather*}
 \lambda_1\ge2\Bigl(\frac5{33}-x\Bigr)>\frac{53}{1000},\\
 \sum_{i\le j}\lambda_i\ge2\Bigl(\frac1{8-j}-x\Bigr)>\frac2{8-j}-\frac14-2\Delta\qquad(4\le j\le6).
\end{gather*}
The partial sums of the $\lambda_i$ are therefore at least those of $(\frac{53}{1000},\frac{53}{1000},\frac{53}{1000},\frac1{11},\frac16,\frac13)$: for $j\le3$ because $\lambda_i\ge\lambda_1$, and for $j=4,5,6$ because $\frac{159}{1000}+\frac1{11}<\frac14-2\Delta$ and the next two partial sums increase by $\frac16$ and $\frac13$ on both sides. Lemma~\ref{lem:prodsum} and Lemma~\ref{lem:densityn}(a) give $\eucl v\le198\cdot10^9/53^3<1{,}330{,}000$, so $\sqrt2\,\eucl v<Q$. As in the proof of Corollary~\ref{cor:seven}(ii), the denominator $q$ of $x=\frac pq$ in lowest terms divides the sum of two elements of $V$, so $q<Q$, and $x-\frac18=\frac{8p-q}{8q}\ge\frac1{8q}>\Delta$, a contradiction.

(ii) As in the proof of Theorem~\ref{thm:eight}, the forms are pairwise distinct up to sign, and the speeds of every repeat line take exactly seven distinct absolute values, because $\frac18+\Delta<\frac17$. Their loneliness is at least $\frac18$ and less than $\frac18+\Delta$, so it is $\frac18$ by (i). By (M) they are a multiple of some $U_i$, and the enumeration in the proof of Theorem~\ref{thm:eight} applies.

(iii) Let $v$ be a vector with the entries of such a set, and suppose that $T_v$ lies on none of the three products. By (ii) every coordinate-free two-dimensional subtorus containing $T_v$ has loneliness at least $\frac18+\Delta$, and Proposition~\ref{prop:flag}(b) applies with $n=8$ and $\gamma=\frac18+\Delta$. With $a=2\Delta=\frac1{4Q}$,
\begin{gather*}
 (d_1,\dots,d_7)=\Bigl(a,\ \frac1{28}-a,\ \frac1{21},\ \frac1{15},\ \frac1{10},\ \frac16,\ \frac13\Bigr),\\
 \frac1{d_1\cdots d_7}=\frac{7{,}641{,}648\cdot10^{12}}{633{,}331}<12{,}065{,}804{,}452{,}964.
\end{gather*}
So $T_v$ lies on one of the products, and Proposition~\ref{prop:lines8} gives the set.
\end{proof}

Marshall derives from his classification the conjecture of Kravitz on one very fast runner for eight speeds \cite{Marshall}. Under (M), Corollary~\ref{cor:eight} adds the classification of the critical subtori and an effective bound for the near-tight sets off them. As for seven speeds, the optimal eventual constant in $\max V\le Cq$ is $1$ under (M), since $U_i\cup\{8s\}$ has largest element $8s$ and value $\frac s{8s+1}$ for $s\ge2$.

\subsection{Counting near-tight sets by height}
For $n\ge3$ let $N_n(H)$ be the number of primitive sets of $n$ distinct positive integers with largest element at most $H$ and $\ML<\frac1n$.

\begin{theorem}\label{thm:count}
\begin{enumerate}[label=\textup{(\roman*)},leftmargin=*]
\item $N_4(H)=H^2/\pi^2+O(H\log H)$.
\item $N_5(H)=2\lfloor H/5\rfloor+12$ for $H\ge27$.
\item $N_6(H)=H^2/(4\pi^2)+O(H\log H)$.
\item There is an integer $C_7\ge27$ with $N_7(H)=\lfloor H/7\rfloor+C_7$ for all $H\ge5{,}145{,}525$.
\item Under \textup{(M)} there is an integer $C_8\ge17$ with $N_8(H)=3\lfloor H/8\rfloor+C_8$ for all $H\ge12{,}065{,}804{,}452{,}964$.
\end{enumerate}
\end{theorem}
\begin{proof}
We use the following count. Let $\alpha,\beta\ge1$ and $m\ge1$ be integers, $\Delta_H=\{(x,y)\in\R^2:x,y>0,\ \alpha x+\beta y\le H\}$, and let $\mathcal C$ be a set of $R$ residue classes modulo $m$, none of which is $\equiv(0,0)$ modulo a prime divisor of $m$. Then the number of $(x,y)\in\Delta_H\cap\Z^2$ with $\gcd(x,y)=1$ and $(x,y)\bmod m\in\mathcal C$ is
\begin{equation}\label{eq:coprimecount}
 \frac R{m^2}\prod_{p\nmid m}\Bigl(1-\frac1{p^2}\Bigr)\operatorname{area}(\Delta_H)+O(H\log H).
\end{equation}
Indeed, write $\mathbf 1[\gcd(x,y)=1]=\sum_{d\mid\gcd(x,y)}\mu(d)$. The terms with $\gcd(d,m)>1$ vanish, since a common prime factor $p$ of $d$ and $m$ would give $(x,y)\equiv(0,0)\pmod p$. For $\gcd(d,m)=1$ the points $(x,y)=d(x',y')$ correspond to the points of $\Delta_{H/d}$ in $R$ residue classes modulo $m$, whose number is $R\operatorname{area}(\Delta_{H/d})/m^2+O(H/d+1)$, and which do not exist for $d>H/(\alpha+\beta)$. Summing over $d$ gives \eqref{eq:coprimecount}.

(i) By Theorem~\ref{thm:main}, apart from the $O(H)$ sets of (a) and (c), the sets counted by $N_4(H)$ are the sets $V(x,y)$ with $x,y\ge1$ coprime, $y\not\equiv x$, $y\not\equiv2x\pmod4$ and $x+2y\le H$. The two largest elements of $V(x,y)$ are $x+2y$ and $x+y$, so $V(x,y)$ determines $(x,y)$. Of the twelve classes modulo $4$ that are not $\equiv(0,0)$ modulo $2$, the conditions keep eight: $x$ even and $y$ odd, and $(x,y)\equiv(1,0),(1,3),(3,0),(3,1)$. The triangle has area $H^2/4$, and \eqref{eq:coprimecount} gives $\frac8{16}\cdot\frac43\cdot\frac6{\pi^2}\cdot\frac{H^2}4=\frac{H^2}{\pi^2}$.

(ii) By Theorem~\ref{thm:fivefull}, for $H\ge27$ the sets are the twelve sets of (c), whose largest elements are at most $27$, the sets $\{1,2,3,4,5s\}$ with $5s\le H$, and the sets $\{1,3,4,7,5s\}$ with $\max(7,5s)\le H$. These are disjoint, and each family contributes $\lfloor H/5\rfloor$ sets.

(iii) By Theorem~\ref{thm:six}, the sets of (a), (b) and (d) contribute $O(H)$. For (c), the maps $(A,B)\mapsto(-A,3A+B)$ and $(A,B)\mapsto(A+B,-B)$ permute the six entries of $w(A,B)$ up to sign, so they preserve the set $V$ of absolute values. Each pair in the orbit of $(A,B)$ has entries in $\pm V$, so the orbit is finite. The two maps change $5A+3B$ by $-A$ and $-B$, so at a pair of the orbit with the largest value of $5A+3B$ we have $A>0$ and $B>0$, as no entry vanishes. For $A,B>0$ the three largest elements of $V$ are $3A+2B>3A+B>2A+B$, and their differences give $B$ and $A$. So $V$ has exactly one representation with $A,B>0$, and the sets of (c) with $\max V\le H$ correspond to the coprime pairs with $A,B>0$, $3A+2B\le H$ and $(A,B)$ in the classes modulo $6$ allowed by Theorem~\ref{thm:linesC}(i). These are twelve of the twenty-four classes that are $\equiv(0,0)$ neither modulo $2$ nor modulo $3$. The triangle has area $H^2/12$, and \eqref{eq:coprimecount} gives $\frac{12}{36}\cdot\frac32\cdot\frac6{\pi^2}\cdot\frac{H^2}{12}=\frac{H^2}{4\pi^2}$.

(iv) By Corollary~\ref{cor:seven}(i), the sets are the sets $\{1,\dots,6,7s\}$ with $7s\le H$ and finitely many further sets, all with Euclidean length below $5{,}145{,}525$. Let $C_7$ be the number of the latter. The census of Proposition~\ref{prop:computed} finds $27$ of them.

(v) By Corollary~\ref{cor:eight}(iii) and Proposition~\ref{prop:lines8}, the sets are the sets $U_i\cup\{8s\}$, which are pairwise distinct and have largest element at most $H$ exactly when $s\le H/8$, for $H\ge13$, and finitely many further sets, all with Euclidean length below $12{,}065{,}804{,}452{,}964$. The census finds $17$ of the latter.
\end{proof}

\begin{remark}\label{rem:schur6}
Every set of Theorem~\ref{thm:six}(c) is the disjoint union of two triples of the form $\{a,b,a+b\}$: with $A,B>0$ they are $\{A,A+B,2A+B\}$ and $\{B,3A+B,3A+2B\}$. So the proportion of primitive near-tight sets of six speeds at most $H$ that admit such a partition is $1-O(1/H)$. Outside (c) only $\{1,3,4,5,9,12\}$ and $\{2,5,6,8,10,11\}$ admit one. The sets of (a) have odd sum. For $s\ge3$ the element $6s$ of a set of (b) exceeds the sum of any two other elements. The set $\{1,3,4,5,6,9\}$ belongs to (c), while $\{1,3,4,5,9,12\}=\{1,4,5\}\cup\{3,9,12\}$ does not, because its three largest elements would force $(A,B)=(4,3)$. Among the nine sets of (d) only $\{2,5,6,8,10,11\}=\{2,8,10\}\cup\{5,6,11\}$ admits such a partition.
\end{remark}

\subsection{The complete classification for seven speeds}\label{sec:sevenfull}
Supplements A, B and C, which follow the appendix, close the finite range that Corollary~\ref{cor:seven} leaves. Their results carry the prefixes A, B and C, and their references without a prefix are to Sections~1 to~\ref{sec:comp}. They are computer assisted, and they describe their own programs. Here we record the outcome and what it implies for the hypothesis \textup{(M)}. Let $E_7$ be the set of the $27$ primitive near-tight sets of seven speeds found in Section~\ref{sec:schur} that are not of the form $\{1,\dots,6,7s\}$. It is Table~\ref{C:tab1} of Supplement~C.

\begin{theorem}[Supplement C, Theorem C1.1 and Corollary C1.2]\label{thm:sevenfull}
A primitive set of seven distinct positive integers has $\ML<\frac17$ if and only if it is $\{1,2,3,4,5,6,7s\}$ for some $s\ge1$ or one of the $27$ sets of $E_7$. The values of $\ML$ below $\frac17$ on sets of seven positive integers are the numbers $\frac s{7s+1}$, $s\ge1$, and $\frac3{23}$, $\frac5{37}$, $\frac9{65}$. Every set of seven positive integers with $\ML<\frac17$ contains three distinct elements $a,b,a+b$.
\end{theorem}

The proof in Supplement~C has four steps. Theorem~C2.1 lowers the bound $5{,}145{,}525$ of Corollary~\ref{cor:seven}(i) to $6{,}000$, with four successive normal vectors and exact tests for balls in the corresponding subtori, as prepared in Supplement~B. Theorem~C3.1 shows, by an exhaustive search of the coverings of $\Z/p\Z$ for $88$ primes $p$, that a set $V$ of seven speeds with $\ML(V)\le\frac17$ and $\sqrt3\,\eucl V<29{,}893$ contains $a,b,a+b$. Theorems~C4.1 and~C5.1 reduce this triple to $74$ shapes and determine all the sets that contain one of them. The last assertion of Theorem~\ref{thm:sevenfull} can also be read off the list, since $\{1,\dots,6,7s\}$ contains $1,2,3$ and every set of $E_7$ contains such a triple.

\begin{corollary}\label{cor:M}
Hypothesis \textup{(M)} holds. A set of seven positive integers with $\ML=\frac18$ is a dilate of $U_1$, $U_2$ or $U_3$.
\end{corollary}
\begin{proof}
Such a set is near-tight, and dividing by the greatest common divisor does not change $\ML$. By Theorem~\ref{thm:sevenfull} the primitive set is $\{1,\dots,6,7s\}$ or belongs to $E_7$. In the first case $\ML=\frac s{7s+1}$ by Corollary~\ref{cor:seven}(i), which equals $\frac18$ only for $s=1$, and this is $U_1$. Among the $27$ sets of $E_7$ exactly two have $\ML=\frac18$, namely $U_2=\{1,2,3,4,5,7,12\}$ and $U_3=\{1,4,5,6,7,11,13\}$. The supplements use \textup{(M)} only in Corollary~A6.2 and Theorem~B9.1, which concern eight speeds, so the argument is not circular.
\end{proof}

\begin{corollary}\label{cor:eightfull}
\begin{enumerate}[label=\textup{(\roman*)},leftmargin=*]
\item Theorem~\ref{thm:eight}, Corollary~\ref{cor:eight} and Theorem~\ref{thm:count}(v) hold without the hypothesis \textup{(M)}.
\item Every primitive set $V$ of eight distinct positive integers with $\ML(V)<\frac18$ and $\eucl V\ge3{,}898{,}150{,}815$ is $U_i\cup\{8s\}$ for some $i\in\{1,2,3\}$ and $s\ge1$ \textup{(Supplement~B, Theorem~B9.1)}.
\item In Theorem~\ref{thm:count}, $C_7=27$ and $N_7(H)=\lfloor H/7\rfloor+27$ for all $H\ge112$, and $N_8(H)=3\lfloor H/8\rfloor+C_8$ for all $H\ge3{,}898{,}150{,}815$.
\end{enumerate}
\end{corollary}
\begin{proof}
(i) and (ii) follow from Corollary~\ref{cor:M}, since Theorem~B9.1 assumes only \textup{(M)}. For (iii), Theorem~\ref{thm:sevenfull} gives $N_7(H)=\lfloor H/7\rfloor+\#\{V\in E_7:\max V\le H\}$, and the largest element of a set of $E_7$ is at most $112$. For eight speeds the proof of Theorem~\ref{thm:count}(v) applies with the bound of (ii).
\end{proof}

\section{The certificate}\label{sec:relcert}
Write $\ee(y)=e^{2\pi iy}$, and for a real trigonometric polynomial $\Psi$ on $\T^k$ write $\Psi(x)=\sum_m\widehat\Psi(m)\ee(m\cdot x)$.

\begin{lemma}\label{lem:cert}
Let $w_1,\dots,w_k\ge0$ and
\[
 \Phi(x)=\prod_{j=1}^k\sin^2(\pi x_j)-\sum_{i=1}^k w_i\cos^2(\pi x_i)\prod_{j\ne i}\sin^2(\pi x_j)\qquad(x\in\T^k).
\]
\begin{enumerate}[label=\textup{(\alph*)},leftmargin=*]
\item $\Phi(x)\le1$ for all $x$, and $\Phi$ agrees with \eqref{eq:Phi} when no $x_j$ is an integer and $w_i=\tan^2(\pi\delta_i)$.
\item If $w_i=\tan^2(\pi\delta_i)$ with $\delta_i\in[0,\frac12)$, and $\norm{x_i}\le\delta_i$ for some $i$, then $\Phi(x)\le0$.
\item The Fourier coefficients of $\Phi$ vanish outside $\{-1,0,1\}^k$, and for $m\in\{-1,0,1\}^k$ with support $S$,
\[
 \widehat\Phi(m)=2^{-k}\Bigl(-\frac12\Bigr)^{|S|}\beta(m),\qquad \beta(m)=1-\sum_{i\notin S}w_i+\sum_{i\in S}w_i.
\]
\end{enumerate}
\end{lemma}

\begin{proof}
(a) The subtracted terms are nonnegative and $\prod_j\sin^2(\pi x_j)\le1$. The second claim is the identity $\cos^2=\cot^2\sin^2$.

(b) If $x_i\in\Z$, then every term containing $\sin^2(\pi x_i)$ vanishes, so that $\Phi(x)=-w_i\prod_{j\ne i}\sin^2(\pi x_j)\le0$. Otherwise $0<\norm{x_i}\le\delta_i<\frac12$. Since $\sin^2(\pi y)$ and $\cos^2(\pi y)$ depend only on $\norm y$ and $\tan$ is increasing on $[0,\frac\pi2)$, we get $\sin^2(\pi x_i)\le\tan^2(\pi\delta_i)\cos^2(\pi x_i)=w_i\cos^2(\pi x_i)$. So the $i$th subtracted term is at least $\prod_j\sin^2(\pi x_j)$, and $\Phi(x)\le0$.

(c) Use $\sin^2(\pi y)=\frac12-\frac14(\ee(y)+\ee(-y))$ and $\cos^2(\pi y)=\frac12+\frac14(\ee(y)+\ee(-y))$. In $\prod_j\sin^2(\pi x_j)$ the coefficient of $\ee(m\cdot x)$ is $2^{-k}(-\frac12)^{|S|}$. In $\cos^2(\pi x_i)\prod_{j\ne i}\sin^2(\pi x_j)$ it is $2^{-k}(-\frac12)^{|S|}$ if $i\notin S$ and $-2^{-k}(-\frac12)^{|S|}$ if $i\in S$. Summing gives the formula.
\end{proof}

\begin{lemma}\label{lem:integral}
For $v\in\Z^k$, $s\in\R^k$ and $\Phi$ as in Lemma~\ref{lem:cert},
\[
 \int_0^1\Phi(tv+s)\,dt=\widehat\Phi(0)+\sum_{m\in\Rel(v)}\widehat\Phi(m)\cos(2\pi m\cdot s).
\]
\end{lemma}

\begin{proof}
$\int_0^1\ee(t\,m\cdot v)\,dt$ is $1$ if $m\cdot v=0$ and $0$ otherwise, and $\widehat\Phi(-m)=\widehat\Phi(m)$ is real, so the terms for $m$ and $-m$ combine to cosines.
\end{proof}

\section{Proofs of the relation theorems}\label{sec:relproofs}
\begin{lemma}\label{lem:thresholds}
The sequences $f(k)=k\tan^2(\pi/(k+1))$, $k\ge1$, and $g(n)=n\tan^2(\pi/n)$, $n\ge3$, are strictly decreasing, with $f(8)>1>f(9)$ and $g(10)>1>g(11)$. So $f(k)<1$ exactly for $k\ge9$ and $g(n)<1$ exactly for $n\ge11$.
\end{lemma}

\begin{proof}
The function $\tan x/x$ is increasing on $(0,\frac\pi2)$. Hence $f(k)=\pi^2\frac{k}{(k+1)^2}\bigl(\frac{\tan x_k}{x_k}\bigr)^2$ with $x_k=\pi/(k+1)$ is a product of positive decreasing factors, and so is $g(n)=\frac{\pi^2}n\bigl(\frac{\tan(\pi/n)}{\pi/n}\bigr)^2$. For $0<x<1$ we have $\tan x\le x/(1-x^2/2)$, from $\sin x\le x$ and $\cos x\ge1-x^2/2$. With $\pi^2<9.8697$ this gives $f(9)\le\frac{9\pi^2}{100}(1-\frac{\pi^2}{200})^{-2}<0.983$ and $g(11)\le\frac{\pi^2}{11}(1-\frac{\pi^2}{242})^{-2}<0.976$. For $0<x<\frac\pi2$ we have $\tan x\ge x+x^3/3$, and with $\pi/9>0.349065$ and $\pi/10>0.314159$ this gives $f(8)>1.05$ and $g(10)>1.05$.
\end{proof}

\begin{proof}[Proof of Theorem~\ref{thm:relgen}]
Put $w_i=\tan^2(\pi\delta_i)$, so $\sum_iw_i<1$, and let $\Phi$ be as in Lemma~\ref{lem:cert}. Every $m\in\Rel(v)$ has $\lone m$ even, so $(-\frac12)^{\lone m}>0$ and $\beta(m)\ge1-\sum_iw_i>0$, hence $\widehat\Phi(m)>0$. Lemma~\ref{lem:integral} with $s=0$ gives
\[
 \int_0^1\Phi(tv)\,dt\ \ge\ \widehat\Phi(0)=2^{-k}\Bigl(1-\sum_iw_i\Bigr).
\]
Let $\Omega=\{t\in[0,1):\norm{tv_i}>\delta_i\text{ for all }i\}$. By Lemma~\ref{lem:cert}(b), $\Phi(tv)\le0$ for $t\notin\Omega$, and $\Phi\le1$ by Lemma~\ref{lem:cert}(a). Therefore $|\Omega|\ge\int_\Omega\Phi(tv)\,dt\ge\int_0^1\Phi(tv)\,dt\ge2^{-k}(1-\sum_iw_i)$.
\end{proof}

A zero entry $v_i=0$ gives the odd signed relation $m=e_i$, so the hypothesis excludes it.

\begin{proof}[Proof of Theorem~\ref{thm:rel}]
Take $\delta_i=1/(k+1)$ for all $i$. Then $\sum_i\tan^2(\pi\delta_i)=f(k)<1$ by Lemma~\ref{lem:thresholds}. If $v$ had no odd signed relation, Theorem~\ref{thm:relgen} would give $t$ with $\norm{tv_i}>1/(k+1)$ for all $i$, contradicting $\ML(v)\le1/(k+1)$.
\end{proof}

\begin{proof}[Proof of Corollary~\ref{cor:rel}]
(a) For $\delta<\delta^*_k$ we have $k\tan^2(\pi\delta)<1$, so Theorem~\ref{thm:relgen} gives $\ML(v)>\delta$. Hence $\ML(v)\ge\delta^*_k$. The lower bound follows from $\arctan y\ge y-y^3/3$. (b) Take $\delta_i=1/n$. Then $\sum_i\tan^2(\pi\delta_i)=g(n)<1$ by Lemma~\ref{lem:thresholds}, and argue as for Theorem~\ref{thm:rel}.
\end{proof}

\begin{proof}[Proof of Theorem~\ref{thm:relcount}]
Let $\Phi$ be as in Lemma~\ref{lem:cert} with all $w_i=w$. Since $\ML(v)\le\delta$, every $t$ has some $\norm{tv_i}\le\delta$, so $\Phi(tv)\le0$ for all $t$ by Lemma~\ref{lem:cert}(b). With $\beta(m)=1+(2\lone m-k)w$, Lemma~\ref{lem:integral} gives
\[
 0\ \ge\ 2^k\int_0^1\Phi(tv)\,dt=1-kw+\sum_{m\in\Rel(v)}\Bigl(-\frac12\Bigr)^{\lone m}\bigl(1+(2\lone m-k)w\bigr),
\]
which is the inequality. For the second statement note that $\lone m\le k$ gives $1+(2\lone m-k)w\le1+\lone m\,w$, and that $x\mapsto2^{-x}(1+xw)$ decreases along the integers $x\ge1$, because $1+(x+1)w\le2(1+xw)$. So each odd signed relation contributes at most $2^{-s}(1+sw)$ to the left side, and the number $N$ of them satisfies $N2^{-s}(1+sw)\ge1-kw$. As $sw\le kw<1$, $N>2^{s-1}(1-kw)$.
\end{proof}

\begin{proof}[Proof of Theorem~\ref{thm:relshift}]
By Lemma~\ref{lem:integral} and Lemma~\ref{lem:cert}(c),
\begin{align*}
 \int_0^1\Phi(tv+s)\,dt&=\widehat\Phi(0)+2^{-k}\sum_{m\in\Rel(v)}2^{-\lone m}\beta(m)\,(-1)^{\lone m}\cos(2\pi m\cdot s)\\
 &\ge\widehat\Phi(0),
\end{align*}
since every $\beta(m)$ is positive. Now argue as in the proof of Theorem~\ref{thm:relgen}, with $tv$ replaced by $tv+s$.
\end{proof}

\begin{remark}
If $s\in(\frac12\Z)^k$, the hypothesis of Theorem~\ref{thm:relshift} says that every signed relation has an even number of indices $i$ with $s_i\in\Z$ in its support: runners that start at the antipode do not count. The same proofs work for real speeds, with $\int_0^1$ replaced by the mean value over $\R$ and relations taken over the reals.
\end{remark}

\section{Optimality of the certificate}\label{sec:relopt}
\begin{proof}[Proof of Theorem~\ref{thm:relopt}(a)]
For $Z\subseteq\{1,\dots,k\}$ let $x_Z\in\{0,\frac12\}^k$ have zero coordinates exactly on $Z$, and let $\nu$ be the uniform probability measure on $\{x_Z:|Z|\text{ odd}\}$. For $m\in\Z^k$ put $\varepsilon_j=(-1)^{m_j}$. Then $\ee(m\cdot x_Z)=\prod_{j\notin Z}\varepsilon_j$, and from $\sum_Zy^{|Z|}\prod_{j\notin Z}\varepsilon_j=\prod_j(y+\varepsilon_j)$,
\begin{align*}
 \int\ee(m\cdot x)\,d\nu(x)&=\frac{\prod_j(1+\varepsilon_j)-\prod_j(\varepsilon_j-1)}{2^k}\\
 &=\begin{cases}1,&\text{all }m_j\text{ even},\\(-1)^{k+1},&\text{all }m_j\text{ odd},\\0,&\text{otherwise.}\end{cases}
\end{align*}
If $\lone m\le k-1$, not all $m_j$ are odd. Hence $\int\Psi\,d\nu=\sum\widehat\Psi(m)$ over the $m$ in the support of $\widehat\Psi$ with all $m_j$ even. These $m$ have even coordinate sum, so the sum is at least $\widehat\Psi(0)>0$. As $\nu$ is a probability measure, $\Psi>0$ at some point of its support.
\end{proof}

\begin{proof}[Proof of Theorem~\ref{thm:relopt}(b)]
Put $\gamma_i=\cos(2\pi\delta_i)$, so $(1+\gamma_i)/(1-\gamma_i)=\cot^2(\pi\delta_i)$ when $\delta_i>0$. Since $\sum_i\tan^2(\pi\delta_i)\ge1$ we can choose $\lambda_i\in[0,\tan^2(\pi\delta_i)]$ with $\sum_i\lambda_i=1$, and then $\lambda_i=0$ whenever $\delta_i=0$. With $x_Z$ as above let $r_i=x_{\{i\}}$, and let $q_i^\pm$ have $i$th coordinate $\pm\delta_i$ and all other coordinates $\frac12$. Define
\[
 \nu=2^{-k}\Bigl[\sum_{Z\ne\emptyset}\mathbf 1_{x_Z}+\sum_{i:\lambda_i>0}\lambda_i\Bigl(\frac{\mathbf 1_{q_i^+}+\mathbf 1_{q_i^-}}{1-\gamma_i}-\frac{1+\gamma_i}{1-\gamma_i}\mathbf 1_{r_i}\Bigr)\Bigr],
\]
where $\mathbf 1_y$ is the unit point mass at $y$. The mass at $r_i$ is $2^{-k}(1-\lambda_i\cot^2(\pi\delta_i))\ge0$, so $\nu\ge0$, and every point of its support has a coordinate $x_i$ with $\norm{x_i}\le\delta_i$. Let $m\in\{-1,0,1\}^k$ have support $S$. Summing over all $Z$ gives $\sum_Z\ee(m\cdot x_Z)=2^k\mathbf 1[S=\emptyset]$, and $\ee(m\cdot x_\emptyset)=(-1)^{|S|}$. For each $i$ the bracket contributes $(-1)^{|S|}$. If $i\notin S$, all three points give $(-1)^{|S|}$, and $(2-1-\gamma_i)/(1-\gamma_i)=1$. If $i\in S$, the points $q_i^\pm$ give $(-1)^{|S|-1}\cos(2\pi\delta_i)$ on average and $r_i$ gives $(-1)^{|S|-1}$, so the bracket is $(-1)^{|S|-1}(2\gamma_i-1-\gamma_i)/(1-\gamma_i)=(-1)^{|S|}$. Since $\sum_i\lambda_i=1$, these contributions cancel the missing term $x_\emptyset$, and $\int\ee(m\cdot x)\,d\nu=\mathbf 1[m=0]$. In particular $\nu$ is a probability measure and $\int\Psi\,d\nu=\widehat\Psi(0)>0$, so $\Psi>0$ somewhere on its support.
\end{proof}

The constant $2^{-k}\bigl(1-\sum_i\tan^2(\pi\delta_i)\bigr)$ of Theorem~\ref{thm:relgen} is the best that a minorant of degree one can give, but on integer orbits the threshold can be crossed. Write $B_\delta=\{x\in\T^k:\norm{x_i}>\delta_i\text{ for all }i\}$.

\begin{proposition}\label{prop:minorant}
Let $\delta_1,\dots,\delta_k\in[0,\frac12)$, $w_i=\tan^2(\pi\delta_i)$ and $W=\sum_iw_i$. Among the real trigonometric polynomials $P$ with Fourier coefficients supported on $\{-1,0,1\}^k$ and $P\le\mathbf 1_{B_\delta}$, the largest value of $\widehat P(0)$ is $2^{-k}\max(1-W,0)$. If $W<1$, the polynomial $\Phi$ of Lemma~\ref{lem:cert} is the only maximizer that is invariant under the sign changes $x_i\mapsto-x_i$.
\end{proposition}
\begin{proof}
If $W\ge1$, Theorem~\ref{thm:relopt}(b) gives $\widehat P(0)\le0$, and $P=0$ attains $0$. Let $W<1$, and use the points $x_Z$, $r_i=x_{\{i\}}$ and $q_i^\pm$ of the proofs above. Start with the mass $2^{-k}$ at every point $x_Z$. For each $i$ with $w_i>0$ remove the masses $2^{-k}w_i$ from $x_\emptyset$ and $2^{-k}$ from $r_i$, and put the mass $2^{-k}(1+w_i)/2$ at each of $q_i^+$ and $q_i^-$. Since $W<1$, the result is a probability measure $\nu$. For $m\in\{-1,0,1\}^k$ each of these moves preserves $\int\ee(m\cdot x)$: the four points agree outside the coordinate $i$, and $(1+w_i)\cos(2\pi\delta_i)=1-w_i$. Hence $\int\ee(m\cdot x)\,d\nu=2^{-k}\sum_Z\ee(m\cdot x_Z)=\mathbf 1[m=0]$, and $\int P\,d\nu=\widehat P(0)$. The only point of the support of $\nu$ in $B_\delta$ is $x_\emptyset$, of mass $2^{-k}(1-W)$, so $\widehat P(0)\le2^{-k}(1-W)$. By Lemma~\ref{lem:cert}, $\Phi\le\mathbf 1_{B_\delta}$ and $\widehat\Phi(0)=2^{-k}(1-W)$.

Let $P$ be a maximizer that is invariant under the sign changes. Then $P$ has degree at most one in each $\cos(2\pi x_i)$, so it is a multi-affine function of the numbers $u_i=\sin^2(\pi x_i)$, and it is determined by its values at the points $x_Z$. Equality gives $\int(\mathbf 1_{B_\delta}-P)\,d\nu=0$, so $P(x_\emptyset)=1$ and $P=0$ at every other point of the support of $\nu$, which consists of the $x_Z$ with $|Z|\ge2$, the $r_i$ with $w_i=0$ and the $q_i^\pm$ with $w_i>0$. If $w_i>0$, then $u_i=w_i/(1+w_i)$ at $q_i^+$, and interpolation between $r_i$ and $x_\emptyset$ gives $0=P(q_i^+)=(P(r_i)+w_i)/(1+w_i)$, so $P(r_i)=-w_i$. These are the values of $\Phi$ at the points $x_Z$, so $P=\Phi$.
\end{proof}

\begin{theorem}\label{thm:orbit}
Let $k\ge2$ and let $v\in\Z^k$ have no odd signed relation. Put
\[
 h_k=\frac1\pi\arctan\frac1{\sqrt{2k}},\qquad a_k=\frac{h_k}{2(k-1)},\qquad c_k=\frac{h_k}{2k+1}\,\sin^{2k-2}(\pi a_k).
\]
If $\delta_1,\dots,\delta_k\in[0,\frac12)$ and $W=\sum_i\tan^2(\pi\delta_i)\ge1$, then the set of $t\in[0,1)$ with $tv\in B_\delta$ has measure at least $c_k-(W-1)$. Consequently
\[
 \ML(v)\ge\frac1\pi\arctan\sqrt{\frac{1+c_k}k}>\frac1\pi\arctan\frac1{\sqrt k}.
\]
\end{theorem}
\begin{proof}
Let $w_i=\tan^2(\pi\delta_i)$, and let $\Phi_w$ be the polynomial of Lemma~\ref{lem:cert} with the weights $w_i$. For the weights $w_i/W$, which sum to $1$, Lemma~\ref{lem:cert}(c) gives $\beta(m)=\frac2W\sum_{i\in S}w_i\ge0$, and $\beta(0)=0$. Every $m\in\Rel(v)$ has even support, so the coefficients $\widehat\Phi_{w/W}(m)$, $m\in\Rel(v)$, are nonnegative, and Lemma~\ref{lem:integral} gives $\int_0^1\Phi_{w/W}(tv)\,dt\ge0$. The difference $\Phi_{w/W}-\Phi_w$ is the sum of the terms $(w_i-w_i/W)\cos^2(\pi x_i)\prod_{j\ne i}\sin^2(\pi x_j)$, which lie between $0$ and $w_i-w_i/W$, so $\int_0^1\Phi_w(tv)\,dt\ge-(W-1)$.

Choose $i$ with $w_i\ge\frac1k$. No $v_j$ vanishes, since $v_j=0$ would be an odd signed relation, so each $tv_j$ is uniformly distributed on $\T$ for $t$ uniform in $[0,1)$. Hence the set $E$ of $t$ with $\norm{tv_i}\le h_k$ and $\norm{tv_j}\ge a_k$ for all $j\ne i$ has measure at least $2h_k-2(k-1)a_k=h_k$. As $\tan^2(\pi h_k)=\frac1{2k}$, for $t\in E$
\[
 \sin^2(\pi tv_i)-w_i\cos^2(\pi tv_i)\le\sin^2(\pi h_k)-\frac1k\cos^2(\pi h_k)=-\frac1{2k+1},
\]
and dropping the other subtracted terms of $\Phi_w$ gives $\Phi_w(tv)\le-\sin^{2k-2}(\pi a_k)/(2k+1)$. Since $\mathbf 1_{B_\delta}-\Phi_w\ge0$ everywhere by Lemma~\ref{lem:cert},
\[
 \int_0^1\mathbf 1_{B_\delta}(tv)\,dt\ \ge\ \int_0^1\Phi_w(tv)\,dt+\int_E\bigl(-\Phi_w(tv)\bigr)\,dt\ \ge\ -(W-1)+c_k.
\]
For the last statement take all $\delta_i$ equal to some $\delta$ with $1\le k\tan^2(\pi\delta)<1+c_k$. Then the measure is positive, so $\ML(v)>\delta$.
\end{proof}

The gain $c_k$ is tiny, $c_2\approx1.6\cdot10^{-3}$ and $c_3\approx1.5\cdot10^{-6}$, but it shows that the threshold $\sum_i\tan^2(\pi\delta_i)=1$, which is sharp for minorants by Proposition~\ref{prop:minorant} and Theorem~\ref{thm:relopt}(b), is not sharp on orbits $\{tv\}$.

\begin{corollary}\label{cor:orbitshift}
Let $v\in\Z^k$ have no zero entry, and let $s\in\R^k$ satisfy $(-1)^{\lone m}\cos(2\pi m\cdot s)\ge0$ for every $m\in\Rel(v)$, for instance any $s$ if $v$ is dissociated. Then the measure bound of Theorem~\ref{thm:orbit} holds with $tv+s$ in place of $tv$.
\end{corollary}
\begin{proof}
By Lemma~\ref{lem:integral} with the shift $s$, the term of $m\in\Rel(v)$ in $\int_0^1\Phi_{w/W}(tv+s)\,dt$ is $2^{-k}2^{-\lone m}\beta(m)\,(-1)^{\lone m}\cos(2\pi m\cdot s)\ge0$. Each $tv_j+s_j$ is uniformly distributed on $\T$, and the rest of the proof is unchanged.
\end{proof}

\begin{corollary}\label{cor:four}
Let $v$ consist of $k$ distinct positive integers, where either $k=8$ and $\ML(v)\le\frac19$, or $k=10$ and $\ML(v)\le\frac1{10}$. If $v$ has a signed relation of support four, then it has an odd signed relation.
\end{corollary}
\begin{proof}
Suppose that $v$ has no odd signed relation. Let $\delta=\frac19$ if $k=8$ and $\delta=\frac1{10}$ if $k=10$, put $w=\tan^2(\pi\delta)$, and apply Lemmas~\ref{lem:cert} and~\ref{lem:integral} with all $\delta_i=\delta$. Every $m\in\Rel(v)$ has even support size $r=\lone m$, and $r\ne2$ because the speeds are distinct and positive. So
\[
 2^k\int_0^1\Phi(tv)\,dt=1-kw+\sum_{m\in\Rel(v)}2^{-r}\bigl(1+(2r-k)w\bigr).
\]
For $r\ge4$ each term is positive, since $1+(2r-k)w\ge1-2w>0$. The relation of support four and its negative contribute $2\cdot2^{-4}(1+(8-k)w)$, so $2^k\int_0^1\Phi(tv)\,dt\ge1-kw+\frac18(1+(8-k)w)$. This is positive if $w<\frac9{64}$ for $k=8$ and if $w<\frac9{82}$ for $k=10$. Both hold: $\tan(\frac\pi9)\le\frac\pi9/(1-\frac{\pi^2}{162})$ and $\pi<\frac{22}7$ give $\tan^2(\frac\pi9)<0.139$, and $\tan^2(\frac\pi{10})=1-\frac2{\sqrt5}<0.106$. So $\Phi(tv)>0$ for some $t$, and Lemma~\ref{lem:cert}(b) gives $\norm{tv_i}>\delta$ for all $i$, which contradicts $\ML(v)\le\delta$.
\end{proof}

The measure in the proof of Theorem~\ref{thm:relopt}(a) belongs to a family, and the family controls the Fourier mass of a certificate in every parity class. For a real trigonometric polynomial $\Psi$ on $\T^k$ and $0\le j\le k$ let $\sigma_j(\Psi)$ be the sum of the coefficients $\widehat\Psi(m)$ over the vectors $m\in\Z^k$ with exactly $j$ odd entries. These numbers are real, because $m$ and $-m$ have the same odd entries.

\begin{proposition}\label{prop:parity}
Let $\Psi$ be a real trigonometric polynomial on $\T^k$ with $\Psi(x)\le0$ for every $x\in\{0,\frac12\}^k$ that has a zero coordinate. Then
\[
 (-1)^j\sigma_j(\Psi)\ \ge\ \binom kj\,\sigma_0(\Psi)\qquad(1\le j\le k).
\]
Equality holds for every $j$ when $\Psi=\prod_i\sin^2(\pi x_i)$.
\end{proposition}
\begin{proof}
For $y\in\{-1,1\}^k$ let $x(y)\in\{0,\frac12\}^k$ be the point with $\ee(x(y)_i)=y_i$, so that $x(y)_i=0$ exactly when $y_i=1$. If the odd entries of $m\in\Z^k$ are at the positions in $S$, then $\ee(m\cdot x(y))=y^S:=\prod_{i\in S}y_i$. Hence $\Psi(x(y))=\sum_SA_Sy^S$, where $A_S$ is the sum of $\widehat\Psi(m)$ over these $m$, and $\sigma_j(\Psi)=\sum_{|S|=j}A_S$. Let $e_j(y)=\sum_{|S|=j}y^S$ and
\[
 \omega_j(y)=1+(-1)^{j+1}\binom kj^{-1}e_j(y).
\]
Then $\omega_j\ge0$ on $\{-1,1\}^k$ because $|e_j|\le\binom kj$, and $\omega_j(-1,\dots,-1)=0$ because $e_j(-1,\dots,-1)=(-1)^j\binom kj$. Every $y\ne(-1,\dots,-1)$ has a coordinate $1$, so $\Psi(x(y))\le0$, and by the orthogonality of the functions $y^S$,
\begin{multline*}
 0\ \ge\ 2^{-k}\sum_y\omega_j(y)\Psi(x(y))=A_\emptyset+(-1)^{j+1}\binom kj^{-1}\sum_{|S|=j}A_S\\
 =\sigma_0(\Psi)+(-1)^{j+1}\binom kj^{-1}\sigma_j(\Psi).
\end{multline*}
For the product, Lemma~\ref{lem:cert}(c) with all $w_i=0$ gives $A_S=2^{-k}(-1)^{|S|}$, hence $\sigma_j=2^{-k}(-1)^j\binom kj$.
\end{proof}

Theorem~\ref{thm:relopt}(a) is the case $j=k$. Under its hypotheses $\sigma_k(\Psi)=0$ and $\sigma_0(\Psi)\ge\widehat\Psi(0)>0$, and $\omega_k$ is supported on the points with an odd number of zero coordinates. The other values of $j$ give the following.

\begin{corollary}\label{cor:allweights}
Let $\Psi$ be a real trigonometric polynomial on $\T^k$ with Fourier coefficients supported on $\{-1,0,1\}^k$, with $\widehat\Psi(0)>0$ and with $\Psi(x)\le0$ whenever some $x_i$ is $0$. For $1\le j\le k$ let $s_j$ be the sum of $\widehat\Psi(m)$ over the $m$ with support of size $j$. Then $s_j\le-\binom kj\widehat\Psi(0)$ for odd $j$ and $s_j\ge\binom kj\widehat\Psi(0)$ for even $j$.
\end{corollary}
\begin{proof}
Here $\sigma_0(\Psi)=\widehat\Psi(0)$ and $\sigma_j(\Psi)=s_j$.
\end{proof}

So a certificate with Fourier support $\{-1,0,1\}^k$ has negative mass on every odd support size, at least as much as the product $\prod_i\sin^2(\pi x_i)$ normalized to the same mean. Suppose that the argument of Section~\ref{sec:relproofs} is to give a relation whose support has a size in a prescribed set $J$ of odd numbers. Then the coefficients on the supports of every other size must be nonnegative, since a relation of such a size may be present. By Corollary~\ref{cor:allweights} this forces $J$ to contain every odd number from $3$ to $k$, for every choice of thresholds. In particular this method cannot bound the size of the support of the odd relation in Corollary~\ref{cor:rel}(b) by less than the largest odd number that is at most $n$. Conjecture~\ref{conj:schur} asks for the support size $3$, so for $n\ge5$ this method cannot prove it. More is true: no trigonometric polynomial with only odd harmonics can.

\begin{corollary}\label{cor:noschur}
Let $n\ge5$, and let $\Psi$ be a real trigonometric polynomial on $\T^n$ with $\widehat\Psi(0)>0$ and with $\Psi(x)\le0$ whenever some $x_i$ is $0$. Suppose that $\widehat\Psi(m)\ge0$ for every $m$ that has at least four nonzero entries, and for every $m$ that has two or three nonzero entries, not all of them $\pm1$. Then
\[
 \sum_{i=1}^n\sum_{c\ne0}\widehat\Psi(2c\,\epsilon_i)\ \le\ -\widehat\Psi(0),
\]
where $\epsilon_i$ is the $i$th unit vector. In particular the Fourier coefficients of $\Psi$ are not supported on $\{-1,0,1\}^n$, nor on the vectors whose nonzero entries are all odd.
\end{corollary}
\begin{proof}
A vector with exactly five odd entries has at least five nonzero entries, so $\sigma_5(\Psi)\ge0$, and Proposition~\ref{prop:parity} with $j=5$ gives $\sigma_0(\Psi)\le0$. A nonzero vector with all entries even and at least two nonzero entries has a nonnegative coefficient by assumption. The other vectors with all entries even are $0$ and the vectors $2c\,\epsilon_i$.
\end{proof}

The assumptions on the coefficients are those under which $\Psi$ proves Conjecture~\ref{conj:schur} at a threshold $\delta>0$. Let $V=\{v_1,\dots,v_n\}$ consist of distinct positive integers and contain no $a,b,a+b$. If $m\cdot v=0$ and $m\ne0$, then $m$ has at least four nonzero entries, or two or three nonzero entries that are not all $\pm1$. So $\int_0^1\Psi(tv)\,dt\ge\widehat\Psi(0)>0$, and if in addition $\Psi\le0$ wherever some $\norm{x_i}\le\delta$, then $\ML(V)>\delta$. By Corollary~\ref{cor:noschur} such a $\Psi$ must contain the even harmonics $\cos(4\pi cx_i)$ of single coordinates, with total coefficient at most $-\widehat\Psi(0)$. The union bound is of this kind and reaches every $\delta<\frac1{2n}$. We do not know how far beyond $\frac1{2n}$ such certificates reach. For $n=5$, a linear program over the polynomials with all $|m_i|\le5$ that are invariant under the signed permutations of the coordinates, with the condition $\Psi\le0$ imposed on the points of $(\frac1{40}\Z/\Z)^5$ that have a coordinate $\norm{x_i}\le\frac18$, has no solution (Section~\ref{sec:comp}). So for five speeds no certificate of this degree reaches $\frac18$, while Conjecture~\ref{conj:schur} asks for $\frac15$.

\section{A second-moment bound for every number of speeds}\label{sec:box}
This section proves Theorem~\ref{thm:l2}. It refines the Fourier argument of Beck and Everett \cite[Section~2]{BeckEverett} and uses their window, and it works for every closed subgroup of a torus. For a closed subgroup $H\le\T^k$, not necessarily connected, let
\[
 L=\{m\in\Z^k:m\cdot x\in\Z\text{ for all }x\in H\}
\]
be its annihilator, a free abelian group of some rank $d\le k$. For $\alpha\in\T^k$ and $0<h_j<\frac12$ put
\[
 \mathcal B(\alpha,h)=\{x\in\T^k:\norm{x_j-\alpha_j}<h_j\text{ for every }j\}.
\]
The class of $m\cdot\alpha$ modulo $1$ does not depend on the representative of $\alpha$.

\begin{theorem}[Box obstruction]\label{thm:box}
Let $H\le\T^k$ be a closed subgroup with $H\cap\mathcal B(\alpha,h)=\emptyset$, let $S=\operatorname{diag}(h_1,\dots,h_k)$, and let $P$ be the orthogonal projection onto $S\,\R L$. Then $d\ge1$, and there is a probability measure $\rho$ on $\{m\in L:m\cdot\alpha\notin\Z\}$ with finite second moment and
\[
 \textstyle\int(Sm)(Sm)^T\,d\rho(m)\preceq\frac12P,
\]
where $\preceq$ is the order of positive semidefiniteness. In particular some $m\in L$ has $m\cdot\alpha\notin\Z$ and $\sum_jh_j^2m_j^2<\frac d2$.
\end{theorem}

At $\alpha=(\frac12,\dots,\frac12)$ the condition $m\cdot\alpha\notin\Z$ says that the coordinate sum of $m$ is odd.

\begin{lemma}\label{lem:window}
For $0<h<\frac12$ let $\psi_h(x)=\sqrt{2/h}\cos(\pi x/h)$ for $|x|\le\frac h2$ and $\psi_h(x)=0$ for $\frac h2<|x|\le\frac12$, extended with period $1$, and put $a_h(r)=|\widehat{\psi_h}(r)|^2$ and $R_h(x)=\int_\T\psi_h(y)\psi_h(y-x)\,dy$. Then
\[
 a_h(r)=\frac{8h\cos^2(\pi hr)}{\pi^2(1-4h^2r^2)^2}\quad(2h|r|\ne1),\qquad a_h(r)=\frac h2\quad(2h|r|=1),
\]
$a_h(-r)=a_h(r)\ge0$, $\sum_ra_h(r)=1$, $\sum_rr\,a_h(r)=0$, $\sum_rr^2a_h(r)=\frac1{4h^2}$ and $R_h(x)=\sum_ra_h(r)\ee(rx)$. Moreover $R_h\ge0$, $R_h(x)>0$ exactly when $\norm x<h$, and $R_h$ is continuously differentiable, with absolutely convergent Fourier series for $R_h$ and $R_h'$.
\end{lemma}

\begin{proof}
The formula is a direct integration. The function $\psi_h$ is real, even, nonnegative and continuous, with $\int\psi_h^2=1$ and $\int(\psi_h')^2=\pi^2/h^2$, so Parseval's identity gives the first and third sums, and symmetry gives the second. The Fourier coefficients of $R_h$ are $\widehat{\psi_h}(r)\overline{\widehat{\psi_h}(r)}=a_h(r)$. The integrand of $R_h$ is nonnegative, and the supports of $\psi_h(y)$ and $\psi_h(y-x)$ overlap in an interval of positive length exactly when $\norm x<h$. Finally $a_h(r)=O(r^{-4})$, so $\sum_r|r|a_h(r)<\infty$ and the series of $R_h$ may be differentiated termwise.
\end{proof}

Let $w(c)=\prod_ja_{h_j}(c_j)$ for $c\in\Z^k$. This is the probability mass function of a random vector $C$ with independent coordinates, mean $0$ and covariance $\operatorname{diag}\bigl((4h_1^2)^{-1},\dots,(4h_k^2)^{-1}\bigr)$, so that $SC$ has covariance $\frac14I$. The function
\[
 F(x)=\prod_{j=1}^kR_{h_j}(x_j-\alpha_j)=\sum_{c\in\Z^k}w(c)\,\ee(-c\cdot\alpha)\,\ee(c\cdot x)
\]
is nonnegative, positive exactly on $\mathcal B(\alpha,h)$, and continuously differentiable, and the Fourier series of $F$ and of its first partial derivatives converge absolutely.

\begin{lemma}[Moment cancellation]\label{lem:moments}
Suppose that $H\cap\mathcal B(\alpha,h)=\emptyset$. Let $Q=C+L$ be the coset of $C$ in $\Z^k/L$ and $U_C=\ee(-C\cdot\alpha)$. For every coset $q$ with $\mathbb P(Q=q)>0$,
\[
 \mathbb E[U_C\mid Q=q]=0,\qquad\mathbb E[C\,U_C\mid Q=q]=0.
\]
\end{lemma}

\begin{proof}
The function $F$ vanishes on $H$, where it attains its minimum, so $F$ and its first partial derivatives vanish on $H$. Fix $c_q\in q$ and integrate the Fourier series of $F$ and of $\partial F/\partial x_j$ against $\ee(-c_q\cdot x)$ with respect to the Haar probability measure of $H$. The integral of $\ee((c-c_q)\cdot x)$ over $H$ is $1$ for $c-c_q\in L$ and $0$ otherwise, so $\sum_{c\in q}w(c)\ee(-c\cdot\alpha)=0$ and $\sum_{c\in q}c_j\,w(c)\ee(-c\cdot\alpha)=0$ for every $j$.
\end{proof}

\begin{proof}[Proof of Theorem~\ref{thm:box}]
Choose $Q$ with its distribution and, given $Q=q$, choose $C$ and $D$ independently with the conditional law $w(c)/\mathbb P(Q=q)$ on $q$. Reweight this law by
\[
 \theta(C,D)=1-\operatorname{Re}\bigl(U_C\overline{U_D}\bigr)=1-\cos\bigl(2\pi(C-D)\cdot\alpha\bigr)\in[0,2].
\]
By Lemma~\ref{lem:moments}, $\mathbb E[\theta\mid Q=q]=1-|\mathbb E[U_C\mid Q=q]|^2=1$, so the reweighted law is a probability law in which $Q$ keeps its distribution. Let $\rho$ be the law of $M=C-D$ under it. Then $M\in L$, and $\theta>0$ forces $M\cdot\alpha\notin\Z$. In the expansion of $\theta\,(C-D)(C-D)^T$, every term containing $U_C\overline{U_D}$ has a factor $\mathbb E[\overline{U_D}\mid q]$, $\mathbb E[U_C\mid q]$, $\mathbb E[CU_C\mid q]$ or the conjugate of the last, all zero by Lemma~\ref{lem:moments}. Hence
\[
 \mathbb E\bigl[\theta\,(C-D)(C-D)^T\bigm|Q=q\bigr]=2\operatorname{Cov}(C\mid Q=q).
\]
Within a coset $C$ varies in $L$, so $\operatorname{Cov}(SC\mid Q=q)=P\operatorname{Cov}(SC\mid Q=q)P$. Let $A=\mathbb E\operatorname{Cov}(SC\mid Q)$. The law of total covariance gives $A\preceq\operatorname{Cov}(SC)=\frac14I$, and $A=PAP$ gives $A\preceq\frac14P$. So $\int(Sm)(Sm)^T\,d\rho=2A\preceq\frac12P$. If $d=0$, then $L=\{0\}$ and $\mathbb E[U_C\mid Q=L]=1$, contrary to Lemma~\ref{lem:moments}. Taking traces gives $\int\sum_jh_j^2m_j^2\,d\rho\le\frac12\operatorname{tr}P=\frac d2$.

For the strict inequality it suffices that $\sum_jh_j^2m_j^2$ is unbounded on the support of $\rho$. The coset $L$ contains $0$, so it has positive probability, and Lemma~\ref{lem:moments} gives $b\in L$ with $w(b)>0$ and $\beta=b\cdot\alpha\notin\Z$. Let $D_0$ be a positive integer with $D_0h_j\in\Z$ for every rational $h_j$. By the formula of Lemma~\ref{lem:window}, $a_h(r)>0$ whenever $hr\in\Z$ or $h$ is irrational, so $w(\ell D_0b)>0$ for every $\ell\ge0$. The pair $(\ell D_0b,b)$ has positive probability in the coset $L$, and positive weight whenever $(\ell D_0-1)\beta\notin\Z$. This holds for infinitely many $\ell$, since otherwise $D_0\beta\in\Z$ and then $\beta\in\Z$. So $\rho$ charges $(\ell D_0-1)b$ for infinitely many $\ell$. If every point of the support of $\rho$ had $\sum_jh_j^2m_j^2\ge\frac d2$, the integral would exceed $\frac d2$.
\end{proof}

Beck and Everett bound the squared norms of the two points separately. The first-moment identity of Lemma~\ref{lem:moments} is what removes the resulting factor $2$. By Wirtinger's inequality, among windows $\psi$ supported in an interval of length $h$ with $\int\psi^2=1$ the cosine minimizes $\sum_rr^2|\widehat\psi(r)|^2$, so changing the window cannot improve the covariance bound used in the proof. Taking the trace of the matrix bound against a positive semidefinite matrix $G$ gives a relation with $m\cdot\alpha\notin\Z$ and $(Sm)^TG(Sm)\le\frac12\operatorname{tr}(GP)$, which helps for unequal widths.

\begin{proof}[Proof of Theorem~\ref{thm:l2}]
Apply Theorem~\ref{thm:box} to the orbit $H=\{tv:t\in\R\}$, whose annihilator $L=\{m:m\cdot v=0\}$ has rank $k-1$, with $\alpha=(\frac12,\dots,\frac12)$ and $h_j=\frac12-\frac1{k+1}=\frac{k-1}{2(k+1)}$. Since $\norm{x-\frac12}=\frac12-\norm x$, the orbit misses $\mathcal B(\alpha,h)$ exactly when no $t$ has $\norm{tv_j}>\frac1{k+1}$ for all $j$. We get a harmful relation with $\eucl m^2<\frac{k-1}{2h^2}=B_k:=\frac{2(k+1)^2}{k-1}$. By the Cauchy--Schwarz inequality $\lone m<\sqrt{kB_k}$, and summing $(|m_j|-1)(|m_j|-2)\ge0$ over $j$ gives $3\lone m-2k\le\eucl m^2<B_k$, that is,
\[
 \lone m<\tfrac13\bigl(4k+6+\tfrac8{k-1}\bigr).
\]
Moreover $\lone m$ is odd, as it has the parity of the coordinate sum. For $k\ge9$ write $k=3q+r$ with $r\in\{0,1,2\}$. The bound is less than $4q+3$, $4q+4$ and $4q+5$ in the three cases, so $\lone m\le4q+1$, $4q+3$, $4q+3$, which is $2\lfloor\frac{2k+1}3\rfloor+1$. For $2\le k\le8$ the smaller of the two bounds gives $5,5,7,9,9,11,13$.
\end{proof}

The same argument with $\alpha=\frac12\mathbf 1-s$ and $h_j=\frac12-\delta_j$ treats shifted runners and unequal thresholds, with no smallness condition on the $\delta_j$, unlike Theorems~\ref{thm:relgen} and~\ref{thm:relshift}.

\begin{corollary}\label{cor:boxshift}
Let $v\in\Z^k\setminus\{0\}$, $s\in\R^k$ and $0<\delta_j<\frac12$. If no $t$ has $\norm{tv_j+s_j}>\delta_j$ for every $j$, then some $m\in\Z^k$ has $m\cdot v=0$, $m\cdot(\frac12\mathbf 1-s)\notin\Z$ and $\sum_j(1-2\delta_j)^2m_j^2<2(k-1)$.
\end{corollary}

Table~\ref{tab:l1} compares the bounds for $\lone m$. For $k\ge9$ Theorem~\ref{thm:rel} is stronger, also in the Euclidean norm, since $\eucl m^2=\lone m\le k<B_k$ for $m\in\{-1,0,1\}^k$. For $k\le8$ Theorem~\ref{thm:l2} is the best bound we have.

\begin{table}[H]
\centering\small
\begin{tabular}{r r r r}
\toprule
$k$ & \cite{BeckEverett}: $2k+3$ & Theorem~\ref{thm:l2} & Theorem~\ref{thm:rel}\\\midrule
$6$ & $15$ & $9$ & none\\
$8$ & $19$ & $13$ & none\\
$9$ & $21$ & $13$ & $9$\\
$11$ & $25$ & $15$ & $11$\\
$20$ & $43$ & $27$ & $19$\\
$100$ & $203$ & $135$ & $99$\\
\bottomrule
\end{tabular}
\caption{Upper bounds for $\lone m$, where $m$ is a harmful relation of a vector $v\in\Z^k$ with $\ML(v)\le\frac1{k+1}$ guaranteed by each result.}
\label{tab:l1}
\end{table}

\section{Few speeds and Schur triples}\label{sec:schur}
For small $n$, near-tight sets satisfy a much shorter relation than Theorem~\ref{thm:rel} provides. The earlier sections give the following, where ``primitive'' means that the greatest common divisor is $1$.

\begin{proposition}\label{prop:few}
\begin{enumerate}[label=\textup{(\alph*)},leftmargin=*]
\item A triple of positive integers with $\ML<\frac13$ either contains $a,b,a+b$ or is a dilate of $\{1,2,3s\}$ for some $s\ge1$ \textup{(the remark after Theorem~\ref{thm:second})}.
\item A quadruple of positive integers with $\ML<\frac14$ contains $a,b,a+b$. Indeed, by Theorem~\ref{thm:main} it is a dilate of $\{1,2,3,4k\}$, of $\{x,y,x+y,x+2y\}$, or of $\{1,3,4,14\}$.
\item A set of five positive integers with $\ML<\frac15$ contains $a,b,a+b$. Indeed, by Theorem~\ref{thm:fivefull} it is a dilate of $\{1,2,3,4,5s\}$, of $\{1,3,4,7,5s\}$, or of one of twelve sets, each of which contains such a triple.
\item A set of six positive integers with $\ML<\frac16$ contains $a,b,a+b$. Indeed, by Theorem~\ref{thm:six} it is a dilate of $\{1,2,3,4,5,6s\}$, of $\{1,3,4,5,9,6s\}$, of one of nine sets, each of which contains such a triple, or of a set $\{|A|,|B|,|A+B|,|2A+B|,|3A+B|,|3A+2B|\}$ with six distinct elements, in which the largest of $|A|$, $|B|$, $|A+B|$ is the sum of the other two.
\item For $3\le k\le6$ every tight instance contains $a,b,a+b$. Up to dilation the tight instances are $\{1,2,3\}$, then $\{1,2,3,4\}$ and $\{1,3,4,7\}$ \cite{Chen}, then $\{1,2,3,4,5\}$ and $\{1,3,4,5,9\}$ \cite[Theorem~3]{BHK}, see also \cite[Section~2.5]{Cordella}, then $\{1,2,3,4,5,6\}$ by Corollary~\ref{cor:six}(iii), see also \cite{Marshall}.
\end{enumerate}
\end{proposition}

\begin{proposition}\label{prop:computed}
For $n=7$ and $n=8$, every primitive set of $n$ positive integers bounded by $256$ with $\ML<1/n$ contains $a,b,a+b$.
\end{proposition}

The census of Section~\ref{sec:comp} gives exactly $63$ primitive near-tight sets of seven speeds and $113$ of eight speeds with largest element at most $256$ (Table~\ref{tab:census}), and each of them contains $a,b,a+b$. It reproduces the $23$ and $13$ sets that the first search of Section~\ref{sec:comp} finds with the bounds $34$ and $28$, among $5{,}359{,}838$ and $3{,}105{,}093$ primitive sets. The first search with the bounds $70$ and $46$ for $n=5$ and $n=6$ covers $11{,}742{,}956$ and $9{,}260{,}789$ primitive sets and finds $40$ and $74$ with $\ML<1/n$, the numbers that Theorems~\ref{thm:fivefull} and~\ref{thm:six} give for these ranges. The relation $a=2b$ is never needed for $4\le n\le8$ in these statements, while for $n=3$ it is needed exactly for the dilates of $\{1,2,3s\}$. The tight instances for $n=7,8$ in these ranges are $\{1,\dots,7\}$, $\{1,2,3,4,5,7,12\}$ \cite{GoddynWong}, $\{1,4,5,6,7,11,13\}$ and $\{1,\dots,8\}$, and each of them contains $1,4,5$ or $1,2,3$. According to the classification announced by Marshall \cite{Marshall}, the first three are all the tight instances for seven speeds, and Corollary~\ref{cor:M} proves this.

\begin{table}[ht]
\centering\small
\begin{tabular}{l r r}
\toprule
 & $n=7$ & $n=8$\\\midrule
near-tight sets, including dilates & $449$ & $473$\\
primitive near-tight sets & $63$ & $113$\\
primitive sets $\{1,\dots,6,7s\}$, resp.\ $U_i\cup\{8s\}$ & $36$ & $96$\\
primitive near-tight sets without $a,b,a+b$ & $0$ & $0$\\
primitive near-tight sets with $q-np\ne1$ & $4$ & $0$\\
\bottomrule
\end{tabular}
\caption{The census of near-tight sets of seven and eight speeds up to $256$, where $\frac pq$ is the value in lowest terms.}
\label{tab:census}
\end{table}

The four sets of seven speeds with $q-7p\ne1$ are $\{1,2,3,4,5,7,18\}$ and $\{1,3,4,5,7,13,18\}$ with value $\frac3{23}$, $\{1,3,4,5,7,11,30\}$ with value $\frac5{37}$, and $\{2,7,9,11,13,20,54\}$ with value $\frac9{65}$. The $27$ primitive near-tight sets of seven speeds and the $17$ of eight speeds outside the families have largest elements at most $112$ and $80$ (Figure~\ref{fig:census}). They give $C_7\ge27$ and, under \textup{(M)}, $C_8\ge17$ in Theorem~\ref{thm:count}. Whether these lists are complete would determine $C_7$ and $C_8$. For seven speeds the list is complete and $C_7=27$ (Theorem~\ref{thm:sevenfull}). Corollaries~\ref{cor:seven} and~\ref{cor:eight} leave a finite range beyond the census.

\begin{figure}[ht]
\centering
\begin{tikzpicture}
\begin{axis}[width=0.86\textwidth,height=5.2cm,xmin=0,xmax=256,ymin=0,ymax=30,
 xtick={0,64,128,192,256},ytick={0,10,20,30},xlabel={largest element $x$},ylabel={number of sets},
 tick label style={font=\small},label style={font=\small},
 legend pos=south east,legend cell align=left,legend style={font=\small,draw=none,fill=none}]
\addplot[const plot,thick] coordinates {(0,0) (11,2) (12,4) (13,5) (14,6) (16,7) (18,11) (20,12) (23,13) (24,15) (26,16) (28,17) (29,18) (30,19) (36,20) (39,22) (52,23) (54,24) (65,25) (71,26) (112,27) (256,27)};
\addlegendentry{seven speeds (27)}
\addplot[const plot,thick,dashed] coordinates {(0,0) (13,1) (19,2) (21,3) (24,4) (29,5) (32,11) (36,12) (40,14) (56,15) (62,16) (80,17) (256,17)};
\addlegendentry{eight speeds (17)}
\end{axis}
\end{tikzpicture}
\caption{The number of primitive near-tight sets of seven and eight speeds with largest element at most $x$ outside the families $\{1,\dots,6,7s\}$ and $U_i\cup\{8s\}$.}
\label{fig:census}
\end{figure}

This is the evidence for Conjecture~\ref{conj:schur}. The conjecture implies that every tight instance with $k\ge4$ contains $a,b,a+b$. It holds for $n=4,5,6$ by Proposition~\ref{prop:few}(b), (c), (d), and for $n=7$ by Theorem~\ref{thm:sevenfull}. The known infinite families of near-tight sets satisfy it, for instance Kravitz's sets $\{1,\dots,n-1,ns\}$ \cite[Theorem~5.1]{Kravitz} and the family $\{8,4s+3,4s+11,4s+19\}$ of Fan and Sun \cite{FanSun}, in which $8+(4s+3)=4s+11$. For $n\ge11$, Corollary~\ref{cor:rel}(b) proves the weaker statement with an odd signed relation in place of $a+b=c$.

By Corollary~\ref{cor:seven}(i), a primitive set of seven speeds with $\ML<\frac17$ and without $a,b,a+b$ has Euclidean length less than $5{,}145{,}525$, since the sets $\{1,\dots,6,7s\}$ contain $1,2,3$. Under \textup{(M)}, Corollary~\ref{cor:eight}(iii) gives the same for eight speeds with the bound $12{,}065{,}804{,}452{,}964$, since $U_1$ and $U_2$ contain $1,2,3$ and $U_3$ contains $1,4,5$. So Conjecture~\ref{conj:schur} for $n=7$ is a finite computation, although far beyond the census. Supplement~C carries it out (Theorem~\ref{thm:sevenfull}). For eight speeds the hypothesis \textup{(M)} is no longer needed, and the bound becomes $3{,}898{,}150{,}815$ (Corollaries~\ref{cor:M} and~\ref{cor:eightfull}).

The inequality in Conjecture~\ref{conj:schur} is strict for a reason that is visible only for four speeds.

\begin{proposition}\label{prop:equal}
\begin{enumerate}[label=\textup{(\alph*)},leftmargin=*]
\item A quadruple of positive integers with $\ML\le\frac14$ contains $a,b,a+b$ or is a dilate of $\{1,2,6,11\}$. The set $\{1,2,6,11\}$ has $\ML=\frac14$, it is dissociated, and it contains no $a,b,a+b$.
\item For $n=5$ and $n=6$, every set of $n$ positive integers with $\ML\le\frac1n$ contains $a,b,a+b$.
\item Every set of seven positive integers bounded by $160$ with $\ML\le\frac17$, and every set of eight positive integers bounded by $88$ with $\ML\le\frac18$, contains $a,b,a+b$.
\end{enumerate}
\end{proposition}
\begin{proof}
By Proposition~\ref{prop:few} it remains to consider the primitive sets with $\ML=\frac1n$. For $n\le6$ they are described in Theorem~\ref{thm:above}. A quadruple whose circle lies on $\Uone$ contains $a,2a,3a$, and one whose circle lies on $\Utwo$ is $\{|x|,|y|,|x+y|,|x+2y|\}$, in which the largest of $|x|$, $|y|$, $|x+y|$ is the sum of the other two. For five speeds the sets $\{a,2a,3a,4a,b\}$ and $\{a,3a,4a,7a,b\}$ contain $a,2a,3a$ and $a,3a,4a$. For six speeds a set on one of the three critical subtori contains $a,2a,3a$, or $a,3a,4a$, or is as in Proposition~\ref{prop:few}(d). Among the $29$, $64$ and $54$ further sets of Theorem~\ref{thm:above}, the only one without $a,b,a+b$ is $\{1,2,6,11\}$, which has no signed relation because $6>1+2$ and $11>1+2+6$. Part (c) is an exhaustive computation (Section~\ref{sec:comp}): there are $2{,}728$ primitive sets of seven speeds and $1{,}539$ of eight speeds with $\ML\le\frac1n$ in these ranges, and $49$ and $50$ of them have $\ML<\frac1n$.
\end{proof}

So $\ML<\frac14$ cannot be replaced by $\ML\le\frac14$ in Conjecture~\ref{conj:schur}, while for $5\le n\le8$ we know no set with $\ML=\frac1n$ and without $a,b,a+b$. For seven speeds, Theorem~C3.1 of Supplement~C covers all sets with $\ML\le\frac17$ and $\sqrt3\,\eucl V<29{,}893$. By (a) and Proposition~\ref{prop:few}(b), the least loneliness of a dissociated quadruple, and of a quadruple without odd signed relations, is exactly $\frac14$.

\section{Dissociated sets with small loneliness}\label{sec:fifth}
\begin{proof}[Proof of Proposition~\ref{prop:fifth}]
Write $a_i=5\cdot2^i-1$, so $a_0=4$ and $a_{i+1}=2a_i+1$. The set $V_j$ is superincreasing: $2>1$, $4>1+2$, and if $a_i$ exceeds the sum $S$ of the smaller elements, then $a_{i+1}=2a_i+1>a_i+S$. In a signed relation the largest element of the support would exceed the sum of all the others, so $V_j$ is dissociated.

Put $N=5\cdot2^j+1$ and $L=(2^j+1)/N$. Then $\frac15<L\le\frac13$ and $5L=1+4/N$.

\emph{Lower bound.} Let $t_0=2^{j+1}/N<\frac12$. Since $5\cdot2^j\equiv-1\pmod N$, we have $a_it_0=(5\cdot2^{i+j+1}-2^{j+1})/N\equiv-(2^{i+1}+2^{j+1})/N$. For $i=j$ this gives $\norm{a_jt_0}=1-2^{j+2}/N=L$, and $\norm{2t_0}=L$ in the same way. For $i<j$ the numerator $2^{i+1}+2^{j+1}$ lies in $[2^{j+1}+2,\,3\cdot2^j]$, so $\norm{a_it_0}\ge\min(2^{j+1}+2,\,N-3\cdot2^j)/N>L$, and $\norm{t_0}=2^{j+1}/N\ge L$. Hence $\ML(V_j)\ge L$.

\emph{Upper bound.} By the symmetries $t\mapsto-t$ and $t\mapsto t+1$ it suffices to consider $t\in[0,\frac12]$. Suppose that $\norm{vt}>L$ for all $v\in V_j$. We derive a contradiction. From $\norm t>L$ we get $t>L$, and then $\norm{2t}>L$ gives $2t<1-L$. So $4t\in(4L,2-2L)$, and since $4L>1-L$, the condition $\norm{4t}>L$ forces $4t\in(1+L,2-L)$. Hence
\[
 \frac{1+L}4<t<\frac{1-L}2 .
\]
For $j=0$ this interval is empty. Let $j\ge1$ and put $\eta=2-5t$, so that $2/N<\eta<\frac12-\frac1N$. For $0\le i\le j$,
\[
 a_it=2^i(2-\eta)-t\equiv-(2^i\eta+t)\pmod 1,\qquad\text{so}\qquad\norm{a_it}=\norm{2^i\eta+t}.
\]
Let $\theta=1-L-t$. The inequality $t>(1+L)/4$ is equivalent to $\eta<\theta$. Using $5L=1+4/N$, the inequality $2^j\eta>\theta$ is equivalent to $\eta(5\cdot2^j-1)>2(5\cdot2^j-1)/N$, that is to $\eta>2/N$. So $i^*=\min\{i:2^i\eta\ge\theta\}$ satisfies $1\le i^*\le j$ and $\theta\le2^{i^*}\eta<2\theta$.

If $2^{i^*}\eta\le\theta+2L$, then $2^{i^*}\eta+t\in[1-L,1+L]$ and $\norm{a_{i^*}t}\le L$, a contradiction.

Otherwise put $\zeta=1-2^{i^*}\eta$. From $2^{i^*}\eta<2\theta=2-2L-2t$ and $t>\frac{1+L}4>\frac12-L$ we get $\zeta>0$, and from $2^{i^*}\eta>\theta+2L$ we get $\zeta<t-L$. For $m\ge0$,
\[
 2^{i^*+m}\eta+t=2^m-2^m\zeta+t\equiv t-2^m\zeta,\qquad\text{so}\qquad\norm{a_{i^*+m}t}=\norm{t-2^m\zeta}.
\]
Let $m^*=\min\{m\ge1:2^m\zeta\ge t-L\}$. Then $t-L\le2^{m^*}\zeta<2(t-L)$, so $2L-t<t-2^{m^*}\zeta\le L$, and $2L-t>-L$ because $t<\frac12<3L$. Hence $\norm{a_{i^*+m^*}t}\le L$, a contradiction as soon as $i^*+m^*\le j$. We check this.

The present case needs $\theta+2L<2\theta$, that is $t<1-3L$, that is $\eta>15L-3=12/N$. Then $2^{j-1}\eta>6\cdot2^j/N>1>\theta$, so $i^*\le j-1$.

If $i^*\ge2$, then $\eta<\theta/2$, which reads $t>(3+L)/9$. So $\zeta>1-2\theta=2L+2t-1>(20L-3)/9>\frac19$ and $2\zeta>\frac29>\frac15>\frac{1-3L}2>t-L$. Thus $m^*=1$ and $i^*+m^*\le j$.

If $i^*=1$, then $j\ge2$ and $\zeta=1-2\eta$. The function $F(\eta)=2^{j-1}(1-2\eta)-\bigl(\frac{2-\eta}5-L\bigr)$ is decreasing, and $F(\frac12-\frac1N)=(5\cdot2^j+5)/(10N)>0$. So $2^{j-1}\zeta>t-L$ for all admissible $\eta$, which gives $m^*\le j-1$ and $i^*+m^*\le j$.
\end{proof}

The minimum is attained at $t_0=2^{j+1}/N$, by the runners $2$ and $5\cdot2^j-1$. For $j\le8$ the formula was also confirmed by exact computation. For small $k$ the sets $V_{k-3}$ are extremal. Among the sets of $k$ positive integers without odd signed relations and with largest element at most $B$, the least loneliness is $\frac27$ for $k=3$ (at $\{1,2,6\}$, $B=60$), $\frac14$ for $k=4$ (at $\{1,2,6,11\}$, $B=50$), $\ML(V_2)=\frac5{21}$ for $k=5$ and $\ML(V_3)=\frac9{41}$ for $k=6$, and $V_2$ and $V_3$ are the only minimizers up to dilation for $B=250$ and $B=120$. For $k\ge7$ this fails. Let $D_k$ be the infimum of $\ML(V)$ over dissociated sets $V$ of $k$ positive integers. Table~\ref{tab:dissociated} lists the least values we found. The entry for $k=7$ is the minimum over all $7$-subsets of $\{1,\dots,100\}$ without odd signed relations, and it is attained only once. The entries for $k\ge8$ come from restricted searches and give upper bounds for $D_k$. Every set in the table is superincreasing, hence dissociated. Each value $x$ in the table is checked as follows. Starting from $[0,\frac12]$, intersect, for each $v\in V$ in turn, the set obtained so far with the union of the closed intervals $[(m+x)/v,(m+1-x)/v]$, $0\le m\le\lfloor v/2\rfloor$. The result is the set of $t\in[0,\frac12]$ with $\min_{v\in V}\norm{vt}\ge x$, and for every row it is the single time in the last column. So $\ML(V)=x$, since a larger value would leave an interval. For the ten-speed row the numbers of intervals after the successive steps are $1,1,2,3,4,5,5,7,3,1$, and at the final time $\frac{62}{209}$ the speeds $4$ and $623$ are at distance $\frac{39}{209}$ from an integer. Lemma~\ref{lem:corner} gives the same values. In particular $D_{10}\le\frac{39}{209}<\frac3{16}$. For comparison, the set $\{1,2,4,9,19,36,73,179,322,660\}$ has no odd signed relation and has loneliness $\frac4{21}$.

\begin{table}[H]
\centering\small
\begin{tabular}{r l l l}
\toprule
$k$ & set & $\ML$ & maximizer\\\midrule
$7$ & $\{1,2,4,9,19,36,76\}$ & $8/39\approx0.2051$ & $31/78$\\
$8$ & $A\cup\{143\}$ & $29/147\approx0.1973$ & $44/147$\\
$9$ & $A\cup\{140,298\}$ & $29/151\approx0.1921$ & $45/151$\\
$10$ & $B$ & $39/209\approx0.1866$ & $62/209$\\
$11$ & $B\cup\{1231\}$ & $229/1235\approx0.1854$ & $366/1235$\\
$12$ & $B\cup\{1231,2420\}$ & $229/1239\approx0.1848$ & $436/1239$\\
$13$ & $C$ & $81/439\approx0.1845$ & $130/439$\\
$14$ & $C\cup\{9689\}$ & $1787/9693\approx0.1844$ & $2870/9693$\\
$15$ & $C\cup\{9689,19417\}$ & $3579/19421\approx0.1843$ & $5750/19421$\\
\bottomrule
\end{tabular}
\caption{Dissociated sets with loneliness below $\ML(V_{k-3})$, where $A=\{1,2,4,8,18,35,69\}$, $B=A\cup\{140,282,623\}$ and $C=B\cup\{1204,2393,4825\}$. The last column is the unique maximizing time in $[0,\frac12]$. For $k\ge8$ the values are below $\frac15$.}
\label{tab:dissociated}
\end{table}

Every nonempty finite set $V$ of positive integers has an extension $V\cup\{w\}$ with $w>\sum_{v\in V}v$ and $\ML(V\cup\{w\})<\ML(V)$. Indeed, $f_V(t)=\min_{v\in V}\norm{vt}$ is piecewise linear with nonzero slopes, so its maximum $\ML(V)>0$ is attained at finitely many rational times. Choose $w>\sum_{v\in V}v$ divisible by all their denominators. Then $\norm{wt}=0$ at every maximizer of $f_V$, and the maximum of the continuous function $\min(f_V(t),\norm{wt})$ is smaller than $\ML(V)$. Every signed sum of elements of $V$ has absolute value less than $w$, so $V\cup\{w\}$ is dissociated if $V$ is, and superincreasing if $V$ is. Starting from the last row of Table~\ref{tab:dissociated}, this gives $D_k<\frac{3579}{19421}$ for every $k>15$.

\begin{question}\label{q:fifth}
Is $\inf\{\ML(V):V\text{ dissociated}\}$ positive? The same question for sets without odd signed relations.
\end{question}

By Corollary~\ref{cor:rel}(a) and Proposition~\ref{prop:fifth}, $\pi^{-1}\arctan(k^{-1/2})\le D_k\le(2^{k-3}+1)/(5\cdot2^{k-3}+1)$ for $k\ge3$. Theorem~\ref{thm:orbit} improves the lower bound slightly, Table~\ref{tab:dissociated} improves the upper bound for $7\le k\le15$, and the extensions above give $D_k<\frac{3579}{19421}$ for $k>15$. A positive answer to Question~\ref{q:fifth} would need more than a certificate supported on $\{-1,0,1\}^k$, by Theorem~\ref{thm:relopt}(b). For lacunary sets the answer is positive.

\begin{proposition}\label{prop:lacunary}
Let $v_1<\dots<v_k$ be positive integers, and let $r\ge2$ and $G>1$ be such that $v_i\ge Gv_{i-r}$ for $r<i\le k$. If $\delta,h>0$ satisfy
\[
 2(\delta+h)\Bigl(1+\frac2{Gh}\Bigr)\le\frac1r\Bigl(1-\frac1r\Bigr)^{r-1},
\]
then for all $s_1,\dots,s_k\in\T$ the set of $t\in\T$ with $\norm{v_it+s_i}\ge\delta$ for all $i$ has measure at least $(1-\frac1r)^k$. For an infinite sequence with $v_i\ge Gv_{i-r}$ for all $i>r$, the same inequality gives a time $t$ with $\norm{v_it+s_i}\ge\delta$ for all $i$.
\end{proposition}
\begin{proof}
For each $i$ let $L_i=2^{-\lceil\log_2(v_i/h)\rceil}$, so that $h/(2v_i)<L_i\le h/v_i$, and let $A_i$ be the union of the closed dyadic intervals of length $L_i$ that meet $\{t:\norm{v_it+s_i}<\delta\}$. On such an interval $v_it+s_i$ varies by at most $v_iL_i\le h$, so $A_i\subseteq\{t:\norm{v_it+s_i}\le\delta+h\}$. This set has period $\frac1{v_i}$ and measure $2(\delta+h)$, so its measure in an interval of length $L$ is at most $2(\delta+h)(L+\frac1{v_i})$. For $i>r$ a dyadic interval of length $L_{i-r}$ has $v_iL_{i-r}>v_ih/(2v_{i-r})\ge Gh/2$, so $A_i$ fills at most the fraction $p=2(\delta+h)(1+\frac2{Gh})$ of it. For $i\le r$, $|A_i|\le2(\delta+h)\le p$.

Put $G_i=\T\setminus(A_1\cup\dots\cup A_i)$ for $i\ge1$ and $G_i=\T$ for $i\le0$. For $i>r$ the set $G_{i-r}$ is a union of dyadic intervals of length $L_{i-r}$, because $L_j\ge L_{i-r}$ for $j\le i-r$, so $|A_i\cap G_{i-1}|\le|A_i\cap G_{i-r}|\le p\,|G_{i-r}|$. This also holds for $i\le r$. Put $z=\frac1r$, so that $p\le z(1-z)^{r-1}$. We show $|G_i|\ge(1-z)|G_{i-1}|$ by induction on $i$. If it holds below $i$, then $|G_{i-r}|\le(1-z)^{-(r-1)}|G_{i-1}|$, also when $i\le r$, and
\[
 |G_i|=|G_{i-1}|-|A_i\cap G_{i-1}|\ge|G_{i-1}|-z(1-z)^{r-1}|G_{i-r}|\ge(1-z)|G_{i-1}|.
\]
So $|G_k|\ge(1-\frac1r)^k$, and every $t\in G_k$ has $\norm{v_it+s_i}\ge\delta$ for all $i$. For an infinite sequence the sets $\{t:\norm{v_it+s_i}\ge\delta\text{ for }i\le k\}$ are closed, nonempty and decreasing in $k$, so they have a common point.
\end{proof}

\begin{corollary}\label{cor:lacunary}
\begin{enumerate}[label=\textup{(\alph*)},leftmargin=*]
\item For every finite or infinite superincreasing sequence $v_1<v_2<\cdots$ and all $s_i\in\T$ there is a $t$ with $\norm{v_it+s_i}\ge\frac1{97}$ for all $i$.
\item If $V_1,V_2,\dots$ are finite sets of positive integers with $\ML(V_j)\to0$, then the largest number of elements of $V_j$ in an interval $[x,2x)$ tends to infinity.
\end{enumerate}
\end{corollary}
\begin{proof}
(a) The partial sums of a superincreasing sequence more than double at every step, so $v_i>2^{14}v_{i-15}$. Take $r=15$, $G=2^{14}$, $\delta=\frac1{97}$ and $h=\frac\delta9$ in Proposition~\ref{prop:lacunary}. The condition reads $\frac{20}9\bigl(\frac1{97}+\frac9{8192}\bigr)\le\frac1{15}\bigl(\frac{14}{15}\bigr)^{14}$, and the left side is below $0.02536$ while the right side is above $0.02537$.

(b) If every interval $[x,2x)$ contains at most $D$ elements of $V$, then $v_{i+D}\ge2v_i$ for the elements $v_1<v_2<\cdots$ of $V$. Apply Proposition~\ref{prop:lacunary} with $r=mD$, $G=2^m$ and $h=G^{-1/2}$. For $\delta=0$ the left side of the condition is $2(2^{-m/2}+2^{1-m})$, which decays exponentially in $m$, while the right side is asymptotic to $1/(emD)$. So for fixed $D$ and large $m$ the condition holds for some $\delta=\delta(D)>0$, and every such $V$ has $\ML(V)\ge\delta(D)$.
\end{proof}

The dyadic argument is the method of Peres and Schlag for lacunary sequences \cite{PeresSchlag}. By (b), dissociated sets whose loneliness tends to $0$, if they exist, have unboundedly many elements in some interval $[x,2x)$. A superincreasing set has at most two elements in every such interval, in line with (a). The question for general dissociated sets remains open.

Corollary~\ref{cor:lacunary}(b) is a qualitative form of a theorem of Peres and Schlag: if $n_{j+M}>2n_j$ for all $j$, where $M\ge4$, then some $\theta$ has $\norm{n_j\theta}\ge1/(240M\log_2M)$ for all $j$ \cite[Theorem~3.1]{PeresSchlag}. A second necessary feature is spread. If $V\subseteq[a,Ka]$, then the time $t=\frac1{(K+1)a}$ puts every $vt$ in $[\frac1{K+1},\frac K{K+1}]$, so $\ML(V)\ge\frac1{K+1}$. Thus a set with $\ML(V)<\delta$ has $\max V>(\frac1\delta-1)\min V$.

The next proposition restates Question~\ref{q:fifth} as a question about infinite sets.

\begin{proposition}\label{prop:qrec}
The following are equivalent.
\begin{enumerate}[label=\textup{(\roman*)},leftmargin=*]
\item $\inf\{\ML(V):V\text{ a finite dissociated set of positive integers}\}=0$.
\item There is an infinite dissociated set $E$ of positive integers that is a set of recurrence for every rotation of the circle, that is, $\inf_{n\in E}\norm{nt}=0$ for every $t\in\R$.
\item For every $\delta>0$ there are $d\ge1$ and a finite set $A\subseteq\Z^d\setminus\{0\}$ without signed relations such that every $\theta\in\T^d$ has $\norm{a\cdot\theta}<\delta$ for some $a\in A$.
\end{enumerate}
The same holds with ``dissociated'' and ``without signed relations'' replaced by ``without odd signed relations''.
\end{proposition}
\begin{proof}
(i)$\Rightarrow$(ii). Choose finite dissociated sets $V_1,V_2,\dots$ with $\ML(V_j)<\frac1j$, put $M_1=1$, choose integers $M_j>\sum_{i<j}M_i\sum_{v\in V_i}v$, and let $E=\bigcup_jM_jV_j$. Every element of $M_jV_j$ is larger than all elements of the earlier blocks. Suppose that a signed relation on $E$ involves elements of $M_jV_j$ and of no later block. Its part in $M_jV_j$ is $M_js$, where $s$ is a signed sum of elements of $V_j$, and the remaining part has absolute value less than $M_j$. So $s=0$, and since $V_j$ is dissociated all coefficients on $M_jV_j$ vanish, a contradiction. Hence $E$ is dissociated. For every $t$ and $j$ some $v\in V_j$ has $\norm{vM_jt}\le\ML(V_j)<\frac1j$.

(ii)$\Rightarrow$(i). Let $\delta>0$. The open sets $\{t\in\T:\norm{nt}<\delta\}$, $n\in E$, cover $\T$, so finitely many of them do. The corresponding finite set $F\subseteq E$ is dissociated, and $\ML(F)<\delta$ because the maximum in the definition of $\ML$ is attained.

(i)$\Rightarrow$(iii) holds with $d=1$. (iii)$\Rightarrow$(i). Let $A$ be as in (iii), choose an integer $R>2\sum_{a\in A}\supn a$, put $w=(1,R,\dots,R^{d-1})$ and $F=\{\lvert a\cdot w\rvert:a\in A\}$. If $\varepsilon_a\in\{-1,0,1\}$ and $b=\sum_a\varepsilon_aa$, then $\supn b<R/2$, so $b\cdot w=0$ forces $b=0$ and then all $\varepsilon_a=0$. Hence the numbers $\lvert a\cdot w\rvert$ are positive and pairwise distinct, and $F$ is a dissociated set with $\lvert F\rvert=\lvert A\rvert$. For every $t$, the point $\theta=tw\in\T^d$ gives some $a$ with $\norm{(a\cdot w)t}=\norm{a\cdot\theta}<\delta$, so $\ML(F)<\delta$.

For sets without odd signed relations the same constructions work. In (i)$\Rightarrow$(ii), the argument above shows that a signed relation on $E$ restricts to a signed relation on the top block $M_jV_j$, possibly trivial, and the rest is a signed relation on the earlier blocks. By induction the relation splits into signed relations on the blocks, and if its support is odd, one of them has odd support. In (iii)$\Rightarrow$(i), a signed relation on $F$ gives one on $A$ with the same support.
\end{proof}

In harmonic analysis our dissociated sets are called quasi-independent, and they are Sidon sets \cite{LopezRoss,GrahamHare}. Let $b\Z$ be the Bohr compactification of $\Z$.

\begin{corollary}\label{cor:bohr}
Suppose that there is $c>0$ with $\ML(V)\ge c$ for every finite dissociated set $V$. Then no finite union of dissociated subsets of $\Z\setminus\{0\}$ has $0$ in its closure in $b\Z$. In particular no dissociated set is dense in $b\Z$, and the Cayley graph on $\Z$ generated by a dissociated set has chromatic number at most $\lceil1/c\rceil$.
\end{corollary}
\begin{proof}
Since $\norm{-nt}=\norm{nt}$ and a dissociated set does not contain both $n$ and $-n$, the hypothesis also holds for finite dissociated subsets of $\Z\setminus\{0\}$. Let $E$ be dissociated, and let $F_1\subseteq F_2\subseteq\cdots$ be finite sets with union $E$. Choose $t_m$ with $\norm{nt_m}\ge c$ for all $n\in F_m$, and let $t$ be a limit point of the $t_m$ in $\T$. Then $\norm{nt}\ge c$ for all $n\in E$, so the Bohr neighbourhood $\{n\in\Z:\norm{nt}<c\}$ of $0$ misses $E$. For a finite union of dissociated sets, intersect the corresponding neighbourhoods. A dissociated set does not contain $0$, so it is not dense in $b\Z$. For the chromatic number, give $m\in\Z$ the colour $\lfloor K\{mt\}\rfloor$ with $K=\lceil1/c\rceil$, where $\{x\}$ is the fractional part. If $m$ and $m+n$ with $n\in E$ had the same colour, then $\norm{nt}<\frac1K\le c$.
\end{proof}

Whether a Sidon subset of $\Z$ can be dense in $b\Z$ is a long-standing open problem \cite{LopezRoss,RamseyBohr,FQQ}. By Corollary~\ref{cor:bohr}, a positive answer to Question~\ref{q:fifth} would show that no quasi-independent set is dense in $b\Z$, and, if every Sidon set is a finite union of quasi-independent sets, which is also open for $\Z$ \cite{Lewko}, that no Sidon subset of $\Z$ is dense in $b\Z$. Ramsey \cite{RamseyBohr} showed that a quasi-independent set with $0$ in its Bohr closure yields a quasi-independent set that is dense in $b\Z$, so for quasi-independent sets these two properties stand or fall together. Conversely, a quasi-independent set dense in $b\Z$ satisfies (ii) of Proposition~\ref{prop:qrec} and answers Question~\ref{q:fifth} negatively. Condition (ii) is weaker than having $0$ in the Bohr closure, which asks for $n\in E$ with $\norm{nt_1},\dots,\norm{nt_r}$ all small for every finite set of times. The corollary would also bound the chromatic numbers of these Cayley graphs, which Erdős asked about for lacunary sequences and which Katznelson related to recurrence \cite{Katznelson,PeresSchlag}. Corollary~\ref{cor:bohr} and its proof carry over to sets without odd signed relations, but such sets need not be Sidon: the odd integers have no odd signed relation.

\section{Computations and reproducibility}\label{sec:comp}
The main enumeration programs are written in C. The four-speed programs use exact $64$-bit integer arithmetic, and the pruning tests for the five- and six-speed trees of Section~\ref{sec:tree} use $128$-bit integer arithmetic. Floating point arithmetic only proposes candidates and sets loop bounds, as described below, and every candidate is then tested exactly. The Hermite normal form enumeration sets its loop bounds by floating point square roots with a margin of one on each side. Standard-library Python scripts independently check the formulas, certificates and enumeration outputs on the stated ranges. For the five-speed enumeration of Section~\ref{sec:five}, we also compare the C output with a separate Python box search up to squared covolume $3{,}000$.

\emph{Five speeds up to $130$.} The program \texttt{five130} reproduces Cordella's bounded search \cite[Section~6.2]{Cordella}. By Lemma~\ref{lem:corner}, for speeds at most $130$ the maximum occurs at a rational time of denominator at most $260$. The program tests these times for each of the $\binom{130}5=286{,}243{,}776$ sets of five distinct speeds at most $130$, and for the $296$ sets in which no time puts every speed at distance at least $\frac15$ it computes $\ML=p/q$ in lowest terms. Every one satisfies $q-5p=1$. Of the $64$ primitive sets among them, $52$ lie on the two critical products and $12$ lie on neither. These twelve are the sets of Theorem~\ref{thm:fivefull}(c), and the $64$ primitive sets are exactly the sets with largest speed at most $130$ in the output of the independent exhaustive search described below.

\emph{Exact loneliness of a two-dimensional subtorus.} For forms $\ell_i(X,Y)=s_iX+r_iY$, the maximum of $\min_i\norm{\ell_i}$ over $\T^2$ is attained at a point where three signed linear branches $\sigma_i(\ell_i-k_i)$ are equal, or where two proportional forms are at a common level and a third form is at a half-integer, or where two independent forms are at half-integers. Every such point is the solution of a $3\times3$ linear system with integer right-hand side in a bounded range. The program enumerates all of them, evaluates the minimum exactly at each, and returns the best point. It agrees with an independent implementation and with a fine grid search on several hundred random subtori with small coefficients. In the proof only the returned points are used, as certificates.

\emph{Enumeration of two-dimensional subtori.} Hermite normal forms for case~(i) of Lemma~\ref{lem:reduce}, and primitive $n\in\Z^3$ for case~(ii). Running times are ten and three seconds. The enumeration of Proposition~\ref{prop:third-enum}, with covolume up to $57$, the special-subtorus test and the threshold $1/3$, runs in $54$ seconds. The enumeration of Proposition~\ref{prop:enum25}, with squared covolume up to $11{,}331$, the tests for the $40$ special characters and the $176$ characters with $|m|_1=5$ up to sign, and the threshold $\frac25$, ran in two parallel jobs of $21$ minutes. When the grids $Q\in\{2,10,20,30\}$ give no point with $G\ge\frac25$, the program calls the exact maximization, in a second implementation of the same candidate enumeration in which the integer offsets of a base pair of independent forms run only over a fundamental domain of the lattice they generate. This happened for $44{,}839{,}680$ lattices, of which $18{,}960$ had $\ML<\frac25$. The two implementations agree on all $916{,}492$ coordinate-free subtori of covolume at most $20$, on $19{,}300$ further subtori with larger coefficients, and on all $18{,}960$ survivors. The formula in the proof of Lemma~\ref{lem:five} was compared with the C maximizer for all subtori $U_n$ of the five tori with $\supn n\le20$, $142{,}547$ subtori, with no difference, and every value in $[\frac13,\frac25)$ among them lies in $\mathcal R$.

\emph{Section~\ref{sec:obstr}.} The exact maximizer confirms Theorem~\ref{thm:obstr}(i) on $255$ random admissible subtori with $N\le126$. The loneliness of every repeated-speed circle was computed by Lemma~\ref{lem:corner} for $10{,}958$ admissible subtori with $N\le196$, $58$ of them with all three greatest common divisors above $1$, and for the examples $m\le4$, $K\le2$, in agreement with parts (ii) and (iii). The formula of Theorem~\ref{thm:fourspeed} was checked for all $4\mid N\le124$ and $\ell\le6$, and the comparison for $(N,\ell)=(60,61)$ and $(120,121)$. Proposition~\ref{prop:coal} was checked on $300$ random subtori with $\gcd(B,S)=1$ against $\ML(S,B,HS)$, and its independence of $A$ and $C$ on $1{,}440$ groups of parameters. Theorem~\ref{thm:phase} was checked on $5{,}460$ subtori with $N\le100$ and $H=\frac N2+jN$, $j\in\{0,1,2,5,9\}$, including the triple values of the second phase in all $3{,}651$ cases there with $g\mid\frac N2$.

\emph{Section~\ref{sec:five}.} The program \texttt{g5enum2} runs through the vectors $c=(0,c_2,\dots,c_5)$ with $\sum_{i<j}(q_ic_j-q_jc_i)^2\le765{,}275$ by nested loops whose ranges come from the partial sums of this positive definite quadratic form, keeps those with $\gcd(c)=1$ and first nonzero entry positive, and tests all directions $r_i\mp r_j$. For the eight vectors $q$ together it visits $65{,}989{,}492{,}206$ vectors, of which $30{,}485{,}136{,}454$ are normalized, in two parallel jobs of $47$ and $40$ minutes, and exactly $16$ pass. They are the products of Theorem~\ref{thm:five}, with squared covolume $30$ or $75$. A Python search over a box, without the pruning, reproduces the number of normalized vectors and the survivors for all eight $q$ with the bound $3{,}000$ in place of $765{,}275$. The program \texttt{torusmlN} is the exact maximizer \texttt{torusml2} extended to any number of forms. It agrees with \texttt{torusml2} on $289$ random inputs with four forms and with a separate Python implementation on $80$ inputs with five forms, and it confirms $\ML=\frac15$ for the $16$ survivors. It also computes $\ML(\Z q+\Z c)$ for all $8{,}960$ normalized $c$ with $|c_i|\le3$: the $14$ with $\ML<\frac29$ all pass the test and have $\ML=\frac15$. On $20{,}000$ random primitive planes in $\Z^5$ with coordinates in $[-3,3]$ it finds no value in $(\frac15,\frac29)$.

\emph{Section~\ref{sec:full}.} The program \texttt{lonely\_s24\_verify} enumerates the $2\times n$ integer matrices with entries of absolute value at most $K$ whose columns $c_1,\dots,c_n$ generate $\Z^2$ and have the property of Lemma~\ref{lem:zon}(i), up to the signs and the order of the columns. For each of them it computes $\ML(U)$ by Lemma~\ref{lem:zon}(ii), runs through every minimizer $t^*$, every edge through it and every $Q$ that satisfies \eqref{eq:Q}, and compares $\ML(g,N-g,d)$, computed by Lemma~\ref{lem:corner}, with $\ML(U)$. For $n=4$ and $K=8$ this covers $14{,}098{,}685$ subtori, of which $11{,}416{,}603$ have $\ML<\frac12$, with $13{,}722{,}416$ edges and $16{,}149{,}274$ triples, and $N$ up to $353$. For $n=5$, $K=4$ and for $n=6$, $K=2$ it covers $985{,}153$ and $14{,}021$ subtori, and a random sample with $n=4$ and $K=60$ covers $305{,}134$ subtori with $N$ up to $17{,}752$. The two sides agree in every case. A third Python script, \texttt{verify\_sections\_12\_14.py}, which again uses only the standard library and exact arithmetic, covers Sections~\ref{sec:full}, \ref{sec:three} and~\ref{sec:tree}. It repeats the comparison, together with (a) and (b) of Proposition~\ref{prop:contact}, in an independent implementation for the $949$ subtori with $n=4$, $K=2$ and with $n=5$, $K=1$. It compares Lemma~\ref{lem:zon}(ii) with its exact two-dimensional maximizer on $60$ random subtori of $\T^4$, checks Lemma~\ref{lem:three} for all $1{,}574$ triples $(N,k,Q)$ with $N<90$ and $0\le Q<N$ that satisfy (a) and (b), checks that the value of each of the $9{,}389$ primitive triples with entries at most $40$ is $\frac12$ or satisfies Corollary~\ref{cor:arith}(iii) with $N\le120$, and checks the examples of Remarks~\ref{rem:read} and~\ref{rem:strict}.

\emph{Section~\ref{sec:three}.} The program \texttt{three\_check} decides by a direct search for $Q$ whether a pair $(N,k)$ is admissible. Among the $43{,}890$ pairs with $k<N\le420$ and $N\equiv k\pmod2$ exactly $5{,}965$ are admissible, among the $122{,}150$ pairs with $N\le700$ exactly $13{,}735$, and in both ranges they are the pairs of the form \eqref{eq:star}. The same program computes $\ML$ by Lemma~\ref{lem:corner} for the $8{,}827{,}132$ primitive triples of distinct positive integers at most $400$. They take $10{,}184$ values below $\frac12$. The $3{,}384$ values that \eqref{eq:star} gives for $N\le400$ are among them, and each of the $10{,}184$ values comes from a pair of the form \eqref{eq:star} with $N\le800$. Of these values $266$ lie below $\frac13$, $422$ between $\frac13$ and $\frac25$ and $552$ between $\frac25$ and $\frac37$, and they agree with the lists of Corollary~\ref{cor:windows}(a), (b), (c): every such value is in the list, and every number of the lists with denominator at most $400$ occurs. The script \texttt{verify\_sections\_12\_14.py} checks this from the output file. It repeats the direct search for $N\le150$ ($5{,}550$ pairs, $1{,}109$ of them admissible) and the comparison with the lists for the $8{,}410$ primitive triples of distinct integers at most $40$. It checks Proposition~\ref{prop:witness} with Lemma~\ref{lem:corner} for all parameters with $N\le110$, which are $1{,}097$ triples in (i) and $1{,}242$ in (ii). It lists the quadruples in the proof of Corollary~\ref{cor:windows} and the finite sets of part (ii) for $h\le13$, and it checks the values named in Remark~\ref{rem:level37} and the two subgroups of Remark~\ref{rem:five}.

\emph{Section~\ref{sec:tree}.} The program \texttt{lrtree} builds the tree depth first. A node of depth $k$ is stored through a basis of $\Lambda=\Lambda_{T(L)}$, its Gram matrix $\Gamma$ and a basis of $L$. A child $L+\Z h$ corresponds to the primitive vector $c\in\Z^{n-k}$ of the values of $h$ on the basis of $\Lambda$, up to sign, and $\eucl{\pi_L(h)}^2=c^{\mathsf T}\Gamma^{-1}c$. The centre of a ball has coordinates in $2^{-40}\Z$ with respect to the basis, and its radius is the largest multiple of $2^{-30}$ for which \eqref{eq:alive} holds with the weights $w_j$ rounded up, which is decided in integer arithmetic because the numbers $w_j^2$ are rational. A centre for which this radius is zero is discarded, so every ball has positive radius, as Proposition~\ref{prop:tree} requires. Conditions (R) and (S) are evaluated in $128$-bit integer arithmetic, with $\norm{h\cdot x}$ rounded up to a multiple of $2^{-30}$, so a child is discarded only if it violates (R) or (S). The children are enumerated exactly. With $\Gamma_{[i]}$ the leading block of $\Gamma$ of size $i$ and $c_{[i]}=(c_1,\dots,c_i)$, the number $c_{[i]}^{\mathsf T}\Gamma_{[i]}^{-1}c_{[i]}$ is the minimum of the quadratic form $c^{\mathsf T}\Gamma^{-1}c$ over the remaining real coordinates, so nested loops over $c_1,\dots,c_{n-k}$ with the tests $c_{[i]}^{\mathsf T}\Gamma_{[i]}^{-1}c_{[i]}<C_{n-k-1}/(4\rho^2)$, for the largest ball of the node, lose no child. Floating point arithmetic is used only to propose centres of balls and to choose bases. Each node has at most ten balls. For a leaf without a ball the program computes the smallest Hermite normal form of $\Lambda_{T(L)}$ over all signed coordinate permutations and compares it with those of the critical subtori named in Proposition~\ref{prop:search}. It stops if a node of depth at most $n-3$ has no ball, and it reports a leaf without a ball that is not one of these subtori. Neither occurred. For a line with nonzero entries of distinct absolute values it looks for a time $a/q$ with $q\le60$ at which all distances are at least $\frac1n$, and if there is none it decides whether $\ML<\frac1n$ by Lemma~\ref{lem:corner}. The runs for $n=4$ and $n=5$ take $0.01$ and $0.45$ seconds, and the run for $n=6$ consists of two jobs of two and four minutes on one core each. Runs with at most two balls per node for $n=4$ and $n=5$, with at most four balls per node for $n=6$, and for $n=5$ without the reduction by symmetry at depth two give the same lines with $\ML<\frac1n$ and the same critical subtori.

An independent exhaustive search discards a set of speeds as soon as one of $256$ times $a/q$ with $q\le41$ has all distances at least $\frac1n$ and otherwise applies Lemma~\ref{lem:corner}. It covers the $19{,}582{,}837{,}560$ sets of five speeds at most $300$ and the $82{,}408{,}626{,}300$ sets of six speeds at most $200$, and it finds $132$ and $1{,}085$ primitive sets with $\ML<\frac1n$. They are exactly the sets of Theorems~\ref{thm:fivefull} and~\ref{thm:six} with largest element in these ranges, with the same values. For six speeds $33$ of them lie on $U_A$, $33$ on $U_B$, $1{,}011$ on $U_C$, one of these on both $U_B$ and $U_C$, and nine on none. They take $76$ distinct values, and every $P\equiv5\pmod6$ from $11$ to $131$ occurs. The script \texttt{verify\_sections\_12\_14.py} solves the $120$ systems of Lemma~\ref{lem:cellsC}, carries out the sector analysis in the proof of Theorem~\ref{thm:linesC}, with $2{,}304$ sectors in the twelve classes, compares the formula of that theorem with Lemma~\ref{lem:corner} for the $1{,}428$ vectors $w(A,B)$ with $|A|,|B|\le24$, and evaluates $P$ for the $2{,}488$ vectors $w(A,B)$ with nonzero entries, $|A|,|B|\le45$ and $\ML<\frac16$. It checks Lemma~\ref{lem:tight5} at $38$ levels, Theorem~\ref{thm:linesAB} for $a\le150$ and $b\le5$, Proposition~\ref{prop:k3} for $P\le131$ and the maximizing time for $P\le250$, the values of the twelve sets of Theorem~\ref{thm:fivefull}(c) and of the nine sets of Theorem~\ref{thm:six}(d), and the two classifications for speeds at most $34$ and $24$ by an independent enumeration. It also checks Lemmas~\ref{lem:flag}(a) and~\ref{lem:prune} at the root on the flags of $41$ near-tight tuples, the comparison of Table~\ref{tab:UC} with \cite[Theorem~B]{Cordella} for $0<A\le60$ and $|B|\le60$, the count of the vectors $h_1$ in Table~\ref{tab:tree}, and the output files of the tree and of the exhaustive search against the statements of Proposition~\ref{prop:search} and of the two theorems.

\emph{Section~\ref{sec:levels}.} The program \texttt{lrlevel} is \texttt{lrtree} with the level $\lambda$ as a rational parameter, which enters \eqref{eq:alive}, the test at the times $a/q$ with $q\le60$ and the use of Lemma~\ref{lem:corner}. At the level $\frac1n$ it reproduces the lines and the node counts of \texttt{lrtree} for $n=4$ and $n=5$. The runs at the levels $\frac7{25}$, $\frac3{14}$ and $\frac7{40}$ for $n=4,5,6$ take $0.3$ seconds, $2$ seconds and fifteen minutes of computing time. The tree for $n=6$ has $43$, $7{,}919$, $1{,}879{,}774$ and $23{,}183{,}642$ nodes of depths one to four. No node of depth at most $n-3$ is without a ball, and the leaves without a ball are the critical subtori of Proposition~\ref{prop:search}. The exhaustive search \texttt{brutelevel}, with the same change, covers the $186{,}043{,}585$ quadruples with largest element at most $260$, the $2{,}535{,}650{,}040$ sets of five speeds at most $200$ and the $14{,}297{,}000{,}725$ sets of six speeds at most $150$. It finds $25{,}274$, $9{,}923$ and $5{,}527$ primitive sets below the three levels. Those with $\ML>\frac1n$ are exactly the $1{,}269$, $267$ and $98$ sets of Theorem~\ref{thm:above}, with the same values. Those with $\ML=\frac1n$ on no critical subtorus are the $29$, $64$ and $54$ sets of that theorem, and all the others lie on a critical subtorus.

\emph{Section~\ref{sec:seven}.} The program \texttt{plane7} carries out the enumeration in the proof of Theorem~\ref{thm:seven}. For each $r$ it generates the $967{,}680$ vectors $v$, discards those with two equal rows, and applies the determinant test to the $42$ signed pairs of rows, in integer arithmetic. It runs in $0.14$ seconds and finds $22$ survivors for each $r$, all of them products. The program \texttt{plane8} does the same for Theorem~\ref{thm:eight}, with $54{,}190{,}080$ vectors $v$ and $56$ signed pairs for each of the $21$ vectors $q$, in about $28$ seconds per vector $q$. The script \texttt{checks\_new.py} verifies in rational arithmetic the constants $n!\,(n-1)!/2^n$ for $n\le8$, the products of the numbers $d_j$ that give the bounds $7{,}290$, $5{,}145{,}525$ and $12{,}065{,}804{,}452{,}964$, the bounds in the proof of Corollary~\ref{cor:eight}(i), the value $\frac9{65}$ with its witness time, the three values $\ML(U_i\cup\{8\})$, and the inequality \eqref{eq:envelope} at the $45$, $69$ and $87$ points of the proof of Proposition~\ref{prop:lines8}.

\emph{The census.} The program \texttt{census} lists, for given $n$ and $H\le256$, the sets of $n$ distinct speeds at most $H$ with $\ML<\frac1n$. For each reduced time $t=\frac kd\in(0,\frac12]$ with $d\le2H$ it forms the set $C_t$ of speeds $a\le H$ with $\norm{at}<\frac1n$. By Lemma~\ref{lem:corner} the maximum of $\min_{v\in V}\norm{vt}$ for a set $V$ of speeds at most $H$ is attained at one of these times, so $\ML(V)<\frac1n$ if and only if $V$ meets every $C_t$. After repeated sets and supersets of other sets are discarded, $9{,}832$ sets $C_t$ remain for $n=7$ and $9{,}841$ for $n=8$. A depth-first search branches on the elements of an unmet set $C_t$ and excludes the elements chosen in earlier branches at the same node, so that every set $V$ is produced exactly once. For each primitive $V$ it prints the reduced value and the results of the family test and of the Schur test. For $n=7$ and $H=256$ it visits $2.16\cdot10^9$ nodes in six minutes. For $n=8$ and $H=256$ it visits $2.80\cdot10^{10}$ nodes in $72$ minutes of computing time, run as two jobs that traverse the root branches in opposite orders. It reproduces the counts $40$, $74$, $23$ and $13$ of the first search for $(n,H)=(5,70)$, $(6,46)$, $(7,34)$ and $(8,28)$, and the formula of Theorem~\ref{thm:count}(ii) for $H=27,30,44,70,100,150$. The script \texttt{dissociated\_table.py} checks each row of Table~\ref{tab:dissociated} in rational arithmetic, by superincrease, by the interval intersection described in Section~\ref{sec:fifth} and by Lemma~\ref{lem:corner}, and it checks that $\{1,2,4,9,19,36,73,179,322,660\}$ has no odd signed relation and loneliness $\frac4{21}$.

\emph{Exhaustive search for four speeds.} All primitive quadruples with $\eucl v<512$, in two parallel runs of under one minute each.

\emph{Verification script.} It checks the formulas of Theorems~\ref{thm:U1} and \ref{thm:U2} against the exact corner computation for all coprime $x,y\le 40$ and $b\le 10$, $a\le 60$, the $21$ classes of characters and the certificates of Table~\ref{tab:chars}, the $123$ certificates of Section~\ref{sec:two} including the covolumes and the coordinate-freeness, the classification of all primitive quadruples with maximum speed at most $60$ by an independent pure Python enumeration, the absence of $m=3$ and the witnesses for all other $m$, the formulas of Lemmas~\ref{lem:W1sub} and \ref{lem:W2sub} against an exact two-dimensional maximizer written in Python ($524$ subtori of $\Wone$ with $|n_1|\le6$, $|n_2|,|n_3|\le3$, $458$ subtori of $\Wtwo$ with $|n_1|,|n_2|\le4$, $|n_3|\le3$, covering all three cases of Lemma~\ref{lem:W2sub}, and $F(n_1,n_2)$ for $|n_i|\le4$), the by-product on triples for maximum speed at most $80$, and the output file of the exhaustive search, including $200$ random exact re-computations. Larger ranges were checked with the C maximizer: $3{,}356$ subtori of $\Wone$ and $3{,}354$ of $\Wtwo$ with $|n_1|\le12$, $|n_2|,|n_3|\le6$, the $624$ subtori $F(n_1,n_2)$ with $|n_i|\le12$, and the by-product on all $1{,}098{,}601$ primitive triples with maximum speed at most $200$, of which $4{,}639$ have $\ML<1/3$. For Section~\ref{sec:upto25} the script checks the arithmetic of Lemma~\ref{lem:generic25}, the five types of characters with $|m|_1=5$ and their parametrizations, Lemma~\ref{lem:fivecells} at all points of the grid $(\frac1{60}\Z)^3$ for three values of $\varepsilon$, the formula in the proof of Lemma~\ref{lem:five} against its own exact maximizer for all $U_n$ with $\supn n\le3$ in the five tori, the formulas for $\ML(2,3,5j)$ and $\ML(1,4,5j+3)$ for $j\le40$, the two descriptions of $\mathcal R$, the example with value $\frac{43}{103}$ at the level $\frac37$, and the example of Remark~\ref{rem:slicing}. For Section~\ref{sec:obstr} it checks $\ML(1,\frac N2,N-1)$ for $4\mid N\le200$, five subtori $U(A,B,C)$ with $N\le16$ and all their repeated-speed circles, Theorem~\ref{thm:fourspeed} for $N\le60$, $\ell\le3$, the comparison in $U(2,55,3)$ with $\ell=61$, Proposition~\ref{prop:coal} on small cases, Theorem~\ref{thm:phase} on both sides of the transition for $N=12$, and the two values $\frac{15}{31}$ and $\frac{195}{401}$. For the $18{,}960$ survivors of Proposition~\ref{prop:enum25} it checks the witness points, the squared covolumes and the values, reduces them to the $141$ orbits under signed permutations, and recomputes the value of each orbit with its own exact maximizer.

\emph{Sections~\ref{sec:relcert} to \ref{sec:fifth}.} The script \texttt{check\_relations.py}, which uses only the standard library, checks the Fourier expansion of $\Phi$ at random points and by a discrete Fourier transform, its sign on random and grid points, the identity of Lemma~\ref{lem:integral} by an exact quadrature rule on random speed vectors, the numerical claims of Lemma~\ref{lem:thresholds} with rational enclosures, Corollary~\ref{cor:rel}(a) on random sets, the inequality of Theorem~\ref{thm:relcount} on tight and near-tight sets with $9\le k\le12$ (for instance $\{1,\dots,9\}$ has $272$ odd and $274$ even signed relations), the measures of Section~\ref{sec:relopt}, Proposition~\ref{prop:fifth} for $j\le5$ together with the witness time for $j\le30$, and Conjecture~\ref{conj:schur} for the primitive sets of five speeds at most $24$ by an independent enumeration. Section~\ref{sec:box} uses no computation. The script checks the integer bounds of Theorem~\ref{thm:l2} for $2\le k\le10^4$, the three moments of $a_h$ in Lemma~\ref{lem:window} for five values of $h$, and the two identities of Lemma~\ref{lem:moments} numerically for $v=(1,2,3)$, $h=\frac14$ and four cosets.

The minima of Section~\ref{sec:fifth} are computed by the program \texttt{relsearch}. It builds a set in increasing order and skips a new element that is a signed sum of an even number of earlier elements, which is exactly when an odd signed relation appears (for dissociated sets, a signed sum of any number). It stores the signed sums of each parity of the current prefix as bit arrays, rejects a complete set as soon as one of the $630$ times $a/c$ with $c\le64$ puts every runner at distance above the target, and computes $\ML$ of the remaining sets by Lemma~\ref{lem:corner}. For $(k,B)=(3,60)$, $(4,50)$, $(5,250)$, $(6,120)$ and $(7,100)$, with the targets $\frac27$, $\frac14$, $\frac5{21}$, $\frac9{41}$ and $\frac8{39}$, it finds no set below the target, and the sets at the target are the dilates of the minimizer named in Section~\ref{sec:fifth}. The run for $k=7$ visits $2.7\cdot10^9$ sets without odd signed relations in $85$ seconds on one core. The sets of Table~\ref{tab:dissociated} for $k\ge8$ were found by restricted searches with an earlier version of this program, and the script \texttt{dissociated\_table.py} described above checks the rows of the table independently of how they were found. The first search for Proposition~\ref{prop:computed}, with the bounds $70$, $46$, $34$ and $28$ for $5\le n\le8$, enumerated all primitive $n$-subsets, and \texttt{census} reproduces its counts.

\emph{Sections~\ref{sec:relopt} and~\ref{sec:schur}, later additions.} The script \texttt{check\_levels.py} also tests the $29$, $64$ and $54$ sets of Theorem~\ref{thm:above} with $\ML=\frac1n$ for triples $a,b,a+b$, and finds only $\{1,2,6,11\}$ without one. For Proposition~\ref{prop:equal}(c) the program \texttt{brutelevel} was run at the level $\frac1n+\frac1{4nM}$, that is $\frac{641}{4480}$ for seven speeds up to $M=160$ and $\frac{353}{2816}$ for eight speeds up to $M=88$. A value $\frac pq>\frac1n$ of $\ML$ on a set bounded by $M$ has $q\le2M$ by the corner lemma, so it is at least $\frac1n+\frac1{2nM}$, and the output consists of the primitive sets with $\ML\le\frac1n$. The runs examine $466{,}251{,}591{,}520$ and $64{,}276{,}915{,}527$ sets and take a quarter of an hour on two cores. The files \texttt{equal7\_160.txt} and \texttt{equal8\_88.txt} contain the $2{,}728$ and $1{,}539$ sets, and \texttt{check\_equal.py} recomputes their loneliness exactly and tests them for triples. The $49$ sets with $\ML<\frac17$ are the $22$ sets $\{1,\dots,6,7s\}$ with $7s\le160$ and the $27$ sets of $E_7$, and the $50$ sets with $\ML<\frac18$ are the $33$ sets $U_i\cup\{8s\}$ with $8s\le88$ and the $17$ further sets of the census. The linear program at the end of Section~\ref{sec:relopt} is \texttt{schur\_cert\_lp.py}. It uses the solver HiGHS through SciPy in floating-point arithmetic, has $251$ coefficient classes and $41{,}502$ grid points, and takes one minute. With degree $4$ and the threshold $\frac1{10}$ the same program finds a solution.

\emph{Supplements A to C.} The supplements name their own programs and output files, among them \texttt{core\_extensions.cpp}, \texttt{third\_normals.cpp}, \texttt{fourth\_normals.cpp}, \texttt{b2\_universal\_verify.py} and \texttt{b2\_schur\_verify.py}. These files are not among the ancillary files of this version. The ancillary script \texttt{check\_supplements.py} does not use them and does not repeat their searches. It recomputes in exact arithmetic what can be checked from the statements alone. This covers the values and attaining times of the $27$ and $17$ sets in Tables~\ref{A:tab1} and~\ref{C:tab1}, the marks of Theorem~A2.1, the Schur triples and the gaps between the largest speeds in these sets, the counts of Proposition~B6.1 and Corollary~C1.2, the weights of Table~\ref{C:tab2} and the count of signed sums in the proof of Theorem~C3.1, the $74$ pairs of Table~\ref{C:tab3} and the $248$ triples of Theorem~C5.1, the constants of Theorem~B7.1 and Proposition~B8.1, and the closed-form bounds in Theorems~A6.1, B2.1, B3.1 and B9.1 and Corollaries~A3.4, A6.2 and B10.1. It prints \texttt{ALL OK} after one second. The census above is a second check of Theorem~\ref{thm:sevenfull} for speeds up to $160$.

\subsection*{Acknowledgments}
We thank Vico Bonfioli for correspondence about his formalizations. AI tools assisted with parts of the research, the programs and the writing. The authors take full responsibility for the content of the paper.

\appendix
\section{Tables}
\begin{table}[H]
\centering\small
\begin{tabular}{l l l}
\toprule
$m$ & point $x\in W_m$ & $G(x)$\\\midrule
$(0,0,1,-1)$ & $(2/5,2/5,2/5,2/5)$ & $2/5$\\
$(1,-1,-1,-1)$ & $(2/5,2/5,2/5,3/5)$ & $2/5$\\
$(0,1,-2,-1)$ & $(2/5,2/5,2/5,3/5)$ & $2/5$\\
$(1,-2,-1,-1)$ & $(2/5,3/5,3/5,3/5)$ & $2/5$\\
$(0,1,-2,-2)$ & $(2/5,2/5,3/5,3/5)$ & $2/5$\\
$(0,0,1,-3)$ & $(2/5,2/5,1/2,1/2)$ & $2/5$\\
$(1,-2,-2,-1)$ & $(2/5,2/5,3/5,2/5)$ & $2/5$\\
$(0,1,-3,-1)$ & $(2/5,2/5,3/5,3/5)$ & $2/5$\\
$(1,-3,-1,-1)$ & $(2/5,2/5,3/5,3/5)$ & $2/5$\\
$(0,0,2,-3)$ & $(2/5,2/5,2/5,3/5)$ & $2/5$\\
$(1,-2,-2,-2)$ & $(2/5,2/5,2/5,2/5)$ & $2/5$\\
$(0,1,-3,-2)$ & $(2/5,2/5,2/5,3/5)$ & $2/5$\\
$(1,-3,-2,-1)$ & $(2/5,2/5,2/5,2/5)$ & $2/5$\\
$(0,0,1,-4)$ & $(2/5,2/5,2/5,3/5)$ & $2/5$\\
$(0,2,-3,-2)$ & $(2/5,2/5,3/5,1/2)$ & $2/5$\\
$(0,1,-4,-1)$ & $(2/5,2/5,1/2,2/5)$ & $2/5$\\
$(1,-3,-2,-2)$ & $(2/5,3/5,2/5,2/5)$ & $2/5$\\
$(0,1,-3,-3)$ & $(2/5,2/5,2/5,2/5)$ & $2/5$\\
$(1,-4,-1,-1)$ & $(2/5,2/5,2/5,2/5)$ & $2/5$\\
\midrule
$(0,0,1,-2)$ & $\Wone$ up to symmetry & $\ML=1/3$\\
$(0,1,-1,-1)$ & $\Wtwo$ up to symmetry & $\ML=1/3$\\
\bottomrule
\end{tabular}
\caption{The $21$ classes of primitive $m\in\Z^4$ with at least two nonzero entries and $\eucl m^2\le19$, up to signed permutations, with certificates for Lemma~\ref{lem:W25}. In each row $m\cdot x\in\Z$.}
\label{tab:chars}
\end{table}

\begin{table}[H]
\centering\small
\begin{tabular}{l r l l}
\toprule
$\ML(U)$ & $\covol^2$ & parametrization $(X,Y)\mapsto$ & maximizer\\\midrule
$1/4$ & $9$ & $(X,\,Y,\,-X-2Y,\,-X-Y)$ & $(3/4,3/4)$ \\
$1/4$ & $14$ & $(X,\,Y,\,-3Y,\,-2Y)$ & $(3/4,3/4)$ \\
$2/7$ & $21$ & $(X,\,Y,\,-3X-3Y,\,-X-Y)$ & $(5/7,5/7)$ \\
$2/7$ & $26$ & $(X,\,Y,\,-4Y,\,-3Y)$ & $(5/7,4/7)$ \\
$2/7$ & $27$ & $(X,\,Y,\,-2X-4Y,\,-X-Y)$ & $(5/7,5/7)$ \\
$2/7$ & $38$ & $(X,\,2Y,\,-5Y,\,-3Y)$ & $(5/7,6/7)$ \\
$2/7$ & $41$ & $(X,\,Y,\,-6Y,\,-2Y)$ & $(5/7,5/7)$ \\
$2/7$ & $62$ & $(X,\,Y,\,-6Y,\,-5Y)$ & $(5/7,5/7)$ \\
$3/10$ & $45$ & $(X,\,Y,\,-4X-5Y,\,-X-Y)$ & $(7/10,7/10)$ \\
$3/10$ & $57$ & $(X,\,Y,\,-3X-6Y,\,-X-Y)$ & $(7/10,7/10)$ \\
$3/10$ & $74$ & $(X,\,3Y,\,-7Y,\,-4Y)$ & $(7/10,9/10)$ \\
$3/10$ & $86$ & $(X,\,Y,\,-9Y,\,-2Y)$ & $(7/10,7/10)$ \\
$3/10$ & $146$ & $(X,\,Y,\,-9Y,\,-8Y)$ & $(7/10,7/10)$ \\
$4/13$ & $75$ & $(X,\,Y,\,-6X-6Y,\,-X-Y)$ & $(9/13,9/13)$ \\
$4/13$ & $81$ & $(X,\,Y,\,-5X-7Y,\,-X-Y)$ & $(9/13,9/13)$ \\
$4/13$ & $86$ & $(X,\,Y,\,-7Y,\,-6Y)$ & $(9/13,8/13)$ \\
$4/13$ & $98$ & $(X,\,3Y,\,-8Y,\,-5Y)$ & $(9/13,7/13)$ \\
$4/13$ & $99$ & $(X,\,Y,\,-4X-8Y,\,-X-Y)$ & $(9/13,9/13)$ \\
$4/13$ & $122$ & $(X,\,4Y,\,-9Y,\,-5Y)$ & $(9/13,12/13)$ \\
$4/13$ & $149$ & $(X,\,Y,\,-12Y,\,-2Y)$ & $(9/13,9/13)$ \\
$4/13$ & $158$ & $(X,\,3Y,\,-10Y,\,-7Y)$ & $(9/13,3/13)$ \\
$4/13$ & $206$ & $(X,\,2Y,\,-11Y,\,-9Y)$ & $(9/13,11/13)$ \\
$4/13$ & $266$ & $(X,\,Y,\,-12Y,\,-11Y)$ & $(9/13,9/13)$ \\
$5/16$ & $117$ & $(X,\,Y,\,-7X-8Y,\,-X-Y)$ & $(11/16,11/16)$ \\
$5/16$ & $129$ & $(X,\,Y,\,-6X-9Y,\,-X-Y)$ & $(11/16,11/16)$ \\
$5/16$ & $134$ & $(X,\,2Y,\,-9Y,\,-7Y)$ & $(11/16,13/16)$ \\
$5/16$ & $153$ & $(X,\,Y,\,-5X-10Y,\,-X-Y)$ & $(11/16,11/16)$ \\
$5/16$ & $182$ & $(X,\,5Y,\,-11Y,\,-6Y)$ & $(11/16,15/16)$ \\
$5/16$ & $230$ & $(X,\,Y,\,-15Y,\,-2Y)$ & $(11/16,11/16)$ \\
$5/16$ & $278$ & $(X,\,3Y,\,-13Y,\,-10Y)$ & $(11/16,9/16)$ \\
$5/16$ & $422$ & $(X,\,Y,\,-15Y,\,-14Y)$ & $(11/16,11/16)$ \\
\bottomrule
\end{tabular}
\caption{The $31$ coordinate-free two-dimensional subtori of $\T^4$ with $\ML\le5/16$, up to signed coordinate permutations, with a maximizer in the given coordinates. The first two are $\Utwo$ and $\Uone$ up to symmetry. The subtori of the form $(X,cY,\ldots)$ with the last three coordinates depending on $Y$ only are products of a three-runner circle with a full coordinate circle, and their values are the three-runner values $s/(3s+1)$.}
\label{tab:tori}
\end{table}

\input{suppA}
\input{suppB}
\input{suppC}

\end{document}

%% file: suppA.tex
\supplement{A}{Further classifications and bounds for seven and eight speeds}

For every primitive septuple of loneliness below $1/7$, we classify all cases having either a five-speed subset of loneliness below $1/5$ or a six-speed subset of loneliness below $1/6$. The classification also holds whenever some four speeds, divided by their common gcd, have maximum at most $26$. With eight speeds the analogous result holds for five speeds of normalized maximum at most $24$. These statements hold at arbitrary total height. We determine the seven-coordinate two-tori below $4/25$ and three-tori below $2/11$. The resulting seven-speed norm cutoff is $140{,}612{,}164/45$, down from $5{,}145{,}525$. The finite residual seven-speed problem is specified in Section~A6. A separate result gives the sharp first-order asymptotic for the worst-case Hausdorff dimension of lonely times for superincreasing speeds.

\emph{Conventions and dependencies.} Result and reference numbers without a prefix refer to Sections~1 to~\ref{sec:comp}. Write $\T=\R/\Z$ and $\|x\|=\operatorname{dist}(x,\Z)$. For a speed set $V$, $\ML(V)=\max_{t\in\T}\min_{v\in V}\|vt\|$. A set is primitive when its gcd is one and near-tight when $\ML(V)<1/|V|$, with $|V|$ here denoting cardinality. Norm bounds use $\|V\|_2$. A coordinate-free subtorus lies in no coordinate hyperplane $x_i=0$. The results use the specified theorems and finite classifications of Sections~\ref{sec:general} to~\ref{sec:seven}. The proofs in this supplement are computer-assisted. Section~\ref{sec:comp} says which of their programs are among the ancillary files.

\section{Completing a prescribed subset of speeds}\label{A:sec1}
For a set $V$ of $n$ distinct positive integers put
\[
 h_k(V)=\min_{A\subset V,\ |A|=k}\frac{\max A}{\gcd(A)}.
\]
This quantity is unchanged by dilation. A subset of arbitrarily large speeds can have small $h_k$. Write $[j]=\{1,\dots,j\}$ and set
\[
 U_1=[7],\quad U_2=\{1,2,3,4,5,7,12\},\quad U_3=\{1,4,5,6,7,11,13\}.
\]
The finite lists $E_7$, $E_8$ appear in the tables at the end of the supplement.

\begin{theorem}\label{A:thm1.1}
The primitive seven-speed sets $V$ satisfying $\ML(V)<1/7$ and $h_4(V)\le26$ are precisely
\[
 [6]\cup\{7s\}\quad(s\ge1),\qquad V\in E_7\quad(27\text{ sets}).
\]
The primitive eight-speed sets $V$ satisfying $\ML(V)<1/8$ and $h_5(V)\le24$ are precisely
\[
 U_i\cup\{8s\}\quad(1\le i\le3,\ s\ge1),\qquad V\in E_8\quad(17\text{ sets}).
\]
The eight-speed conclusion holds independently of the seven-speed tight-instance classification~\textup{(M)}.
\end{theorem}

For $A$ primitive, write $G_A=\{x\in\T:\min_{a\in A}\|ax\|\ge1/n\}$. The following bound allows us to enumerate extensions of $dA$ without knowing $d$ or the eventual maximum speed.

\begin{lemma}\label{A:lem1.2}
Suppose $dA\subset V$, $\ML(V)<1/n$ and $m=n-|A|$ satisfies $0<2m<n$. If $I=[l,r]\subset G_A$ has length $L>0$, there is a speed $w\in V\setminus dA$ such that, writing $w/d=p/q$ in lowest terms and $\alpha=\{2q/n\}$, one has
\[
 q\le\frac{mn}{n-2m},\qquad pqL\le\frac{mn\alpha(1-\alpha)}{n-2m},\qquad q\le m\lceil 2q/n\rceil.
\]
On the $d$ lifts of a uniform point of $I$, the probability that $w$ has distance strictly less than $1/n$ is at least $1/m$.
\end{lemma}
\begin{proof}
Choose $x$ uniformly in $I$, and choose $j$ uniformly in $\{0,\dots,d-1\}$ independently. At $t=(x+j)/d$ all the speeds in $dA$ are good. The $m$ events $\|wt\|<1/n$ cover all outcomes. Some event has probability at least $1/m$.

For that speed the conditional values of $wt$ form a translate of the $q$-point grid, each with equal weight. Apart from the finitely many endpoints, the number in the bad arc is
\[
 c+\mathbf 1_{\{\{px+q/n\}<\alpha\}},\qquad c=\lfloor 2q/n\rfloor.
\]
Indeed it counts the integers in the open interval $(px-q/n,px+q/n)$. Its maximum is $\lceil 2q/n\rceil$, yielding the third bound and, using $\lceil 2q/n\rceil\le 2q/n+1$, the first. An antiderivative of the indicator of $[0,\alpha)$ modulo $1$ is
\[
 K(z)=\alpha\lfloor z\rfloor+\min(\{z\},\alpha).
\]
The exact bad probability is therefore
\[
 \frac1q\Bigl(c+\frac{K(pr+q/n)-K(pl+q/n)}{pL}\Bigr).
\]
The periodic function $K(z)-\alpha z$ has range $[0,\alpha(1-\alpha)]$. The probability is at most $2/n+\alpha(1-\alpha)/(pqL)$. Rearrangement proves the remaining bound.
\end{proof}

\begin{proof}[Proof of Theorem~A1.1]
Start with each primitive $k$-subset of $[H]$, for $(n,k,H)=(7,4,26)$ or $(8,5,24)$. Any set in the theorem contains a dilate of one of these seeds. At each node compute $G_A$ on $[0,1/2]$ by intersecting the closed intervals
\[
 \Bigl[\frac{nj+1}{na},\frac{nj+n-1}{na}\Bigr],\qquad j\in\Z,
\]
for the speeds $a\in A$. Symmetry gives the other half of the circle. If $m\ge2$, choose any interval of positive length, enumerate the coprime $p,q$ from Lemma~A1.2, and keep those passing its exact probability test. The next node is the primitive set $qA\cup\{p\}$, deleting repeated speeds. Indeed $d/q=\gcd(d,w)$ and $V$ contains $(d/q)(qA\cup\{p\})$. Induction gives at least one surviving path for every possible completion. Accepting an extra path has no effect on the final answer.

When $m=1$, the bound for $q$ forces $q=1$. By primitivity of $V$, the dilation of $A$ is now one. On a positive interval we need $pL<2/n$, so only finitely many $p$ remain. Accept a candidate exactly when every closed component of $G_A$, including each isolated point, lies wholly inside one open bad interval for $p$. If $G_A$ has only isolated points, it is handled as a separate terminal node. No empty good sets occurred; any node with just isolated points had $m=1$. The distinct terminal cores of this kind were $[6]$ at seven speeds and $U_1,U_2,U_3$ at eight. The arguments of Corollary~\ref{cor:seven}(i) and Proposition~\ref{prop:lines8} give the extensions of these four cores, including all possible last speeds and dilations. Their hypotheses here are verified directly by the enumeration.
\begin{center}\small
\begin{tabular}{crrrr}
\toprule
$n$ & seed nodes & next size & last proper size & finite outputs (off family)\\
\midrule
$7$ & $14{,}161$ & $775{,}318$ & $29{,}606{,}104$ & $33$ $(27)$\\
$8$ & $41{,}656$ & $967{,}584$ & $19{,}850{,}119$ & $27$ $(17)$\\
\bottomrule
\end{tabular}
\end{center}
The last proper size counts visits, which can repeat when different branches lead to the same core. The eight-speed tree starts with five speeds. The finite outputs left after removing family members are exactly the tables $E_7$, $E_8$. The source file \texttt{core\_extensions.cpp} implements these bounds; all decisions use integer arithmetic apart from the choice among intervals, for which any positive interval suffices. The independent check \texttt{verify\_lists.py} evaluates each surviving set by all the corner times in Lemma~\ref{lem:corner}, verifies the times in the tables, and checks their normalized subset heights. Each infinite family contains a fixed subset meeting the stated height bound.
\end{proof}

\begin{corollary}\label{A:cor1.3}
In the seven-speed class of Theorem~A1.1, the spectrum below $1/7$ is
\[
 \Bigl\{\frac{s}{7s+1}:s\ge1\Bigr\}\ \cup\ \{3/23,\,5/37,\,9/65\}.
\]
In its eight-speed class the spectrum below $1/8$ is exactly $\{s/(8s+1):s\ge1\}$. Every set in both classes contains three distinct speeds $a,b,a+b$.
\end{corollary}
\begin{proof}
The regular seven-speed family has value $s/(7s+1)$. For eight speeds the family $U_1\cup\{8s\}$ gives all $s/(8s+1)$; the other two have that value for $s\ge2$ and $2/17$ for $s=1$ by Proposition~\ref{prop:lines8}. The two tables supply the remaining fractions. Every regular core contains $1,2,3$ or $1,4,5$. The distinct sums for the finite rows are tested in \texttt{verify\_lists.py}.
\end{proof}

\section{All seven-speed extensions of a six-speed root configuration}\label{A:sec2}
Write $R(A,B)$ for the set of absolute values of
\[
 A,\ B,\ A+B,\ 2A+B,\ 3A+B,\ 3A+2B.
\]
The result applies to all such sextuples, including directions where the six speeds attain loneliness $1/6$.

\begin{theorem}\label{A:thm2.1}
Exactly $18$ primitive near-tight seven-speed sets contain a six-element set $R(A,B)$ for integers $A,B$. They are the sets marked with an asterisk in the seven-speed table and $[7]$.
\end{theorem}
\begin{proof}
First, any six speeds of a primitive near-tight septuple have gcd one. Otherwise put $d>1$ for the gcd and $C$ for the remaining speed. There is a time at which the six normalized speeds have distances at least $1/7$ by the six-speed lonely runner theorem. At its $d$ lifts the last speed runs through a coset of the $d$-grid because $(d,C)=1$. Its greatest distance is at least $1/2-1/(2d)\ge1/4$, a contradiction. Thus $(A,B)=1$. The reflections $(A,B)\mapsto(-A,3A+B)$ and $(A,B)\mapsto(A+B,-B)$ permute the absolute forms. Their orbit is finite. Maximizing $5A+3B$ on it gives $A,B>0$, as in Theorem~\ref{thm:count}(iii).

In the $(x,y)$-torus the six forms are all at distance at least $1/7$ throughout
\[
 x\ge1/7,\qquad y\ge1/7,\qquad 3(x-1/7)+2(y-1/7)\le1/7.
\]
Indeed the forms there lie between $1/7$ and $6/7$. Its vertical edge has length $1/14$ and its horizontal edge length $1/21$. On $At=1/7\pmod 1$, the values of $Bt$ form an equally spaced grid of step $1/A$, so at least $N=\lfloor A/14\rfloor$ consecutive grid values lie on the vertical edge. Choose a balanced integer $k$ with $kB\equiv C\pmod A$. Consecutive values of $Ct$ along this ordered grid differ by $k/A$ modulo one. Since the septuple is near-tight, their representatives in $(-1/7,1/7)$ have differences in $(-2/7,2/7)\subset(-1/2,1/2)$. Each difference must equal $k/A$. Therefore for $N\ge2$,
\[
 |k|(N-1)<2A/7.
\]
If $A\ge231$ this implies $|k|<2A/[7(A/14-2)]<5$. Write $C=jA+kB$. The plane given by the six root forms together with $jx+ky$ is coordinate-free. Its first six rows are pairwise nonproportional, while the unique seven-speed plane below $5/33$ has six proportional rows. Hence this plane has loneliness at least $5/33$ by Theorem~\ref{thm:seven}. At a maximizing point fix $x$ and choose $t$ with $At=x$ and $Bt$ within $1/(2A)$ of its target. The largest absolute $y$-coefficient is at most four. This gives
\[
 \ML(V)\ge5/33-2/A\ge1/7,
\]
a contradiction. Similarly on the horizontal edge there are at least $\lfloor B/21\rfloor$ grid points. For $B\ge347$ the balanced $j$ with $jA\equiv C\pmod B$ satisfies $|j|<2B/[7(B/21-2)]<7$. The largest absolute $x$-coefficient is at most six and $\ML(V)\ge5/33-3/B\ge1/7$.

It remains to test $1\le A\le230$, $1\le B\le346$. Coprimality and six distinct speeds leave $48{,}497$ bases. The last-speed test in the proof of Theorem~A1.1 computes their full closed good sets and all possible seventh speeds. There are no isolated-only terminal cores. The $18$ results are verified by corner times and marked in the table; the full output is in \texttt{root-extensions.txt}.
\end{proof}

\begin{theorem}\label{A:thm2.2}
If a primitive near-tight seven-speed set contains either a five-speed subset $P$ with $\ML(P)<1/5$ or a six-speed subset $S$ with $\ML(S)<1/6$, it belongs to $[6]\cup\{7s\}$ or to $E_7$.
\end{theorem}
\begin{proof}
For a five-speed subset, Theorem~\ref{thm:fivefull} gives two families with fixed four-speed cores $\{1,2,3,4\}$ or $\{1,3,4,7\}$ after normalization, and twelve exceptional quintuples of maximum at most $27$. Their four smallest speeds have maximum at most $26$. Thus $h_4(V)\le26$, and Theorem~A1.1 applies even when the five speeds have a common factor.

For a six-speed subset, $S$ is primitive by the grid argument above. Apply Theorem~\ref{thm:six}. Each of the two product types contains either $\{1,2,3,4\}$ or $\{1,3,4,5\}$. Each of the nine exceptional sextuples in part~(d) has a four-subset with maximum at most $26$, by its displayed entries. Theorem~A1.1 applies to all these cases. The remaining sextuples are covered by Theorem~A2.1.
\end{proof}

\section{The first seven-speed plane outside the products}\label{A:sec3}
Write $T_S$ for the coordinate circle with primitive positive speed set $S$.

\begin{theorem}\label{A:thm3.1}
Up to signed coordinate permutations, the coordinate-free two-dimensional subtori $W\subseteq\T^7$ with $\ML(W)<4/25$ are the products $T_S\times\T$ for the $22$ primitive six-speed sets $S$ listed below, and the plane
\[
 (x,y)\longmapsto(y,\,6y,\,x-2y,\,x-y,\,x,\,x+y,\,2x-y).
\]
The latter has loneliness $3/19$. Thus every such plane of loneliness strictly below $3/19$ is a product, and $3/19$ is the best possible endpoint of that statement.
\end{theorem}
\begin{center}\small
\begin{tabular}{c|l}
$\ML(S)$ & $S$ (listed in increasing order)\\
\hline
$1/7$ & $(1,2,3,4,5,6)$\\
$5/33$ & $(1,5,6,11,16,17)$\\
$2/13$ & $(1,2,3,4,5,12)$, $(1,2,5,6,7,8)$, $(1,3,4,5,6,9)$,\\
 & $(1,3,4,5,9,12)$, $(1,4,5,6,7,11)$, $(1,4,5,6,7,22)$,\\
 & $(1,6,7,8,9,15)$, $(2,5,6,8,10,11)$\\
$8/51$ & $(5,6,11,17,23,28)$\\
$3/19$ & $(1,2,3,4,5,18)$, $(1,2,3,5,8,18)$, $(1,3,4,5,9,18)$,\\
 & $(1,5,6,7,8,13)$, $(2,3,5,6,8,11)$, $(2,3,5,7,9,12)$,\\
 & $(2,5,6,7,8,11)$, $(5,7,12,17,22,29)$, $(6,7,13,20,27,33)$\\
$11/69$ & $(1,11,12,23,34,35)$, $(6,11,17,23,29,40)$
\end{tabular}
\end{center}
\begin{proof}
In the notation of Corollary~\ref{cor:six}, the inequality $(P-1)/(6P)<4/25$ gives $P<25$. The admissible indices are $7,11,13,17,19,23$. Their reduced denominators are at most $69$. Theorem~\ref{thm:six} now gives a small finite enumeration. In (a) and (b) it is enough to test $s=1,2,3$; the exceptional $s=1$ of (b) is included. In (c), Corollary~\ref{cor:six}(iv) gives $\max(|A|,|B|)<207$, and we may take $A>0$ by overall sign. We form the six absolute values, discard repetitions, zeros and nonprimitive pairs, and evaluate Lemma~\ref{lem:corner}. Testing the nine inputs of (d) finishes the displayed list. The exact evaluation on any set $S$ checks all times $k/(u+v)$ with $u,v\in S$ and $0\le k\le(u+v)/2$ by symmetry.

Write $p_i$ for the seven coordinate forms on a prospective plane. If $p_i=\pm p_j$, deleting one coordinate leaves a coordinate-free plane in $\T^6$, whose repeat line has at most five distinct speeds and hence loneliness at least $1/6$. The same lower bound holds for any coordinate-free line with at most five distinct speeds. Consequently each repeat line of $W$ with $\ML(W)<4/25$ has exactly six distinct absolute speeds and yields, after rescaling, one of the displayed sets. Lemma~\ref{lem:repeat} gives two distinct repeat lines spanning $W$. After signs and permutations the first is $q=(S,r)$ with $r\in S$; the second is a signed permutation $v$ of $(S',r')$, with $r'\in S'$ and $v_1>0$. We allow the two lines to be exchanged, so a fixed ordering of the $22$ models allows us to require $S\le S'$. There are no further identifications at this stage.

Now $p_i=(q_i,v_i)$ and $E=\R q+\R v$. We discard equal rows at the repeated entry of $q$. This also excludes proportional spanning vectors. For every $i<j$ and $\epsilon\in\{-1,1\}$ compute
\[
 d_k=\det(p_k,\,p_i+\epsilon p_j),\qquad 1\le k\le7.
\]
If some $d_k=0$ the kernel is contained in a coordinate hyperplane and supplies no condition. Otherwise it is a genuine repeat line: if $p_i,p_j$ had been proportional, its kernel would be contained in the kernel of $p_i$. Its distinct absolute determinants must therefore be a positive integer multiple of one of the $22$ models. The integer scale follows from the primitivity of the model. All formed determinants have magnitude at most $4H^2$, where $H=40$ is the greatest speed in the list.

The following are the exact counts for the indicated independent runs, treating distinct spanning pairs as distinct presentations.
\begin{center}\small
\begin{tabular}{c|r|r|r|r|r}
bound & $\#S$ & pairs & unequal rows & survivors & products\\
\hline
$2/13$ & $2$ & $17{,}418{,}240$ & $17{,}003{,}520$ & $264$ & $264$\\
$3/19$ & $11$ & $383{,}201{,}280$ & $374{,}077{,}440$ & $1{,}452$ & $1{,}452$\\
$4/25$ & $22$ & $1{,}468{,}938{,}240$ & $1{,}433{,}963{,}520$ & $3{,}032$ & $2{,}904$
\end{tabular}
\end{center}
Six proportional rows whose primitive multipliers give an admissible $S$, and one independent row, give the product $T_S\times\T$. For each of the remaining $128$ presentations choose an ordered independent pair $a,b$ of rows and rewrite each row $z$ as $(\det(z,b),\pm\det(a,z))$, with a common choice of sign. This amounts to an invertible rational change of the two parameters. Remove the common gcd of the entries, choose each row sign so its first nonzero entry is positive, and sort. Minimizing first the largest absolute entry, then the sum of absolute entries and finally the ordered list, over all such choices, gives in every case
\[
 (0,1),\ (0,6),\ (1,-2),\ (1,-1),\ (1,0),\ (1,1),\ (2,-1).
\]
Every operation here preserves the connected image subtorus up to signed coordinate permutation.

It remains to determine the loneliness of this plane. The point $(3,9)/19$ gives $3/19$. On each chamber the maximum is the optimum of the inequalities $m_i+z\le p_i\cdot(x,y)\le m_i+1-z$, $z\ge0$, for fixed integers $m_i$. The rank of the forms makes this polytope bounded. At a maximizing vertex $0<z<1/2$, since $y$ and $6y$ cannot both have distance $1/2$. Three independent equalities are active, from distinct rows. By Cramer's rule $(x,y)$ belongs to the grid $(1/D)\Z^2$, where $D$ is a nonzero absolute value of
\[
 \pm\det(p_i,p_j)\pm\det(p_j,p_k)\pm\det(p_k,p_i).
\]
The distinct positive $D$ are $1,2,3,4,5,6,7,9,10,11,13,14,15,17,19,21$. Testing their $2{,}123$ grid points in $[0,1)^2$ gives the upper bound $3/19$. A separate exact check, \texttt{plane\_polygon.py}, intersects the good strips with the unit square and finds no positive-area polygon at $3/19$. Any larger value would give an open region at this threshold. Conversely all listed products have loneliness $\ML(S)$, since the free coordinate can be set to $1/2$.
\end{proof}

\begin{corollary}\label{A:cor3.2}
The accumulation points in $(0,4/25)$ of loneliness values of sets of seven distinct positive integer speeds are precisely
\[
 \frac17,\quad\frac5{33},\quad\frac2{13},\quad\frac8{51},\quad\frac3{19},\quad\frac{11}{69}.
\]
In particular the seven-speed instance of Giri and Kravitz's accumulation-point conjecture \cite[Conjecture~1.5]{GiriKravitz} holds on this interval after the change of variables $D=1/2-\ML$.
\end{corollary}
\begin{proof}
By \cite[Theorem~1.4]{GiriKravitz}, a seven-speed accumulation value must be the loneliness of a coordinate-free two-dimensional subtorus of $\T^7$. The preceding classification restricts it to the displayed six values. Conversely, choose a six-speed set $S$ attaining any one of these values $\mu$. The function $t\mapsto\min_{a\in S}\|at\|$ has finitely many maximizers in $[0,1)$, all rational, since it is piecewise linear with nonzero slopes and rational breakpoints. Let $b$ run through positive multiples of all of their denominators, with $b>\max S$. Then $\ML(S\cup\{b\})<\mu$. Given any $\epsilon>0$, an interval around a maximizer keeps the six distances at least $\mu-\epsilon$; for all sufficiently large $b$ this interval contains a time where $\|bt\|=1/2$. The seven-speed values therefore approach $\mu$.
\end{proof}

\begin{proposition}\label{A:prop3.3}
Up to signs and permutations, the primitive seven-speed lines on the non-product plane of Theorem~A3.1 having distinct nonzero speeds and loneliness below $1/7$ are exactly
\begin{center}
\begin{tabular}{c|c}
speeds & $\ML$\\
\hline
$(1,2,3,4,5,6,7)$ & $1/8$\\
$(1,2,3,4,5,7,18)$ & $3/23$\\
$(1,2,3,4,5,7,24)$ & $4/29$\\
$(2,7,9,11,12,13,20)$ & $4/29$\\
$(2,7,9,11,13,20,54)$ & $9/65$.
\end{tabular}
\end{center}
\end{proposition}
\begin{proof}
Write the primitive line as $(x,y)=(At,Bt)$. The coordinate forms include both $x$ and $y$, so their values at $(A,B)$ have gcd $1$ exactly when $\gcd(A,B)=1$. At $(x_0,y_0)=(3,9)/19$ the margin above the threshold is $\delta=3/19-1/7=2/133$. All seven coordinates remain at distance at least $1/7$ on the horizontal segment $|x-x_0|\le\delta/2$, $y=y_0$, since their $x$-coefficients have magnitudes at most $2$. They also remain good on the vertical segment $x=x_0$, $|y-y_0|\le\delta/6$. The primitive circle consists of the points satisfying $Bx-Ay\in\Z$. The two images of the segments under $Bx-Ay$ have lengths $|B|\delta$ and $|A|\delta/3$. If either were at least one, the circle would meet the corresponding closed segment. Thus $|B|\le66$ and $|A|\le199$. We may assume $B>0$, since $B=0$ gives a zero coordinate.

Within this rectangle there are $16{,}065$ primitive directions, of which $16{,}044$ have seven nonzero, distinct absolute speeds. These yield $8{,}686$ distinct sorted sets. Testing them by Lemma~\ref{lem:corner} gives the list. A second test intersects, for each speed $u$, the closed intervals $[(7j+1)/(7u),(7j+6)/(7u)]$ with the current good set in $[0,1/2]$. This gives precisely the same five empty intersections.
\end{proof}

\begin{corollary}\label{A:cor3.4}
Let $v$ be a primitive septuple with distinct positive speeds and $\ML(v)<1/7$. If it is covered by neither Theorem~A2.2 nor the five entries of Proposition~A3.3, then
\[
 \|v\|_2<\frac{26{,}796{,}875}{8}.
\]
For such a septuple the reduced denominator of $\ML(v)$ is at most $4{,}737{,}063$.
\end{corollary}
\begin{proof}
Every product plane in Theorem~A3.1 contains only lines whose six fixed coordinates are a common multiple of its six-speed model. Their six-speed loneliness is strictly below $1/6$, so Theorem~A2.2 applies. The proposition disposes of the remaining plane in that list. Every coordinate-free plane containing the line of $v$ therefore has loneliness at least $4/25$.

In the flag proof of Proposition~\ref{prop:flag}(b), write $\lambda_1\le\dots\le\lambda_6$ for the projected successive minima. Its strengthened rank-two inequality gives $\lambda_1>2(4/25-1/7)=6/175$. Monotonicity therefore gives a strict lower bound $12/175$ on $\lambda_1+\lambda_2$. For prefix lengths $j=3,4,5,6$ retain the bounds $2(1/(8-j)-1/7)$ from the same proof. The successive differences of the six lower bounds are
\[
 \frac6{175},\quad\frac6{175},\quad\frac8{175},\quad\frac1{10},\quad\frac16,\quad\frac13.
\]
They are positive and increasing in the weak sense. The product inequality of Lemma~\ref{lem:prodsum} and $\prod_{i=1}^6\lambda_i\le1/\|v\|_2$ now give the asserted strict Euclidean bound. By Lemma~\ref{lem:corner} the reduced denominator divides $u+w$ for two distinct speeds, so it is at most $\lfloor\sqrt2\,(26{,}796{,}875/8)\rfloor=4{,}737{,}063$.
\end{proof}

\section{The first two values for three-dimensional subtori of seven coordinates}\label{A:sec4}
Write $T_u=\{tu:t\in\T\}$ and use the three two-dimensional subtori of six coordinates from Section~\ref{subsec:search}:
\begin{gather*}
 U_A=T_{(1,2,3,4,5)}\times\T,\qquad U_B=T_{(1,3,4,5,9)}\times\T,\\
 U_C=\{(x,y,x+y,2x+y,3x+y,3x+2y):x,y\in\T\}.
\end{gather*}
All permutations and sign changes below act on the ambient coordinates. A subtorus is coordinate-free when no coordinate vanishes identically on it.

\begin{theorem}\label{A:thm4.1}
Up to permutations and sign changes, the only coordinate-free two-dimensional subtori of $\T^6$ with loneliness below $2/11$ are $U_A,U_B,U_C$. The only coordinate-free three-dimensional subtori of $\T^7$ below the same threshold are $U_A\times\T$, $U_B\times\T$, and $U_C\times\T$. All six of these models have loneliness $1/6$. In both dimensions the next value is exactly $2/11$.
\end{theorem}
\begin{proof}
A collision subtorus is the connected subtorus cut out in the tangent space by $\ell_i=\pm\ell_j$ for two nonproportional coordinate forms. There are at least two distinct coordinate-free collision subtori of codimension one in every coordinate-free torus of dimension at least two. Indeed, in a two-plane of the dual spanned by coordinate forms, their distinct projective directions determine at least two open arcs. In each arc take the sum of signed coordinate forms spanning the bounding rays as in the proof of Lemma~\ref{lem:coalesce}. The sum belongs to no coordinate direction, so its kernel is coordinate-free; distinct arcs give distinct kernels.

First let $W$ be a two-dimensional subtorus in six coordinates below $2/11$. Its coordinates are distinct up to sign, since otherwise deleting a repeated coordinate and applying Lemma~\ref{lem:coalesce} gives loneliness at least $1/5$. On each coordinate-free collision circle there are exactly five distinct absolute speeds: four or fewer would give the same contradiction. By Theorem~\ref{thm:fivefull}, the primitive five-speed sets with loneliness below $2/11$ are precisely
\[
 E_A=\{1,2,3,4,5\},\qquad E_B=\{1,3,4,5,9\}.
\]
Two distinct collision circles span $W$. Fix the first vector as $E_A$ or $E_B$ with one entry repeated at the end. Enumerate the second vector in the same ten forms with all distinct coordinate permutations and signs and positive first entry. For each pair of row forms $r_i,r_j$ and either sign, its kernel circle has speeds $\det(r_k,r_i\pm r_j)$. An identically zero defining form excludes the plane. Each coordinate-free collision must, after discarding one repeated entry and normalizing the absolute values, yield $E_A$ or $E_B$. Of $1{,}152{,}000$ ordered pairs, $190$ pass. Every one spans a signed coordinate permutation of $U_A$, $U_B$ or $U_C$. Each of these has loneliness $1/6$ by Section~\ref{subsec:search}.

Now let $S$ be three-dimensional in seven coordinates below $2/11$. Global repeated coordinates are impossible by Lemma~\ref{lem:coalesce} in six coordinates. A coordinate-free collision plane in $S$ has exactly one duplicate up to sign, since two deletions would give a plane in five coordinates of loneliness at least $1/5$. Deleting its repeated coordinate leaves one of the three models just found. If $S$ contains a coordinate axis, it splits as the product of that circle and its projection to the other six coordinates. Its projection is coordinate-free and has the same loneliness, giving one of the three claimed products.

For an axis-free $S$, no collision plane can contain a coordinate axis. In a lift of $U_A$ or $U_B$ this forces the free coordinate to be the one duplicated. In a lift of $U_C$ any one of the six coordinates can be duplicated. We therefore have eight choices for a fixed first collision plane $F$. For each, enumerate the eight choices for a second plane $G$, all $7!$ coordinate permutations, and the $2^6$ sign choices modulo simultaneous negation. Retain the sums $F+G$ of rank three, discard those containing an axis, and identify equal tangent spaces by their primitive signed vectors of maximal minors. Every sought $S$ occurs because it has two distinct collision planes. For each remaining sum, restrict the seven coordinate forms to each coordinate-free collision plane, testing the forms up to sign, permutation and change of basis against the same eight models. The exact counts are
\begin{center}\footnotesize\setlength{\tabcolsep}{3.5pt}
\begin{tabular}{c|rrrrrrrr}
$F$ & $A$ & $B$ & $C_1$ & $C_2$ & $C_3$ & $C_4$ & $C_5$ & $C_6$\\
\hline
rank-three pairs & $15608$ & $15464$ & $6792$ & $6568$ & $6792$ & $6792$ & $6568$ & $6568$\\
distinct axis-free sums & $3820$ & $3827$ & $501$ & $477$ & $501$ & $501$ & $477$ & $477$\\
pass all collisions & $0$ & $0$ & $0$ & $0$ & $0$ & $0$ & $0$ & $0$
\end{tabular}
\end{center}
Here $C_i$ duplicates the $i$th coordinate of $U_C$ in the order displayed above. The verifier \texttt{endpoint\_planes.cpp} uses integer determinants. To check equivalence of row configurations, it expresses the rows in every ordered pair of independent rows as basis, allowing sign changes of that basis. The models all contain the two standard basis rows. It then fixes the sign of each row and compares the sorted rows. This tests precisely equality of the tangent spaces up to a signed coordinate permutation.

Finally, $T_{(1,2,3,5,8)}$ has loneliness $2/11$ by Theorem~\ref{thm:fivefull}(c). Adjoining one or two independent circle factors preserves that value, proving attainment at the upper endpoint.
\end{proof}

\section{Gaps between the largest speeds of a near-tight septuple}\label{A:sec5}
\begin{theorem}[The two largest speeds]\label{A:thm5.1}
Let $v_1<\dots<v_7$ be primitive positive integers with $\ML(v)<1/7$. If $v_7\ge6v_6$, then $\{v_i\}=\{1,2,3,4,5,6,7s\}$ or $\{v_i\}=\{1,5,6,11,16,17,112\}$. The latter has loneliness $16/113$. In particular, outside the product family the largest possible ratio $v_7/v_6$ is $112/17$, attained by a unique tuple.
\end{theorem}
\begin{proof}
Put $U=\{v_1,\dots,v_6\}$, $H=v_6$, $w=v_7$ and $r=\ML(U)$. If $r\le1/7$ the required inequality $r<1/6$ already holds. For $r>1/7$, the function $t\mapsto\min_{u\in U}\|ut\|$ is $H$-Lipschitz. It therefore stays at least $1/7$ on a closed interval of length $2(r-1/7)/H$ about a maximizing time. Such an interval must fit inside a single open bad interval for $w$, of length $2/(7w)$. Consequently $r<1/7+H/(7w)\le1/6$ whenever $w\ge6H$. Theorem~A2.2 now applies. Of its finite exceptions, only $\{1,5,6,11,16,17,112\}$ satisfies $w\ge6H$; its value is $16/113$. Since $112/17>6$, the stated maximum follows.
\end{proof}

\begin{proposition}\label{A:prop5.2}
Every primitive near-tight septuple $v_1<\dots<v_7$ satisfies $v_6<5v_5$.
\end{proposition}
\begin{proof}
Put $U=\{v_1,\dots,v_5\}$, $H=v_5$, $a=v_6$ and $b=v_7$. Every component of the union of the open bad-time sets of $a,b$ at level $1/7$ has length at most $4/(7a)$. Indeed no $b$-band bridges two $a$-bands, since $2/(7b)<5/(7a)$. When $b\le5a/2$, no $a$-band bridges two $b$-bands either, giving component length below $2/(7a)+2/(7b)$. When $b>5a/2$, at most two $b$-bands extend the ends of one $a$-band, giving length below $2/(7a)+4/(7b)<4/(7a)$.

Assume $a\ge5H$. If $r=\ML(U)\ge1/5$, the $H$-Lipschitz bound gives a common closed good interval of length $2(r-1/7)/H\ge4/(35H)$, which the open union cannot cover. Thus $r<1/5$. The five-speed case of Theorem~A2.2 gives the regular family and the table $E_7$. In the regular family $v_6/v_5=6/5$; inspection of the $27$ finite rows gives $v_6<5v_5$ in every case. This contradicts $a\ge5H$.
\end{proof}

\section{A smaller residual range for seven and eight speeds}\label{A:sec6}
\begin{theorem}\label{A:thm6.1}
Let $V$ be a primitive seven-speed set with $\ML(V)<1/7$. If $\|V\|_2>140{,}612{,}164/45$, then $V=[6]\cup\{7s\}$ for some $s\ge1$. Every reduced value $p/q<1/7$ off the sequence $s/(7s+1)$ has $q\le4{,}419{,}013$. Consequently every seven-speed value in
\[
 \Bigl(\frac17-\frac{2}{7\cdot4{,}419{,}013},\ \frac17\Bigr)
\]
belongs to that sequence. Also every seven-speed value strictly above $1/8$ is at least $1/8+1/(8Q)$, where $Q=1{,}292{,}960$.
\end{theorem}
\begin{proof}
It suffices for the norm bound to consider $V$ outside the regular family and $E_7$, since the latter sets have maximum speed at most $112$. No plane of loneliness below $4/25$ contains its line, by Section~A3 and Theorem~A2.2. Nor can its line lie in a three-torus of loneliness below $2/11$. Indeed Theorem~A4.1 lists three products. On the first two, $V$ contains a dilate of $\{1,2,3,4,5\}$ or $\{1,3,4,5,9\}$ and is covered by Theorem~A1.1. On the third, the six coordinates form $R(A,B)$ with distinct nonzero absolute values, so Theorem~A2.1 applies.

Put $g=(4/25,2/11,1/5,1/4,1/3,1/2)$. For any such $V$ of loneliness $\mu$, Lemmas~\ref{lem:coalesce} and~\ref{lem:densityn} give, for the ordered successive minima of the projected lattice,
\[
 \sum_{i=1}^j\lambda_i\ge2(g_j-\mu)\quad(1\le j\le6),\qquad\|V\|_2\le(\lambda_1\cdots\lambda_6)^{-1}.
\]
Here $g_1,g_2$ use the two classifications, and the four other entries use the lonely runner theorem in dimensions at most four as in Proposition~\ref{prop:flag}.

For clarity we describe the small rational optimization used twice below. Given positive prefix bounds $b_j<1$, minimize $\prod_ix_i$ subject to $0<x_1\le\dots\le x_r$ and $\sum_{i\le j}x_i\ge b_j$. We may cap all entries at $1$: any prefix changed still sums to at least one. The feasible region is compact since $x_1\ge b_1>0$. The sum of the logarithms is concave, so its minimum occurs at a vertex. At a vertex split the vector into $m$ maximal constant blocks of lengths $s_1,\dots,s_m$. The $m$ common values must be determined by $m$ independent equalities chosen from the $r$ prefix bounds and $x_r=1$. Conversely it is harmless to test any such solution passing the original inequalities. The verifier \texttt{flag\_bounds.py} tries every ordered composition of $r$ and every such subset of rows. It solves the linear systems and compares the products over the rationals.

With $b_j=2(g_j-1/7)$ a minimizing vector is
\[
 \Bigl(\frac3{77},\frac3{77},\frac3{77},\frac{15}{154},\frac16,\frac13\Bigr).
\]
The reciprocal product is $140{,}612{,}164/45$. For any exceptional value $p/q$ the two-speed corner rule gives $q\le\sqrt2\,\|V\|_2$, and the integer $q-7p$ is at least two. These facts give the denominator bound and the asserted interval.

Finally suppose $1/8<\mu<1/8+1/(8Q)$. Near-tight regular sets in this interval would have $s\ge2$ and $\mu\ge2/15$. The finite table $E_7$ also has no value in this interval. The two subtorus lower bounds thus apply. Put $x=1/8+1/(8Q)$ and $u=2(4/25-x)$. With $b_j=2(g_j-x)$ the vertex check gives a minimizing vector
\[
 (u,\ u,\ u,\ u,\ 2(1/3-x)-4u,\ 1/3).
\]
For its reciprocal product $C$ we get $\lfloor\sqrt2\,C\rfloor=1{,}292{,}958<Q$, by exact squaring. Hence $q<Q$ in the reduced fraction $\mu=p/q$, contradicting $\mu-1/8=(8p-q)/(8q)\ge1/(8q)$. For repeated speeds, the ordinary lonely runner bound on at most six distinct speeds gives $\mu\ge1/7$.
\end{proof}

\begin{corollary}\label{A:cor6.2}
Under hypothesis \textup{(M)} of Section~\ref{sec:seven} on the seven-speed tight instances, every primitive eight-speed set with $\ML(V)<1/8$ and $\|V\|_2\ge8{,}210{,}857{,}637{,}041$ is $U_i\cup\{8s\}$.
\end{corollary}
\begin{proof}
Use the last assertion of Theorem~A6.1 in the proof of Corollary~\ref{cor:eight}(ii). Under \textup{(M)}, the only two-planes in eight coordinates with loneliness below $1/8+1/(8Q)$ are the three products there. Away from them, the successive prefix bounds for $r=7$ start with $1/(4Q)$, followed by $2(1/(9-j)-1/8)$ for $2\le j\le7$. Their successive differences are
\[
 \Bigl(\frac1{4Q},\ \frac1{28}-\frac1{4Q},\ \frac1{21},\ \frac1{15},\ \frac1{10},\ \frac16,\ \frac13\Bigr).
\]
By Lemmas~\ref{lem:densityn} and~\ref{lem:prodsum} their reciprocal product bounds $\|V\|_2$ from above. It equals $10{,}616{,}253{,}014{,}384{,}640{,}000/1{,}292{,}953$, strictly below the stated integer. On the products Proposition~\ref{prop:lines8} gives the result.
\end{proof}

\begin{corollary}\label{A:cor6.3}
Let $v_1<\dots<v_8$ be a primitive eight-speed set of loneliness below $1/8$ outside the families and exceptions of Theorem~A1.1. If $v_8\ge7v_7$, its seven smallest speeds form a primitive septuple outside $[6]\cup\{7s\}$ and $E_7$ with loneliness below $1/7$. In particular this octuple would supply a counterexample to the proposed complete seven-speed list.
\end{corollary}
\begin{proof}
The gcd of the seven smaller speeds is one by the lift argument of Theorem~A2.1, using the seven-speed lonely runner theorem at level $1/8$. Let $r$ be their loneliness. If $r\le1/8$ the strict inequality $r<1/7$ already holds. Otherwise their good set at level $1/8$ contains a closed interval of length $2(r-1/8)/v_7$. It must fit in a single open bad interval of length $2/(8v_8)$, so $r<1/8+v_7/(8v_8)\le1/7$. Every member of the regular septuple family contains $[5]$. Each row in $E_7$ has five elements at most $24$, by the table. Thus any of the known septuples would give $h_5(V)\le24$ and hence the asserted eight-speed classification.
\end{proof}

\subsection*{Common factors of near-tight speed sets}
The pre-jump of Bienia, Goddyn, Gvozdjak, Seb\H{o} and Tarsi \cite{Bienia} is particularly effective when the time before the jump can be chosen to control some of the omitted speeds. The case of one omission below is Fan and Sun \cite[Theorem~4]{FanSun}.

\begin{proposition}\label{A:prop6.4}
Let $V$ be a primitive set of $n\in\{7,8\}$ positive integer speeds with $\ML(V)<1/n$. The gcd of any $n-1$ members is $1$, and the gcd of any $n-2$ members is $1$ or $2$. The gcd $d$ of any $n-3$ members satisfies
\begin{center}
\begin{tabular}{c|l}
$n$ & $d$\\
\hline
$7$ & $1,2,3,4,5,6,9,10,12$\\
$8$ & $1,2,3,4,5,6,10$.
\end{tabular}
\end{center}
In particular, at most $n-4$ speeds are divisible by any given prime $p\ge7$.
\end{proposition}
\begin{proof}
Choose a subset $J$ of $n-r$ speeds, put $d=\gcd(J)$, and write $w$ for a speed outside $J$. Set $h_w=\gcd(w,d)$ and $q_w=d/h_w$. The lower-dimensional lonely runner bound, known through seven speeds, gives a real $s$ with $\|(v/d)s\|\ge1/n$ for every $v\in J$. The $d$ times $(s+j)/d$, $0\le j<d$, all retain these inequalities. The number of these times at which $w$ has distance strictly below $1/n$ is
\[
 h_w\,\#\{m\in\Z:|m+(w/h_w)s|<q_w/n\}.
\]
Indeed, multiplication by $w/h_w$ permutes residues modulo $q_w$, and an interval of length $2q_w/n<q_w$ meets each residue class at most once. Put $F(q)=\lceil 2q/n\rceil$. Since every lift must be excluded by an omitted speed,
\begin{equation}\tag{A.1}\label{A:eq1}
 d\le\sum_{w\notin J}h_wF(q_w).
\end{equation}

Up to $r-1$ auxiliary positive integer speeds may be included when choosing $s$, since the total number of speeds then remains at most $n-1$. Write $2q=bn+c$, $0\le c<n$. If $b$ is even and $c\in\{1,2\}$, the open interval in the preceding count can contain $b+1$ integers only if $\|(w/h_w)s\|<c/(2n)$. If $b$ is odd and $c=1$, having $b+1$ integers requires the distance of $(w/h_w)s$ from a half-integer to be strictly less than $1/(2n)$. These assertions follow by taking the midpoint of the $b+1$ consecutive integers. Thus choosing $w/h_w$ in the first case or $2w/h_w$ in the second as an auxiliary speed decreases the corresponding upper bound in \eqref{A:eq1} by $h_w$. Let $E(q)=1$ in these two cases and $E(q)=0$ otherwise. Subtract from the right hand side of \eqref{A:eq1} the largest $r-1$ members of the multiset $\{h_wE(q_w):w\notin J\}$. The resulting expression is still an upper bound for $d$.

For $r=1$ we have $h_w=1$ by primitivity, giving $d\le\lceil 2d/n\rceil$ and hence $d=1$. For $r=2$ both $h_w$ are $1$, by the $r=1$ conclusion. For $r=3$ each $h_w$ is $1$ or $2$, by the $r=2$ conclusion, and at most one equals $2$, since two such speeds together with $J$ would give $n-1$ even speeds. For transparency the finite arithmetic is given here. Since $F(q)<2q/n+1$,
\[
 d<\frac{n}{n-2r}\sum_{w\notin J}h_w.
\]
In each row let $h=(h_1,\dots,h_r)$, with $h_i\mid d$, and retain the positive $d$ below the indicated strict upper bound for which
\[
 d\le\sum_ih_iF(d/h_i)-\sum_{j=1}^{r-1}e_j,
\]
where $e_1\ge\dots\ge e_r$ is the decreasing rearrangement of $(h_iE(d/h_i))_{i=1}^r$.
\begin{center}
\begin{tabular}{c|c|c|l}
$n$ & $h$ & strict bound & $d$ retained\\
\hline
$7$ & $(1,1)$ & $14/3$ & $1,2$\\
$8$ & $(1,1)$ & $4$ & $1,2$\\
$7$ & $(1,1,1)$ & $21$ & $1,2,3,4,5,6,9,12$\\
$7$ & $(1,1,2)$ & $28$ & $2,4,6,10,12$\\
$8$ & $(1,1,1)$ & $12$ & $1,2,3,5,6$\\
$8$ & $(1,1,2)$ & $16$ & $2,4,6,10$
\end{tabular}
\end{center}
This proves all the assertions.
\end{proof}

\begin{proposition}\label{A:prop6.5}
Let $n\ge4$, $2\le d\le n$, and let $b=\max V$ be the only multiple of $d$ in a set $V$ of $n$ positive integer speeds. If $\ML(V)<1/n$, every unit class modulo $d$ contains a speed $a\in V\setminus\{b\}$ satisfying
\[
 a>(n/d-1)b.
\]
In particular $d>n/2$. For $n=d=7$ the six other speeds represent all six nonzero residue classes modulo $7$. For $n=8$, $d=7$, six different speeds greater than $b/7$ represent those classes.
\end{proposition}
\begin{proof}
For $d=2$ the single time $1/2+1/(nb)$ works for every speed: the odd speeds have distance at least $1/2-1/n\ge1/n$ and $b$ has distance $1/n$. Suppose $d\ge3$ and let $u$ be a unit modulo $d$. Examine $t=u/d+1/(nb)$. The speed $b$ has distance exactly $1/n$. For $a\ne b$, with $ua\bmod d$ taken in $\{1,\dots,d-1\}$, the shift to the right is $a/(nb)\le1/n$. If $ua\not\equiv-1\pmod d$ then the starting point is at least $1/d$ from $0$ and at least $2/d$ from $1$. After the shift it lies in $[1/n,1-1/n]$. For a speed in the remaining class to have distance less than $1/n$ requires $1/d-a/(nb)<1/n$, or $a>(n/d-1)b$. Since the full set is near-tight there must be such a speed. As $u$ varies, the classes $-u^{-1}$ range over all units.
\end{proof}

\emph{The remaining seven-speed question.} Any further primitive near-tight septuple must have maximum speed above $256$ by the census of Section~\ref{sec:seven}, Euclidean norm at most $140{,}612{,}164/45$, $h_4\ge27$, loneliness at least $1/5$ on each of its $21$ five-element subsets, and loneliness at least $1/6$ on each of its seven six-element subsets. It must satisfy $v_6<5v_5$ and $v_7<6v_6$ in increasing order. Excluding all tuples satisfying these simultaneous restrictions would prove the full classification and $C_7=27$ in Theorem~\ref{thm:count}. For eight speeds equality $C_8=17$ still requires excluding sets with $h_5\ge25$ outside the census. The unconditional Theorem~A1.1 and the conditional Corollary~A6.2 have different hypotheses. Supplement~C excludes all these septuples.

\section{Lonely times for superincreasing speeds}\label{A:sec7}
Let $v=(v_i)_{i\ge1}$ be a sequence of positive integers satisfying $v_i>\sum_{j<i}v_j$, and let $s_i\in\T$ be arbitrary starting positions. For $0<\delta<1/2$, write
\[
 F(v,s;\delta)=\{t\in[0,1]:\|v_it+s_i\|\ge\delta\text{ for every }i\}.
\]
We consider the infima of its Hausdorff dimension with arbitrary starting positions and with all runners starting at the origin, respectively,
\[
 d(\delta)=\inf_{v,s}\dim_{\mathrm H}F(v,s;\delta),\qquad d_0(\delta)=\inf_v\dim_{\mathrm H}F(v,0;\delta).
\]
The infima in both definitions range over infinite superincreasing sequences. The following bounds strengthen Corollary~\ref{cor:lacunary}(a). For related Hausdorff-dimension estimates for lacunary sequences, see Bugeaud and Moshchevitin \cite[\S6]{BugeaudMoshchevitin}.

\begin{theorem}\label{A:thm7.1}
For every $v,s$ as above, $F(v,s;1/69)$ is nonempty. Also,
\[
 d(1/71)\ge\frac17,\qquad d(1/72)\ge\frac{\log3}{\log128},\qquad d(1/97)\ge\frac{\log21}{\log128}.
\]
For $0<\delta<1/4$ one has
\[
 d(\delta)\le d_0(\delta)\le1+\log_2(1-2\delta).
\]
As $\delta\downarrow0$, each of $d(\delta)$ and $d_0(\delta)$ equals
\[
 1-\frac{2\delta}{\log2}+O(\delta^{5/4}),
\]
where $\log$ denotes the natural logarithm.
\end{theorem}
\begin{proof}
\emph{Nested intervals.} Fix an integer $B\ge3$ and $0<K\le1$. Start with an interval of length $K$ in $[0,1]$. At stage $j\ge0$, let $M=B^j$ and process the speeds in $[M,BM)$. In each parent interval of length $L=K/M$, seek closed children of length $\ell=L/B$ throughout which the current block satisfies the threshold. Earlier blocks already satisfy it throughout the parent.

Trim $\ell/2$ at either end of the parent to keep each child inside it. For each speed $v_i$ exclude centers within distance $\delta/v_i+\ell/2$ of a point $(a-s_i)/v_i$, $a\in\Z$. The remaining centers are admissible. If they occupy at least the fraction $g>0$ of the parent, a grid of spacing $\ell$ with a uniformly random translate in $[0,\ell)$ has at least $Bg$ admissible centers in expectation. Some translate therefore gives at least $m=\lceil Bg\rceil$ children with disjoint interiors. A uniform bound $m\ge1$ at every stage gives a common point by nesting. For $m\ge2$ retain $m$ children per parent and assign mass $m^{-j}$ to each level-$j$ interval. Equivalently, push forward the uniform measure on the space of paths in this tree. An interval of length at most $KB^{-j}$ meets at most three level-$j$ intervals, so it has mass at most $3m^{-j}$. Comparing consecutive scales gives Hausdorff dimension at least $\log m/\log B$.

\emph{Explicit thresholds.} Write $H=1-1/B$. In units of the parent length, the trimmed interval has length $H$. For a speed $x=v_i/M\in[1,B)$, put $z=Kx$. The forbidden centers have period $1/z$ and occupy length at most $w=2\delta/z+1/B$ per period. The trimmed interval contains $h=\lfloor Hz\rfloor$ complete periods and a remainder. Its forbidden length is at most
\[
 hw+\min\{w,\,H-h/z\}=\min\bigl\{(h+1)w,\ h/B+H-h(1-2\delta)/z\bigr\}.
\]
For fixed $h$ the two functions on the right are respectively decreasing and nondecreasing in $z$. They agree at $z_*=(h+2\delta)/(1-2/B)$. Consequently an upper bound for the loss at $x$ is
\[
 p(x)=(h+1)\Bigl(\frac{2\delta}{\max(z,z_*)}+\frac1B\Bigr).
\]
This bound is nonincreasing between successive jumps of $\lfloor HKx\rfloor$.

To estimate the sum of the losses, allow real $x\in[1,B)$ and relax superincrease to $x\ge y$, where $y$ is the sum of the preceding speeds in the block, in units of $M$. Let $U(y)$ be the supremum of the total remaining cost in this relaxed problem. Then $U(y)=0$ for $y\ge B$ and $U$ is nonincreasing, since increasing the initial sum only restricts the admissible continuations. For $0\le y<B$, set
\[
 a=\max(1,y),\qquad J(y)=\{a\}\cup\{j/(HK):j\in\N,\ a<j/(HK)<B\}.
\]
Since $p(x)+U(y+x)$ is nonincreasing on each interval between its possible upward jumps, its supremum is attained at one of the indicated left endpoints. Thus
\[
 U(y)=\max_{x\in J(y)}\{p(x)+U(y+x)\}.
\]
The recursion is finite, since the first speed is at least $1$ and the running sum thereafter at least doubles. The union bound gives $g\ge H-U(0)$. The following table evaluates the recursion in rational arithmetic, as in the verifier \texttt{dimension.py}.
\begin{center}\small
\begin{tabular}{c|c|c|c|c}
$\delta$ & $B$ & $K$ & $B-1-BU(0)$ & $m$\\
\hline
$1/69$ & $64$ & $13/64$ & $3158371406/51247973595$ & $1$\\
$1/71$ & $128$ & $7/36$ & $16365763749873/10072483745324$ & $2$\\
$1/72$ & $128$ & $7/36$ & $142197976240/54732153763$ & $3$\\
$1/97$ & $128$ & $7/36$ & $17099988305/837073237$ & $21$
\end{tabular}
\end{center}
These child counts give the claimed explicit bounds. The same conclusion on existence applies to a finite sequence by extending it superincreasingly.

\emph{The lower bound as $\delta\downarrow0$.} Choose $B=2^r$. If $w_1<\dots<w_N$ are the speeds of a nonempty block, then $w_i>2^{i-2}w_1$ for $i\ge2$ and $\sum_{i<N}w_i<w_N$. Hence
\[
 N\le r+1,\qquad\sum_{i=1}^Nw_i/M<2B,\qquad\sum_{i=1}^NM/w_i\le3.
\]
For speed $v_i$, at most $v_iL+2\delta+1$ lattice centers can exclude part of the trimmed interval. Each excludes length at most $2\delta/v_i+\ell$. Summing their products and including the initial trimming bounds the excluded fraction of the parent by
\[
 E=2\delta(r+1)+2K+\frac{6\delta(1+2\delta)}{K}+\frac{1+(r+1)(1+2\delta)}{B}.
\]
Take $K=\sqrt\delta$ and $r=\lceil\delta^{-3/4}\rceil$. Uniformly in the speeds and shifts, $E=2\delta r+O(\sqrt\delta)=O(\delta^{1/4})$. The branching bound with $g=1-E$ and $\log(1-E)=-E+O(E^2)$ gives
\[
 d(\delta)\ge1+\frac{\log(1-E)}{r\log2}\ge1-\frac{2\delta}{\log2}-O(\delta^{5/4}).
\]

\emph{The upper bound with all starting positions zero.} Fix $0<\delta<1/4$. Choose $v_i$ independently and uniformly among the integers in $J_i=[2^ii^2,2^i(i^2+i)]$, $i\ge1$. Every realization is superincreasing, since for $i\ge2$
\[
 \sum_{j=1}^{i-1}2^j(j^2+j)=2^i(i^2-3i+4)-4<2^ii^2.
\]
For $0<\epsilon<2\delta$, choose $\eta>0$ and a real trigonometric polynomial $P$ majorizing the indicator of the arc $A_\eta=[\delta-\eta,1-\delta+\eta]$ modulo one, with mean below $1-2\delta+\epsilon/2$. To obtain such a polynomial, enlarge the arc slightly, interpolate linearly outside it, approximate the resulting continuous majorant uniformly, and add the approximation error. Let $D$ be the degree of $P$. For each integer $b\ge1$ put
\[
 C_b=\{t\in[0,1]:\|ht\|\ge1/b\text{ for }1\le h\le D\}.
\]
Their union contains every irrational. For $t\in C_b$ and $1\le|h|\le D$, the absolute value of the average of $e^{2\pi ihvt}$ over $v\in J_i\cap\Z$ is at most $b/(2\#(J_i\cap\Z))$ by the geometric-sum formula. Averaging $P$ therefore gives, for all sufficiently large $i$ depending on $b$ and $P$,
\[
 \Pr(v_it\bmod1\in A_\eta)\le1-2\delta+\epsilon\qquad(t\in C_b).
\]

Partition $[0,1]$ into intervals of length $\ell_n=2^{-n}n^{-3}$, and choose a deterministic representative of $C_b$ in each interval meeting $C_b$. Retain a cell when its representative maps into $A_\eta$ under all $v_i$ with $i\le n$. For sufficiently large $n$ the retained cells cover $C_b\cap F(v,0;\delta)$, since $\ell_n\max J_n\le2/n<\eta$. Independence bounds the probability of retaining each cell by $O_{b,P}((1-2\delta+\epsilon)^n)$. For $u>1+\log_2(1-2\delta+\epsilon)$, the expected sum of the $u$th powers of the retained diameters is at most
\[
 O_{b,P}\Bigl(n^{3(1-u)}\bigl[2^{1-u}(1-2\delta+\epsilon)\bigr]^n\Bigr).
\]
This is summable over $n$, so the covering sums tend to zero almost surely. Take a countable intersection of these probability-one events over $b$, rational $u$ above the bound and a sequence $\epsilon\downarrow0$. Since rational times form a countable set, we obtain $d_0(\delta)\le1+\log_2(1-2\delta)$. Finally $d(\delta)\le d_0(\delta)$ by inclusion of the unshifted case, and expansion of the logarithm gives both asymptotics.
\end{proof}

\begin{proposition}\label{A:prop7.2}
For each $k\ge1$, the infimum $D_k$ over dissociated $k$-element sets is an attained rational value, and $D_{k+1}<D_k$. Both conclusions also hold for the infima over sets without odd signed relations.
\end{proposition}
\begin{proof}
Every accumulation point of the fixed-cardinality spectrum in the $\ML$ convention is approached only from below, by Giri and Kravitz \cite[Theorem~1.4]{GiriKravitz}, with their quantity $D=1/2-\ML$. Each of the two classes is nonempty at every cardinality, as seen from the powers of two. If the infimum for either class were unattained, a strictly decreasing sequence of values from the class would converge to it, contradicting the one-sided accumulation theorem. Rationality follows from the corner formula.

Let $V$ attain the minimum for either class. On $[0,1]$, the function $t\mapsto\min_{v\in V}\|vt\|$ has finitely many maximizing times, all rational, since it is piecewise linear with nonzero slopes and rational breakpoints. Choose an integer $w>\sum_{v\in V}v$ divisible by all their denominators. Then $\|wt\|=0$ at every maximizing time, so compactness gives $\ML(V\cup\{w\})<\ML(V)$. No signed relation can involve $w$, and the extension stays in the required class.
\end{proof}

\section{Finite exceptions and attaining times}\label{A:sec8}
The column $t$ gives an attaining time. An asterisk marks a set containing a six-speed root configuration. Together with $[7]$, the marked septuples give Theorem~A2.1.

\begin{table}[ht]
\centering\footnotesize\setlength{\tabcolsep}{3pt}
\begin{tabular}{lcc@{\quad}lcc}
\multicolumn{3}{c}{\textbf{The $27$ sets in $E_7$}} & \multicolumn{3}{c}{\textbf{The $17$ sets in $E_8$}}\\[2pt]
\toprule
$V$ & $\ML(V)$ & $t$ & $V$ & $\ML(V)$ & $t$\\
\midrule
$(1,2,3,4,5,7,12)^*$ & $1/8$ & $1/8$ & $(1,2,3,4,5,7,18,40)$ & $5/41$ & $5/41$ \\
$(1,2,3,4,5,7,18)^*$ & $3/23$ & $4/23$ & $(1,2,3,4,5,7,18,80)$ & $10/81$ & $10/81$ \\
$(1,2,3,4,5,7,24)^*$ & $4/29$ & $5/29$ & $(1,2,5,6,7,8,11,13)$ & $2/17$ & $6/17$ \\
$(1,2,5,6,7,8,39)$ & $3/22$ & $5/22$ & $(1,3,4,5,6,7,11,32)$ & $4/33$ & $4/33$ \\
$(1,3,4,5,6,7,11)^*$ & $2/15$ & $7/15$ & $(1,3,4,5,7,11,12,32)$ & $4/33$ & $4/33$ \\
$(1,3,4,5,6,9,28)^*$ & $4/29$ & $4/29$ & $(1,3,4,5,7,11,13,24)$ & $3/25$ & $3/25$ \\
$(1,3,4,5,7,11,12)$ & $2/15$ & $7/15$ & $(1,3,4,5,7,11,18,32)$ & $4/33$ & $4/33$ \\
$(1,3,4,5,7,11,18)$ & $2/15$ & $7/15$ & $(1,3,4,5,7,11,24,62)$ & $9/73$ & $34/73$ \\
$(1,3,4,5,7,11,24)$ & $2/15$ & $7/15$ & $(1,3,4,5,7,11,30,56)$ & $7/57$ & $7/57$ \\
$(1,3,4,5,7,11,30)$ & $5/37$ & $6/37$ & $(1,3,4,5,7,13,18,40)$ & $5/41$ & $5/41$ \\
$(1,3,4,5,7,11,36)$ & $6/43$ & $7/43$ & $(1,3,4,5,7,13,24,29)$ & $3/25$ & $3/25$ \\
$(1,3,4,5,7,13,18)$ & $3/23$ & $4/23$ & $(1,4,5,6,7,10,11,32)$ & $4/33$ & $4/33$ \\
$(1,3,4,5,7,24,71)$ & $11/78$ & $35/78$ & $(1,4,5,6,7,11,23,32)$ & $4/33$ & $4/33$ \\
$(1,3,4,5,13,14,18)$ & $2/15$ & $2/15$ & $(1,4,5,6,7,11,26,32)$ & $4/33$ & $4/33$ \\
$(1,4,5,6,7,10,11)^*$ & $2/15$ & $7/15$ & $(1,4,5,6,7,11,32,36)$ & $4/33$ & $4/33$ \\
$(1,4,5,6,7,11,13)^*$ & $1/8$ & $1/8$ & $(1,5,6,7,8,11,13,19)$ & $3/25$ & $12/25$ \\
$(1,4,5,6,7,11,16)^*$ & $2/15$ & $7/15$ & $(1,5,6,7,8,11,13,21)$ & $2/17$ & $6/17$ \\
$(1,4,5,6,7,11,23)^*$ & $4/29$ & $9/29$ &  & &  \\
$(1,4,5,6,7,11,26)^*$ & $2/15$ & $7/15$ &  & &  \\
$(1,4,5,6,7,11,29)^*$ & $5/36$ & $11/36$ &  & &  \\
$(1,4,5,6,7,11,39)^*$ & $2/15$ & $7/15$ &  & &  \\
$(1,4,5,6,7,11,52)^*$ & $4/29$ & $9/29$ &  & &  \\
$(1,4,5,6,7,11,65)^*$ & $10/71$ & $22/71$ &  & &  \\
$(1,5,6,11,16,17,112)^*$ & $16/113$ & $16/113$ &  & &  \\
$(2,6,7,8,10,13,14)$ & $2/15$ & $4/15$ &  & &  \\
$(2,7,9,11,12,13,20)^*$ & $4/29$ & $6/29$ &  & &  \\
$(2,7,9,11,13,20,54)^*$ & $9/65$ & $11/65$ &  & &  \\
\bottomrule
\end{tabular}
\caption{The finite exceptions $E_7$ and $E_8$ with attaining times.}\label{A:tab1}
\end{table}

%% file: suppB.tex
\supplement{B}{The remaining seven-speed cases}

Every primitive set of seven speeds with loneliness below $1/7$ and Euclidean norm at least $12{,}000$ is $\{1,2,3,4,5,6,7s\}$. Below maximum speed $329$, this family and $27$ exceptions form the complete list. Any missing set therefore has maximum speed between $329$ and $11{,}830$. The upper bound comes from four successive normal vectors and exact tests for balls inside the corresponding subtori. At every height we classify all septuples containing $d,2d,3d,4kd$ for any positive $d,k$. We also classify $630$ families of six linear forms with an arbitrary seventh speed. For eight speeds the norm cutoff under hypothesis \textup{(M)} is $3{,}898{,}150{,}815$.

\emph{Notation and scope.} Write $[k]=\{1,\dots,k\}$. For $V$ consisting of $n$ distinct positive integer speeds, put $\ML(V)=\max_{t\in\T}\min_{v\in V}\|tv\|$, where $\T=\R/\Z$ and $\|\cdot\|$ is distance to the nearest integer. A near-tight set has $\ML(V)<1/n$. Euclidean lengths of speed vectors use $\|V\|_2$. Numbers prefixed A refer to Supplement~A. Numbers without a prefix refer to Sections~1 to~\ref{sec:comp}. The full seven-speed classification is proved in Supplement~C. The proofs below use those earlier results and exact finite computations.

\section{Higher-dimensional subtori near their first value}\label{B:sec1}
Write $T_u=\{tu:t\in\T\}$ for the circle of an integer vector and set
\begin{gather*}
 B=\{(x,y,x+y,x+2y):x,y\in\T\},\\
 P=\{(y,4y,x,x-y,x-2y):x,y\in\T\}.
\end{gather*}
A product with $\T^0$ means the displayed torus itself. All classifications are up to permutations and sign changes of coordinates.

\begin{theorem}\label{B:thm1.1}
Let $S\subseteq\T^{d+c}$ be a coordinate-free torus of dimension $d\ge1$.
\begin{enumerate}[label=\arabic*.]
\item For $c=2$, the tori with $\ML(S)<2/7$ are $T_{(1,2,3)}\times\T^{d-1}$ and, for $d\ge2$, $B\times\T^{d-2}$. Both have loneliness $1/4$.
\item For $c=3$, the tori with $\ML(S)<3/13$ are $T_u\times\T^{d-1}$ for
\begin{center}
\begin{tabular}{c|l}
$\ML(u)$ & $u$\\
\hline
$1/5$ & $(1,2,3,4)$, $(1,3,4,7)$\\
$2/9$ & $(1,2,3,8)$, $(1,3,4,5)$, $(1,4,5,6)$, $(3,4,7,11)$\\
$5/22$ & $(1,7,8,15)$,
\end{tabular}
\end{center}
and, for $d\ge2$, the torus $P\times\T^{d-2}$, of loneliness $2/9$.
\end{enumerate}
The upper endpoints are attained in every stated dimension.
\end{theorem}
\begin{proof}
We first record two facts about coordinate supports. If the tangent space of $S$ contains a vector supported on a single coordinate, $S$ splits as the product of that coordinate circle and its projection, with the same loneliness. Suppose there is no such axis but there is a nonzero rational tangent vector supported on two coordinates. Write its integer entries as $p,q\ne0$. Deleting the two coordinates then leaves a coordinate-free torus of dimension $d-1$ and codimension $c-1$. By Lemma~\ref{lem:coalesce} its loneliness is at least $1/(c+1)$ in the two codimensions under consideration. In each fiber, the two deleted phases vary as $a+pt$, $b+qt$. Each open bad-time set at level $1/4$ has measure $1/2$. Their union cannot be the circle: equality would require two disjoint nonempty open sets to cover it. Thus for $c=3$ we have $\ML(S)\ge1/4$.

For $c=2$ the projected torus has loneliness at least $1/3$, and we claim $\ML(S)>1/4$. There is a nonempty open set in $S$ where all retained distances exceed $1/4$. If the two \emph{closed} bad sets in a fiber cover the circle, their interiors must be disjoint by measure. Their boundaries must then coincide, and the boundary sizes $2|p|,2|q|$ give $|p|=|q|$. In this case covering requires $b-(q/p)a=1/2\pmod1$: the two bad half-circles must be complementary after bringing the two slopes to the same sign. The last affine condition cannot hold on a nonempty open set of the connected torus $S$. The character on the left is either identically zero or has level sets with empty interior. Choosing a fiber avoiding the condition gives strict inequalities in the two deleted coordinates as well.

We use coordinate-free collision subtori obtained by setting two nonproportional coordinate forms equal up to sign. Every torus of dimension at least two contains two distinct such codimension-one subtori. Indeed, in the plane spanned by two coordinate forms take two different gaps between consecutive projective coordinate directions. The signed sum of suitable forms at the ends of each gap lies inside it. Its kernel is coordinate-free, and the two kernels differ (Lemmas~\ref{lem:coalesce} and~\ref{lem:repeat}). At either threshold in the statement, a global repeated coordinate is impossible, because deleting it gives loneliness at least $1/(c+1)$. On any of these collision subtori exactly two coordinates become equal up to sign: a second deletion would give the same lower bound. Hence each collision subtorus is obtained by repeating one row of a lower-dimensional model in the same codimension.

For $c=2$, Theorem~\ref{thm:full-intro} and the three-speed spectrum imply that $1/4$ is the only value below $2/7$. The unique primitive triple at $1/4$ is $(1,2,3)$, and Proposition~\ref{prop:list} lists the two-dimensional models at $1/4$ as the product and $B$ above.

For $c=3$, the one-dimensional list follows from Theorem~\ref{thm:main}. In its first family $k/(4k+1)<3/13$ means $k<3$. In the root family $m/(4m+2)<3/13$ gives $m<6$, and formula~\eqref{eq:m} gives $m\ge(x+y-2)/4$. It therefore suffices to evaluate the corner formula on the indicated bounded positive $x,y$ and the one sporadic quadruple. The resulting seven sets are those displayed in the theorem.

To classify the two-dimensional tori in five coordinates, choose two distinct collision lines. Fix the first with entries $u_1,u_2,u_3,u_4$ in increasing order followed by one repeated entry; the second ranges over all signed permutations of those forms with positive first entry. For rows $r_i$ of the resulting two-column matrix the circle defined by $r_i\pm r_j$ has speeds $\det(r_k,r_i\pm r_j)$. Reject a global duplicate, and, whenever all determinants are nonzero, require the primitive quartet after deleting row $j$ to be one of the seven listed above. Of $752{,}640$ pairs, $482$ pass. Of these $90$ have no coordinate axis, producing eight distinct labeled tangent spaces. All eight are signed permutations of $P$. To check this last assertion, express the rows in each ordered pair of independent rows as a basis and compare the resulting multiset up to signs with the displayed five forms of $P$. Those forms contain the two unit rows, so this comparison is exhaustive.

For completeness, at fixed $0\le y\le1/2$ the best possible minimum for $x,x-y,x-2y$ is half the longest circular gap among $0,y,2y$, namely $\max(y/2,1/2-y)$. If this and $\|y\|$ both exceeded $2/9$, then either $2/9<y<5/18$ or $4/9<y\le1/2$. In the former interval $\|4y\|<1/9$, and in the latter $\|4y\|<2/9$. Taking $y=4/9$ and $x=2/9$ gives equality in the upper bound for $P$.

In dimension three consider a torus with no coordinate axis whose loneliness is below the applicable threshold. A collision lift of any product model from dimension two either preserves an axis or duplicates the free coordinate. The latter case supplies a tangent vector supported on two coordinates. Both possibilities have already been excluded (for $c=2$, strict improvement above $1/4$ suffices). Thus two distinct collision planes must each duplicate one row of the root model, $B$ or $P$. The following small exact computation shows that their sums cannot give further tori. For either model order the rows as $(y,x,x+y,x+2y)$ or $(y,4y,x,x-y,x-2y)$. Fix the repeated row of the first plane, vary the second over all signed permutations of the possible duplicates, and keep sums having rank three. Identify equal labeled tangent spaces by their primitive vectors of maximal minors. The computation also discards spaces with a vector supported on at most two coordinates: deleting the supporting pair lowers rank precisely when all surviving maximal minors vanish.
\begin{center}\footnotesize
\begin{tabular}{c|r|r|r}
model and repeated row & rank-three sums & distinct sums & require a witness\\
\hline
$B_1$ & $272$ & $40$ & $0$\\
$B_2$ & $344$ & $44$ & $6$\\
$B_3$ & $272$ & $40$ & $0$\\
$B_4$ & $344$ & $44$ & $6$\\
\hline
$P_1$ & $444$ & $188$ & $114$\\
$P_2$ & $384$ & $162$ & $106$\\
$P_3$ & $428$ & $184$ & $120$\\
$P_4$ & $404$ & $176$ & $100$\\
$P_5$ & $428$ & $184$ & $120$
\end{tabular}
\end{center}
For the twelve listed $B$ cases a rational grid point gives a value greater than $1/4$; for all $560$ listed $P$ cases a rational grid point gives a value at least $3/13$. The stand-alone verifier \texttt{higher\_tori.py} generates the sums, finds each witness and checks the counts by integer arithmetic.

For $d\ge4$ any collision subtorus, with its repeated coordinate deleted, is a product from dimension $d-1$ by induction and has at least one free coordinate. Restoring the duplicate either leaves a coordinate axis or gives a vector supported on two coordinates. The second alternative has already been excluded for an axis-free candidate. Splitting off an axis reduces the classification to the previous dimension. All listed products attain the stated values. Finally the circles of $(1,2,6)$ and $(1,2,3,12)$ have loneliness $2/7$ and $3/13$ respectively; adding free factors gives the endpoint examples.
\end{proof}

\begin{corollary}\label{B:cor1.2}
Let $v$ be a primitive near-tight septuple outside the regular family and the twenty-seven listed exceptions. Every four-dimensional torus containing its circle has loneliness at least $2/9$. If one has loneliness less than $3/13$, five of the speeds have the form $|y|,|4y|,|x|,|x-y|,|x-2y|$ for integers $x,y$, after reordering. Every five-dimensional torus containing its circle with loneliness less than $2/7$ is a signed coordinate permutation of $T_{(1,2,3)}\times\T^4$ or $B\times\T^3$.
\end{corollary}
\begin{proof}
Each four-speed product in the theorem has four normalized speeds at most $15$. Theorem~A1.1 classifies all near-tight septuples containing such a set. The conclusion for five-dimensional tori is the codimension-two case of Theorem~B1.1.
\end{proof}

\section{A quotient-lattice bound for the remaining septuples}\label{B:sec2}
Let $V$ be primitive, with seven different positive speeds. Call a near-tight $V$ \emph{unlisted} if it is outside $[6]\cup\{7s\}$ and the $27$ sets $E_7$ in Supplement~A. We first prove an intermediate cutoff. The next reduction will use the actual normal spaces that can approach it.

\begin{theorem}\label{B:thm2.1}
Every unlisted septuple satisfies $\|V\|_2<34{,}860$. If the last shortest-vector step described below has a normal of squared length other than $3$ or $5$, the bound is $21{,}198$.
\end{theorem}

\begin{lemma}\label{B:lem2.2}
Let $T_v\subset W\subset\T^n$ be a primitive speed circle in a connected subtorus of dimension $r+1$. Write $L=\Z^n\cap E$, where $E$ is the tangent space of $W$. There is a flag $T_v=S_1\subset S_2\subset\dots\subset S_{r+1}=W$ and there are positive squared lengths $m_1,\dots,m_r$ such that
\[
 \prod_{i=1}^rm_i=\frac{\operatorname{covol}(L)^2}{\|v\|_2^2},\qquad m_{i+1}\ge\tfrac34m_i.
\]
If $S_{j+1}$ contains an intrinsic Euclidean ball of radius $R$ all of whose points have coordinate distances at least $u$, and $\ML(v)<u$, then $\sum_{i=1}^jm_i>4R^2$. In particular one may use $R=\ML(S_{j+1})-u$ when this is nonnegative.
\end{lemma}
\begin{proof}
First project $L$ orthogonally off $v$. Take a shortest nonzero vector in the quotient lattice and a lift $b_1\in L$. Project off its span as well, and continue. The shortest vector is primitive, so each step extends a basis of the previous saturated sublattice. At the end $v,b_1,\dots,b_r$ are a basis of $L$. Their orthogonalized vectors after $v$ have squared lengths $m_i$, giving the product formula. For consecutive steps, lift the second vector and subtract a multiple of the first to leave a parallel component of length at most $\sqrt{m_i}/2$. The preceding shortest-vector property gives $m_i\le m_{i+1}+m_i/4$.

Take any lift of the ball center in $\operatorname{span}_{\R}(v,b_1,\dots,b_j)$. In descending order, subtract an integer multiple of $b_i$ to bring the coefficient along its orthogonalized direction into $[-1/2,1/2]$. Later subtractions at smaller indices preserve that coefficient. The remaining displacement from $\R v$ has squared length at most $\frac14\sum_{i\le j}m_i$. The lattice subtractions preserve the torus point, and projecting onto $\R v$ gives a point of the speed circle. If the bound were at most $R^2$, the circle would meet the ball. The last assertion follows because each coordinate distance is $1$-Lipschitz.
\end{proof}

For $W=\T^7$ the product in the lemma is $1/\|V\|_2^2$. The classifications in Supplement~A and Theorem~B1.1 give the flag lower bounds
\[
 \ML(S_2)\ge4/25,\quad\ML(S_3)\ge2/11,\quad\ML(S_4)\ge2/9,\quad\ML(S_5)\ge1/4.
\]
The first two hold by Supplement~A even at arbitrary height for unlisted sets. The only additional four-torus below $3/13$ that can contain such a circle is $P\times\T^2$. The matrix of the two parameters of $P$ has Gram determinant $57$; the rows include $x$ and $y$, so its image basis is primitive. At $(x,y)=(7,23)/28$ with both free coordinates equal to $1/2$, the fixed coordinate distances are $5/28$, $2/7$, $1/4$, $3/7$, $11/28$. For a fixed row $(a,b)$ the squared length of its projected coordinate vector is $(22a^2+6ab+3b^2)/57$, by inverting the two-parameter Gram matrix. Direct squaring shows that each distance is at least $1/7$ plus $3/20$ times this projected length, also for the free coordinates. There is thus an intrinsic good ball of radius $3/20$. Using it in the three-step flag on this product and retaining the first two displayed gaps gives
\[
 \|V\|_2^2\le\frac{2403039912109375}{20365344}<10{,}863^2.
\]
For the intermediate global cutoff we may therefore use $\ML(S_4)\ge3/13$. In this paragraph and below, a repeated coordinate or a zero coordinate on a flag space is impossible since that space contains $T_v$.

\subsection*{The two final normal vectors}
Put $u=1/7$ and $C_j=\sum_{i=0}^{j-1}(4/3)^i$. Let $h$ be a primitive normal to $S_6$ and let $h,f$ be a basis of the saturated normal lattice to $S_5$. Put
\[
 N=h\cdot h,\quad t=h\cdot f,\quad M=f\cdot f,\quad D=NM-t^2.
\]
The squared covolumes of the tangent lattices of $S_6,S_5$ are $N,D$ respectively. This follows by taking primitive exterior products and applying the integral Hodge star to their normal lattices (equivalently, Section~\ref{subsec:flags}). The flag product gives $m_6=1/N$ and $m_5=N/D$. Applying the lemma at $(1/2,\dots,1/2)$ and using $\sum m_i\le C_6m_6$ gives $N<C_6/(1-2u)^2<28$. Signed coordinate permutations let us sort the absolute coordinates of $h$. Its $\ell^1$ norm is at least three. Replace $f$ by $f+kh$ so $2|t|\le N$.

Here is an exact procedure to obtain ball radii on either section. For the normal matrix $H$ with one or two rows, let $c_i$ be its columns and put $w_i^2=1-c_i^{\top}(HH^{\top})^{-1}c_i$. This is the squared length of the projection of the $i$th coordinate vector onto the tangent space. A ball of radius $R$ is good if its center has distances at least $u+Rw_i$. Put $a_i=1/2-u-Rw_i$. Provided $a_i\ge0$, such a center exists precisely when the zonotope
\[
 \sum_i[-2a_i,2a_i]c_i
\]
contains a vector $z\in\Z^{\operatorname{rank}H}$ with $z\equiv He\pmod 2$, where $e=(1,\dots,1)$: write the center as $e/2+y$ with $|y_i|\le a_i$ and impose $H(e/2+y)\in\Z^{\operatorname{rank}H}$. Saturation ensures this describes the connected section.

For $H=h$ it suffices to use $z=0$ for even parity and $z=1$ for odd parity. For $H=(h,f)$ the test in each facet-normal direction is
\[
 |\det(c_i,z)|\le\sum_j\bigl(1-2u-2Rw_j\bigr)|\det(c_i,c_j)|.
\]
These inequalities describe a planar zonotope, whose edge directions are its generators. The verifier enumerates all parity points in $|z_1|\le\|h\|_1$, $|z_2|\le\|f\|_1$, identifying $z$ and $-z$. It rounds $w_i$ upwards to millionths and the resulting nonnegative radius down to millionths. Thus every stored positive radius is certified. If the search gives zero we use only the trivial nonnegative bound. Denote the results by $R_1,R_2$ for the one- and two-normal sections.

Since $4R_1^2<C_5m_5$, one need only test $f$ with
\[
 M=1/m_5+t^2/N\le C_5/(4R_1^2)+N/4.
\]
All one-normal sections in the search have $R_1>0$. The test retains primitive pairs by the gcd of the two-row minors and rejects spaces forcing a zero coordinate or equal absolute coordinates. Permutations preserving $h$, and signs where $h_i=0$, reduce the enumeration. It reports $431{,}305$ reduced pairs and $409{,}106$ admissible primitive presentations before the remaining length tests. At the center of the whole torus the last rounding displacement has squared length $0$ for even $\sum_ih_i$ and $m_6/4$ for odd sum. It is orthogonal to all previous components. Since $\sum_ih_i\equiv N\pmod 2$, the lower bounds for the four unknown-length prefixes are therefore
\begin{gather*}
 b_j=4(g_j-u)^2,\quad(g_1,g_2,g_3)=(4/25,2/11,3/13),\\
 b_4=\max\{4(1/4-u)^2,\ 4R_2^2,\ 4R_1^2-m_5,\\
 \qquad\qquad(1-2u)^2-(N\bmod2)m_6-m_5\}.
\end{gather*}
We also require $4D\ge3N^2$ and $m_4\le4m_5/3$. For fixed $N,D$, taking the smallest $b_4$ among all presentations is sufficient.

\subsection*{Exact product minimization}
For any given prefix bounds, minimize $\prod_{i=1}^jm_i$ subject to those bounds, $m_{i+1}\ge3m_i/4$, and the last upper bound. The feasible set is compact: the first prefix is positive and the last upper bound propagates backwards. For an intrinsic flag without a last upper bound, capping every variable at one preserves all these inequalities when every prefix bound is at most one, and can only decrease the product. Thus a cap of one suffices there. By concavity of $\sum\log m_i$ it suffices to test vertices. At a vertex group consecutive entries joined by $m_{i+1}=3m_i/4$. The initial values of the blocks are fixed by as many independent tight prefix constraints or final upper bounds as there are blocks. Trying all ordered block partitions and all such choices is an exact rational algorithm.

With four unknown entries this gives the largest squared norm bound
\[
 \frac{8713202129754161834716796875}{7170347468731601943}<34{,}860^2
\]
at $(N,D)=(3,26)$. For $N=5$ the ceiling is $30{,}512$; over the other $N\ne3,5$ it is at most $21{,}198$ using this calculation. The exceptional intrinsic four-torus bound was accounted for separately. Together with the intrinsic cutoff $10{,}863$, this proves Theorem~B2.1.

\section{A third normal and the residual norm cutoff}\label{B:sec3}
\begin{theorem}\label{B:thm3.1}
Every unlisted primitive near-tight septuple satisfies
\[
 \|V\|_2^2\le\frac{10{,}325{,}562{,}122{,}344{,}970{,}703{,}125}{40{,}824{,}480{,}836{,}736}<15{,}904^2.
\]
If $\|V\|_2\ge12{,}000$, the last three normal spaces of its flag have, up to the symmetries used below, one of $195$ presentations determined in the proof.
\end{theorem}
\begin{proof}
Put $u=1/7$. Use the rank-four bound $g_3=2/9$ from Corollary~B1.2, so this argument applies also to the exceptional four-torus $P\times\T^2$. Apply the exact four-variable vertex calculation of Section~B2 to each last-two-normal presentation. Discard one when its squared norm bound is below $12{,}000^2$. A pair $(h,f)$ specifies the first normal and the saturated plane of normals, so we identify presentations with the same $h$ and the same $h\wedge f$ up to sign. Of the enumerated presentations $511$ distinct flags remain. Their counts by $N=\|h\|_2^2$ are
\begin{center}
\begin{tabular}{c|ccccccc}
$N$ & $3$ & $4$ & $5$ & $6$ & $7$ & $8$ & $9$\\
\hline
flags & $170$ & $3$ & $97$ & $29$ & $110$ & $30$ & $72$
\end{tabular}
\end{center}

Extend $h,f$ by $k\in\Z^7$ to a basis of the saturated normal lattice of $S_4$. With $N,D,t=h\cdot f$ as in the preceding section, write $\Delta=\|h\wedge f\wedge k\|_2^2$, the determinant of the three-row Gram matrix. Therefore $m_4=D/\Delta$, and $4N\Delta\ge3D^2$ follows from $m_5\ge3m_4/4$. Subtract multiples of $f$ and then $h$ from $k$ to obtain
\[
 2\,|(Nf-th)\cdot k|\le D,\qquad2|h\cdot k|\le N.
\]
In its orthogonal decomposition along $h$, $f-(t/N)h$ and their common orthogonal complement, the first two contributions to $\|k\|_2^2$ are at most $N/4$ and $D/(4N)$. The last equals $\Delta/D=1/m_4$. Since $4R_2^2\le C_4m_4$ with $C_4=175/27$, each such $k$ occurs in the finite search
\[
 \|k\|_2^2\le\frac{175}{108R_2^2}+\frac14(N+D/N).
\]
All $511$ initial radii $R_2$ are positive. The search keeps vectors with primitive three-row exterior product, rejects tangent spaces forcing zero or equal absolute coordinates, and uses permutations of equal $(h_i,f_i)$ and sign changes where both entries are zero.

For the normal matrix $H=(h;f;k)$ let $c_i$ be its columns, and $w_i^2=1-c_i^{\top}(HH^{\top})^{-1}c_i$. In the parity-coset ball test of Section~B2, a complete set of directions for the facet inequalities is given by the nonzero $n=c_i\times c_j$. Indeed every facet of a three-dimensional zonotope is spanned by its generators orthogonal to the facet normal. Thus, with $a_i=1/2-u-Rw_i\ge0$, the test at the parity point $z$ is
\[
 |n\cdot z|\le\sum_i(1-2u-2Rw_i)|n\cdot c_i|\quad\text{for all such }n.
\]
The inequalities also hold by continuity when some $a_i=0$. Only $|z_j|\le(5/7)\sum_i|H_{ji}|$ need be considered. Rounding the weights upward and each feasible radius downward to millionths gives a certified lower radius $R_3$. For each fixed initial flag and $\Delta$ the verifier retains the least of these certified radii.

The three remaining squared lengths have prefix lower bounds
\begin{gather*}
 4(4/25-u)^2,\qquad4(2/11-u)^2,\\
 b_3'=\max\{4(2/9-u)^2,\ 4R_3^2,\ b_4-D/\Delta\},
\end{gather*}
where $b_4$ is the fourth-prefix bound of Section~B2. Also $m_3\le4D/(3\Delta)$ and $\|V\|_2^2=\Delta/(m_1m_2m_3)$. The $99{,}852$ radius-table rows give $12{,}123$ grouped prefix tests after taking the least $b_3'$ at fixed $(D,\Delta)$ and applying the elementary last-cap test. Exact vertex enumeration gives the largest squared norm bound in the statement at $(N,D,\Delta)=(3,9,132)$. The radii in millionths are $(233284,153736,155520)$, and a minimizing vertex is
\[
 \Bigl(\frac{36}{30625},\ \frac{2592}{529375},\ \frac{5{,}250{,}061{,}836}{57{,}900{,}390{,}625}\Bigr).
\]
Any flag removed at the initial stage already has norm below $12{,}000$. Finally the same test at $12{,}000^2$ leaves $30$ initial-flag and determinant combinations. Re-evaluating the radius for each individual third normal, and identifying its primitive exterior product up to sign within its initial flag, leaves $195$ presentations. The source files \texttt{third\_selection.py}, \texttt{third\_normals.cpp} and \texttt{third\_bounds.py} reproduce the selection, radii, and bounds.
\end{proof}

\begin{corollary}\label{B:cor3.2}
Every near-tight primitive septuple of Euclidean norm at least $12{,}000$ contains three distinct speeds $a,b,a+b$.
\end{corollary}
\begin{proof}
Each of the $195$ last normals has exactly three nonzero entries, all of absolute value one. Its equation on distinct positive speeds forces the claimed triple. Every member of the regular family contains $1,2,3$, and the same Schur test holds for all $27$ listed exceptions. The verifier \texttt{verify\_third\_flags.py} checks the normal and independently recomputes the radius of every listed flag from its exterior minors.
\end{proof}

\section{The fourth normal}\label{B:sec4}
\begin{theorem}\label{B:thm4.1}
Every primitive seven-speed set $V$ with $\ML(V)<1/7$ and $\|V\|_2\ge12{,}000$ is of the form $[6]\cup\{7s\}$. In particular, every unlisted primitive near-tight septuple satisfies $\|V\|_2<12{,}000$.
\end{theorem}
\begin{proof}
Suppose an unlisted set has norm at least $T=12{,}000$. The preceding third-normal calculation leaves $195$ presentations of its last three normals. Write $K$ for the three-row normal matrix of one presentation and let $N,D,E$ be the successive Gram determinants of its first one, two and three rows. Thus $m_4=D/E$. Extend its saturated normal lattice by a fourth normal $g$ for $S_3$, with Gram determinant $F$. Then
\[
 m_3=E/F,\qquad\|V\|_2^2=F/(m_1m_2),\qquad4DF\ge3E^2.
\]
Let $R_3$ denote the certified ball radius on $S_4$. With $b_4$ as in Section~B2, we may use the strict three-variable prefix bound
\[
 B_3=\max\{4R_3^2,\ b_4-D/E,\ 4(3/13-1/7)^2\}.
\]
The exceptional four-torus in that section already has norm below $10{,}863$. From $m_1+m_2+m_3\le(37/9)m_3$ we obtain $F\le(37/9)E/B_3$.

Reduce $g$ successively along the third, second and first normal, making each coefficient in their orthogonalized directions at most $1/2$ in absolute value. This leaves
\[
 \|g\|_2^2\le F/E+\tfrac14(N+D/N+E/D).
\]
Integer vectors within this cap can be tested directly. Coordinates with identical columns of $K$ may be sorted within their class. A sign may also be chosen when the column of $K$ is zero.

\emph{The required radius.} For each remaining integer $F$, put $P=F/T^2$, $c=4E/(3F)$, $a=36/30625$, and $\rho=3/4$. If $x=m_1$, $y=m_2$, our assumptions imply
\[
 x\ge a,\qquad\rho x\le y\le c,\qquad xy\le P.
\]
The maximum possible sum $S=x+y$ on this region is found by maximizing $x+\min(c,P/x)$ between $a$ and $\min(c/\rho,\sqrt{P/\rho})$. On either side of $x=P/c$ the function is affine or convex, so only the endpoints and the crossing, when present, need testing. The only irrational endpoint value has square $49P/12$. Thus $S^2$ is rational and is compared exactly. If $\max\{4(2/11-1/7)^2,\ B_3-E/F\}>S$, the determinant is excluded. Otherwise take the least integer $r$ satisfying $16(r/10^6)^4>S^2$. A good ball of radius $r/10^6$ on $S_3$ now rules out this case by Lemma~B2.2.

\emph{Finding a ball or a short four-speed subset.} Let $w_i$ be the length of the projection of the $i$th coordinate vector onto the tangent space of $S_4$, rounded up to a millionth. The projection length for $S_3\subset S_4$ is at most $w_i$. For $p\in\{5,7,11,13,17,19\}$ enumerate the kernel of $K$ over $\mathbb F_p$, a space of dimension four by saturation. If the reduction of $g$ vanishes on this kernel, the four-row matrix has deficient rank modulo $p$, contradicting saturation of the four-normal lattice. Otherwise test the vectors $z$ in the kernel with $g\cdot z=0\pmod p$. Any such $z/p$ lies on the connected torus $S_3$. If
\[
 \|z_i/p\|\ge1/7+(r/10^6)w_i\qquad(1\le i\le7),
\]
it is a center of the required ball. All comparisons use integers.

The calculation examines $109{,}752$ reduced vectors after the determinant tests. Orthogonal projection eliminates spaces forcing a coordinate to vanish or two coordinates to be equal up to sign. Of the remaining $107{,}672$ vectors, the modular tests eliminate $107{,}608$. In each of the last $64$ cases the normal matrix has exactly five nonzero columns. The five associated speeds are proportional to the absolute values of its five signed maximal minors, since this four-by-five submatrix has a one-dimensional kernel. In every case four of these integers have maximum at most $26$ after division by their common gcd. The four-speed subset classification of Supplement~A therefore applies. This exhausts all possibilities for an unlisted set. The $27$ finite listed exceptions have maximum speed at most $112$, hence norm less than $12{,}000$.

The program \texttt{fourth\_thresholds.py} generates the determinant thresholds. The program \texttt{fourth\_normals.cpp} performs the vector and modular tests. The counts for each input are recorded in \texttt{fourth-flags.txt}.
\end{proof}

\section{A smaller bound for a septuple without a Schur triple}\label{B:sec5}
\begin{theorem}\label{B:thm5.1}
Every primitive seven-speed set $V$ with $\ML(V)<1/7$ and $\|V\|_2\ge6{,}000$ contains three distinct speeds $a,b,a+b$.
\end{theorem}
\begin{proof}
Suppose that a counterexample exists. A signed vector with exactly three nonzero entries, all of magnitude one, cannot be a normal to any flag space containing $T_v$. Indeed its equation on three distinct positive speeds forces one to be the sum of the other two. Every previously listed septuple contains such a triple.

We may use the higher flag thresholds $g_3=3/13$ and $g_4=2/7$ without a height qualification. In codimension three, the four-speed products listed in Theorem~B1.1 force a classified small four-speed subset. The remaining low torus $P\times\T^2$ has the relation $(x-y)+y=x$. In codimension two the two low models contain respectively the three coordinate forms $t,2t,3t$ and $x,y,x+y$. Sign changes preserve the existence of a signed three-coordinate relation.

Repeat the last-normal calculation at $T=6{,}000$ with these thresholds. Discard a normal space whenever it contains one of the $140$ signed three-coordinate vectors, taken up to simultaneous negation. Membership is tested over the rationals in the normal row span. The selection leaves $1{,}325$ last-two flags. The third-normal computation gives $294{,}193$ grouped radius rows. The two- and three-prefix bounds leave $428$ initial-flag and determinant combinations to recheck individually, and $1{,}968$ concrete third flags reach the threshold. Of these, $13$ force a Schur relation, leaving $1{,}955$.

Apply the fourth-normal test of Section~B4 with $T=6{,}000$. In its determinant range using the weaker lower bound $g_4=1/4$ remains valid. Among $218{,}956$ vectors, $201{,}670$ pass the coordinate tests and the modular tests exclude $201{,}662$. The eight remaining presentations all have fourth Gram determinant $424$ and require a ball radius $52903/10^6$. To certify them, use the actual coordinate projection lengths on the three-dimensional section rather than their four-dimensional upper bounds. If $H$ is the four-row normal matrix, then by Cauchy--Binet
\[
 w_i^2=\frac1{424}\sum_{|I|=4,\ i\notin I}\det(H_I)^2.
\]
For every triple of linearly independent columns take their normal $n\in\Z^4$ consisting of their signed maximal minors. With $R=52903/10^6$, test
\[
 |n\cdot z|\le\sum_{j=1}^7(5/7-2Rw_j)|n\cdot c_j|,\qquad z\equiv H(1,\dots,1)^{\top}\pmod 2.
\]
These are the facet inequalities of the four-dimensional zonotope in the proof of Section~B2, with weights rounded upward to millionths. The witness pairs, in order of the saved third-flag identifiers, are
\begin{center}
\begin{tabular}{c|c}
$1502$ & $(-2,-1,1,\pm1)$\\
$1511$ & $(-2,-1,-1,\pm1)$\\
$1542$ & $(-2,1,-1,\pm1)$\\
$1547$ & $(-2,1,1,\pm1)$.
\end{tabular}
\end{center}
The programs \texttt{schur\_flags.py} and \texttt{schur\_fourth.py} reproduce the selections and verify all eight witnesses directly. This excludes every counterexample of the stated norm.
\end{proof}

\section{A larger census by maximum speed}\label{B:sec6}
For a septuple $V$ write $H(V)=\max V$. The following test covers every choice of six smaller speeds for each new value of $H(V)$.

\begin{proposition}\label{B:prop6.1}
For $H(V)\le328$, the primitive sets of seven distinct positive integers with $\ML(V)<1/7$ are the $27$ members of $E_7$ and the $46$ sets $\{1,2,3,4,5,6,7s\}$ with $7s\le328$. There are $73$ primitive sets and $586$ sets including all dilates. In particular, a further primitive exception must have maximum speed at least $329$.
\end{proposition}
\begin{proof}
The census through $256$ is Proposition~\ref{prop:computed}. Fix a maximum $h\in[257,328]$, select $h$ initially, and allow the six remaining speeds to range over $[1,h-1]$. For each reduced rational $t=k/d\in(0,1/2]$ with $d\le150$, form the set of speeds $a$ with $7\min(ak\bmod d,\,-ak\bmod d)<d$. Every near-tight set must meet each such set. We may discard duplicate clauses and supersets of other clauses.

At any node select a clause disjoint from the chosen speeds. Branch over its available elements in order, ruling out all earlier elements in the branch that selects a given one. These branches partition the possible completions. For speed, the program chooses pivots among $256$ clauses. If these have all been met it branches on all available speeds in order. With one speed left it intersects the available choices with every unmet cached clause. Completed sets are tested against all retained clauses.

For each remaining septuple, test the times $k/(a+b)\in[0,1/2]$ for $a,b\in V$, allowing $a=b$. A maximum at the peak of an individual distance function has denominator dividing $2a$. At any other positive maximum, two active distances have opposite slopes, and their speeds sum to a multiple of the reduced denominator of the maximizing time. Consequently the existence of a good time at threshold $1/7$ is decided exactly by these tests. Equality excludes a candidate. We count all surviving sets and use the gcd-one records for the primitive count.

The new interval contributes $137$ near-tight sets including dilates, of which $10$ are primitive. The latter are precisely the new members of the stated product family. The census program records the counts for every height, and a separate check sums the dilates of each primitive set.
\end{proof}

\section{Seven-speed completions of four-speed families}\label{B:sec7}
The infinite four-speed families in Theorem~\ref{thm:main} force an additional short relation on any near-tight seven-speed completion. The common factor of the four speeds and the sizes of the three new speeds are unrestricted.

\begin{theorem}\label{B:thm7.1}
Let $V$ consist of seven distinct positive integer speeds with $\ML(V)<1/7$.
\begin{enumerate}[label=\arabic*.]
\item Suppose $V$ contains $x$, $y$, $x+y$, $x+2y$, with $H=\max(x,y)/\gcd(x,y)$. If $H>27$, one of the three other speeds $w$ satisfies
\[
 qw=ax+by,\qquad1\le q\le27,\qquad|a|,|b|\le24.
\]
If $H>36$, there is also such a relation with $q\le21$ and $|a|,|b|\le32$. The two relations may use different speeds.
\item Suppose $V$ contains $d,2d,3d,4kd$. For $k\ge7$ one of the three other speeds satisfies
\[
 qw=d(a+4kb),\qquad1\le q\le24,\quad|a|\le24,\quad|b|\le4.
\]
For $k\ge10$ there is also such a relation with $q\le21$, $|a|\le32$ and $|b|\le5$.
\end{enumerate}
\end{theorem}
\begin{proof}
Put $I=[1/7,2/7]$ and $c=3/14$. If $(X,Y)\in I^2$, all four forms $X,Y,X+Y,X+2Y$ lie in $[1/7,6/7]$. If $X\in I$ and $\|Y\|\ge1/7$, the same assertion about distances holds for $X,2X,3X,Y$.

We use the periodic Selberg majorant: for an interval $J$ on the circle and an integer $Q\ge1$, there is a real trigonometric polynomial $P_Q\ge\mathbf 1_J$ of degree at most $Q$ with constant coefficient $|J|+1/(Q+1)$ \cite[Appendix~E, Theorem~E.4]{MontgomeryVaughanII}. Take $J=[-1/7,1/7]$.

Fix $K,m$, and extend $p_j=\binom{j+m}{m}\binom{K-j+m}{m}$ for $0\le j\le K$ by zero on the rest of $\Z$. Write
\[
 F(z)=\frac{\bigl|\sum_jp_je^{2\pi ij(z-c)}\bigr|^2}{\sum_jp_j^2}.
\]
This nonnegative degree-$K$ polynomial has integral $1$. Outside $I$, $|\sin\pi(z-c)|\ge\sin(\pi/14)>2/9$. Parseval and the formula for the $m$th finite difference give
\[
 \int(9/2)^{2m}\sin^{2m}(\pi(z-c))F(z)\,dz=\eta_{K,m}:=(9/2)^{2m}\frac{\sum_j(\Delta^mp_j)^2}{4^m\sum_jp_j^2}.
\]
The integrand majorizes $\mathbf 1_{\T\setminus I}F$ and has degree at most $K+m$.

For the first family equip the time circle with density $F(xt)F(yt)$. If $H>K+m$, no nonzero relation $ix+jy=0$ has both coefficients of magnitude at most $K+m$. Consequently this density integrates to $1$, and each of the events $xt\notin I$, $yt\notin I$ has probability at most $\eta_{K,m}$ by the preceding polynomial majorant and orthogonality. If a remaining speed $w$ has no relation with $1\le q\le Q$, $|a|,|b|\le K$, then orthogonality also gives
\[
 \int F(xt)F(yt)P_Q(wt)\,dt=2/7+1/(Q+1).
\]
Indeed a nonconstant Fourier term contributing to the integral would provide exactly such a relation or a relation between $x$ and $y$. If all three remaining speeds lacked a relation, the bad events would cover the circle with total probability at most $2\eta_{K,m}+3(2/7+1/(Q+1))$.

For the second family set $X=dt$, $Y=4kdt$, and use instead the density $F(X)L_h(Y)$, where
\[
 L_h(z)=\frac{2^h}{\binom{2h}{h}}(1-\cos(2\pi z))^h.
\]
Its second factor is a nonnegative degree-$h$ polynomial of mean $1$. On $\|z\|<1/7$ it is at most $\rho_h=(2/5)^h2^h/\binom{2h}{h}$ since $\cos(2\pi/7)>3/5$. When $4k>K+m$ the integrals of all products used for the first two bad events again factor by orthogonality. If no extra speed supplied $qw=d(a+4kb)$ with $|a|\le K$, $|b|\le h$, $1\le q\le Q'$, the total bad probability would be at most $\eta_{K,m}+\rho_h+3(2/7+1/(Q'+1))$. The following choices make both bounds strictly below $1$.
\begin{center}
\begin{tabular}{cc|c|ccc}
$K$ & $m$ & $\eta_{K,m}$ is less than & $Q$ & $h$ & $Q'$\\
\hline
$24$ & $3$ & $0.01678$ & $27$ & $4$ & $24$\\
$32$ & $4$ & $0.00266$ & $21$ & $5$ & $21$
\end{tabular}
\end{center}
The verifier \texttt{four\_kernel.py} evaluates the finite differences and the two strict inequalities as rational numbers. In either family a good time would follow from total bad probability below one, proving the required relation.
\end{proof}

\begin{theorem}\label{B:thm7.2}
Take the four positive speeds $x,y,x+y,x+2y$. Append the absolute values of any two integer forms $ax+by$ with $\max(|a|,|b|)\le4$ and one arbitrary positive integer. Among primitive sets of seven distinct positive speeds obtained this way, the sets with loneliness below $1/7$ are precisely $[6]\cup\{7s\}$, $s\ge1$, and $25$ of the $27$ entries of $E_7$. The two entries absent are
\[
 (1,3,4,5,7,24,71),\qquad(2,6,7,8,10,13,14).
\]
There is no upper bound on the initial parameters or the last speed.
\end{theorem}
\begin{proof}
Let $U$ be the two-torus defined by the first six forms. Among primitive near-tight septuples every six speeds have gcd one, by the lift argument of Supplement~A, so $\gcd(x,y)=1$. Let $\beta=\ML(U)$ and write $M$ for the largest absolute second coefficient of the six forms. By the collision-line argument of Lemma~\ref{lem:coalesce}, $\beta\ge1/6$. At a maximizing point $(X_0,Y_0)$, the entire vertical segment $X=X_0$, $|Y-Y_0|\le(\beta-1/7)/M$ is good for the six forms. Put $L=2(\beta-1/7)/M$.

On the $x$ times with $xt=X_0$ modulo one, the values $yt$ form a translate of the $x$-point grid. At least $N=\lfloor xL\rfloor$ consecutive grid points meet the segment. Let $w$ be the seventh speed and write $w=jx+ky$, choosing $k$ as the balanced residue modulo $x$. This is possible since $(x,y)=1$. Along the ordered grid the values of $wt$ then advance by $k/x$ modulo one. If all $N$ points are bad for $w$, their representatives in $(-1/7,1/7)$ advance by the actual difference $k/x$, since $2/7<1/2$. Hence $|k|(N-1)<2x/7$ whenever $N\ge2$. For any positive integer $K$ with $KL>2/7$, the condition $x\ge2K/(KL-2/7)$ implies $N\ge2$ and $|k|<K$.

The seven-form plane obtained by adjoining $jX+kY$ to $U$ has loneliness at least $3/19$ by the seven-plane catalogue of Supplement~A. Indeed the four original rows have four distinct directions, while every product plane in that catalogue has six proportional rows. At a maximizing point of this extended plane fix the first coordinate. Of the $x$ available positions of the second coordinate, one differs from the target by at most $1/(2x)$. Thus
\[
 \ML(V)\ge3/19-\max(M,K-1)/(2x).
\]
As $3/19-1/7=2/133$, near-tightness requires
\[
 x<\max\Bigl(\frac{2K}{KL-2/7},\ \frac{133}{4}\max(M,K-1)\Bigr).
\]
Interchanging the coordinates gives the same bound for $y$ with $M$ the maximum first coefficient. The program takes the smallest integer ceiling of these upper bounds over $1\le K\le100$ with positive denominator.

There are $40$ nonzero coefficient pairs up to sign in the box. Removing the four prescribed rows leaves $36$, giving $\binom{36}{2}=630$ models. To compute each $\beta$ the verifier uses the finite vertex grids of the seven-plane calculation of Supplement~A, with the six rows in place of seven. An optimizing height lies strictly between $0$ and $1/2$ since the forms include $X,Y,X+Y$. Three independent active constraints give the required grid denominator by Cramer's rule.

For the resulting bounded coprime pairs we obtain $11{,}906{,}349$ visits to six-speed sets with distinct nonzero entries. By the large-gap theorem of Supplement~A any last speed $w\ge6\max S$ for such a six-set $S$ is already classified: the regular case $S=[6]$ or the single septuple $(1,5,6,11,16,17,112)$. For $w<6\max S$ a good time $k/d$ for the seven speeds with $d\le128$ excludes that extension. Exactly $853$ candidate visits survive this test. Intersecting all six closed good-time sets on $[0,1/2]$ and testing whether the remaining speed is strictly bad on each component gives the stated list. The entire regular family occurs since $(x,y)=(1,2)$ provides the prescribed four speeds $1,2,3,5$, and $2y,3y$ supply $4,6$. The files \texttt{four\_planes.cpp} and \texttt{four-planes-counts.txt} give the model-by-model test and bounds.
\end{proof}

\begin{proposition}\label{B:prop7.3}
A primitive seven-speed set of loneliness below $1/7$ containing $x,y,x+y,x+2y$ with $\max(x,y)/\gcd(x,y)\le64$ is either $[6]\cup\{7s\}$ or a member of $E_7$. Every displayed family member and every member of $E_7$ occurs, even if one requires the four-speed subset itself to have loneliness below $1/4$ and its normalized parameter height to be at most $9$.

If the set contains $d,2d,3d,4kd$ for any integers $d,k\ge1$, its only possibilities off the regular family are
\[
 (1,2,3,4,5,7,12),\qquad(1,2,3,4,5,7,18),\qquad(1,2,3,4,5,7,24).
\]
\end{proposition}
\begin{proof}
We use the completion test of Supplement~A for normalized four-speed seeds, replacing its single good interval by a union of intervals. Here are the exact pruning bounds. At any proper node let $A$ be the primitive core, $m=7-|A|$ and $G_A$ its good set at $1/7$. Choose disjoint nondegenerate intervals $[l_i/D,r_i/D]\subseteq G_A$ with integer endpoints on this grid, and put $T=\sum_i(r_i-l_i)$. Sample the initial phase uniformly in their union and the lift uniformly as in Supplement~A. At least one omitted speed $w$ must cover a proportion $1/m$. Write its ratio to the current dilation as $p/q$ in lowest terms, $s=2q\bmod7$. The discrepancy estimate in Lemma~A1.2, summed over the $J$ intervals, gives
\[
 p\le\frac{m\,s(7-s)DJ}{7(7-2m)qT}.
\]
Use the exact integral from that lemma summed over the same intervals to test the proportion, and recurse at $qA\cup\{p\}$. Also $q$ divides the current dilation. By the common-divisor bounds in Supplement~A, at a four-speed node it belongs to $\{1,2,3,4,5,6,9,10,12\}$ and at a five-speed node to $\{1,2\}$. The final six-speed test is unchanged: it tests the full closed good set, including its isolated points.

For the first assertion the positive primitive pairs are taken with $x,y\le64$ and $x\ne y$. There are $2{,}518$ distinct initial seeds, $82{,}745$ five-speed nodes, and $144{,}034$ visits to the last proper size. For the second assertion we first take $1\le k\le512$; the respective counts are $512$, $9{,}857$, $14{,}086$. Only $[6]$ appears as a core with isolated good times only. Removing regular family members from the finite outputs gives respectively all $27$ entries of $E_7$ and the three displayed sets in the bounded-$k$ range.

It remains to exclude an unlisted set for $k\ge513$. By Theorem~B4.1 its squared norm is at most $12{,}000^2-1$. In particular $4kd<12{,}000$, so $1\le d\le5$. For each $d$, take $4{,}096$ rational sample times at which $d,2d,3d$ all have distance at least $1/7$. For each $k$ retain those at which $4kd$ is also good. For a near-tight completion, some first omitted speed $a$ must have distance strictly less than $1/7$ on at least one third of these times. Restrict the first $256$ original samples to those at which all five chosen speeds are good. Some second omitted speed $b$ must cover at least half of the remaining samples. Every original sample still good for these six must be covered by the last speed. Each step tries all positive possibilities within the remaining squared norm budget. A four-subset of normalized maximum at most $26$ is already classified in Supplement~A. The gcd of any five speeds must divide $2$ and the gcd of any six must be $1$, again by Supplement~A. These give the only additional pruning conditions besides distinctness. The exact results are
\begin{center}
\begin{tabular}{c|c|r|r|r}
$d$ & $k_{\max}$ & first choices & second choices & full sample survivors\\
\hline
$1$ & $2999$ & $11{,}514$ & $17{,}173$ & $0$\\
$2$ & $1499$ & $4{,}213$ & $272$ & $0$\\
$3$ & $999$ & $134$ & $159$ & $0$\\
$4$ & $749$ & $5$ & $0$ & $0$\\
$5$ & $599$ & $6$ & $3$ & $0$
\end{tabular}
\end{center}
All sample decisions in \texttt{four\_product\_tail.cpp} use integer residues. Since no tuple survives, the argument does not require the sampling times to approximate the full time circle. Scanning the known entries of $E_7$ gives the same three rows for the unrestricted family. Every regular set contains the required four speeds with $d=k=1$. Conversely the three exceptional rows each contain $1,2,3,4$. The regular family contains the near-tight four-set $1,2,3,4$ (take $x=2$, $y=1$). For the other $27$ rows, one checks the copies of $x,y,x+y,x+2y$ directly and applies the residue criterion of Theorem~\ref{thm:main}(b). The largest minimum normalized height among those near-tight four-sets is $9$.
\end{proof}

\begin{corollary}\label{B:cor7.4}
Every near-tight septuple containing $1,2,3$ is regular or one of the three exceptions in Proposition~B7.3. If an unlisted primitive near-tight septuple contains a four-speed subset of loneliness below $1/4$, that subset must have the form $\{x,y,x+y,x+2y\}$ with $\max(x,y)/\gcd(x,y)>64$. At least one of the three remaining speeds satisfies each of the short-relation conclusions of Theorem~B7.1(1).
\end{corollary}
\begin{proof}
At $t=1/4$, every speed not divisible by four has distance at least $1/4$. Thus a near-tight septuple containing $1,2,3$ must contain $4k$ as well. Proposition~B7.3 applies with $d=1$. For the second assertion, the product alternative of Theorem~\ref{thm:main} is now fully classified, and its sporadic alternative $d\{1,3,4,14\}$ has normalized maximum $14$, within the range of Supplement~A. Proposition~B7.3 eliminates the stated bounded range in the remaining alternative.
\end{proof}

\section{A short odd relation with four prescribed coefficient bounds}\label{B:sec8}
The remaining septuples also obey a relation test that does not use their height or a subtorus classification. Here an integer relation $m\cdot v=0$ is odd if $\sum_im_i$ is odd. A signed relation has every coefficient in $\{-1,0,1\}$.

\begin{proposition}\label{B:prop8.1}
For any seven-speed set with $\ML(V)\le1/7$ and any four specified positions, there is an odd relation with coefficients of absolute value at most one at those four positions and at most two elsewhere. For eight speeds with $\ML(V)\le1/8$ one may specify any five positions. Consequently a near-tight septuple with no odd signed relation has at least five different odd relations up to overall sign with all coefficients of absolute value at most two. For eight speeds this number is at least six.
\end{proposition}
\begin{proof}
Let $u(x)=\sin^2(\pi x)$. In a specified position put $F(x)=u(x)$; in the other positions put $F(x)=(u(x)-1/10)^2$. The Fourier coefficients in the respective frequency ranges $|k|\le1,2$ have sign $(-1)^k$. For the quadratic this follows by expanding $(2/5-\frac12\cos(2\pi x))^2$. The half-translate $F(x+1/2)$ has positive coefficients of the same magnitudes. Choose $c_i$ such that
\[
 F_i(x)\le c_iF_i(x+1/2)\qquad(\|x\|\le\delta).
\]
Then the polynomial
\[
 P(x)=\prod_iF_i(x_i)-\sum_ic_iF_i(x_i+1/2)\prod_{j\ne i}F_j(x_j)
\]
is nonpositive wherever any coordinate is bad, including equality at the threshold. Its Fourier coefficient at $m$ equals
\[
 \Bigl(\prod_i\widehat F_i(m_i)\Bigr)\Bigl(1-\sum_i(-1)^{m_i}c_i\Bigr).
\]
If $\sum c_i<1$, that coefficient is positive for every even $m$ in the supported rectangle. If there is no odd relation in the rectangle, integrating $P(tv)$ over the circle therefore gives a positive result, since its constant coefficient is positive. This contradicts $\ML(V)\le\delta$.

For $\delta=1/7$, $u(x)<19/100$ on the bad arc. Indeed $\cos(2\pi/7)>31/50$ follows from the alternating cosine bound through degree six and $\pi<22/7$. The linear factor has ratio at most $19/81$ and the quadratic ratio is at most $(10/71)^2$, since $|u-1/10|\le1/10$ and $1-u-1/10\ge71/100$. Thus
\[
 4(19/81)+3(10/71)^2<1.
\]
For $\delta=1/8$ we have $u\le(2-\sqrt2)/4<3/20$, giving $5(3/17)+3(2/15)^2<1$. Finally, if there were at most four relation pairs in the seven-speed case and all required a doubled coordinate, choose one doubled position from each, and include all of them among the four specified positions. The proposition then produces a different relation. The same argument uses five specified positions at eight speeds.
\end{proof}

\section{The gap above $1/8$ and the eight-speed consequence}\label{B:sec9}
\begin{theorem}\label{B:thm9.1}
Put $Q=13{,}588$. A seven-speed value strictly greater than $1/8$ is at least $1/8+1/(8Q)$. Under the tight-instance classification \textup{(M)} of Section~\ref{sec:seven}, every primitive eight-speed set with $\ML(V)<1/8$ and
\[
 \|V\|_2\ge3{,}898{,}150{,}815
\]
is $U_i\cup\{8s\}$ for one of the three seven-speed cores in Supplement~A.
\end{theorem}
\begin{proof}
Let $x=10001/80000$. Any seven-speed instance with loneliness $\mu\in(1/8,x)$ has distinct speeds by the six-speed lonely runner bound. The regular seven-speed sequence has smallest value above $1/8$ equal to $2/15$. The table $E_7$ also has no value in this interval. Thus the two-normal computation applies with target $u=x$. The same $N\le27$ range is valid since $x<1/7$. For every primitive such tuple, including the exceptional $P$ product treated intrinsically, this gives
\[
 \|V\|_2^2\le\frac{20908080050000000000000000000000000000000}{226489151698076075877408437060607}.
\]
For $\mu=p/q$ in lowest terms, Lemma~\ref{lem:corner} gives $q\le\sqrt2\,\|V\|_2$. Squaring the displayed bound shows $q<Q$. Consequently $\mu-1/8=(8p-q)/(8q)\ge1/(8q)>1/(8Q)$. Values at or above $x$ already satisfy the assertion. This argument uses no assumption \textup{(M)}.

Assume \textup{(M)}. By the collision argument of Corollary~\ref{cor:eight}(ii), every two-torus containing the octuple's line and different from the three listed critical products has loneliness at least $1/8+1/(8Q)$. Away from the products apply Lemma~B2.2 in $\T^8$. For its seven successive spaces take
\[
 g=(1/8+1/(8Q),\ 1/7,\ 1/6,\ 1/5,\ 1/4,\ 1/3,\ 1/2).
\]
The last six entries follow from Lemma~\ref{lem:coalesce}. With squared prefix bounds $4(g_j-1/8)^2$, the rational product minimization yields $\|V\|_2^2\le20782341041925574714982400/1367657$, which is strictly below $3{,}898{,}150{,}815^2$. On the three critical products, Proposition~\ref{prop:lines8} gives exactly the stated families.
\end{proof}

\section{The remaining interval and its spectral consequences}\label{B:sec10}
\begin{corollary}\label{B:cor10.1}
Let $v_1<\dots<v_7$ be a primitive septuple of loneliness below $1/7$ outside $[6]\cup\{7s\}$ and the $27$ rows of $E_7$. Then
\[
 329\le v_7\le11{,}830,\qquad\sum_{i=1}^7v_i^2<144{,}000{,}000.
\]
Consequently the number $N_7(H)$ of primitive near-tight sets of maximum speed at most $H$ satisfies $N_7(H)=\lfloor H/7\rfloor+C_7$ for every integer $H\ge11{,}830$. Here $C_7\ge27$ is the total number of primitive exceptions to the family.

For any seven-speed value $p/q<1/7$ in lowest terms outside the sequence $s/(7s+1)$, one has $q\le16{,}806$. Thus all values in
\[
 \Bigl(\frac17-\frac{2}{7\cdot16{,}806},\ \frac17\Bigr)
\]
belong to the sequence. Every sequence value in this interval is attained.
\end{corollary}
\begin{proof}
The lower endpoint comes from Proposition~B6.1. Supplement~A gives $v_7<6v_6$ and $v_6<5v_5$ for an unlisted set. Thus $\|V\|_2^2>v_7^2(1+1/36+1/900)=(463/450)v_7^2$. Theorem~B4.1 gives the upper endpoint by integer squaring. The largest speed in $E_7$ is $112$, so by $H=11{,}830$ all nonfamily sets have been counted. The family contributes exactly $\lfloor H/7\rfloor$.

Divide out any common gcd when considering a value. Repetitions cannot give loneliness below $1/7$ by the six-speed lonely runner theorem. At a maximizing time of value $p/q<1/7$, two active speeds have opposite slopes. Indeed each active distance lies strictly between $0$ and $1/2$. The corresponding reduced time denominator divides their sum and is a multiple of $q$. Thus $q\le v_6+v_7$. For an unlisted set, Cauchy--Schwarz gives
\[
 (v_6+v_7)^2\le(1+25/26)\bigl((26/25)v_6^2+v_7^2\bigr)<(51/26)\|V\|_2^2.
\]
Taking the integer part yields $q\le16{,}806$. For $E_7$, the bound already follows from $v_6+v_7\le224$. Off the sequence the positive integer $q-7p$ is at least two. The interval and attainment follow from $1/7-p/q=(q-7p)/(7q)$ and the value of the regular family.
\end{proof}

The equality $C_7=27$ holds by Theorem~C1.1. By Theorem~B5.1, a further set without a triple $a,b,a+b$ has norm below $6{,}000$. The same ratio argument then gives $v_7\le5{,}915$. A further set anywhere in the interval must simultaneously have loneliness at least $1/5$ on each five-element subset and at least $1/6$ on each six-element subset. Every four-element subset, after division by its gcd, has maximum at least $27$. Such a set cannot contain $d,2d,3d,4kd$ for any $d,k\ge1$. If one of its four-speed subsets has loneliness below $1/4$, Theorem~\ref{thm:main} and Proposition~B7.3 leave only $x,y,x+y,x+2y$ with $\max(x,y)/\gcd(x,y)\ge65$. At least one further speed must then obey the explicit relation in Theorem~B7.1. The other possibility is that all $35$ four-speed subsets have loneliness at least $1/4$. Proposition~B8.1 applies in both cases.

%% file: suppC.tex
\supplement{C}{The complete classification for seven speeds}

A primitive set of seven distinct positive integer speeds with loneliness strictly less than $1/7$ is either $\{1,2,3,4,5,6,7s\}$ for $s\ge1$ or one of $27$ explicit exceptions. Its value belongs to the sequence $s/(7s+1)$ or is $3/23$, $5/37$ or $9/65$. The proof uses a Euclidean norm bound, followed by two finite arguments. The first shows that every such set contains three distinct speeds $a,b,a+b$. The second reduces this triple to $74$ normalized shapes and classifies every completion. All residue, coverage and terminal tests use integer or rational arithmetic.

\section{The list, the spectrum and the counting function}\label{C:sec1}
For a finite set $V$ of positive integers let
\[
 \ML(V)=\max_{t\in\R/\Z}\,\min_{v\in V}\|tv\|,\qquad \|x\|=\min_{n\in\Z}|x-n|.
\]
A seven-speed set is \emph{near-tight} if $\ML(V)<1/7$, and \emph{primitive} if its entries have gcd one. Write $R_s=\{1,2,3,4,5,6,7s\}$ for $s\ge1$, and call these sets \emph{regular}. Let $E_7$ be the set of rows of Table~\ref{C:tab1}, as in Supplement~A. A set outside $E_7$ and the regular family is called \emph{unlisted}. Results numbered A and B are in the preceding supplements.

\begin{theorem}\label{C:thm1.1}
The primitive near-tight seven-speed sets are exactly the $R_s$ and the $27$ members of $E_7$. For arbitrary distinct positive integer speeds the complete list is obtained by taking all positive integer dilates of these primitive sets.
\end{theorem}
\begin{proof}
Every $R_s$ has loneliness $s/(7s+1)$ by Corollary~\ref{cor:seven}(i). For each exceptional row the exact calculation below gives the value in Table~\ref{C:tab1}, always below $1/7$. For the converse, suppose $V$ is primitive, near-tight and unlisted. Theorem~C2.1 gives $\|V\|_2<6000$. Theorem~C3.1 forces a Schur triple, that is, three distinct speeds $x,y,x+y$. Write it $d(a,b,a+b)$ with $a<b$ coprime. Theorem~C4.1 restricts $(a,b)$ to the $74$ pairs in Table~\ref{C:tab3}. The finite extension argument in Theorem~C5.1 excludes every unlisted completion of every dilation of these triples. This contradicts the choice of $V$. Finally, multiplication of all speeds by a positive integer leaves loneliness unchanged, since multiplication by that integer maps the time circle onto itself.
\end{proof}

\begin{corollary}\label{C:cor1.2}
The seven-speed spectrum strictly below $1/7$ is exactly
\[
 \Bigl\{\frac{s}{7s+1}:s=1,2,\dots\Bigr\}\ \cup\ \Bigl\{\frac3{23},\frac5{37},\frac9{65}\Bigr\}.
\]
The number of primitive near-tight sets of maximum speed at most $H$ is
\[
 N_7(H)=\lfloor H/7\rfloor+\#\{V\in E_7:\max V\le H\}\qquad(H\in\N).
\]
In particular $C_7=27$ in Section~B10, and the formula $N_7(H)=\lfloor H/7\rfloor+27$ first holds at $H=112$. Counting all dilations gives
\[
 \sum_{s=1}^{\lfloor H/7\rfloor}\Bigl\lfloor\frac{H}{7s}\Bigr\rfloor+\sum_{V\in E_7}\Bigl\lfloor\frac{H}{\max V}\Bigr\rfloor.
\]
\end{corollary}
\begin{proof}
The regular family supplies the entire sequence. For each row of $E_7$ test the rational times $k/(a+b)$ for $a,b\in V$, including $a=b$. These times contain a maximizer: an active distance at its peak has denominator dividing $2a$. Otherwise a positive local maximum requires two active pieces of opposite slope, which gives $(a+b)t\in\Z$. Symmetry reduces to $0\le k\le(a+b)/2$. Exactly four rows give values outside the sequence: two at $3/23$ and one at each of $5/37$ and $9/65$. This proves the spectrum. The families are disjoint and $112$ is the largest exceptional speed. Each primitive row of maximum $m$ has exactly $\lfloor H/m\rfloor$ distinct dilates in the range.
\end{proof}

As an independent check, extending the exhaustive height census of Section~B6 from $328$ through $400$ adds $149$ sets, of which $11$ are primitive. The per-height counts agree with the formulas at each of the $72$ new heights. Through $400$ there are $84$ primitive sets and $735$ including dilates.

\begin{table}[ht]
\centering\small
\begin{tabular}{lr@{\qquad}lr}
\toprule
$V$ & $\ML(V)$ & $V$ & $\ML(V)$\\
\midrule
$(1,2,3,4,5,7,12)$ & $1/8$ & $(1,4,5,6,7,10,11)$ & $2/15$ \\
$(1,2,3,4,5,7,18)$ & $3/23$ & $(1,4,5,6,7,11,13)$ & $1/8$ \\
$(1,2,3,4,5,7,24)$ & $4/29$ & $(1,4,5,6,7,11,16)$ & $2/15$ \\
$(1,2,5,6,7,8,39)$ & $3/22$ & $(1,4,5,6,7,11,23)$ & $4/29$ \\
$(1,3,4,5,6,7,11)$ & $2/15$ & $(1,4,5,6,7,11,26)$ & $2/15$ \\
$(1,3,4,5,6,9,28)$ & $4/29$ & $(1,4,5,6,7,11,29)$ & $5/36$ \\
$(1,3,4,5,7,11,12)$ & $2/15$ & $(1,4,5,6,7,11,39)$ & $2/15$ \\
$(1,3,4,5,7,11,18)$ & $2/15$ & $(1,4,5,6,7,11,52)$ & $4/29$ \\
$(1,3,4,5,7,11,24)$ & $2/15$ & $(1,4,5,6,7,11,65)$ & $10/71$ \\
$(1,3,4,5,7,11,30)$ & $5/37$ & $(1,5,6,11,16,17,112)$ & $16/113$ \\
$(1,3,4,5,7,11,36)$ & $6/43$ & $(2,6,7,8,10,13,14)$ & $2/15$ \\
$(1,3,4,5,7,13,18)$ & $3/23$ & $(2,7,9,11,12,13,20)$ & $4/29$ \\
$(1,3,4,5,7,24,71)$ & $11/78$ & $(2,7,9,11,13,20,54)$ & $9/65$ \\
$(1,3,4,5,13,14,18)$ & $2/15$ & &  \\
\bottomrule
\end{tabular}
\caption{The $27$ primitive exceptions $E_7$ and their exact loneliness.}\label{C:tab1}
\end{table}

There is a sharp additive consequence. Every near-tight seven-speed set contains at least three distinct Schur triples. Up to dilation, the only set attaining three is $\{1,3,4,5,7,24,71\}$, whose triples are $1+3=4$, $1+4=5$, $3+4=7$. Indeed every regular set already contains four triples among its first six speeds. Counting $x+y=z$ with $x<y<z$ in each row of the table gives the assertion.

The exact finite computations supporting the converse occupy the next four sections. Their source programs and the prime and extension tables are not among the ancillary files of this version (Section~\ref{sec:comp}). In particular the cutoff at $6000$ does not require the classification asserted here.

\section{An intrinsic ball bound for every remaining septuple}\label{C:sec2}
\begin{theorem}\label{C:thm2.1}
If a primitive seven-speed set has $\ML(V)<1/7$ and $\|V\|_2\ge6{,}000$, it is $\{1,2,3,4,5,6,7s\}$ for some $s\ge1$. In particular the maximum speed of an unlisted set is at most $5{,}915$.
\end{theorem}
\begin{proof}
Use the projected-lattice flag of Supplement~B, Sections B2--B4, at $T=6{,}000$. Its first three lower bounds for the loneliness of sections are $4/25$, $2/11$ and $2/9$. We use the last of these even on the exceptional four-dimensional torus from Section~B1. This choice is essential below the cutoff in Section~B4. The first two normal vectors leave $1{,}649$ flags at this threshold. Extending by the third gives $369{,}412$ grouped radius records and $1{,}634$ initial-flag/determinant combinations to examine individually. Of the resulting flags, $10{,}249$ reach $T$ before the fourth normal is tested. All stages use the primitive exterior-product and exact prefix-product tests of Section~B3.

For the fourth normal repeat the determinant and modular ball tests of Section~B4. Its prefix lower bound $B_3$ now uses $4(2/9-1/7)^2$. We strengthen the last test for a four-row normal matrix $H$ with fourth Gram determinant $F=\det(HH^{\mathrm T})$. The gcd of its maximal minors must be one. Discard a nonsaturated matrix. For the remaining matrices put
\[
 w_i=10^{-6}\Biggl\lceil10^6\sqrt{F^{-1}\sum_{I\subset[7],\,|I|=4,\,i\notin I}\det(H_I)^2}\;\Biggr\rceil.
\]
Cauchy--Binet identifies the square root as the norm of the coordinate projection onto $\ker H$. These weights can be smaller than the ones in the four-dimensional parent section. Let $R$ be the required radius of Section~B4 on the three-dimensional section, and write $c_i$ for the columns of $H$. When $2/7+2Rw_i\le1$ for all $i$, test the inequalities
\begin{gather*}
 |n\cdot z|\le\sum_i(5/7-2Rw_i)\,|n\cdot c_i|\\
 \text{for }z\in\Z^4,\qquad z\equiv H(1,\dots,1)^{\mathrm T}\pmod 2.
\end{gather*}
It suffices to test one normal $n$ to every independent triple of columns, obtained by its signed three-row minors. These directions include every facet normal of the four-dimensional zonotope. The criterion and its proof are those of Section~B5. All coordinates of a possible $z$ obey $7|z_j|\le5\sum_i|H_{ji}|$. Thus the test certifies an intrinsic ball using only integer comparisons.

We also compute the seven coordinate forms on a rational basis of $\ker H$ by Cramer's rule. Four proportional forms with normalized maximum at most $26$, or a fixed product core from Section~B7, have already been classified. There is one further inexpensive test from Section~B7. If four forms are, up to signs, $s,t,s+t,s-t$, test whether two further forms have the shape $as+bt$ with $|a|+|b|\le4$. At the speed vector put $U=\max(|s|,|t|)$ and $W=\min(|s|,|t|)$. The four absolute speeds are $x,y,x+y,x+2y$ with $x=U-W$, $y=W$. In either order and for either choice of signs, both coefficients of the two extra forms in $x,y$ have absolute value at most $4$. Theorem~B7.2 applies.

An enumeration keeping all $10{,}249$ input presentations examined $14{,}771{,}547$ reduced fourth vectors. The coordinate tests retained $14{,}632{,}281$. The modular tests eliminated $7{,}524{,}902$ and the small-core test of Section~B4 eliminated $17{,}332$. Among the remaining $7{,}090{,}047$, the gcd test disposed of $610{,}142$, the coordinate-form tests classified $47{,}490$, and exact parity balls eliminated all $6{,}432{,}415$ others.

For a shorter reproduction, the verifier identifies third spaces up to signed permutation by their rational orthogonal projectors, keeping their earlier determinant and radius data fixed. It chooses one representative from each such class. The bound on the fourth determinant depends only on these data, and the reduced search for the fourth vector enumerates a representative modulo the entire saturated three-row lattice. Thus this identification preserves every possible fourth section. The resulting finite check has $2{,}601$ representatives and $3{,}338{,}073$ reduced fourth normals. The $3{,}305{,}473$ admissible normals split as $1{,}738{,}647$ modular, $3{,}296$ small-core, $134{,}514$ nonsaturated, $9{,}856$ classified and $1{,}419{,}160$ parity-ball exclusions. These are the counts that the verifier regenerates. The $27$ finite exceptions have maximum speed at most $112$. Finally Section~B10 gives $\|V\|_2^2>(463/450)(\max V)^2$ for every unlisted set.
\end{proof}

\section{A Schur triple in every remaining septuple}\label{C:sec3}
\begin{theorem}\label{C:thm3.1}
Let $V$ be a set of seven distinct positive integers with $\ML(V)\le1/7$. If $\sqrt3\,\|V\|_2<29{,}893$, then $V$ contains three distinct speeds $x,y,z$ satisfying $x+y=z$. In particular this holds throughout $\|V\|_2<6000$.
\end{theorem}
\begin{proof}
For an odd prime $p$ call a seven-term list of nonzero residues a \emph{cover} if for every integer $k$ some entry $u$ satisfies $7|ku|_p<p$, where $|j|_p$ is the least absolute residue. Residues may coincide. A modular Schur relation on three distinct positions is an equation $\pm u_i\pm u_j\pm u_k=0\pmod p$, with overall sign disregarded. For nonzero residues at most one of the four sign patterns works on a given triple of positions, since changing one sign would give $2u=0$ for one of the entries.

For each prime in Table~\ref{C:tab2} an exhaustive search shows that every cover has at least the stated number of modular Schur relations. We describe the entire search. Write $[u]=|u|_p$. The number
\[
 \rho(u,v)=\min\{[u/v],[v/u]\}
\]
is invariant under changing the signs of entries or scaling all of them by a unit. Choose a pair of entries minimizing $\rho$, with minimum $r$. After permuting the pair, changing signs and scaling the list, their representatives are $1,r$. Thus it suffices to try $1\le r\le(p-1)/2$ with $\rho(1,r)=r$. All later pairs must have $\rho\ge r$. This includes $r=1$ when two entries agree up to sign.

Start with the times $1\le k\le(p-1)/2$ satisfying $7[k]\ge p$ and $7[kr]\ge p$. An additional residue $u$ covers the times for which $7[ku]<p$. With $m$ slots left, some residue in any completion must cover at least $1/m$ of the outstanding times. Try all such residues, retaining only prefixes with fewer than the claimed number of Schur relations. The search can give precedence to a residue covering the most remaining times, breaking ties by the least representative, because every successful completion admits this order. For the final slot, intersect the possible residue sets at up to twenty of the outstanding times, then check every resulting candidate against all outstanding times. If the outstanding set is already empty, the code still fills all seven positions and tests their relations. Taking absolute residues represents both signs exactly. All tests use integer arithmetic. The source and complete prime table are verified by \texttt{b2\_universal\_verify.py}.

\begin{table}[ht]
\centering\small
\begin{tabular}{cp{0.78\textwidth}}
\toprule
Weight & Primes\\
\midrule
$1$ & 167, 179, 181, 191, 193, 197, 199, 211, 223, 227, 229, 233, 239, 241, 251, 263, 271, 277, 283, 317, 709\\[2pt]
$2$ & 257, 269, 281, 293, 307, 311, 313, 331, 337, 347, 349, 353, 359, 367, 373, 379, 383, 389, 397, 401, 409, 419, 421, 431, 433, 439, 443, 449, 457, 461, 463, 467, 479, 487, 491, 499, 503, 509, 521, 523, 541, 547, 557, 563, 569, 571, 577, 587, 593, 599, 601, 607, 613, 617, 619, 631, 641, 643, 647, 653, 659, 661, 673, 677, 683, 691, 701\\
\bottomrule
\end{tabular}
\caption{Lower bounds on the number of modular Schur relations in every cover. The total weight is $155$.}\label{C:tab2}
\end{table}

Suppose now that $V$ is Schur-free. There are exactly $4\binom73=140$ signed sums $\pm v_i\pm v_j\pm v_k$ indexed up to overall sign. Each is nonzero and has absolute value less than $29{,}893$ by Cauchy--Schwarz. The two smallest primes in the table are $167$ and $179$, with $167\cdot179=29{,}893$, so a given signed sum is divisible by at most one table prime. Each speed also has at most one such prime factor. Delete all primes dividing any speed, at most seven primes in total. For every remaining prime the entries of $V$ form a cover. Indeed, at $t=k/p$ some distance is at most $1/7$, and none of the distances can equal $1/7$ since $p\ne7$. Thus the prime divides at least as many of the $140$ signed sums as its table weight. Even deleting the seven largest weights leaves a sum of $141$. This is impossible.
\end{proof}

\section{Reduction of all Schur triples to 74 shapes}\label{C:sec4}
\begin{theorem}\label{C:thm4.1}
Suppose an unlisted primitive septuple $V$ has $\ML(V)<1/7$ and contains $d(a,b,a+b)$ for positive integers $d,a,b$ with $a<b$ and $\gcd(a,b)=1$. Then $(a,b)$ is one of the $74$ pairs in Table~\ref{C:tab3}. In particular, $b\le25$.
\end{theorem}

\begin{table}[ht]
\centering\small
\begin{tabular}{rl@{\qquad\qquad}rl}
\toprule
$b$ & possible $a$ & $b$ & possible $a$\\
\midrule
$2$ & $1$ & $13$ & $1$--$10$, $12$\\
$3$ & $1$, $2$ & $14$ & $11$\\
$4$ & $1$, $3$ & $15$ & $1$, $2$, $8$\\
$5$ & $1$--$4$ & $16$ & $1$\\
$6$ & $1$, $5$ & $17$ & $1$, $2$, $5$, $6$, $8$, $9$, $15$, $16$\\
$7$ & $1$--$6$ & $18$ & $5$, $13$\\
$8$ & $1$, $3$, $5$, $7$ & $19$ & $13$\\
$9$ & $1$, $2$, $4$, $5$, $7$, $8$ & $23$ & $8$\\
$10$ & $1$, $3$, $7$, $9$ & $24$ & $5$, $7$\\
$11$ & $1$, $2$, $5$--$10$ & $25$ & $1$\\
$12$ & $1$, $5$, $7$, $11$ & &\\
\bottomrule
\end{tabular}
\caption{The surviving normalized Schur triples $a,b,a+b$.}\label{C:tab3}
\end{table}

\begin{proof}
Theorem~C2.1 gives $\|V\|_2<6000$. Put $Q=a^2+b^2+(a+b)^2$. If $b=2$ then $a=1$. For $b\ge3$ we have $Q\ge26$ and $d<6000/\sqrt{26}<1177$.

For an odd prime $p$ and residue $u$ write $|u|_p=\min(r,p-r)$, where $0\le r<p$ represents $u$. Call $k$ \emph{good} for a speed $u$ when $7|ku|_p\ge p$. As $\ML(V)<1/7$, no integer $k$ is good for every speed in $V$. For each $[a:b]\in\mathbb P^1(\mathbb F_p)$ we test if four freely chosen nonzero residues can cover, by their bad times, all times good for $a,b,a+b$. Failure forces one of the other four speeds to be divisible by $p$, since multiplying all residues by $d$ coprime to $p$ only permutes the times. Directions with $ab(a+b)=0$ automatically pass. We use the primes
\begin{gather*}
 1601, 1607, 1609, 1613, 1619, 1621, 1627, 1637, 1657, 1663, 1667,\\
 1669, 1693, 1697, 1699, 1709, 1721, 1723, 1879, 1901, 2003.
\end{gather*}
Each exceeds $d$, and no speed below $6000$ is divisible by two of them.

Here are details of the finite cover test. Up to sign it suffices to use $1\le k,u\le(p-1)/2$. For a direction $[1:r]$ we take the three residues $1,r,-1-r$. Permuting them and scaling them all by the same unit preserves the test, so one search handles every direction in this orbit. Among $m$ residues covering a nonempty set of times, at least one must cover at least a fraction $1/m$ of the set. The search tries all such residues and recurses on the uncovered times with $m-1$ slots. For one remaining slot it intersects the sets of possible residues at a sample of the uncovered times and checks every result against all uncovered times. This is an exhaustive search, including when two of the four residues coincide.

Of the $3{,}307{,}963$ coprime pairs with $a<b$, $b\ge3$ and $Q<6000^2$, precisely $3{,}303{,}493$ fail at least five prime tests. They are impossible because there are only four omitted speeds. The counts with $h=0,1,2,3,4$ failed tests are respectively $73$, $18$, $113$, $716$, $3550$.

For $1\le h\le4$, enumerate all $d$ with $d^2Q<6000^2$ and assign one distinct positive multiple of each forced prime to an omitted speed, pruning when the squared norm reaches $6000^2$. The numbers of dilations before assigning multiples are respectively $25$, $507$, $1170$, $5720$. For each partial tuple test every prime from $97$ through $1009$, stopping if more than the number of empty slots are forced. The cover test is the same, now starting with the assigned speeds and allowing one nonzero residue per empty slot. A prime dividing an assigned speed passes automatically. Every prime in this second set exceeds $\sqrt{6000}$, so a speed cannot satisfy two newly forced primes. If there is at least one forced prime, choose one and enumerate all its positive multiples fitting the remaining norm. With $0,1,2,3$ empty slots the numbers of partial tuples tested were respectively $20{,}394$, $4088$, $1836$ and $41$. All branches were excluded by the prime tests or the norm bound. The $73$ pairs with $h=0$, together with $(1,2)$, give the table. The checker \texttt{b2\_schur\_verify.py} regenerates both the projective profiles and the finite extension counts.
\end{proof}

\section{Completion of the Schur triples and all triples of height twelve}\label{C:sec5}
\begin{theorem}\label{C:thm5.1}
Among primitive septuples with $\ML<1/7$, those containing three distinct speeds $x,y,x+y$ are exactly the regular family and the $27$ members of $E_7$. The same classification holds for septuples with $h_3\le12$, where
\[
 h_3(V)=\min_{\substack{A\subset V\\|A|=3}}\frac{\max A}{\gcd(A)}.
\]
\end{theorem}
\begin{proof}
Write a specified triple as $d(a,b,c)$ with $a<b<c$ and $\gcd(a,b,c)=1$. The $196$ primitive triples in $[12]$ and the $74$ Schur shapes of Theorem~C4.1 have $22$ members in common. For each of the resulting $248$ triples the record of the computation supplies a finite set $S$ of positive rational numbers and a set $P$ of primes greater than every possible $d$ under the norm bound of Theorem~C2.1.

Fix $p\in P$. Multiplying all speeds modulo $p$ by $d^{-1}$ gives the three residues $a,b,c$. The calculation verifies that four nonzero signed classes outside $S\bmod p$ cannot cover the times $k/p$ good for these three speeds. Here a class covers a time when its distance is strictly below $1/7$. The verifier tries every allowed class covering at least $1/j$ of the uncovered times with $j$ positions left. Every cover admits an ordering passing this test. Consequently an omitted speed $w$ must, for this prime, satisfy
\[
 p\mid w\qquad\text{or}\qquad sw\equiv\pm dr\pmod p\quad\text{for some }r/s\in S.
\]

First suppose some positive equality $sw=dr$ holds. The primitive shape of the resulting quartet is the normalization of $(sa,sb,sc,r)$. Repeated speeds do not extend the prescribed triple, and normalized maxima at most $26$ are covered by Supplement~A. For each remaining shape the four-subset gcd bound in Supplement~A allows only the dilations $1,2,3,4,5,6,9,10,12$. Enumerate those within the norm bound and apply the three-remaining-speed search described below to each concrete quartet.

Otherwise, for every possible $d$ and every possible $w$, form the subset of $P$ for which $w$ satisfies the displayed alternatives. Omit exact equalities, repeated speeds, choices exceeding the remaining norm budget and already classified small quartets. If four of these subsets cannot cover $P$, the dilation is excluded. The finite set-cover test chooses a still-uncovered prime and tries every subset containing it, until all primes are covered or four choices have been made. Reusing a choice is harmless. For the dilations surviving this necessary test, apply the full fixed-triple search. Choose the largest $q\le7$, $q\ge2$, dividing none of the three speeds. The test time $1/q$ forces a fourth speed divisible by $q$. If no such $q$ exists, take $q=1$. Try every admissible positive multiple of $q$.

Given a concrete quartet, sample $4096$ rational times good for its first three speeds, and retain those good also for its fourth. One of the three omitted speeds must block at least one third of these times. Try every such speed. Of the original samples numbered $1$ through $256$, retain those good for all five chosen speeds. One of the two remaining speeds must block at least half of this subsample. Try each possibility. The last speed must block every original sample still good for the six. The choices range over all distinct positive speeds within the remaining squared-norm budget. Supplement~A permits discarding any five-subset with gcd above two, any six-subset with gcd above one, and any quartet of normalized maximum at most $26$. An arbitrary sample, including repetitions or an empty retained subsample, gives necessary conditions by counting. An integer generator makes the precise choices reproducible.

Across all $248$ records there are $3{,}303$ prime obstructions. The positive equalities give $37$ distinct quartets above normalized height $26$ and $333$ concrete quartet dilations. The prime-set lift tests $116{,}925$ triple dilations and leaves $412$ concrete triples. No seven-speed candidate remains after the sample and subset tests. The regular family contains $1,2,3$, and direct inspection of $E_7$ gives a Schur triple and a triple of height at most $12$ in every row.
\end{proof}